\documentclass[12pt]{article}
\usepackage{graphicx}
\usepackage{latexsym}
\usepackage{amsmath}
\usepackage{verbatim}
\usepackage{amssymb}
\usepackage{mathrsfs}
\usepackage{amsmath,amssymb}
\usepackage{enumitem}
\DeclareMathAlphabet      {\mathbf}{OT1}{cmr}{bx}{n}
\DeclareFontFamily{OT1}{pzc}{}
\DeclareFontShape{OT1}{pzc}{m}{it}%
{<-> s * [1.15] pzcmi7t}{}
\DeclareMathAlphabet{\mathpzc}{OT1}{pzc}{m}{it}

\usepackage{ dsfont }
\usepackage{ mathrsfs }

\usepackage{amsbsy }
\usepackage{graphics}

\usepackage{amsthm}
\newtheorem{thm}{Theorem}
\newtheorem{theorem}{Theorem}[section]

\newtheorem{lemma}[theorem]{Lemma}
\newtheorem{corollary}[theorem]{Corollary}

\newtheorem{definition}[theorem]{Definition}

\newtheorem{prop}[theorem]{Proposition}

\newtheorem{remark}{Remark}

\usepackage{amsmath}
\newcommand\mathitem{\item\leavevmode\vspace*{-\dimexpr\baselineskip+\abovedisplayskip\relax}}
\usepackage{romannum}
\usepackage[toc,page,title,titletoc,header]{appendix}
\usepackage{hyperref}

\usepackage[all,cmtip]{xy}

\def\ack{\section*{Acknowledgements}%
  \addtocontents{toc}{\protect\vspace{6pt}}%
  \addcontentsline{toc}{section}{Acknowledgements}%
}

\usepackage{abstract}

\begin{document}\small
 \pagenumbering{arabic}
\title{On cobordism maps on periodic Floer homology}
\author{{\normalfont Guanheng Chen}}
\date{}

\maketitle
\begin{abstract}    % type your abstract below
In this article, we  investigate the cobordism maps on periodic Floer homology (PFH).   In the first part of the paper, we define  the cobordism maps on  PFH   via Seiberg Witten theory as well as  the isomorphism  between PFH and Seiberg Witten cohomology. Furthermore, we show that the maps satisfy the holomorphic curve axiom.   In the second part of the paper, we give an alternative  definition  of these maps by using holomorphic curve  method, provided that the symplectic cobordisms  are  Lefschetz fibration satisfying certain nice  conditions.  Under additionally certain monotonicity  assumptions, we show that  these two definitions   are equivalent.
\end{abstract}

\maketitle

%%%%%%%%%%%%%%%%%%%%   Start of main body of article

\section{Introduction }
Let $(\Sigma, \omega_{\Sigma})$ be a connected closed $2$-dimensional symplectic manifold   and $\phi: \Sigma \to \Sigma $ be a symplectomorphism. Let $Y$ be a mapping torus of $(\Sigma, \phi)$, i.e.,
\begin{eqnarray*}
 Y= \mathbb{R}\times \Sigma /(t+ 2 \pi, x) \sim (t, \phi(x)).
\end{eqnarray*}
The coordinate vector field $\partial_t$ and the volume form $\omega_{\Sigma}$    descend respectively to a vector field  $\tilde{\partial_t}$ and a   closed $2$-form $\omega_{\phi}$ on $Y$ under the natural projection $\mathbb{R} \times \Sigma \to Y$. The closed integral curves of $\tilde{\partial_t}$ are called periodic orbits.

In \cite{H1}, M. Hutchings proposes an  invariant $HP_*(Y, \omega_{\phi},  \Gamma, J, \Lambda_P)$ for the pair $(Y,  \omega_{\phi})$  which is so-called periodic Floer homology, abbreviated as PFH. Here $J$ is a generic almost complex structure, $\Gamma \in H_1(Y, \mathbb{Z})$, and $\Lambda_P$ is a local coefficient. Roughly speaking, the chain complex of  $HP(Y, \omega_{\phi}, \Gamma, J,  \Lambda_P)$ is generated by periodic orbits, and its differential is defined by counting  holomorphic currents with ECH index one.  PFH is   shown to be well defined by  Hutchings  and  Taubes in  \cite{HT1} and \cite{HT2}. PFH is a sister version of a more well-known invariant called embedded contact homology, abbreviated as ECH. In Tuabes' series of papers  \cite{Te1}, \cite{Te2}, \cite{Te3},   \cite{Te4} and \cite{Te5}, he shows that ECH is isomorphic to Seiberg Witten cohomology.
Likewise, periodic Floer homology is also isomorphic to a version of Seiberg Witten cohomology.

\begin{theorem}[Theorem 6.2 of \cite{LT}] \label{Thm0}
Given a fixed homology class $\Gamma \in H_1(Y, \mathbb{Z})$,  let $\Lambda_P$ be a $(c_{\Gamma}, [\omega])$-complete local coefficient for the periodic Floer homology. Let  $\Lambda_S$ denote a corresponding $(c_1(\mathfrak{s}_{\Gamma}), 2\pi r[\omega])$-complete local coefficient for Seiberg Witten cohomology in the sense of \cite{LT}. There  exists  an isomorphism between periodic Floer homology and a perturbed version of Seiberg Witten cohomology
\begin{equation*}
\mathcal{T}_{r*}: HP_*(Y, \omega_{\phi}, \Gamma, J, \Lambda_P) \to HM^{-*}(Y,  \mathfrak{s}_{\Gamma}, -\pi \varpi_r, J, \Lambda_S),
\end{equation*}
which reverses the relative grading.
\end{theorem}

In \cite{U}, M. Usher defines an invariant $HF_*(Y,  [\omega], \Gamma, \Lambda)$ for a tuple   $(Y, \pi, [\omega], \Gamma)$, where $\pi: Y \to S^1$ is a surface fibration over   circle, $[\omega] \in H^2(Y, \mathbb{R})$,  $\Gamma \in H_1(Y, \mathbb{Z})$ and $\Lambda$ is  a local coefficient. Additionally, he shows that  $HF$ is  a covariant functor  from  the fibered cobordism category (FCOB) to the category of modules over universal Novikov ring.  Roughly speaking, the object in FCOB is surface fibration over the circle and the morphism is Lefschetz fibration. %In our paper,  these   Lefschetz fibration are also called symplectic fibered cobordism.
The invariants  $HF_*(Y,  [\omega], \Gamma, \Lambda)$ and $HP_*(Y, \omega_{\phi}, \Gamma,  \Lambda_P)$ are respectively  three dimensional analogy of  Donaldson-Smith's invariant \cite{DS} and Taubes' Gromov invariant \cite{T2}, thus one might expect that these two homologies are isomorphism. Motivated by  \cite{U}, it is also expected to construct a  TQFT structure on  PFH %$HP_*(Y, \omega_\phi, \Gamma,  \Lambda_P)$
which is an analogy  to Usher's on $HF$. %$HF_*(Y,  [\omega], \Gamma, \Lambda)$.

To this end, we need to define the cobordism maps on PFH induced by Lefschetz fibration. The natural way to   define the cobordism maps is by counting  holomorphic curves with zero ECH index. However, as explained in Section 5.5 of \cite{H3}, the ECH index can be negative in general. Consequently,  we need to count broken holomorphic curves with zero ECH index.  The  moduli space of broken holomorphic curves has many components of various dimensions, so it is difficult to handle it.  To overcome the technical difficulties, Hutchings and Taubes  define the cobordism maps  on embedded contact homology by using the isomorphism ``HM=ECH''\cite{Te1} and Seiberg Witten theory \cite{HT}.

% like the case in embedded contact homology, there are several technical issues for defining the cobordism maps  by using  holomorphic curve method.  The technical issues for defining the cobordism maps on  ECH are explained in section 5.5 of \cite{H3}.
 Since PFH is a cousin  of embedded contact homology, it is natural to expect that the parallel  results hold for PFH. In some special cases, the ECH index is still possible to be nonnegative. In these cases, can we define the cobordism maps on PFH without Seiberg Witten theory? If we can, it is natural to ask weather  we can generalize  the isomorphism in  Theorem \ref{Thm0} to cobordism level.
 % it is equivalent to the cobordism maps on Seiberg Witten cohomology under the isomorphism  in Theorem \ref{Thm0}.

The purpose of this article is to explore the cobordism maps on PFH   and we answer  the   questions above for certain Lefschetz fibrations.    Firstly, we  follow the idea in \cite{HT}  to define the cobordism maps ${HP}_{sw}(X, \Omega_X, J, \Lambda_X)$ on PFH  via Seiberg Witten theory and Theorem \ref{Thm0}.  Since the Seiberg Witten theory works for any 4-manifold,  the maps ${HP}_{sw}(X, \Omega_X, J, \Lambda_X)$ are well-defined for general symplectic cobordisms, not limited to Lefschetz fibration.
These  maps  ${HP}_{sw}(X, \Omega_X, J, \Lambda_X)$ inherit the natural properties of  Seiberg Witten cobordism maps. But beyond that, the maps satisfy an additional  property called ``holomorphic curve  axiom", which relates the Seiberg Witten equations and holomorphic curves.
Secondly, we focus on the case that the symplectic cobordisms admit  Lefschetz fibration structures. Under certain technical assumptions $(\spadesuit)$, we give an alternative definition of the cobordism maps, denoted by ${HP}(X, \Omega_X, J, \Lambda_X)$. % we define these maps by counting of embedded holomorphic curves in a completion of the Lefschetz fibration.
These maps are defined by  counting embedded holomorphic curves with zero ECH index, by using the same  techniques in  \cite{HT1} and \cite{HT2}, and without using Seiberg Witten theory. Furthermore, in most cases,
%except for some exceptions,
we can deform the  symplectic  structure  on  the   Lefschetz fibration   such that it meets these  technical assumptions $(\spadesuit)$.
 Finally, under certain   monotonicity assumptions,  we show that    these two definitions   are equivalent. In other words, Theorem \ref{Thm0} holds in cobordism level  under the assumptions $(\spadesuit)$.  The proof of this part relies heavily on  Tuabes'  series of papers  \cite{Te1}, \cite{Te2}, \cite{Te3},  and \cite{Te4}.

Finally, we mention  that some  relevant  results on ECH side  have been achieved   recently.  In order to generalize Taubes' Gromov invariant to nearly symplectic manifolds, C.Gerig defines the cobordism maps on ECH for certain symplectic cobordisms  \cite{CG}. In addition, he compares the cobordism maps on ECH with the cobordism maps on Seiberg Witten cohomology \cite{CG2}. There are various similarities between the results here and C.Gerig's.  We will compare our results with Gerig's  in Remark \ref{r9}.    Besides,  J. Rooney defines the  cobordism maps on ECH by counting holomorphic curves for   general settings \cite{JR}.
In his definition, the cobordism maps are defined by counting holomorphic building and the holomorphic curves are allowed to have negative ECH index.
 %The holomorphic curves in his definition are allowed to have negative ECH index.
%\subsection{Relevant results}

%in the special case that the cobordism $(X, \pi_X)$ is a Lefschetz fibration over higher genus base,

\section{Preliminaries }
Before we state the main results, let us introduce  some essential definitions in this section.
\subsection{Surface fibration over $S^1$ and periodic orbits} \label{section11}
Let $\pi: Y \to S^1$ be a $3$-dimensional fibration over $S^1$.  Let  $\omega$ be a fiberwise non-degenerate closed $2$-form on $Y$, then there is a decomposition $TY=TY^{hor} \oplus TY^{vert}$ with respect to $\omega$, where $TY^{vert}=\ker \pi_*$ and $T_xY^{hor}= \{v \in T_xY: \omega(v,w)=0,\forall w \in T_xY^{vert}\}$.  Let $\partial_t$ be the coordinate vector field of $S^1$.  The  horizontal lift of $\partial_t$  is the unique vector field $R$ on $Y$ such that $R \in TY^{hor}$ and $\pi_*(R)=\partial_t$. Such a  vector field $R$ is called Reeb vector field.

Fix a base point $0 \in S^1$, let $\phi_t$ be the flow generated by $R$ starting at $\pi^{-1}(0)$ and $\phi_{\omega}=\phi_1$, then $Y$ can be identified with a mapping torus $\Sigma_{\phi_{\omega}}$ via the following diffeomorphism
\begin{eqnarray} \label{e47}
F: x \in \pi^{-1}(t) \to (t, \phi_t^{-1}(x)) \in \mathbb{R} \times \Sigma,
\end{eqnarray}
where $\Sigma=\pi^{-1}(0)$ and the mapping torus $\Sigma_{\phi_{\omega}}$  is defined by
\begin{eqnarray*}
 \Sigma_{\phi_{\omega}}= \mathbb{R}\times \Sigma /(t+ 2\pi, x) \sim (t, \phi_{\omega}(x)).
\end{eqnarray*}
Moreover, it is easy to check  that  $F_*(R) = \tilde{\partial_t}$ and $F^*\omega_{\phi_{\omega}}=\omega$.
The closed integral curves of $R$ are called periodic orbit. Let $\gamma$ be a periodic orbit of $R$, then the intersection pairing of $\gamma$ with the fiber, namely $d= [\gamma] \cdot [\Sigma]$, is called  degree or period of $\gamma$.  The linearization of the  flow  along $\gamma$ defines a symplectic linear map:
$$P_{\gamma}: ({\ker \pi_{*}}\vert_{\gamma(0)}, \omega \vert_{\gamma(0)}) \to ({\ker \pi_{*}}\vert_{\gamma(0)}, \omega \vert_{\gamma(0)}). $$
A periodic orbit $\gamma$ is  called non-degenerate if $1$ is not eigenvalue  of $P_{\gamma}$. For any  non-degenerate periodic orbit $\gamma$,  $P_{\gamma}$  has eigenvalues $\lambda$ and $\lambda^{-1}$  which are either real and positive, in which case $\gamma $ is called positive hyperbolic, or real and negative, in which case $\gamma$ is called negative hyperbolic,  or on the unit circle, in which case $\gamma$ is called elliptic.

\begin{definition} \label{def6}
Let $\pi: Y \to S^1$ be a surface  fibration over $S^1$. A $2$-form $\omega \in \Omega^2(Y)$ is called admissible if $d\omega=0$ and $\omega$ is fiberwise nondegenerate.  Given  $Q>0$,  an admissible $2$-form $\omega$ is called $Q$-admissible if all the  periodic orbits of $\phi_{\omega}$ with degree less than $Q$ are nondegenerate.
\end{definition}

\begin{definition}
Fix $\Gamma \in H_1(Y, \mathbb{Z})$, an orbit set with homology class $\Gamma$ is a finite set of pairs $\alpha =\{(\alpha_i ,m_i)\}$,  where $\{\alpha_i\}_i$ are distinct  irreducible simple periodic orbits and $\{m_i\}_i$ are positive integers. In addition, we require that $\sum\limits_i m_i[\alpha_i] = \Gamma$. An orbit set is called admissible if $m_i=1$ whenever $\alpha_i$ is hyperbolic. The set of admissible orbit sets with homology class $\Gamma$ is denoted by $\mathcal{ P}( Y, \omega, \Gamma)$.
\end{definition}

We will need the following definition when we state the main results and define the ``self-intersection number" (Definition \ref{def1}) of holomorphic curves later.     %It help us to rule out the holomorphic curve with negative ECH index.
\begin{definition}(Cf. Definition 4.1 of \cite{H4})
Let $Q>0$ and $\gamma$ be a simple elliptic orbit with degree $d \le Q$.
\begin{itemize}
\item
$\gamma$ is called $Q$-positive elliptic if the rotation number $\theta \in (0, \frac{d}{Q}) \mod 1$.
\item
$\gamma$ is called $Q$-negative elliptic if the rotation number $\theta \in ( -\frac{d}{Q},0) \mod 1$.
\end{itemize}
\end{definition}
%Keep in mind that $\mu_{\tau}(\gamma^q)=1$ whenever $\gamma$ is called $Q$-positive elliptic  and $q \le Q$, and  $\mu_{\tau}(\gamma^q)=-1$ whenever $\gamma$ is called $Q$-positive elliptic  and $q \le Q $, where $\mu_{\tau}(\gamma^q)$ is the Conley Zehnder index of $\gamma^q$.

\subsection{Symplectic cobordism and Lefschetz fibration  } \label{section10}
%In this subsection, let use introduce some basic definition about Lefschetz fibration which  we will be using.
In this subsection, we clarify the meaning of symplectic cobordism  between fibered 3-manifolds.
\begin{definition} \label{def9}
Let $(Y_{\pm}, \pi_{\pm}, \omega_{\pm})$ be a surface fibration over circle together with admissible 2-form. A symplectic   cobordism from $(Y_+, \pi_+, \omega_+)$ to $(Y_-, \pi_-, \omega_-)$ is a  symplectic manifold $(X, \Omega_X)$ such that  $\partial X =Y_+ \sqcup (-Y_-)$  and  $\Omega_X \vert_{Y_{\pm}}=\omega_{\pm}$. %satisfying the following conditions:
%\begin{enumerate}
%\item
%$\Omega_X$ is a symplectic  form over $X$ such that $\Omega_X \vert_{Y_{\pm}} =\omega_{\pm}$.
%\item
%There exist  vector fields $V_+$ and $V_-$ near $Y_+$ and $Y_-$ respectively, such that   $V_+$ points transversely outward at $Y_+$,    $V_-$ points transversely inward at $Y_-$,  and $V_{\pm} \lrcorner \Omega_X =\pi^*_{\pm} dt $.
%\end{enumerate}
\end{definition}

\begin{definition} \label{def11}
%An admissible $2$-form $\omega_X$ is called monotone if $[\omega_X]= \tau c_1(T\overline{X}) $ for some $\tau  \ne 0 \in \mathbb{R}$.
Let $(X, \Omega_X) $ be a symplectic cobordism between fibered 3-manifolds. Given $\Gamma_X \in H_2(X, \partial X, \mathbb{Z})$,  the pair $(\Gamma_X, \Omega_X)$ is called monotone if $$c_1(T{X}) + 2PD(\Gamma_X)=\tau[\Omega_X]$$  for some $\tau  \ne 0 \in \mathbb{R}$. If $[\Omega_X]= \tau c_1(T{X}) $ for some $\tau  \ne 0 \in \mathbb{R}$, we just say that $(X, \Omega_X)$ is monotone.
\end{definition}

When we define the cobordism maps on PFH by using holomorphic curve method, we  focus on the symplectic cobordisms  with fibration structures. (See Definition \ref{def10}.)
\begin{definition} (Cf. \cite{REG}) \label{def7}
Let $X$ be a compact, connected, oriented,  smooth 4-manifold and $B$ be a compact, connected oriented surface   possibly with boundary. A Lefschetz fibration is a map $\pi_X: X \to B$ with the following properties:
\begin{enumerate}
\item
$\pi_X^{-1}(\partial B) = \partial X$.

\item
Each critical point of $\pi_X$ lies in the interior of $X$.

\item
For each critical point of $\pi_X$, we can find a pair of orientation preserving complex coordinate
charts, one on $X$, centered at the critical point, and one on $B$, such that $\pi_X(z_1, z_2)=z_1^2 + z_2^2$ on these charts.
\end{enumerate}
\end{definition}
Throughout   this article,  we assume that the regular fiber of $\pi_X: X \to B$ is an oriented connected closed  surface and  $B$ has two boundary components $\partial B=S_+^1 \bigsqcup (-S_-^1)$ so that $\partial X= Y_+ \bigsqcup (-Y_-)$. Note that $Y_{+}$ and $Y_-$ are surface fibrations over circles. We allow the boundary component to be empty.
In addition,  we assume that $\pi_X$ is injective on the set of critical points.

The singular fiber of $\pi_X: X \to B $ can be classified as separating and non-separating, depending on whether the corresponding vanishing cycle is   separating or  non-separating.
Let $z$ be a critical value of $\pi_X$ and $g$ is genus of regular fiber. In  the case  that $\pi_X^{-1}(z)=\Sigma$ is non-separating, then $\Sigma$ is an immersed surface of $g-1$ with a double point.   If the singular fiber $\pi_X^{-1}(z)=\Sigma$ is separating, then  $\Sigma$ is a transversely intersecting pair of embedded surfaces with square $-1$ and genus adding to $g$. For more details, please refer to \cite{REG}.

 \begin{definition} (Cf. \cite{REG})
A  Lefschetz fibration   is called relatively minimal if  there is no fiber   containing an embedded sphere with $-1$ self-intersection number. Such sphere is called exceptional sphere.
\end{definition}

The following definition is  an analogy to  Definition \ref{def6}.
\begin{definition}
Let $\pi_X: X \to B$ be a Lefschetz fibration. A $2$-form  $\omega_X \in \Omega^2(X)$ is called admissible if $d\omega_X=0$  and $\omega_X$ is fiberwise nondegenerate. Near critical points of $\pi_X$, $\omega_X$ is $K\ddot{a}hler$ with respective to the coordinate charts $(z_1, z_2)$ in Definition \ref{def7}. Given  $Q>0$,  we say that  an admissible $2$-form $\omega_X$ is $Q$-admissible if  $\omega_X \vert_{Y_{+}}$ and $\omega_X \vert_{Y_{-}}$ are $Q$-admissible.
\end{definition}

\begin{definition} \label{def10}
Let $\pi_{\pm} :Y_{\pm} \to S^1$ be a $3$-dimensional fibration over $S^1$  together with a   $Q$-admissible $2$-form $\omega_{\pm}$. A fiberwise  symplectic   cobordism from $(Y_+, \pi_+, \omega_+)$ to $(Y_-, \pi_-, \omega_-)$ is a $4$-dimensional Lefschetz fibration $\pi_X: X \to B$ together with an admissible $2$-form $\omega_X$ such that $\partial X = Y_+\bigsqcup (-Y_-)$, $\pi_X \vert_{Y_{\pm}}= \pi_{\pm}$ and $\omega_X \vert_{Y_{\pm}}=\omega_{\pm}$.
\end{definition}
Given a  fiberwise  symplectic   cobordism $(X, \pi_X, \omega_X)$, we can  construct a symplectic cobordism $(X, \Omega_X)$   as follows:
Firstly, we need the following lemma.
\begin{lemma}
Let $(X, \pi_X) $ be a Lefschetz fibration,  $Y= \partial X$ and $U$ be a collar neighborhood  of $Y$ in $X$ such that $U\simeq (-\varepsilon, 0 ] \times Y$ for some $\varepsilon>0$. Also, the projection $\pi_X \simeq id \times \pi$ under above identification. Let $\eta$  be a closed $2$-form of $X$, then $\eta \vert_U = \eta\vert_{\{0\} \times Y} + d \mathfrak{a}$  for some $\mathfrak{a} \in \Omega^1(U)$. Moreover, $\mathfrak{a} \vert_{\{0\} \times Y} =0$.  \label{C33}
\end{lemma}
\begin{proof}
Under the identification $U\simeq (-\varepsilon, 0 ] \times Y$, we can write $\eta= \eta_s + ds \wedge \gamma_s$, where $\eta_s \in \Omega^2(Y)$ and $\gamma_s \in \Omega^1(Y)$ for each $s \in (-\varepsilon, 0 ] $. Since $d\eta=0$,  we have
\begin{eqnarray*}
d_3 \eta_s + ds \wedge  \frac{\partial \eta_s}{\partial s} - ds \wedge d_3 \gamma_s=0,
\end{eqnarray*}
where $d_3$ is the exterior derivative of $Y$. Therefore, $d_3 \eta_s=0 $ and $ \frac{\partial \eta_s}{\partial s}= d_3 \gamma_s$  for each $s \in (-\varepsilon, 0 ] $. Integrate both side of the second identity,  we have
\begin{eqnarray*}
\eta_s = \eta \vert_{s=0} + \int_0^s  \frac{\partial \eta_{\tau}}{\partial {\tau}} =  \eta \vert_{\{0\} \times Y} + d_3 \left( \int_0^{s} \gamma_{\tau} \right).
\end{eqnarray*}
Let $\mathfrak{a}= \int_0^s \gamma_{\tau}$, then $\eta= \eta \vert_{\{0\} \times Y}  + d_3 \mathfrak{a} + ds \wedge \frac{\partial \mathfrak{a}}{ \partial s} = \eta \vert_{\{0\} \times Y} + d \mathfrak{a}$.
\end{proof}

 According to Lemma \ref{C33},  we can identify respectively neighborhoods of $Y_{+}$ and $Y_-$ in $(X, \pi_X, \omega_X)$  with collars of the form
\begin{eqnarray*}
&&( (-2\varepsilon, 0 ]\times Y_+ , \omega_+  + d\mu_+ ) ,\\
&& ( (0, 2\varepsilon]\times Y_- , \omega_-  + d\mu_-  ).
\end{eqnarray*}
Moreover, $\pi_X=id \times \pi_{\pm}$ under the identification. Fix  a cut off functions $\phi$ so that $\phi(s)=0$ when $s \le \frac{1}{2}$ and $\phi(s)=1$ when $s \ge \frac{3}{4}$.   Let $\phi_{\pm}(s)=\phi(\frac{s}{{\mp} 2 \varepsilon })$. We define a  new admissible $2$-form on $X$ by
\begin{eqnarray} \label{e46}
\omega_{X \phi}=
\begin{cases}
\omega_X , &  \mbox{on $X \setminus ((-2\varepsilon, 0 ]\times Y_+ \bigcup  (0, 2\varepsilon]\times Y_- )$}
\cr   \omega_+ + d(\phi_+(s_+) \mu _+) & \mbox{ on $(-2\varepsilon, 0 ]\times Y_+ $}
\cr \omega_- + d(\phi_-(s_-)\mu_-) & \mbox { on $(0, 2\varepsilon]\times Y_- $},
\end{cases}
\end{eqnarray}
where $s_{+}$ and $s_-$ are coordinates  of  $(-2\varepsilon, 0 ]$ and $(0, 2\varepsilon]$ respevtively. It is worth noting that  $d \omega_{X\phi} =0$ and  $\omega_{X\phi}$ is monotone whenever $\omega_X$ is monotone.

% Given an admissible $2$-form $\omega_X$ such that $\omega_X \vert_{Y_{\pm}}= \omega_{\pm}$.  From now on, we
Fix a volume form  $\omega_B$ on $B$.  We can define a symplectic form  $\Omega_X=\omega_{X \phi}+ K\pi_X^* \omega_B$ on $X$, where $K$ is a large positive number so that $\Omega_X \wedge \Omega_X >0$ everywhere. Fix a large $K$,  we find collar neighborhoods $(-2\varepsilon, 0 ] \times S_+^1$ and $(0, 2\varepsilon] \times S_-^1$ of $\partial B$ such that $K\omega_B=ds \wedge dt$ in these neighborhoods.  Then $(X, \Omega_X)$ is a symplectic cobordism from    $(Y_+, \pi_+, \omega_+)$ to $(Y_-, \pi_-, \omega_-)$. %The vector fields $V_{\pm}$ in Definition \ref{def9} is given by $\partial_s$.

\begin{remark}
 One can construct the symplectic cobordism  from $(X, \pi_X,\omega_X)$  directly by taking $\Omega_X= \omega_X + K\pi_X^*\omega_B$. The advantage of the above construction  is that the symplectic completion of $(X, \Omega_X)$ is still a Lefschetz fibration,  see the next section.
\end{remark}
\textbf{Notation.} To  simplify the notation,  from now on, we suppress the subscript $\phi$  from the notation  and  assume that $K=1$ all time. Also, given a fiberwise symplectic cobordism $(X, \pi_X, \omega_X)$, the symplectic form $\Omega_X$ over $X$  always refers to the symplectic form constructed above.

\subsection{Symplectic completion  } \label{section14}
In order to introduce holomorphic curves with asymptotic ends, we  define the completion of the symplectic cobordism  or fiberwise symplectic cobordism by adding cylindrical ends.  %To prepare for this, the following lemma  describes the behavior of a closed $2$-form near the $\partial X$.

Suppose that $(X, \Omega_X)$  is a symplectic   cobordism from $(Y_+, \pi_+, \omega_+)$ to $(Y_-, \pi_-, \omega_-)$. By   Moser's trick (Cf. Proposition 6.4 of \cite{Wen2}),  we can identify the neighborhoods of $Y_{+}$ and $Y_-$ in $(X, \Omega_X)$ symplectically with collars of the form
\begin{eqnarray*}
&&( (-\varepsilon, 0 ]\times Y_+ , \omega_+  + ds \wedge \pi^*_+dt ) ,\\
&& ( (0, \varepsilon]\times Y_- , \omega_-   + ds \wedge \pi^*_-dt ).
\end{eqnarray*}
Granted this identification,   define  the  completion of $(X, \Omega_X)$ by adding cylindrical ends
\begin{align*}
\overline{X}= ((-\infty, 0] \times Y_-) \cup_{Y_-}  X \cup_{Y_+} ([0, +\infty) \times Y_+).
\end{align*}
We call $ [0, +\infty) \times Y_+$ and  $(-\infty, 0] \times Y_- $ the ends of $\overline{X}$. The symplectic form $\Omega_X$ extends  to be $\omega_{\pm} + ds \wedge \pi^*_{\pm} dt $ over the ends.

%Given  a  fiberwise  symplectic   cobordism $(X, \pi_X, \omega_X)$, we can   construct a symplectic cobordism $(X, \Omega_X)$ as follows:

   It is worth noting that when $(X, \Omega_X)$ is constructed from a fiberwise symplectic cobordism $(X, \pi_X, \omega_X)$, then the fibration structure  automatically extends  to $\pi_X: \overline{X} \to \overline{B}$, where $\overline{B}= ((-\infty, 0] \times S^1_-) \cup_{S_-^1}  B \cup_{S^1_+} ([0, +\infty) \times S^1_+).$ % We extend the  admissible $2$-form $\omega_X$ and symplectic form $\Omega_X$  to be $\omega_{\pm}$  and $\omega_{\pm}+ ds \wedge \pi_{\pm}^*dt$ on the ends respectively.

Recall that $\omega_B = ds \wedge dt $ near $\partial B$. In order to define the almost complex structure adapted to fibration latter, we  fix a $\omega_B$-compatible complex structure $j_B$ over $\overline{B}$
such that  $j_B(\partial_s) =\partial_t$ over the ends of $\overline{B}$ throughout.

\subsection{Almost complex structures }
In order to define the periodic Floer homology and its cobordism maps in different settings, we need the following several types of almost complex structures.
\subsubsection{Almost complex structures on $\mathbb{R} \times Y$ }
\begin{definition}
Let $\pi:Y \to S^1$ be a surface fibration as before and $\omega$ be  an admissible $2$-form on $Y$. An almost complex structure $J$ on $\mathbb{R} \times Y$ is called symplectization admissible if $J$ satisfies the following properties:
\begin{enumerate}
\item
$J$ is compatible with  $\Omega=\omega + ds \wedge \pi^*dt$.
\item
$J$ is $\mathbb{R}$-invariant and $J(\partial_s)= R$.
\item
$J$ maps $\ker \pi_*$ to $ \ker \pi_*$ .
\end{enumerate}
\end{definition}

Let $\mathcal{J}_{comp}(Y, \pi, \omega)$   denote the space of all symplectization admissible almost complex structures. It is contractible and carries a natural $C^{\infty}$-topology. Let $\mathcal{J}_{comp}(Y, \pi, \omega)^{reg}$   denote the Baire subset of admissible almost complex structures that satisfy the criteria [J1], [J2a] and [J2b] in \cite{LT}.

\subsubsection{Almost complex structures on $\overline{X}$ }  \label{section12}

\begin{definition}\label{def8}
An almost complex structure $J$ on $\overline{X}$ is called  cobordism admissible if $J$ satisfies the following properties:
\begin{enumerate}
\item
There exists $J_{\pm} \in \mathcal{J}_{comp}(Y_{\pm}, \pi_{\pm}, \omega_{\pm})$ such that  $J$ agrees with $J_{+}$ and $J_-$  respectively on $[-\epsilon, +\infty)\times Y_+$ and $(-\infty, \epsilon] \times Y_-$ for some $\epsilon>0$.

\item
$J \vert_X $ is compatible with $\Omega_X$.
\end{enumerate}
We  use  $\mathcal{J}_{comp}(X, \Omega_X)$ to denote the space of   cobordism admissible  almost complex structures. It is contractible and carries a natural $C^{\infty}$-topology. Fix symplectization admissible almost complex structures $J_{+}$ and $J_-$, %$J_{\pm} \in \mathcal{J}_{comp}(Y_{\pm}, \pi_{\pm}, \omega_{\pm})$,
let $ \mathcal{J}_{comp}(X,  \Omega_X,  J_{\pm})$    denote a subset of $ \mathcal{J}_{comp}(X, \Omega_X)$  such that $J \vert_{\mathbb{R}_{\pm} \times Y_{\pm}}$ agrees with $J_{\pm}$ along $\mathbb{R}_{\pm} \times \gamma_{\pm}$, where $\gamma$ runs over all periodic orbits with degree less than $Q$, and $\mathbb{R}_+=[0, \infty)$ and  $\mathbb{R}_-=(-\infty, 0 ]$.
%Let  $ \mathcal{J}_{comp}(X, \Omega_X,  J_{\pm})$% similarly to the one in Definition \ref{def2}.
%to denote a subset of $ \mathcal{J}_{comp}(X, \pi_X, \omega_X)$  such that $J \vert_{\mathbb{R}_{\pm} \times Y_{\pm}}$ agrees with $J_{\pm}$ in a $\mathbb{R}$ invariant neighborhood of $\mathbb{R}_{\pm} \times \gamma $(including $J=J_{\pm}$ along $\mathbb{R}_{\pm} \times \gamma_{\pm}$ ), here $\gamma$ runs over all periodic orbits with degree less than $Q$.
\end{definition}

The following definition only for the case that $(X, \pi_X, \omega_X)$ is a fiberwise symplectic cobordism.
\begin{definition} \label{def2}
An almost complex structure $J$ on $\overline{X}$  is called adapted to the fibration if $J$ satisfies the following properties:
\begin{enumerate}
\item
There exists $J_{\pm} \in \mathcal{J}_{comp}(Y_{\pm}, \pi_{\pm}, \omega_{\pm})$ such that  $J$ agrees with the $J_{+}$ and $J_-$  respectively on $[-\epsilon, +\infty)\times Y_+$ and $(-\infty, \epsilon] \times Y_-$   for some $\epsilon>0$.
\item
$\pi_X: \overline{X} \to \overline{B}$ is complex linear with respect to $(J, j_B)$, i.e., $j_B \circ d\pi_X =  d\pi_X \circ J $.
\item
Away from the critical points of $\pi_X$, $J \vert_{\ker d\pi_{X}}$ is compatible with $\omega_X$.
 \item
 Near the critical points of $\pi_X$, under the coordinate charts  $(z_1, z_2)$ in Definition \ref{def7}, $J$ agrees with the  standard  complex structure.
\end{enumerate}
We use  $\mathcal{J}_{tame}(X, \pi_X, {\omega_X})$ to denote  the space of   almost complex structures adapted to the fibration, this space is contractible and carries a natural $C^{\infty}$-topology. Fix $J_{\pm} \in \mathcal{J}_{comp}(Y_{\pm}, \pi_{\pm}, \omega_{\pm})$,  define $ \mathcal{J}_{tame}(X, \pi_X, \omega_X,  J_{\pm})$ similarly as in Definition \ref{def8}.

%let $ \mathcal{J}_{tame}(X, \pi_X, \omega_X,  J_{\pm})$ to  denote a subset of $ \mathcal{J}_{tame}(X, \pi_X, \omega_X)$  such that $J \vert_{\mathbb{R}_{\pm} \times Y_{\pm}}$ agrees with $J_{\pm}$ along $\mathbb{R}_{\pm} \times \gamma_{\pm}$, here $\gamma$ runs over all periodic orbits with degree less than $Q$, and $\mathbb{R}_+=[0, \infty)$ and  $\mathbb{R}_-=(-\infty, 0 ]$.

%agrees with $J_{\pm}$ in a $\mathbb{R}$ invariant neighborhood of $\mathbb{R}_{\pm} \times \gamma $(including the case $J=J_{\pm}$ along $\mathbb{R}_{\pm} \times \gamma_{\pm}$ ), here $\gamma$ runs over all periodic orbits with degree less than $Q$, and $\mathbb{R}_+=[0, \infty)$ and  $\mathbb{R}_-=(-\infty, 0 ]$.
\end{definition}

\begin{remark}
The symplectic form  $\Omega_X$ gives a decomposition   $TX = TX^{hor} \oplus  TX^{vert}$ of the  tangent bundle of $X$, where $TX^{vert}= \ker \pi_{X*}$ and $TX_{x}^{hor}=\{v \in TX_x \vert   \Omega_X(v, w)=0  \mbox{ } \forall  w \in  TX_x^{vert} \}$. With respect to this  splitting, $J$ can be written as $ J=\left[
 \begin{matrix}
   J^{hh} & J^{vh}  \\
   J^{hv} & J^{vv}
  \end{matrix}
  \right]  $.
Note that the complex linear condition implies that $J^{vh}=0$.   In addition, it is easy to check that $J$  is ${\Omega}_X$-tame, however,  $J$ is not compatible with ${\Omega}_X$ unless $J^{hv}=0$.
\end{remark}

In order to meet the  transversality condition, we often require an almost complex structure $J$ satisfies  certain  criteria.  In the rest of the paper, we  refer to such an almost complex structure $J$ as a generic almost complex structure. The precise meaning will be explained in Section \ref{section15}.

Let us explain why we need two different kinds of almost complex structures here.
Let $(X, \pi_X, \omega_X)$ be a fiberwise symplectic cobordism. To define the cobordism maps $HP (X, \Omega_X, J, \Lambda_X)$  by  using holomorphic curve method,  a genus bound  on holomorphic curves (Lemma \ref{C43}) plays a key role in our proof.  To obtain this bound, we require that $\pi_X$ is complex linear. This is  the reason why we need Definition \ref{def2}.

To define the cobordism maps $HP_{sw}(X, \Omega_X, \Lambda_X)$, in order to apply the techniques in \cite{HT}, we want to perturb the Seribrg Witten equations  by symplectic form $\Omega_X$.  To this end, we need to find a metric such that $\Omega_X$ is self-dual. The natural way is to define  the metric by $g(\cdot, \cdot)=\Omega_X( \cdot, J \cdot)$, where $J $ is an almost complex structure compatible with $\Omega_X$. So one may consider the subset in  $\mathcal{J}_{tame}(X, \pi_X, \omega_X)$   consisting  of almost complex structures  that  are compatible with $\Omega_X$. We denote this subset by  $\mathcal{J}$. Of course, one also can use the more general almost complex structures in  $\mathcal{J}_{comp}(X, \Omega_X)$.

When we show that these two definitions of cobordism maps agree, we need the almost complex structures to be nice enough  so that   both of $HP_{sw}(X, \Omega_X, J,\Lambda_X)$ and  $HP (X, \Omega_X, J, \Lambda_X)$  are well defined  simultaneously. Unfortunately, the potential candidate  $\mathcal{J}$ doesn't work.   Therefore, we need to  search the almost complex structure in a larger  space    $\mathcal{J}_{comp}(X, \Omega_X)$.

\subsection{Composition of symplectic cobordism} \label{section22}
Let $(X_+,   \Omega_{X_+})$ and  $(X_-,  \Omega_{X_-})$ be  respectively symplectic  cobordisms from $(Y_+, \pi_+, \omega_+)$ to $(Y_0, \pi_0, \omega_0)$ and  from $(Y_0, \pi_0, \omega_0)$ to  $(Y_-, \pi_-, \omega_-)$. There is a natural operation on these two cobordisms  by gluing them along   their common boundaries. Here we only focus on the fiberwise symplectic cobordism case, the general case can be formulated similarly.

Define $X_R= X_+ \cup_{\{R\} \times Y_0} [-R,R]_s\times Y_0 \cup_{\{-R\} \times Y_0} X_-$.
Note that in a collar neighborhood of $Y_0$, both of $\omega_{X+}$ and $\omega_{X_-}$   are equal to $\omega_0$ by our construction in Section \ref{section10}, so we can glue them together smoothly.  The admissible $2$-form on $X_R$ is defined by
\[\omega_{X_R}=
\left\{
\begin{array}
    {r@{\quad:\quad}l}
    \omega_{X+} & \mbox{on $X_+$} \\
    \omega_{X-}  & \mbox{ on $X_-$} \\
%    \omega_+ + ds\wedge \pi_+^*dt& \mbox{ on $Y_+\times [R,\infty)$}\\
    \omega_0 & \mbox { on $Y_0\times [-R,R]$}\\
 %   \omega_- + ds\wedge  \pi_-^*dt& \mbox{ on $Y_-\times [-R,-\infty)$}
\end{array}
\right.
\]
Observe  that if $(X_{+}, \omega_{X_{+}})$ and  $(X_{-}, \omega_{X_{-}})$ are monotone with  the same coefficient $\tau$, then $(X_R, \omega_{X_R})$ is also monotone. % When $R=0$,
The result $(X_R, \pi_{X_R}, \omega_{X_R})$ is called the composition of  $(X_+, \pi_{X_+}, \omega_{X_+})$ and $(X_-, \pi_{X_-}, \omega_{X_-})$. % denoted by $(X_+ \circ X_-, \pi_{X_+\circ X_-}, \omega_{X_+ \circ X_{-}})$.

 Let   $\overline{X}_R$  be the usual symplectic completion by adding cylindrical ends.
 Let $J_{X_{{\pm}}} \in \mathcal{J}_{tame}(X_{\pm}, \pi_{X_{\pm}}, \omega_{X_{\pm}})$ be a generic almost complex structure on $X_{\pm}$ so that $J_{X_+}=J_+ $ on $\mathbb{R}_{+} \times Y_{+}$, $J_{X{\pm}}=J_0$ on $\mathbb{R}_{\mp} \times Y_0$ and $J_{X_-} =J_-$ on $\mathbb{R}_- \times Y_-$. These almost complex structures $J_{X_+} $ and $J_{X_-}$ induces a natural   almost complex structure $J_R$ over  $\overline{X_{R}}$. It is given by
%The  almost complex structure on $\overline{X_{R}}$ is given by
\[J_R=
\left\{
\begin{array}
    {r@{\quad:\quad}l}
   J_{X_+}{} & \mbox{on $X_+$} \\
    J_{X_-}   & \mbox{ on $X_-$} \\
    J_+& \mbox{ on $Y_+\times [R,\infty)$}\\
   J_0& \mbox { on $Y_0\times [-R,R]$}\\
     J_-& \mbox{ on $Y_-\times [-R,-\infty)$}.
\end{array}
\right.
\]
 Note that $J_R \in \mathcal{J}_{tame}(X_{R}, \pi_{X_{R}}, \omega_{X_{R}}) $. % By Definition  \ref{def2}, there exists $\epsilon >0$ such that $J_R =J_0 $ over $[-\epsilon -R, R+ \epsilon ] \times Y_0$.  Let $F: \overline{X_0} \to \overline{X_R}$ be the diffeomorphism such that $F(s, y) =(f(s), y)$ for $(s,y)  \in [-\epsilon,   \epsilon ] \times Y_0$ and $F=id$ outside this region, where $f:  [-\epsilon,   \epsilon ] \to  [R-\epsilon,   R+\epsilon ]$ be an increasing function.

 %$J_R \vert_{R=0}$ is called the composition of almost complex structures $J_{X_+} $ and $J_{X_-}$, denoted by $J_{X_+} \circ J_{X_-}$.

\subsection{ Local coefficient.} \label{section9}
The  periodic Floer homology cannot be defined over $\mathbb{Z}$ coefficient in general. We need to introduce a  concept called local coefficient in this subsection.
\begin{definition}
Given orbit  sets $\alpha_+=\{(\alpha_{+, i}, m_i)\}$ and $\alpha_-=\{(\alpha_{-,j}, n_j)\}$, define $H_2(X, \alpha_+, \alpha_-)$ to be the set of relative homology classes of   2-chains in $X$ such that $\partial Z = \sum_i m_i \alpha_{+,i} - \sum_j n_j \alpha_{-,j} $.  Here two 2-chains  are equivalent if and only if their difference is  boundary of a 3-chains. If $X = [0,1] \times Y$, we just denote the set of relative homology classes  by $H_2(Y, \alpha_+, \alpha_-)$.
\end{definition}
A local coefficient, $\Lambda_P$, of $R$-module for periodic Floer homology is defined as follows. For each orbit set $\alpha$, we have a $R$-module $\Lambda_{\alpha}$, and for  each relative homology class $Z \in H_2(Y, \alpha, \beta)$, we assign  an isomorphism $\Lambda_Z: \Lambda_{\alpha} \to \Lambda_{\beta}$ satisfying the composition law $\Lambda_{Z_1+Z_2}= \Lambda_{Z_2} \circ \Lambda_{Z_1}$, where $Z_1 \in H_2(Y, \alpha, \beta)$ and $Z_2 \in H_2(Y,  \beta, \gamma)$.

Let $\omega$ be an admissible $2$-form on $(Y, \pi)$.  Let $S \subset H_2(Y, \alpha, \beta)$ be a $[\omega]$-finite set, that is  $S \bigcap \{Z \in H_2(Y, \alpha, \beta): \int_Z \omega \le C \}$ is finite for any $C \in \mathbb{R}$.    To ensure that a series of isomorphisms $\sum\limits_{Z \in S} n_Z \Lambda_Z$ is convergent and the terms of this series can be rearranged,
%$\left(\sum\limits_{Z\in S_1} n_{Z} \Lambda_{Z}\right) \circ \left(\sum\limits_{W \in S_2} n_{W} \Lambda_{W}\right)= \sum\limits_{Z+W} n_{Z}n_{W}\Lambda_{W} \circ \Lambda_{Z} $,
we required that the local coefficient $\Lambda_P$ is $[\omega]$-complete in the the sense of \cite{KM}.
%\begin{enumerate}
%\item
%$\Lambda_{\alpha}$ is a topological $R$-module and $0 \in \Lambda_{\alpha}$ has a neighborhood consisting of sub-modules.
%\item
%For any $[\omega]$-finite set $S \subset H_2(Y, \alpha, \beta)$, the set of isomorphism $\{\Lambda_Z \in Hom( \Lambda_{\alpha}, \Lambda_{\beta}) \vert Z \in S\}$ is equicontinuous. In addition, $\{\Lambda_Z\}_{Z \in S}$ converges to zero with respect to compact-open topology.
%\end{enumerate}

Let $(X, \Omega_X)$ be a symplectic  cobordism from $(Y_+, \pi_+, \omega_+)$ to $(Y_-, \pi_-, \omega_-)$ and $\Lambda^{\pm}_{P}$ be $[\omega_{\pm}]$-complete local coefficient. A $X$-morphism $\Lambda_X$ between $\Lambda^+_{P}$ and $\Lambda^-_{P}$ means that for each relative homology class $Z \in H_2(X, \alpha_+, \alpha_-)$, we have an isomorphism $\Lambda_X(Z): \Lambda_{\alpha_+} \to \Lambda_{\alpha_-}$  satisfying the  composition law $\Lambda_X(Z_+ +Z +Z_-) = \Lambda^-_{Z_-} \circ \Lambda_X (Z) \circ \Lambda^+_{Z_+}$, where $Z_{\pm} \in H_2(Y_{\pm}, \alpha_{\pm}, \beta_{\pm})$ and $Z \in H_2(X, \beta_+, \alpha_-)$. For the same reasons mentioned in last paragraph, we require that $\Lambda_X$ satisfies the following condition:
\begin{enumerate}
\item[]
For any $\Omega_X$-finite set $S \subset H_2(X, \alpha_+, \alpha_-)$, the set of isomorphism $\{\Lambda_X(Z) \in Hom( \Lambda_{\alpha_+}, \Lambda_{\alpha_-}) \vert Z \in S\}$ is equicontinuous. In addition, $\{\Lambda_X(Z)\}_{Z \in S}$ converges to zero as  $Z$ runs through $S$.
\end{enumerate}
It is an analogy of the $[\omega]$-completeness  in \cite{KM}.

The local coefficient, $\Lambda_S$, of $R$-module for Seiberg Witten cohomology can be defined similarly.  The orbit sets and  relative homology class  are to be replaced by configurations and by relative homotopy class respectively. For the definition of  configurations and relative homotopy class,  please refer  to Section \ref{section3}.

\begin{remark}
We  assume that $R= \mathbb{Z}_2$  unless otherwise stated.
\end{remark}

\subsubsection{  Local coefficient and $2$-form.} \label{section23}
Typical examples of $\Lambda_P$ and $\Lambda_X$ are  given as follows.
\begin{definition}
Let $R$ be a ring,  the universal Novikov ring $\mathcal{R}$ over $R$ is defined by $\mathcal{R}= \{ \sum\limits_i a_i t^{\lambda_i} \vert  i \in \mathbb{R}, a_i \in R, \forall C>0, \# \{i \vert \lambda_i <C , a_i \ne 0\} < \infty\}$.
\end{definition}
Fix an  admissible $2$-form  $\omega$ on $Y$. We introduce a local coefficient system $\mathcal{R}_{\omega}$  which is associated to $\omega$.  For the fiber of $\mathcal{R}_{\omega}$ we take everywhere the universal Novikov ring $\mathcal{R}$. If $Z$ is a relative homology class  from $\alpha$ to $\beta$, then $\mathcal{R}_{\omega}(Z): \mathcal{R}_{\alpha} \to \mathcal{R}_{\beta}$ is defined to be multiplication by $t^{\int_Z \omega}$. One can check that it is $ [\omega]$-complete in the sense of \cite{KM}. %where  $c_{\Gamma}=2PD(\Gamma) + c_1(\ker \pi_*)$.

Let $(X, \pi_X, \omega_X)$ be a fiberwise symplectic   cobordism from $(Y_+, \pi_+, \omega_+)$ to $(Y_-, \pi_-, \omega_-)$. Then the $X$-morphism $\mathcal{R}_{\omega_X}: \mathcal{R}_{\omega_+} \to \mathcal{R}_{\omega_-}$ is defined as follows.  If  $Z \in H_2(X, \alpha_+, \alpha_-)$ is a relative homology class  from $\alpha_+$ to $\alpha_-$, then $\mathcal{R}_{\omega_X}(Z): \mathcal{R}_{\alpha_+} \to \mathcal{R}_{\alpha_-}$ is defined to be multiplication by $t^{\int_Z \omega_X}$.

\section{Statement of the main results}
Let us clarify the notation and conventions we will be using. We fix  a number $Q> g(\Sigma)$  and   a homology class $\Gamma_{\pm} \in H_1(Y_{\pm}, \mathbb{Z})$ satisfies $\Gamma_{\pm} \cdot [\Sigma] \le Q$ throughout. We assume that the manifolds $Y_{+}$, $Y_-$ and $X$ are connected.  In addition, the 2-forms $\omega_{+}$, $\omega_-$ and $\omega_X$ are supposed to be $Q$-admissible.    % and the genera of the regular fibers will be implicitly assumed to be at least two, unless indicated otherwise.    Through out the note, we assume that the manifolds $Y_{\pm}$ and $X$ are connected. \item
%If both of $Y_+$ and $Y_-$ are non-empty, by homological reasons,   we assume that $\Gamma_+ \cdot [\Sigma] = \Gamma_- \cdot [\Sigma]$, otherwise, the cobordism maps are zero.
%If both of $Y_+$ and $Y_-$ are non-empty, by Poincar\'e duality,  there is  no relative class $\Gamma_X \in H_2(X, \partial X, \mathbb{Z})$ such that $\partial_{Y_{\pm}} \Gamma_X =\Gamma_{\pm}$ whenever $\Gamma_+ \cdot [\Sigma] \ne \Gamma_- \cdot [\Sigma]$. Hence, we assume that $\Gamma_+ \cdot [\Sigma] = \Gamma_- \cdot [\Sigma]$ when $Y_+$ and $Y_-$ are nonempty throughout.

%The periodic Floer homology cannot be defined over $\mathbb{Z}$ coefficient in general, we need to introduce the concept called local coefficient.  Similarly, to define the cobordism map,

The following theorems are the main results of this paper.
In the first theorem, we define the cobordism maps on PFH for general cases, the maps are denoted by $HP_{sw}$.  The subscript ``sw"  is used to emphasize that they are defined by Seiberg Witten theory.
\begin{thm}\label{Thm3}
Let $(Y_{\pm}, \pi_{\pm})$ be a  $3$-dimensional fibration over $S^1$  with fiber $\Sigma$, together with a  $Q$-admissible $2$-form $\omega_{\pm}$.  Fix  $\Gamma_{\pm} \in H_1(Y_{\pm}, \mathbb{Z})$ satisfying $g(\Sigma)<\Gamma_{\pm} \cdot [\Sigma] $. Let  $(X, \Omega_X) $ be a symplectic   cobordism  from $(Y_+, \pi_+, \omega_+)$ to   $(Y_-, \pi_-, \omega_-)$.   Let $\Lambda_P^{\pm}$ be a $ [\omega_{\pm}]$-complete local coefficient for the periodic Floer homology and $\Lambda_X$ be a  $X$-morphism between $\Lambda_P^{+}$ and $\Lambda_P^-$.   Then   $(X,   \Omega_X) $  induces a module homomorphism
\begin{eqnarray*}
 {HP}_{sw}(X, \Omega_X, \Lambda_X ): HP_*(Y_+, \omega_+,\Gamma_+, \Lambda^+_P) \to HP_*(Y_-, \omega_-, \Gamma_-,  \Lambda^-_P)
\end{eqnarray*}
with natural decomposition
\begin{equation} \label{e27}
HP_{sw}(X, \Omega_X, \Lambda_X )=\sum\limits_{\Gamma_X \in H_2(X, \partial X, \mathbb{Z}), \partial_{Y_{\pm}}\Gamma_X= \Gamma_{\pm } }HP_{sw}(X, \Omega_X, \Gamma_X,  \Lambda_X )
\end{equation}
and  satisfying the following properties:
\begin{enumerate}
\item (Composition rule)
Let $(X_+,  \Omega_{X+}) $ and $(X_-, \Omega_{X-}) $  be   symplectic  cobordisms from $(Y_+, \pi_+, \omega_+)$ to $(Y_0, \pi_0, \omega_0)$  and  from $(Y_0, \pi_0, \omega_0)$  to $(Y_-, \pi_-, \omega_-)$ respectively.     Then  we have
\begin{eqnarray*}
&& \sum_{\Gamma_X \vert_{X_{\pm}}= \Gamma_{X_{\pm}}, \Gamma_X \in H_2(X, \partial X, \mathbb{Z})} HP_{sw}(X,\Omega_X,  \Gamma_X, \Lambda_X)\\
&=& HP_{sw}(X_-,\Omega_{X_-},  \Gamma_{X_-}, \Lambda_{X_-}) \circ HP_{sw}(X_+,\Omega_{X_+},   \Gamma_{X_+}, \Lambda_{X_+}),
\end{eqnarray*}
where $(X,  \Omega_X)$ is the composition of $(X_+, \, \Omega_{X_+})$ and $(X_-, \Omega_{X_-})$ defined in Section \ref{section22},  and $\Lambda_{X}= \Lambda_{X+} \circ \Lambda_{X-}$.

\item (Invariance)
Suppose that $\Omega_{X_1}$ and $\Omega_{X_2}$ are   symplectic forms on $X$ such that $\Omega_{X_1} =\Omega_{X_2} $ in  collar neighborhoods of $Y_{+}$ and $Y_-$,   then $$HP_{sw}(X, \Omega_{X_1}, \Lambda_X )=HP_{sw}(X, \Omega_{X_2}, \Lambda_X ).$$

\item (Commute with U-map)
\begin{equation*}
 U_+\circ HP_{sw}(X, \Omega_X,   \Lambda_X)=HP_{sw} (X, \Omega_X, \Lambda_X)\circ U_-.
\end{equation*}

\item (Holomorphic curve  axiom)
Given  a cobordism admissible almost complex structure $J \in \mathcal{J}_{comp}(X,  \Omega_X)$ such that $J_{\pm} =J \vert_{\mathbb{R}_{\pm}} \times Y_{\pm}$ is generic, then there is a chain map
$$CP_{sw}(X, \Omega_X, J, \Lambda_X): CP_*(Y_+, \omega_+, \Gamma_+, J_+, \Lambda_P^+) \to CP_*(Y_-, \omega_-, \Gamma_-,  J_-, \Lambda^-_P) $$ inducing $HP_{sw}(X, \Omega_X, J, \Lambda_X)$ with the following properties:
If there is no $J$ holomorphic current from  $\alpha_+$ to  $\alpha_-$ with zero ECH index, then $$<CP_{sw}(X, \Omega_X, J, \Lambda_X) \alpha_+, \alpha_-> = 0.$$
\end{enumerate}
\end{thm}

Before we state the second result, let us introduce the assumptions $(\spadesuit)$ as follows. % under certain assumptions, we give an alternative definition of the cobordism map by using holomorphic curve method.
%\textbf{Assumption $(\spadesuit)$ :}
\begin{definition}
Let $(X, \pi_X, \omega_X)$ be a fiberwise symplectic  cobordism   from $(Y_+, \pi_+,\omega_+)$ to $(Y_-, \pi_-,\omega_-)$. Let $B$ be the base manifold of $(X, \pi_X)$.  We say that the fiberwise symplectic  cobordism satisfies assumptions  $(\spadesuit)$  if any of the following  hold: \label{asump}
\begin{enumerate} [label=\textbf{$\spadesuit$.\arabic*}]
\item  \label{assumption1}
 $g(B) \ge 2$,
\item \label{assumption2}  %[ $(\spadesuit$\Romannum{2}).]
When  $g(B)=1$,  each periodic orbit of $(Y_+, \pi_+,\omega_+)$ with degree $1$ is either $Q$-negative elliptic or hyperbolic, and  each periodic orbit of $(Y_-, \pi_-,\omega_-)$ with degree $1$ is either $Q$-positive elliptic or hyperbolic.
%, where $D \ge 1$.
\item  \label{assumption3} % [ $(\spadesuit$\Romannum{3}).]
When $g(B)=0$, then we have the following properties:
\begin{enumerate}
\item
$Y_{+} \ne \emptyset$ and $Y_- \ne \emptyset$.
\item
Each  periodic orbit of $(Y_+, \pi_+,\omega_+)$ with degree less than $Q$ is either $Q$-negative elliptic or hyperbolic.
%, where $D \ge Q$.
\item
 Each periodic orbit of $(Y_-, \pi_-,\omega_-)$ with degree less than $Q$ is either $Q$-positive elliptic or hyperbolic.
 %, where $D\ge Q$.
\end{enumerate}
\end{enumerate}
\end{definition}
The assumptions on $Q$-positive (negative) elliptic orbits are inspired by \cite{H4}, the existence of these orbits are used to ensure that the ECH indexes  of holomorphic curves  are nonnegative.  Under the assumptions $(\spadesuit)$, we give a Seiberg Witten free definition of the cobordism maps in the following theorem.

\begin{thm}\label{Thm1}
Let $(Y_{\pm}, \pi_{\pm})$ be a $3$-dimensional fibration over $S^1$ with fiber $\Sigma$, together with a $Q$-admissible $2$-form $\omega_{\pm}$ on $Y_{\pm}$.  Fix $\Gamma_{\pm} \in H_1(Y_{\pm}, \mathbb{Z})$ satisfying $g(\Sigma)<\Gamma_{\pm} \cdot [\Sigma]  $. Let  $( X, \pi_X, \omega_X) $ be a fiberwise symplectic   cobordism  from $(Y_+, \pi_+, \omega_+)$ to   $(Y_-, \pi_-, \omega_-)$ satisfying  $(\spadesuit)$.
%condition  \Romannum{1} or \Romannum{2} or \Romannum{3}.
Let $\Lambda_P^{\pm}$ be a $ [\omega_{\pm}]$-complete local coefficient for the periodic Floer homology and $\Lambda_X$ be a $X$-morphism between $\Lambda_P^{+}$ and $\Lambda_P^-$.  Suppose that  $(X, \pi_X)$ contains no separating singular fiber, then  for generic $J \in \mathcal{J}_{tame}(X, \pi_X, \omega_X)$,
%or $\omega_X$ is monotone, then for generic $J \in \mathcal{J}_{tame}(X, \pi_X, \omega_X)$ in the former case and   $J \in \mathcal{V}_{comp}(X, \pi_X, \omega_X)$ in the later case,
$(X, \pi_X, \omega_X) $  induces a module homomorphism
\begin{eqnarray*}
HP(X, \Omega_X,  J, \Lambda_X ): HP_*(Y_+, \omega_+,\Gamma_+, J_+, \Lambda^+_P) \to HP_*(Y_-, \omega_-, \Gamma_-,  J_-, \Lambda^-_P)
\end{eqnarray*}
with   natural decomposition (\ref{e27}). The map  $HP(X, \Omega_X,  J, \Lambda_X )$ is defined by counting
  embedded  holomorphic curves, without using Seiberg Witten theory. (Recall that the symplectic form here is given by $\Omega_X=\omega_X + \pi_X^*\omega_B$.) Moreover, the cobordism maps satisfy the following properties:

\begin{enumerate}
\item (Composition rule)
Let $(X_+, \pi_{X+}, \omega_{X+}) $ and $(X_-, \pi_{X-}, \omega_{X-}) $  be   respectively  fiberwise  symplectic    cobordisms from $(Y_+, \pi_+, \omega_+)$ to $(Y_0, \pi_0, \omega_0)$  and  from $(Y_0, \pi_0, \omega_0)$  to $(Y_-, \pi_-, \omega_-)$.  Assume   both of them  satisfy  assumptions $(\spadesuit)$. Let $(X_R, \pi_{R}, \omega_{X_R})$ be  the composition of $(X_+, \pi_{X_+}, \omega_{X_+})$ and $(X_-, \pi_{X_-}, \omega_{X_-})$, and $J_R$  is defined in  Section \ref{section22}. Let $\Lambda_{X_R}= \Lambda_{X+} \circ \Lambda_{X-}$.  Assume that  $J_{X \pm}$ are generic   and there exists $R_0 \ge 0$ such that  $J_{R}$ is generic for any $R \ge R_0$, then we have
\begin{equation*}
\begin{split}
& \sum_{\Gamma_{X_R} \vert_{X_{\pm}}= \Gamma_{X_{\pm}}, \Gamma_{X_R} \in H_2(X_R, \partial X_R, \mathbb{Z})} {HP}(X_R,\Omega_{X_R},   \Gamma_{X_R},  J_R,  \Lambda_{X_R}) \\
&= HP(X_-,\Omega_{X_-},  \Gamma_{X_-},  J_{X_-}, \Lambda_{X_-}) \circ HP(X_+,\Omega_{X_+}, \Gamma_{X_+},  J_{X_+}, \Lambda_{X_+}),
\end{split}
\end{equation*}
for any $R \ge R_0$.

\item (Commute with U-map)
\begin{equation*}
 U_+\circ HP(X, \Omega_X, J,  \Lambda_X)=HP(X, \Omega_X, J,  \Lambda_X)\circ U_-.
\end{equation*}
\item (Blowup)
Suppose that $(\Gamma_X, \omega_X)$ is monotone. Take a point $x \in X$ in regular fiber and a generic $J \in \mathcal{J}_{tame}(X, \pi_X, \omega_X)$ such that $J$ is integral near $x$. Let $(X, 'J')$ be  the  blowup of $(X, J)$ at $x$. Then there exists an admissible $2$-form $\omega_{X'}$ such that $J' \in \mathcal{J}_{tame}(X', \pi_{X'}, \omega_{X'})$. Let $E$ be the homology class of the exceptional sphere. Regard   $\Gamma_X \in H_2(X', \partial X', \mathbb{Z})$ and assume that $\Gamma_X \cdot E=0$. Let  $\Lambda_{X'}$ be  a $X'$-morphism such that $\Lambda_X(Z)=\Lambda_{X'}(Z)$ for any $Z$ with relative class $\Gamma_X$, then we have
\begin{equation*}
HP(X, \Omega_X,  \Gamma_X, J, \Lambda_X)= HP(X', \Omega_{X'}, \Gamma_X, J',  \Lambda_{X'}).
\end{equation*}

%\item
%The following two statements are  about the case that $d=0$. In the following two cases, $HP(X, \omega_X, \Lambda_X)$ is still well-defined and independent on the almost complex structures.
%\begin{enumerate}
%
%
%
%
%\item
%Suppose that $Y_{\pm}=\emptyset$ and $b_2^+(X) >1$,  then
%\begin{eqnarray*}
%HP(X, \omega_X , \Lambda_X)(\Lambda_{[\emptyset]})= \sum\limits_{A \in H_2(X, \mathbb{Z})}Gr(X, \Omega, A) \Lambda_X(A) \circ \Lambda_{[\emptyset]}
%\end{eqnarray*}
%
%\item
%Suppose that $Y_{+} \ne \emptyset$ or $Y_{-} \ne \emptyset$ and  $X$ is relative minimal.  When $d=0$, $\overline{HP}(X, \omega_X, \Lambda_X)$ is a canonical isomorphism.
%\end{enumerate}
\end{enumerate}
\end{thm}

\begin{remark}  \label{r2}
Here are some remarks about Theorem \ref{Thm1}:
\begin{itemize}
\item
Note that  if  $(X_{+},  \pi_{X{+}}, \omega_{X_+})$ and  $(X_{-}, \pi_{X{-}}, \omega_{X_-})$  satisfy  assumptions $(\spadesuit)$, then their composition also satisfies the assumptions $(\spadesuit)$. Thus the statement about composition rule  in Theorem \ref{Thm1} makes sense.
\item
The cobordism maps $HP(X, \Omega_X, J, \Lambda_X)$ in Theorem \ref{Thm1}  also can be defined  over local coefficient of  $\mathbb{Z}$ module. See the discussion in Section \ref{section19}.

\item
 Theorem \ref{Thm1} is also true  when $Y_{+}$ and $Y_-$ are disconnected.

\item
If $X$ has separating singular fibers and  each  separating singular fiber contains a exceptional sphere. Let $\{E_i\}_i$ denote the homology class of these  exceptional spheres.  Assume that  $\Gamma_X \in H_2(X, \partial X, \mathbb{Z})$ such that  $\Gamma_X \cdot E_i \le 0$ for  all $i$.  %Suppose that $g(\Sigma_i')=0$, i.e. $\Sigma_i'$ is exceptional sphere.
%Then for any $\Gamma_X \in H_2(X, \partial X, \mathbb{Z})$ such that $\partial_{Y_{\pm}} \Gamma_X =\Gamma_{\pm}$  and $\Gamma_X \cdot E_i=0$ for  all $i$,
Then the cobordism maps $HP(X, \Omega_X, \Gamma_X, J,     \Lambda_X)$ are  still well defined. We will discuss this in Section \ref{section4}.

% More detail please see remark \ref{r1}.

\end{itemize}
\end{remark}

\begin{thm}\label{Thm2}
\begin{comment}
Given the same assumption of Theorem \ref{Thm1}, suppose that  $(X, \pi_X)$ is relative minimal and $\omega_X$ is monotone.
For generic $J \in \mathcal{V}_{comp}(X,\pi_X, \omega_X)$,
$HP(X, \omega_X,  J, \Lambda_X)$ is equivalent to $HM(X, \omega_X,  r, \Lambda_X)$ under the isomorphism $\mathcal{T}^{\pm}_{r*}$. In other words, the following diagram commutes: \label{C21}
\begin{displaymath}
\xymatrix{
HP_*(Y_+, \omega_+ ,\Gamma_+, J_+ , \Lambda_P^+) \ar[d]^{HP(X, \omega_X,  J,  \Lambda_X) } \ar[r]^{\mathcal{T}_{r*}^+} & HM^{-*}(Y_+, \mathfrak{s}_{\Gamma_+}, -\pi \varpi^+_r,  J_+, \Lambda_S^+)  \ar[d]^{ HM(X, \omega_X,  r,  \Lambda_X) }\\
HP_*(Y_-, \omega_-, \Gamma_-,  J_- ,\Lambda_P^-) \ar[r]^{\mathcal{T}_{r*}^-} &HM^{-*}(Y_-, \mathfrak{s}_{\Gamma_-}, -\pi \varpi^-_r, J_-, \Lambda_S^-) }
\end{displaymath}
\end{comment}
Given the same assumptions of Theorem \ref{Thm1}, suppose that $(X, \omega_X)$ or $(\Gamma_X, \omega_X)$ is monotone for some $\Gamma_X \in H_2(X, \partial X, \mathbb{Z})$. Then there exists a nonempty open subset
$\mathcal{V}_{comp}(X, \pi_X, \omega_X)$ of $ \mathcal{J}_{comp}(X, \Omega_X)$ such that for
generic $J \in \mathcal{V}_{comp}(X,\pi_X, \omega_X)$, both of  ${HP}_{sw}(X, \Omega_X, J, \Gamma_X, \Lambda_X)$ and $HP(X, \Omega_X, J, \Gamma_X, \Lambda_X)$ are well defined. Moreover,  we have
%$HP(X, \omega_X,  J, \Lambda_X)=\overline{HP}(X, \omega_X, J, \Lambda_X)$ and
\begin{equation*}
 HP(X, \Omega_X,   \Gamma_X, J, \Lambda_X)= {HP}_{sw}(X, \Omega_X, \Gamma_X, J,  \Lambda_X).
\end{equation*}
In particular, the cobordism map $HP(X, \Omega_X,   \Gamma_X,  J, \Lambda_X)$ is independent of  the choice of $J \in \mathcal{V}_{comp}(X,\pi_X, \omega_X) $.
\end{thm}
\begin{remark}
This remark concerns the  monotonicity assumptions  in  Theorem \ref{Thm2}.  These assumptions are only used to guarantee that  the existence of  $\mathcal{V}_{comp}(X,\pi_X, \omega_X)$ and do not play any role in other argument.

When we define the cobordism maps  for  cases \ref{assumption1} and \ref{assumption2}, the methods in this paper rely on a genus bound  in Lemma \ref{C1}.  Roughly speaking, $\mathcal{V}_{comp}(X,\pi_X, \omega_X)$ is the set of  cobordism admissible almost complex structures such that the genus bound holds.  Under the monotonicity assumptions, we show that it is nonempty and open.

 The author  believes that the conclusions in Theorem \ref{Thm2}  should be  true   without the  monotonicity assumptions. Here are some thoughts of the author which support the  above opinion. For convenience, the local coefficients are taken to be the one in Subsection \ref{section23}. Then   the chain maps   ${CP}(X, \Omega_X, \Gamma_X, J) $  and ${CP}_{sw}(X, \Omega_X, \Gamma_X, J) $ can be regarded  as  Taylor series. In fact, Section   \ref{section21} provides a  nonempty  open subset  of almost complex structures $\mathcal{V}^L_{comp}(X, \pi_X, \omega_X, J_{\pm})$ such that  for generic  $J \in \mathcal{V}^L_{comp}(X, \pi_X, \omega_X, J_{\pm})$, the genus bound holds for any curve with energy less than $L$, and the coefficients of  ${CP}(X, \Omega_X, \Gamma_X, J) $ are well defined up to order $L$.

If one could  define     $HP(X, \Omega_X,   \Gamma_X, J)$  for generic $J$ and show that it  is independent of  $J$ by using any method  which beyond this paper, then take a generic $J_L \in  \mathcal{V}^L_{comp}(X, \pi_X, \omega_X, J_{\pm})$,  the methods in this paper   still can show that  the chain maps of $HP$ and $HP_{sw}$ are equal up to order $L$, i.e.,
\begin{equation*}
 CP_{sw}(X, \Omega_X,   \Gamma_X, J_L)= {CP}(X, \Omega_X, \Gamma_X, J_L) + o(L).
\end{equation*}
As  the cobordism maps   are  independent of $J$, then  $CP_{sw}$ and $CP$ are chain homotopic to  $CP_{sw}(X, \Omega_X,   \Gamma_X, J)$ and  ${CP}(X, \Omega_X, \Gamma_X, J)$ respectively, where $J$ is an almost complex structure which is independent of $L$. Take $L \to \infty$, then we get the same conclusions of Theorem \ref{Thm2}.
\end{remark}

\begin{remark}
In case  \ref{assumption3},  we do not use the  genus bound  in Lemma \ref{C1}.  The conclusions in Theorem \ref{Thm2} are  true without the monotonicity assumptions. Moreover, the open set $\mathcal{V}_{comp}(X,\pi_X, \omega_X)=\mathcal{J}_{comp}(X,  \Omega_X)$.
\end{remark}

%
%\begin{remark}
%To define the cobordism maps in Theorems \ref{Thm3} and \ref{Thm1}, we only need to assume that $[\Gamma_{\pm}] \cdot[\Sigma]>g(\Sigma)-1$. The stronger assumption  $[\Gamma_{\pm}] \cdot[\Sigma]>g(\Sigma) $ is to guarantee that  the $U$-map is well defined for generic $J_{\pm} \in \mathcal{J}_{comp}(Y_{\pm}, \pi_{\pm}, \omega_{\pm})$.
%\end{remark}
%

%\begin{remark}
%If both of $Y_+$ and $Y_-$ are non-empty, by Poincar\'e duality,  there is  no relative class $\Gamma_X \in H_2(X, \partial X, \mathbb{Z})$ such that $\partial_{Y_{\pm}} \Gamma_X =\Gamma_{\pm}$ whenever $\Gamma_+ \cdot [\Sigma] \ne \Gamma_- \cdot [\Sigma]$. Hence, we assume that $\Gamma_+ \cdot [\Sigma] = \Gamma_- \cdot [\Sigma]$ when $Y_+$ and $Y_-$ are nonempty throughout.
%\end{remark}

\begin{remark}
For the case that $\Gamma_{\pm} \cdot[\Sigma]=0$, the map $HP(X, \Omega_X, J, \Lambda_X)$ is still well-defined by counting  holomorphic curves with zero ECH index.   The manifold $(X, \Omega_X)$  can be arbitrary symplectic cobordism.  Moreover, $HP(X, \Omega_X, J, \Lambda_X)$  satisfies the following properties:
\begin{enumerate}
\item
Suppose that $Y_{\pm}=\emptyset$ and $b_2^+(X) >1$,  then
\begin{eqnarray*}
HP(X, \Omega_X ,  J, \Lambda_X)(\Lambda_{[\emptyset]})= \sum\limits_{A \in H_2(X, \mathbb{Z})}Gr(X, A) \Lambda_X(A) \circ \Lambda_{[\emptyset]},
\end{eqnarray*}
where $Gr(X, A)$ is the Gromov invariant defined in \cite{T2}.
\item
Suppose that $Y_{+} \ne \emptyset$ or $Y_{-} \ne \emptyset$ and  $X$ is relatively minimal Lefschetz fibration, also, the genus of the fiber is at least two,   then $HP(X, \Omega_X, J, \Lambda_X)$ is a canonical isomorphism.
\item
In above two cases,   $HP(X, \Omega_X, J, \Lambda_X)$  is independent on the choice of  almost complex structures.
\end{enumerate}
The first statement follows from the definition of Taubes' Gromov invariant.  The second statement follows from the observation that the only closed holomorphic curve in $\overline{X}$ with zero ECH index is the empty curve. In both cases,   $HP(X,  \Omega_X, \Lambda_X)=HP_{sw}(X, \Omega_X, \Lambda_X)$.
\end{remark}

\begin{remark}
In many applications, the tuple $(Y, \pi, \omega, \Gamma)$ satisfies the monotonicity property which is similar to  Definition \ref{def11}.  In this case, the periodic Floer homology  of $(Y, \pi, \omega, \Gamma)$  can be defined  with $\mathbb{Z}_2$ or $\mathbb{Z}$ coefficient.  It follows from the observation that the holomorphic curves with the same ECH index have  same energy. The compactness of the moduli space implies that there are only finitely many holomorphic curves that contribute  to the differential.

This observation is   true for the cobordism case as well. (See the proof of Corollary \ref{C53}.) If the pair $(\Gamma_X, \Omega_X)$ is monotone in the sense of Definition \ref{def11}, then  the cobordism maps  $HP_{sw}(X, \Omega_X, \Gamma_X)$ and $HP(X, \Omega_X, \Gamma_X)$ given  in Theorems  \ref{Thm3} and \ref{Thm1}    can be defined with     $\mathbb{Z}_2$ or $\mathbb{Z}$ coefficients.
In Remark \ref{r8}, we   explain more about the case in  Theorem   \ref{Thm3}, as the cobordism maps are defined by counting solutions to Seiberg Witten equations rather than holomorphic curves.

Note that the same argument is not available when $(X, \Omega_X)$ is monotone.  One may use a similar argument as in Lemma \ref{C32} to control the energy of holomorphic curves with fixed genus. However, the curves that contribute to the cobordism maps do not have priori upper bound on the genus.
 \end{remark}

Besides the main results above,  here we state a result which concerns the assumptions $(\spadesuit)$.  The proposition that follows  asserts   that after suitable perturbation, the assumptions $(\spadesuit)$  holds in most cases, the exceptional cases are that the base manifold  $B$ is a disk.  The proposition    is an analogy of   Theorem 2.5.2 of \cite{VPK} in PFH setting.  It will be proved in Section \ref{section1}.
\begin{prop} \label{C42}
Let $(X, \pi_X, \omega_X)$ be a fiberwise symplectic   cobordism from $(Y_+, \pi_+, \omega_+)$ to $(Y_-, \pi_-, \omega_-)$.  Given $Q \ge 1$, we can always find $\omega_{\pm}'$ such that
\begin{itemize}
\item
Every periodic orbit of $\omega_+'$ with degree less than $Q$ is either $Q$-negative  elliptic  or hyperbolic.
\item
Every periodic orbit of $\omega_-'$ with degree less than $Q$ is either $Q$-positive elliptic  or hyperbolic.
\item
The cohomology class of $\omega_{\pm}'$ is unchanged, i.e.,  $[\omega_{\pm}']=[\omega_{\pm}]. $
\item
Given a metric $g_{\pm}$ over $Y_{\pm}$, for  any $\delta>0$, there exists a constant $c_0 $ depending on $g_{\pm}$, such that we can arrange that $|\omega'_{\pm} - \omega_{\pm}|_{g_{\pm}} \le c_0 \delta$.
%Moreover, $[\omega_{\pm}']=[\omega_{\pm}]. $
\end{itemize}
Moreover, we can find an admissible $2$-form $\omega_X'$ such that $(X, \pi_X,\omega_X')$ is  a fiberwise symplectic cobordism from $(Y_+, \pi_+, \omega_+')$ to $(Y_-, \pi_-, \omega_-')$.
\end{prop}

We have organized the rest of the  paper in the following way: Section $4$  is an overview of  the foundation of holomorphic curve  and  periodic Floer homology.  We prove  Theorem \ref{Thm3} in Section $5$.  In Section $6$, we prove   Proposition \ref{C42}.   Section $7$ presents some partial results about defining cobordism maps $HP(X, \Omega_X, J, \Lambda_X)$ on PFH by using holomorphic curve method. Finally, in Sections 8 and 9, we explore the relationship of ${HP}_{sw}(X, \Omega_X, J,  \Lambda_X)$  and  $HP(X, \Omega_X, J, \Lambda_X)$  under    monotonicity assumptions.
%%%%%%%%%%%%%%%%%%%%%%%%%%%%%%%

\ack{The author would like to thank  Prof.Yi-Jen Lee  for suggesting the  problem, as well as  for  her helpful comments.  He also wants to thank the anonymous referees whose comments, suggestions, and corrections, have greatly improved this paper. Without their support, this paper would not have been possible.}

\section{ECH index and periodic Floer homology}
%\subsection{Orbits  set}
This section supplies the background material that is used subsequently to define the periodic Floer homology and its cobordism maps.  For more details,  please refer to \cite{H3}.

% we  review the foundation of periodic Floer homology   like holomorphic curve, ECH index and ECH partition,etc. We also need these kind of concepts when we define the cobordism map. More detail please refer to \cite{H3}.

\subsection{J-holomorphic curves and currents}
\begin{definition}
A $J$ holomorphic curve is a map $u:(C, j) \to \overline{X}$ satisfying $du+ J\circ du\circ j=0$, where $(C, j)$ is  a Riemann surface possibly with punctures. Two holomorphic curves $u:(C, j) \to \overline{X}$  and $u':(C', j') \to \overline{X}$ are equivalent if there exists a biholomorphism $\varphi: (C,j)\to (C', j')$ such that $u=u'\circ \varphi$.
\end{definition}

\begin{definition}
A holomorphic curve $u$ is called somewhere injective if $u^{-1}\{u(z)\}=z$ for some $z \in C$,  we also call it simple.
\end{definition}
\begin{remark}
%If $u: C \to \overline{X}$ is somewhere injective, then we don't distinguish the map and its image.
Sometimes we just use $C$ to denote the holomorphic curves,   especially in the case that $u$ is somewhere injective.
\end{remark}

%The number $\int_{\mathcal{C}}\omega_X= \sum\limits_{a} d_a \int_{C_a} \omega_X$ is called $\omega_X$-energy of $\mathcal{C}$.

%If $\int_C u^* \omega_X < \infty$, then  the holomorphic curve $u$ is asymptotic to periodic orbits.

Given orbit sets $\alpha_+=\{(\alpha_{+, i}, m_i)\}$ and $\alpha_-=\{(\alpha_{-,j}, n_j)\}$,   the   moduli space  $\mathcal{M}_X^J(\alpha_+, \alpha_-)$  is a collection of equivalent class  of holomorphic maps $u: C \to \overline{X}$ with the following properties:
 $u$ has positive ends at covers of $\alpha_{+, i}$ with total multiplicity $m_i$, negative ends at covers of $\alpha_{-,j}$ with total multiplicity $n_j$, and no other ends.
Here $C$ is a  compact Riemann surface possibly with punctures and may be disconnected. When $\overline{X}=\mathbb{R} \times Y$ and $J \in \mathcal{J}_{comp}(Y, \pi, \omega)$, we just use $\mathcal{M}_Y^J(\alpha_+, \alpha_-)$ to denote the moduli space. Given $Z \in H_2(X, \alpha_+, \alpha_-)$, we use $\mathcal{M}_X^J(\alpha_+, \alpha_-, Z)$ to denote the elements in $\mathcal{M}_X^J(\alpha_+, \alpha_-)$ with relative homology class $Z$.
% Given a holomorphic curve $u$,

Given a curve $u \in \mathcal{M}_X^J(\alpha_+, \alpha_-) $, % the $\omega_X$-energy of $u$ %$\mathcal{C}$
%is defined by $\int_C u^*\omega_X$.
the $\Omega_X$-energy $E_{\Omega_X}(u)$ of $u$ which is defined in \cite{FYHKE} is given by
\begin{equation*}
E_{\Omega_X}(u)= \int_{C \cap \mathbb{R}_+ \times Y_+} \omega_+  + \int_{C \cap \mathbb{R}_- \times Y_-} \omega_-  + \int_{C \cap u^{-1}(X)} u^* \Omega_X.% \int_C u^*\omega_X + d vol(B).
\end{equation*}
By Stokes formula, the $\Omega_X$-energy only depends on $\alpha_+$, $\alpha_-$ and  its relative homology class.  For $i \in \mathbb{Z}$ and $L \in \mathbb{R}$, define
$$\mathcal{M}_{X,I=i}^{J,L}(\alpha_+, \alpha_-)= \bigsqcup\limits_{I(\alpha_+, \alpha_-, Z)=i \  E_{\Omega_X}(Z) < L} \mathcal{M}_X^J(\alpha_+, \alpha_-, Z), $$
where $I(\alpha_+, \alpha_-, Z)$ is the ECH index. We will review the definition of $I(\alpha_+, \alpha_-, Z)$ in the upcoming section.

%\begin{remark}
%Let $u: C \to \overline{X}$ be a holomorphic curve from $\alpha_+$ to $\alpha_-$ with $d=[\Sigma] \cdot [\alpha_{\pm}]$. The $\Omega_X$-energy $E_{\Omega_X}(u)$ of $u$ which is defined in \cite{FYHKE} is given by $$E_{\Omega_X}(u)= \int_{C} u^*\omega_X + \int_{C \cap u^{-1}(X)} u^*\pi_X^* \omega_B= \int_C u^*\omega_X + d vol(B).$$ Therefore, a uniform bound  on $\omega_X$-energy ensures that we can use the Gromov compactness in \cite{FYHKE}.
%\end{remark}

\begin{remark}
In the case  that  $(X, \pi_X, \omega_X)$ is a  fiberwise symplectic cobordism,  the $\Omega_X$-energy of a holomorphic curve $u \in \mathcal{M}^J(\alpha_+, \alpha_-)$ is      $E_{\Omega_X}(u) =\int_C u^*\omega_X + d vol(B),$ where  $d=[\Sigma] \cdot [\alpha_{\pm}]$.  Therefore, a uniform upper  bound  on $\int_C u^*\omega_X$ ensures that we can use the Gromov compactness in \cite{FYHKE}.
\end{remark}

 Most of the time, we only care  about the underlying current of the holomorphic curves, rather than the holomorphic maps.  Thus it is  convenient to introduce the notion of holomorphic currents.

\begin{definition} \label{def4}
 A $J$-holomorphic current from $\alpha_+=\{(\alpha_{+, i}, m_i)\}$   to $\alpha_-=\{(\alpha_{-,j}, n_j)\}$ is a finite set of pairs  $\mathcal{C}=\{(C_a, d_a)\}$, where  $\{C_a\}_a$ are distinct, irreducible, somewhere injective $J$ holomorphic curves with $E_{\Omega_X} (C_a) < \infty$ and $\{d_a\}_a$ are positive integers. If $C_a$ is a $J$-holomorphic curve whose positive ends are asymptotic to  $\alpha_+(a)=\{(\alpha_{+,i}, m_{ia})\}$ and negative ends are asymptotic to  $\alpha_-(a)=\{(\alpha_{-,j}, n_{ja})\}$, then $m_i=\sum\limits_a d_am_{ia}$ and  $n_j=\sum\limits_a d_a n_{ja}$. Let  $\widetilde{\mathcal{M}}_X^J(\alpha_+, \alpha_-)$   denote the moduli space of   $J$ holomorphic currents from $\alpha_+$ to $\alpha_-$.
\end{definition}

%In addition to the moduli space $\widetilde{\mathcal{M}}_X^J(\alpha, \beta)$,

Note that for each  $u: \tilde{C} \to \overline{X}$ in $\mathcal{M}^J_{{X}}(\alpha_+, \alpha_-)$, we can associate a holomorphic current $\mathcal{C} \in \widetilde{\mathcal{M}}_X^J(\alpha_+, \alpha_-) $ in the following way.    According to Theorem 6.19 of  \cite{Wen2}, there is a factorization $u= v \circ  \varphi$, where $v: C \to \overline{X}$ is a somewhere injective holomorphic curve and $\varphi: \tilde{C} \to C$ is a  branched covering. Then the  underlying  holomorphic current of $u$ is $\mathcal{C}=\{(C, d)\}$, where $d$ is degree of $\varphi$. If $\tilde{C}$ is reducible, then $\mathcal{C} $ is a union of holomorphic currents of all irreducible components.

%In our setting, the holomorphic curves which we deal with are not closed, but we can still talk about the intersection number in the following sense.
The following definition generalizes the concept of intersection numbers to punctured holomorphic curves.
\begin{definition} [Cf. Definition 4.7 of \cite{H4}] \label{def1}
Let $C$   be a  simple irreducible holomorphic curves from $\alpha_+$ to $\alpha_-$, the self intersection number is defined to be
\begin{equation*}
C\star C = \frac{1}{2}(2g(C)-2 + {\rm ind} C + h(C) + 2e_Q(C)+ 4\delta(C)) \in \frac{1}{2} \mathbb{Z}.
\end{equation*}
where $h(C)$ is the number of ends of $C$ at hyperbolic orbits and $\delta(C) \ge 0$ is a count of the singularities of $C$ in $\overline{X}$ with positive integer weights, and $e_Q(C)$ is the  total multiplicity of all elliptic orbits in $\alpha_+$ that are $Q$-negative, plus   the  total multiplicity of all elliptic orbits in $\alpha_-$ that are $Q$-positive.

If $C$ and $C'$ are two distinct simple irreducible holomorphic curves, then $C\star C' \in \mathbb{Z}$ are defined to be the algebraic count of  their intersection points.
\end{definition}

\begin{remark}
It is worth noting that if $C$ is closed, then $C \star C$ agrees with the usual self intersection  number  $C \cdot C$ by adjunction formula.
%Thus above definition can be  view as a generalization of the usual intersection number in closed cases.
\end{remark}
\begin{definition} \label{def3}
A holomorphic current $\mathcal{C} =\{(C_a, d_a)\} \in \widetilde{\mathcal{M}}_X^J(\alpha_+, \alpha_-) $  is called embedded if it satisfies  the following  properties：
\begin{itemize}
\item
For any $a$, $C_a$ is embedded and $d_a=1$.
\item
$C_a \star C_b =0$ for any $a \ne b$.
\end{itemize}
A holomorphic curve $u \in   {\mathcal{M}}_X^J(\alpha_+, \alpha_-)$ is called embedded if its underlying holomorphic current is embedded.
\end{definition}
\subsection{ECH index and Fredholm index}
In this subsection, we briefly review the  ECH index and Fredholm index for holomorphic curves.

\textbf{ECH Index.} Let $\mathcal{C}$ be a holomorphic current from  $\alpha_+= \{(\alpha_{+, i}, m_i)\}$ to $\alpha_-= \{ (\alpha_{-, j}, n_j)\}$.  Let $\tau$ be a homotopy class of symplectic trivializations $\tau^+_i$ of the restriction of $(\ker{\pi_{+*} }, J_+)$ along $\alpha_{+,i}$ and $\tau_j^-$ of the restriction of $(\ker{\pi_{-*} }, J_-)$ along $\alpha_{_-, j}$.  The ECH index is  defined by the following formula:
\begin{eqnarray*}
I(\mathcal{C}) = c_{\tau}(\mathcal{C}) + Q_{\tau}(\mathcal{C}) + \sum_i \sum\limits_{p=1}^{m_i} \mu_{\tau}(\alpha_{+, i}^{p})- \sum_j \sum\limits_{q=1}^{n_j} \mu_{\tau}(\alpha_{-,j}^{q}),
\end{eqnarray*}
where $c_{\tau}(\mathcal{C})$ and $Q_{\tau}(\mathcal{C})$ are respectively the relative Chern number and the relative self-intersection numbers (See Section 4.2 of \cite{H2}), and $\mu_{\tau}$ is the  Conley-Zehender index.  The ECH index $I$ only depend on orbit sets $\alpha_+$, $\alpha_-$ and relative homology class of $\mathcal{C}$.

\textbf{Fredholm Index.} Let $u: C \to \overline{X}$ be a $J$-holomorphic curve from $\alpha_+= \{(\alpha_{+, i}, m_i)\}$ to $\alpha_-= \{ (\alpha_{-, j}, n_j)\}$. For each $i$, let $k_i$ denote the number of ends of $u$ at $\alpha_{+, i}$, and let $\{p_{ia}\}^{k_i}_{a=1}$ denote their multiplicities. Likewise,
for each $j$, let $l_j$ denote the number of ends of $u$ at $\alpha_{-,j}$, and let $\{q_{jb}\}^{lj}_{b=1}$ denote their multiplicities. Then the Fredholm index of $u$ is defined by
\begin{eqnarray*}
{\rm ind} u = -\chi(C) + 2 c_{\tau}(u^*T\overline{X}) + \sum\limits_i \sum\limits_{a=1}^{k_i} \mu_{\tau} (\alpha^{p_{ia}}_{+, i}) -  \sum\limits_j \sum\limits_{b=1}^{l_j} \mu_{\tau} (\alpha^{q_{jb}}_{-,j}).
\end{eqnarray*}
%When $u$ is somewhere injective,
Sometimes we write ${\rm ind} C$ instead of  ${\rm ind} u$.

The following theorem summarizes two important relationships between  Fredholm index and ECH index for simple holomorphic curve.
\begin{theorem} [\cite{H1}, \cite{H2}] \label{ECH1}
Let $C \in \mathcal{M}^J_X(\alpha_+, \alpha_-)$ be a simple holomorphic curve, then
\begin{itemize}
\item
${\rm{ind}}(C) \le I(C) - 2\delta(C)$, equality holds if and only if $C $ is admissible. (Definition of admissible will be given in next section.)
%&\mbox{If $\alpha$ and $\beta$ are admissible, then ${\rm{ind}}(C) =I(C) \mod 2 $ },

\item
If $\alpha_+$ and $\alpha_-$ are admissible, then ${\rm{ind}}(C) =I(C) \mod 2 $,
\end{itemize}
where $\delta(C)$ is a count of the singularities of $C$ in $\overline{X}$ with positive integer weights.
\end{theorem}
%The proof of these two inequalities pleased refer to \cite{H1} and $\cite{H2}$.

The following theorem will be also used frequently:
\begin{theorem} [Theorem 5.1 \cite{H2}] \label{ECH2}
Let  $\mathcal{C}=\{(C_a, d_a) \}$  be  a holomorphic current  in $\overline{X}$, then
%\begin{equation}
%I(\mathcal{C} + \mathcal{C}') \ge I(\mathcal{C}) + I(\mathcal{C}') + 2 \mathcal{C} \cdot \mathcal{C}'.
%\end{equation}
%In particular,
%\label{A5}
\begin{equation} \label{e18}
I(\mathcal{C}) \ge \sum\limits_a d_aI(C_a) + \sum\limits_a d_a(d_a-1) C_a \star C_a + 2\sum\limits_{a \neq b }d_ad_b C_a \star C_b.
\end{equation}
%where $C_a \cdot C_b  \ge 0$ is the algebraic count of intersections of $C_a$ and $C_b$.
\end{theorem}

\subsection{ECH partition}
Given a simple periodic orbit $\gamma$ and a holomorphic curve $C$, suppose that $C$ has positive ends at covers of $\gamma$ whose total covering multiplicity is $m$, the multiplicities of these covers form a partition of the integer $m$, which is denoted by  $p_+(\gamma^m, C)$. We called it a positive partition of $C$ at $\gamma$. Similarly, we  define $p_-(\gamma^m, C)$  for negative ends and called it a  negative partition of $C$ at $\gamma$.
\begin{definition}
Let $u: C \to \overline{X} $ be a $J$ holomorphic curve from $\alpha_+$ to $\alpha_-$ without $\mathbb{R}$-invariant cylinder. $u$ is called positively admissible if the   positive partition    at $\gamma$ is $p_{+}(m,\gamma)$, i.e., $p_{+}(\gamma^m,C)=p_{+}(m,\gamma)$.   Likewise, $u$ is called negatively admissible if the negative partition   at $\gamma$ is $p_{-}(m,\gamma)$. $u$ is called admissible if $u$ are both positively and negatively admissible.
%If $\overline{X}=\mathbb{R} \times Y$ and $J$ is symplectization admissible, $u$ is admissible if
Here $p_{+}(m,\gamma)$ and $p_{-}(m,\gamma)$ are defined in \cite{H1}, \cite{H2}. When $u$ is admissible, we also say that $u$ satisfies the ECH partition condition.
\end{definition}

\begin{definition}
A connector   $u: C \to \mathbb{R} \times Y$  is a union of branched cover of trivial cylinders with zero Fredholm index. Here each connected component of $C$ is a punctured sphere. A connector is trivial if it is a union of unbranched cover of trivial cylinders, otherwise, it is nontrivial.
\end{definition}
\begin{remark}
For the sake of  convenient,  we use  stronger definition here. %the definition of connector here is stronger than the original definition in \cite{H1}.
Note that the original definition  of connector in   \cite{H1}  doesn't involve any constrain on Fredholm index.
\end{remark}
%\begin{remark}
%It is worth noting that the definition of connector is a little different from \cite{H1}.
%\end{remark}
In this article, we don't use the precise definition of $p_{+}(m,\gamma)$ and $p_{-}(m,\gamma)$, we only need to know the following two facts.

 \begin{itemize}
 \item
 Let $\gamma$ be an  elliptic orbit, for any partition $p(m,\gamma)$, there is no nontrivial connector from $p_{+}(m,\gamma)$ to $p(m,\gamma)$. Likewise, there is no nontrivial  connector from $p(m,\gamma)$ to $p_{-}(m,\gamma)$. Cf. exercise 3.13, 3.14 of \cite{H3}.
 \item
Suppose that  $u$ is a simple holomorphic curve, then $u$ is admissible if and only if $I(u)={\rm ind}(u)$. Cf. \cite{H1}, \cite{H2}.
\end{itemize}

\subsection{Meaning of generic almost complex structure} \label{section15}
 The purpose of this subsection is to explain the precise meaning of ``generic almost complex structure''.
To this end, we first need to  recall the behavior of holomorphic curve at the ends. The following descriptions follow Section 2 of \cite{HT2}.

Let $\gamma$ be a non-degenerate periodic orbit of $(Y, \pi, \omega)$. Let $\varphi_{\gamma} : S^1 \times D \to Y$ be the coordinate in Lemma 2.3 of \cite{LT}. Granted this coordinate, we can express the ends of the holomorphic curve as a graph of certain functions. More details are explained as follows:  Suppose that $C$ is a somewhere injective $J$-holomorphic curve in $\overline{X}$ and  $\mathpzc{E} \subset C$ is a positive end  at $\gamma^{q_{\mathpzc{E}}}$. Let $\hat{\varphi}_{\gamma}= Id \times \varphi_{\gamma} : \mathbb{R} \times S^1 \times D \to \mathbb{R} \times Y$, then there exists a constant $s_0>0$  such that $\hat{\varphi}_{\gamma}^{-1}(\mathpzc{E} \cap ([s_{0} , +\infty) \times Y))$ is the image of a map
\begin{equation} \label{e35}
\begin{split}
&[s_{0}, +\infty) \times \mathbb{R} / 2\pi q_{\mathpzc{E}} \to \mathbb{R} \times S^1 \times D \\
&(s, \tau) \to (s, (\tau, \varsigma(\tau, s))),
\end{split}
\end{equation}
and $ \varsigma(\tau, s)= e^{-\lambda_{q_{\mathpzc{E}}}s} (\varsigma_{q_{\mathpzc{E}}}  + \mathfrak{r} (\tau, s)) $, where the notation is explained as follows.  $\lambda_{q_{\mathpzc{E}}}$ is the samllest positive eigenvalue of the set of $2\pi q_{\mathpzc{E}}$ periodic eigenfunctions of operator:
\begin{eqnarray} \label{e60}
&& L_{\gamma}: C^{\infty}(\mathbb{R}; \mathbb{C}) \to C^{\infty}(\mathbb{R}; \mathbb{C}) \nonumber \\
&& \eta \to \frac{i}{2} \frac{d \eta}{dt} + \nu \eta + \mu \bar{\eta},
\end{eqnarray}
and $\varsigma_{q_{\mathpzc{E}}}$ (possibly zero) is a $2\pi q_{\mathpzc{E}} $-periodic eigenfunction of $L_{\gamma}$. $\nu$ and $\mu$ are respectively $2\pi$-periodic real and complex  functions, they  are determined by the  pair $(\omega, J)$.  $\mathfrak{r}$ is an error term such that $|\mathfrak{r}| \le e^{-\varepsilon |s|}$ for some $\varepsilon >0 $.

%As a  consequence of Lemma \ref{C15}  later,
 Let  $\mathcal{J}_{1tame}(X, \pi_X,  {\omega_X}, J_{\pm}) $  and $\mathcal{J}_{1comp}(X, \Omega_X,  J_{\pm})$  be  subsets  of $\mathcal{J}_{tame}(X, \pi_X,  {\omega_X})$ and   $\mathcal{J}_{comp}(X,    {\Omega_X}, J_{\pm})$ respectively with the following properties:
\begin{enumerate} [label=\textbf{J.0}]
\item \label{j0}
If $J \in  \mathcal{J}_{1tame}(X, \pi_X,  {\omega_X}, J_{\pm}) $ or $ \mathcal{J}_{1comp}(X,    {\Omega_X}, J_{\pm}) $,  then all simple $J$ holomorphic curves are Fredholm regular in the sense of \cite{Wen2},   possibly  except for  fibers.
 \end{enumerate}
%  Lemma \ref{C15} subsequently show that they are Baire  subsets.
\begin{lemma}
$ \mathcal{J}_{1tame}(X, \pi_X,  {\omega_X}, J_{\pm})$  and $\mathcal{J}_{1comp}(X,   {\Omega_X}, J_{\pm})$  are respectively Baire subset of  $\mathcal{J}_{tame}( X, \pi_X, {\omega_X}, J_{\pm})$    and $\mathcal{J}_{comp}(X,  {\Omega_X}, J_{\pm})$.
%, every somewhere injective $J$ holomorphic curve $u$  which is not tangent to $T\overline{X}^{vert}$ everywhere is Fredholm regular.
\label{C15}
\end{lemma}
\begin{proof}
The lemma can be proved by the standard argument, for example, Chapter 7 of  \cite{Wen2} or  Section 3.4 of \cite{MS}.
\end{proof}

%Let
%\begin{equation*}
%\begin{split}
%& \mathcal{J}_{1tame}(X, \pi_X,  {\omega_X}, J_{\pm})\subset  \mathcal{J}_{tame}(X, \pi_X,  {\omega_X}, J_{\pm})\\
%&\mathcal{J}_{1comp}(X, \pi_X,  {\omega_X}, J_{\pm})\subset  \mathcal{J}_{comp}(X, \pi_X,  {\omega_X}, J_{\pm})
% \end{split}
%\end{equation*}
%$\mathcal{J}_1(X, \pi_X,  {\omega_X}, j_B)^{reg} \subset \mathcal{J}(X, \pi_X,  {\omega_X}, j_B)$, $\mathcal{J}_{comp1}(X, \pi_X,  {\omega_X}, j_B)^{reg} \subset \mathcal{J}_{comp}(X, \pi_X,  {\omega_X}, j_B)$, $\mathcal{J}'_1(X, \pi_X,  {\omega_X}, j_B)^{reg} \subset \mathcal{J}'(X, \pi_X,  {\omega_X}, j_B)$ and  $\mathcal{J}_{comp1}'(X, \pi_X,  {\omega_X}, j_B)^{reg} \subset \mathcal{J}_{comp}'(X, \pi_X,  {\omega_X}, j_B)$ respectively.
%be Baire subsets such that for  $J \in  \mathcal{J}_{1tame}(X, \pi_X,  {\omega_X}, J_{\pm}) $ or $ \mathcal{J}_{1comp}(X, \pi_X,  {\omega_X}, J_{\pm}) $,  all simple $J$-holomorphic curve are Fredholm regular in the sense of \cite{Wen2} except of fibers. The existence of the Baire subsets  will be provided by  Lemma \ref{C15}.

%For $J \in \mathcal{J}^{reg}_{1\omega_X}$,  transversality holds for all somewhere injective curves except fiber. When $J \in \mathcal{J}^{reg'}_{1\omega_X}  $, transversality holds for all somewhere injective curves.

Let $C$ be an embedded holomorphic curve with  zero ECH index. Let $\mathpzc{E} \subset C$ denote an end, reintroduce the notation $q_{\mathpzc{E}}$, $\lambda_{\mathpzc{E}}$ and $\varsigma_{q_{\mathpzc{E}}}$ as above. Introduce a subset $div_{\mathpzc{E}} \subset \{1, \dots , q_{\mathpzc{E}}\}$ as follows. An integer $q \in div_{\mathpzc{E}}$ if either of the following is true:
\begin{enumerate}
\item
$q=q_{\mathpzc{E}}$,
\item
$q$ is a proper divisor of $q_{\mathpzc{E}}$, and there is a $2\pi q$-periodic eigenfunction $\varsigma_q$ of $L_{\gamma}$ with eigenvalue $\lambda_q$ of same sign as $\lambda_{q_{\mathpzc{E}}}$, such that $|\lambda_{q_{\mathpzc{E}}}| \ge|\lambda_q|>0$.
\end{enumerate}

The following  conditions are analogs of the  criteria [Ja], [Jb] in \cite{LT}. We introduce   subsets $\mathcal{J}_{2tame}(X, \pi_X,  {\omega_X}, J_{\pm}) \subset \mathcal{J}_{1tame}(X, \pi_X,   {\omega_X}, J_{\pm})$ and  $\mathcal{J}_{2comp}(X,  {\Omega_X}, J_{\pm}) \subset \mathcal{J}_{1comp}(X,  {\Omega_X}, J_{\pm})$ as follows: $J \in  \mathcal{J}_{2tame}(X, \pi_X,  {\omega_X}, J_{\pm}) $ or $\mathcal{J}_{2comp}(X,    {\Omega_X}, J_{\pm})$ if every embedded $J$ holomorphic curve $\mathcal{C}$ with zero ECH index
%or  $\mathcal{J}_{\omega_X}^{reg'} \subset \mathcal{J}_{1\omega_X}^{reg'}$
satisfies the following properties:
\begin{enumerate} [label=\textbf{J.\arabic*}]
\item \label{j1}

For each end $\mathpzc{E} \subset \mathcal{C}$,  $div_{\mathpzc{E}}=\{q_{\mathpzc{E}}\}$.

\item \label{j2}

[Non-degenerate] For each end $\mathpzc{E} \subset \mathcal{C}$, $ \varsigma_{q_{\mathpzc{E}}} \ne 0$.

\item \label{j3}

[Non-overlapping ]Let $\mathpzc{E}$ and $\mathpzc{E}'$ denote two distinct pairs of either negative or positive ends of $\mathcal{C}$ such that $\gamma_{\mathpzc{E}}=\gamma_{\mathpzc{E}'}$, $q_{\mathpzc{E}}=q_{\mathpzc{E}'}=q$. Fix   $2\pi q$ periodic functions  $\varsigma_{q_{\mathpzc{E}}}, \varsigma_{q_{\mathpzc{E}'}}$ for $\mathpzc{E}$ and $\mathpzc{E}'$-version of $L_{\gamma}$. Then
$\varsigma_{q_{\mathpzc{E}}}(t) \ne \varsigma_{q_{\mathpzc{E}'}}(t+ 2\pi k)$  for any $k$.
\end{enumerate}

\begin{lemma}
 $\mathcal{J}_{2tame}(X, \pi_X,  {\omega_X}, J_{\pm}) $  and $\mathcal{J}_{2comp}(X, {\Omega_X}, J_{\pm}) $  are respectively   Baire subsets  of $\mathcal{J}_{tame}(X, \pi_X,  {\omega_X}, J_{\pm})$  and $\mathcal{J}_{comp}(X,  {\Omega_X}, J_{\pm})$.
\end{lemma}
\begin{proof}
The proof follows the same argument in Section 3   of \cite{HT2}.   Propositions 3.2 and 3.9 of \cite{HT2} show that the symplectization
admissible almost complex structures that satisfy the criterions analogy to \ref{j0}, \ref{j2} and \ref{j3}  form a Baire subset.   Note  that our lemma differs from the analogs in \cite{HT2} only in the following aspects: The symplectization $\mathbb{R} \times Y$  and index $1$ holomorphic curves are respectively  replaced by $\overline{X}$ and index $0$ holomorphic curves. Also, the symplectic form and almost complex structures are changed to the one in our setting accordingly.  But observe that our symplectic form and almost complex structures agree with the one in symplectization case over the ends of $\overline{X}$.   We sketch the proof as follows.

Similar to \cite{HT2}, we  define the  universal moduli space $\mathcal{C}_n$, the essential difference  is to    replace  index 1 curves by index zero curves. Let $\hat{\mathcal{J}}_n $ be subset of $J \in \mathcal{J}_{comp}(X,   {\Omega_X}, J_{\pm})  $  such that there is no $J$-holomorphic curve in  $\mathcal{C}_n$     violates  \ref{j0} and \ref{j2}. % is not Fredholm regular and has degenerate ends.
It is easy to check that  $\hat{\mathcal{J}}_n $  is open, so it suffices to show that   $\hat{\mathcal{J}}_n $ is dense. Fix a $J$, let $\mathcal{U}$ be a suitable neighborhood of $J$.  Fix a section $\psi: \mathcal{U} \to \mathcal{C}_n$, where $\psi(J')=(J', C')$. Recall that our $J'$ is fixed along the periodic orbits.  Therefore,
  fix a positive   end  $\mathpzc{E}$ of   $C$,  expression (\ref{e35})  defines a map $  \mathcal{U} \to  \mathcal{B}_{\mathpzc{E}}$,  where $ \mathcal{B}_{\mathpzc{E}} $ is eigenspace corresponding to the smallest positive eigenvalue.  Equivalently, we can express this map as a function $\hat{a}_{_{\mathpzc{E}}} : \mathcal{U} \to \mathbb{R} \ or \ \mathbb{C}$.  The holomorphic curve $C'$ has degenerate end if and only if $J' \in  \hat{a}_{_{\mathpzc{E}}}^{-1}(0)$.

  The strategy of \cite{HT2} is to show that $\hat{a}_{_{\mathpzc{E}}}$ is submersion, then by implicit function theorem, $codim \hat{a}_{_{\mathpzc{E}}}^{-1}(0) = dim  \mathcal{B}_{\mathpzc{E}} \ge 1$. As a consequence,  $\hat{\mathcal{J}}_n $   is dense.  Lemmas  3.6 and 3.7 of \cite{HT2} can be copied to our case with only notation change. To show that   $\hat{a}_{_{\mathpzc{E}}}$ is submersion, we need to construct a $j \in T_J\mathcal{U}$ such that $\nabla_j \hat{a}_{_{\mathpzc{E}}} = (d  \hat{a}_{_{\mathpzc{E}}} )_J \circ d\psi_J(j) \ne 0 $.  Note that  $d\psi_J(j)= D_C^{-1}(j_C)$ in our case, since there is no $\mathbb{R}$-action, where $j_C \in Hom^{0, 1}(TC, N_C)$  is determined by $j$ (Cf. (3.7) of \cite{HT2}).  We can construct   $f \in Hom^{0, 1}(TC, N_C)$   as in   \cite{HT2}, because,  the construction only  takes place in the end $\mathpzc{E}$.    Since $f$ is supported  away from trivial half cylinder $\mathbb{R}_+ \times \gamma$, we can choose a $j$ such that $j_C =f$.
 Again,    the analysis in pages 50-52 of \cite{HT2}  only  takes place in the end $\mathpzc{E}$,
   the argument can be applied  to our case to show that $\nabla_j \hat{a}_{_{\mathpzc{E}}} \ne 0 $.

 Consider a subset  $\tilde{\mathcal{J}}_n $  of $ J \in \hat{\mathcal{J}}_n $ such that there is  no $J$ holomorphic curve in $\mathcal{C}_n$  violates  \ref{j3}.   Define $\mathcal{B} =\oplus_{\mathpzc{E}}\mathcal{B}_{\mathpzc{E}}$. Compare to \cite{HT2}, there is no need to modulo  $\mathbb{R}$ here, because of the lack of global $\mathbb{R}$-action. Define $p: \mathcal{U} \to \mathcal{B}$ by using expression (\ref{e35}). In fact, $p$ is submersion by applying the same argument Lemma 3.10 of \cite{HT2}. The holomorphic curve $C$ corresponding $J \in \mathcal{U}$   has overlapping ends  if and only if $p(J ) \in Z$, where $Z$ is a  subspace of $\mathcal{B}$ with positive codimension. By implicit function theorem, $\tilde{\mathcal{J}}_n $  is dense.

%the space $\mathbb{P}$ is to be replaced by $\mathcal{B}$ since there is no $\mathbb{R}$-translation in cobordism case, where $\mathcal{B}$  is eigenfunction spaces of the asymptotic operators.  The precise definition of $\mathcal{B}$ and $\mathbb{P}$ please refer to \cite{HT2}.

Define $\mathcal{J}_{2comp}(X,  {\Omega_X}, J_{\pm}) =\cap_n \tilde{\mathcal{J}}_n$, then   for $J \in  \mathcal{J}_{2comp}(X,   {\Omega_X}, J_{\pm})$,  there is no simple index zero $J$-holomorphic curve with degenerate ends and overlapping pairs.  If $C$ is embedded with $I(C)=0$, then $C$ satisfies \ref{j1} because of the ECH partition condition.

The case for $\mathcal{J}_{2tame}(X, \pi_X,  {\omega_X}, J_{\pm})$ is similar.
\end{proof}

Let   $\mathcal{J}_{3comp}(X,  {\Omega_X}, J_{\pm})$ be subset of   $\mathcal{J}_{comp}(X,   {\Omega_X}, J_{\pm})$ such that $J \vert_{\mathbb{R} \times Y_{\pm} } $ belongs to $ \mathcal{J}_{comp}(Y_{\pm}, \pi_{\pm}, \omega_{\pm})^{reg}$. It is easy to check that it is a Baire subset. The intersection of $\mathcal{J}_{2comp}(X,  \Omega_X, J_{\pm})$ and $\mathcal{J}_{3comp}(X, \Omega_X, J_{\pm})$ is denoted by $\mathcal{J}_{comp}(X,    {\Omega_X}, J_{\pm})^{reg}$.
$\mathcal{J}_{tame}(X, \pi_X,  {\omega_X}, J_{\pm})^{reg} $ is defined in the same way with only notation change.

%Let $\mathcal{J}_{3tame}(X, \pi_X,  {\omega_X}, J_{\pm})$ and  $\mathcal{J}_{3comp}(X, \pi_X,  {\omega_X}, J_{\pm})$ be subsets of $\mathcal{J}_{tame}(X, \pi_X,  {\omega_X}, J_{\pm})$ and  $\mathcal{J}_{comp}(X, \pi_X,  {\omega_X}, J_{\pm})$ such that $J \vert_{\mathbb{R} \times Y_{\pm} } \in \mathcal{J}_{comp}(Y_{\pm}, \pi_{\pm}, \omega_{\pm})^{reg}$ respectively. It is easy to check that they are Baire subsets. We define  $\mathcal{J}_{tame}(X, \pi_X,  {\omega_X}, J_{\pm})^{reg}$ and $\mathcal{J}_{comp}(X, \pi_X,  {\omega_X}, J_{\pm})^{reg}$ to be respectively the intersection of $\mathcal{J}_{2tame}(X, \pi_X. \omega_X, J_{\pm})$ and $\mathcal{J}_{3tame}(X, \pi_X. \omega_X, J_{\pm})$,  and $\mathcal{J}_{2comp}(X, \pi_X. \omega_X, J_{\pm})$ and $\mathcal{J}_{3comp}(X, \pi_X. \omega_X, J_{\pm})$.

In this paper, an almost complex structure $J$ is called generic if $J$ belongs to $\mathcal{J}_{tame}(X, \pi_X,  {\omega_X}, J_{\pm})^{reg} $ or $J \in \mathcal{J}_{comp}(X,    {\Omega_X}, J_{\pm})^{reg} $.

\subsection{ Closed holomorphic curves in $\overline{X}$} Let $\pi_X: \overline{X} \to \overline{B}$ be the completion of a  fiberwise symplectic cobordism as before and  $J \in \mathcal{J}_{tame}(X, \pi_X, \omega_X)$. Note that if $u: C \to \overline{X}$ is a somewhere injective closed $J$-holomorphic curve, then $\pi_X \circ u: C \to \overline{B}$ is holomorphic by assumption. Hence, $\pi_X \circ u$ is constant by maximum principle. In other words, the  closed $J$ holomorphic curves in $\overline{X} $ are contained in fibers. Conversely, $\pi_X$ is complex linear implies that  all irreducible components  of the fibers are somewhere injective holomorphic curves.

\subsection{Review of   periodic Floer homology }
Let $\pi: Y \to S^1$ be  a surface fibration over $S^1$ together with a $Q$-admissible $2$-form $\omega$. Fix $\Gamma \in H_1(Y, \mathbb{Z})$ with $g(\Sigma)< d=\Gamma \cdot [\Sigma] \le Q$.
% Let $\omega$ be a $Q$-admissible two form on $Y$ and $c_{\Gamma}= 2 P D (\Gamma) + c_1(K^{-1})$,
Take a generic symplectization admissible almost complex structure $J$ on $\mathbb{R} \times Y$ and a $[\omega]$-complete local coefficient system $\Lambda_P$ in the sense of \cite{KM}, then one can define  periodic Floer homology  $HP_*(Y, \omega, \Gamma, J, \Lambda_P)$ as follows.  The chain complex of  periodic Floer homology is freely generated by admissible orbit sets with homology class $\Gamma$,
and the differential is defined by counting holomorphic currents between generators with ECH index $1$. More precisely,
\begin{equation*}
\partial  = \sum\limits_{\alpha  \in \mathcal{P}(Y, \omega, \Gamma) }\sum\limits_{\beta \in \mathcal{P}(Y, \omega, \Gamma)} \sum_{Z \in H_{2}(Y, \alpha, \beta)} \left(\#_2 \widetilde{\mathcal{M}}^J_{Y, I=1}(\alpha, \beta, Z)/\mathbb{R} \right) \Lambda_{Z}.
\end{equation*}

There is an additional  structure on periodic Floer homology which is called $U$-map. It is defined as follows:
Fix a point $y \in \mathbb{R} \times Y$, let $\mathcal{M}^J_{Y, I=2}(\alpha, \beta, Z)_y$  be the holomorphic curves in $\mathcal{M}^J_{Y, I=2}(\alpha, \beta, Z)$ such that $y$ lies in the image of the curves.  Suppose that $\alpha$ and $ \beta$ are  admissible orbit sets, then  it is a compact  one dimensional manifold. The $U$-map is defined by
\begin{equation*}
U_y =\sum\limits_{\alpha  \in \mathcal{P}(Y, \omega, \Gamma) }\sum\limits_{\beta \in \mathcal{P}(Y, \omega, \Gamma)} \sum_{Z \in H_{2}(Y, \alpha, \beta)} \left(\#_2 \mathcal{M}^J_{Y, I=2}(\alpha, \beta, Z)_y  / \mathbb{R}\right) \Lambda_{Z}.
\end{equation*}
It satisfies $\partial \circ U =U \circ \partial$ and it is independent of $y$ in homology level. (Cf. \cite{HT3})
%More details please refer to  \cite{H1} and \cite{LT}.

\begin{remark}
When $d \le g(\Sigma)$, the periodic Floer homology and $U$-map are  still well-defined by using  a larger class of  almost complex structures. (Cf. \cite{LT})
\end{remark}

It is worth noting that the cases $d=\Gamma \cdot [\Sigma]=0$ or $Y = \emptyset$, in either case, $\mathcal{P}(Y, \omega, \Gamma)$ consists of only one element, the empty set. For any $J \in \mathcal{J}_{comp}(Y, \pi, \omega)$(not necessarily generic), the closed holomorphic curve in $\mathbb{R} \times Y$ has homology class $m[\Sigma]$  for some $m \ge 1$ and $I(m[\Sigma])=0 \mod 2$.  Hence, $\mathcal{M}_{Y, I=1}^{J}(\emptyset, \emptyset)= \emptyset$ and $\partial \emptyset=0$. In conclusion,    $HP(Y, \omega, \Gamma, \Lambda_P)=\Lambda_P$.

%%%%%%%%%%%%%%%%%%%%%%%%%%%%%%%%%%%%%%%%%%%%%%%%%%%%%%%
\section{Cobordism maps on HP via Seiberg Witten theory}
In this section, we  follow   \cite{HT} to define the cobordism maps $ {HP}_{sw}(X, \Omega_X, J, \Lambda_X)$ via  Seiberg Witten theory and Theorem \ref{Thm0}.
%We have organized this section in the following way:
Firstly, let us have a quick review of Seiberg Witten cohomology and its cobordism maps.  In Section \ref{section17},  we show that the isomorphism in Theorem \ref{Thm0} is canonical. This is a counterpart  of  Section 3 of \cite{HT}.   The amount of work here is relatively  fewer, because, ECH admits a filtration  structure while PFH does not.  Sections \ref{section7} and  \ref{section18}  present a convergence result which is  used to prove the holomorphic curve axiom for Theorem \ref{Thm3}. This part is corresponding to Section 7 of \cite{HT}.  With all these preparations, we prove   Theorem \ref{Thm3}.

\subsection{Review  of Seiberg Witten theory}\label{section3}
For our purpose,  we only introduce a version of Seiberg Witten cohomology and its cobordism maps defining by non-exact perturbation.   Most of what follows here paraphrase  parts of the accounts in \cite{LT} and chapter $29$ of \cite{KM}.
\subsubsection{Seiberg Witten coholomogy}
Given a closed Riemannian $3$-manifold $(Y, g)$, a $Spin^c$ structure $\mathfrak{s}$ on $Y$  is a  pair $(S, \mathfrak{cl} )$, where $S$ is a rank $2$  Hermitian
vector bundle over $Y$ and  $\mathfrak{cl}: TY \to End(S)$  is a bundle map such that
\begin{eqnarray} \label{e6}
\mathfrak{cl}(u)\mathfrak{cl}(v)+ \mathfrak{cl}(v)\mathfrak{cl}(u)=-2g(u,v),
\end{eqnarray}
for any $u$, $v \in TY$ and $\mathfrak{cl}(e_1)\mathfrak{cl}(e_2)\mathfrak{cl}(e_3)=1$,  where  $\{e_1, e_2, e_3\}$ is an  orthonormal  frame for $TY$. Let $\mathds{A}$ be a Hermitian connection on $\det S$, it determines a    $Spin^c$ connection on $S$ compatible with metric and satisfying
\begin{eqnarray} \label{e17}
\nabla_\mathds{A} (\mathfrak{cl}(u) \psi) =  \mathfrak{cl}(\nabla^g u)\psi +  \mathfrak{cl} (u) \nabla_\mathds{A} \psi,
\end{eqnarray}
 where $u \in TY$, and $\psi$ is a section of $S$,  and $\nabla^g$ is the Levi-Civita connection of $g$.
 %A $Spin^c$ connection is equivalent to a connection on $\det S$, so we don't distinguish them.
 The Dirac operator is defined by $D_{\mathds{A}} \psi = \sum\limits_{i=1}^3\mathfrak{cl}(e_i) \nabla_{ \mathds{A}, e_i}\psi$.
 The set of  isomorphism class of  $Spin^c$ structures is an affine space over $H^2(Y, \mathbb{Z})$.

Let $(Y, \pi)$ be a surface fibration over $S^1$ together with an admissible $2$-form $\omega$.  Let $J \in \mathcal{J}_{comp}(Y, \pi, \omega)$, %be a symplectization admissible almost complex structure.
 now we define a metric $g$ on $Y$ such that
\begin{itemize}
\item
  $|R|_g=1$ and $g(R, v)=0$ for any $ v \in \ker \pi_*$.
 \item
$g \vert_{\ker\pi_*} $ is given by $\omega(\cdot, J\cdot)$.
\end{itemize}
It is worth noting that $*_3 \omega = \pi^*dt$.

With  the  above choice of metric, given a $Spin^c$ structure $\mathfrak{s}=(S, \mathfrak{cl})$,  then  there is a  decomposition
\begin{equation} \label{e2}
S=E \oplus E K^{-1},
\end{equation}
where $E$ is a Hermitian line bundle and $K^{-1}=\ker \pi_*$, regards as complex line bundle via $J$. Given $\Gamma \in H_1(Y, \mathbb{Z})$, we  define a  $Spin^c$  structure $\mathfrak{s}_{\Gamma}$ such that $c_1(E)=PD(\Gamma)$.  When $E$ is the trivial line bundle, there is a canonical connection $A_{K^{-1}}$ such that $D_{A_{K^{-1}}}(1,0)=0$.  In general case, any connection  $\mathds{A}$  on $\det S$ can be written as  $2A +A_{K^{-1}}$  for some connection $A$ of $E$.

Given a connection $A$ of $E$  and a section $\psi$ of $S$, this pair $(A, \psi)$ is called configuration.
A configuration $(A ,\psi)$ satisfies the   three dimensional Seiberg Witten equations if the following equations hold:
\begin{equation} \label{e1}
  \begin{cases}
   D_{A} \psi =0 \\
*_3F_A= r(q_3(\psi )-i*_3\omega)- \frac{1}{2}*_3F_{A_{K^{-1}}} -\frac{1}{2} i*_3\wp_3, \\
  \end{cases}
\end{equation}
where $q_3(\psi )$ is a  certain quadratic form and $r$ is a large positive constant and $\wp_3$ is a closed $2$-form with cohomology class $2\pi c_1(\mathfrak{s}_{\Gamma})$.   In fact, the solutions to equations (\ref{e1}) are critical points of the Chern-Simon-Dirac functional:
\begin{equation} \label{e49}
\begin{split}
&\mathfrak{a}(A, \psi)= -\frac{1}{2}\int_Y (A-A_0)\wedge d(A-A_0)- \int_Y(A-A_0) \wedge (F_{A_0}+ \frac{1}{2}F_{A_{K^{-1}}})\\
& - \frac{i}{2}\int_Y(A-A_0)\wedge (2r\omega + \wp_3) + r\int_Y \psi^*D_A \psi,
\end{split}
\end{equation}
where $A_0$ is a fixed reference connection of $E$.

Two pairs $(A, \psi)$ and $(A', \psi')$ are  gauge equivalent if there exists $u \in Map(Y, S^1)$ such that
$A'=A-u^{-1}du$  and  $ \psi'=u\psi.$ The chain complex ${CM}^{*}(Y,  -\pi \varpi_{r}, \mathfrak{s}_{
\Gamma},  \Lambda_S)$ is a  free module generated by gauge equivalent classes  of  solutions to  (\ref{e1}). By Lemma 29.12 of \cite{KM}, there is no reducible solution $(\psi =0)$ to (\ref{e1}) whenever $\Gamma \cdot [\Sigma] \ne g(\Sigma)-1$.

Given $\mathfrak{c}_{\pm} \in {CM}^{*}(Y,  -\pi \varpi_{r}, \mathfrak{s}_{\Gamma}, \Lambda_S)$, let $(A_{\pm}, \psi_{\pm})$ be a representative of $\mathfrak{c}_{\pm} $, then the differential $<\delta \mathfrak{c}_+, \mathfrak{c}_->$ is defined by counting index one solutions to the following flow line equations (possibly with an  additional  choice of abstract perturbation):
\begin{equation}
  \begin{cases}
  \frac{\partial}{\partial s}\psi(s) + D_{A(s)} \psi(s) =0 \\
\frac{\partial}{\partial s}A(s) + *_3F_A= r(q_3(\psi )-i*_3\omega)- \frac{1}{2}*_3F_{A_{K^{-1}}} - \frac{1}{2}i*_3 \wp_3 \\
  \end{cases}
\end{equation}
with condition $\lim\limits_{s \to \pm\infty}(A(s), \psi(s))=u_{\pm} \cdot (A_{\pm}, \psi_{\pm})$, modulo gauge equivalence and $\mathbb{R}$ action.  For more details about the abstract perturbation,  please refer to \cite{KM}.  The homology of $({CM}^{*}(Y,  -\pi \varpi_{r}, \mathfrak{s}_{\Gamma}, \Lambda_S), \delta)$ is denoted by ${HM}^{*}(Y,  \mathfrak{s}_{\Gamma}, -\pi \varpi_{r}, J, \Lambda_S)$.

\subsubsection{Cobordism maps on Seiberg Witten cohomology }
Let $(Y_{+}, g_{+})$ and $(Y_{-}, g_{-})$ be two Riemannian $3$-manifolds. Let $X$ be a cobordism from $Y_+$ to $Y_-$, i.e., $\partial X=Y_+ \bigsqcup (-Y_-)$. Given  a Riemannian metric $g$ on $X$ such that $g \vert_{Y_{\pm}} =g_{\pm}$, a $Spin^c$ structure $\mathfrak{s}_X$ on $X$  consists of a rank $4$  Hermitian
vector bundle $S=S_+\oplus S_-$ over $X$,  where $S_+$ and $S_-$ are  rank $2$  Hermitian vector bundles, together with a  map $\mathfrak{cl}_X: TX \to End(S)$ satisfying  (\ref{e6}) and
\begin{equation*}
\mathfrak{cl}_X(e_1)\mathfrak{cl}_X(e_2)\mathfrak{cl}_X(e_3)\mathfrak{cl}_X(e_4)=\left(
  \begin{array}{ccc}
    -1 & 0 \\
    0 & 1 \\
  \end{array}
\right),
\end{equation*}
 whenever $\{e_1, e_2, e_3, e_4\}$ is an orthonormal frame for $TX$. The definitions of $Spin^c$ connection and Dirac operator are similar to the three dimensional case.

Let $\mathfrak{s}_{\pm}=(S_{Y_{\pm}}, \mathfrak{cl}_{Y_{\pm}})\in Spin^c(Y_{\pm})$ and $\mathfrak{s}_X=(S_+\oplus S_-, \mathfrak{cl}_X) \in  Spin^c(X)$, we say that    $\mathfrak{s}_X \vert_{Y_{\pm}}=\mathfrak{s}_{\pm}$  if  $S_{Y_{\pm}}=S_+ \vert_{Y_{\pm}}$ and $\mathfrak{cl}_{Y_{\pm}}(\cdot) = \mathfrak{cl}_X(\nu_{\pm})^{-1}\mathfrak{cl}_X(\cdot)$, where $\nu_+$ is outward unit normal vector of $Y_+$ and $\nu_-$ is  inward unit normal vector of $Y_-$.

Let $(X, \Omega_X)$ be a symplectic   cobordism from $(Y_+, \pi_+, \omega_+)$ to $(Y_-, \pi_-, \omega_-)$ and $J\in \mathcal{J}_{comp}(X,  \Omega_X)$ which agrees with symplectization-admissible almost complex structures $J_{+}$ and $J_-$ on the ends.  The metric on $\overline{X}$ is given by $g(\cdot, \cdot)=\Omega_X(\cdot, J\cdot)$. It is worth noting that $\Omega_X$ is self-dual with respect to the  metric $g$ and $|\Omega_X|_g=\sqrt{2}$. (Cf. Lemma 1 of \cite{Kim})  Moreover,  $g$ agrees with  $ds^2+ g_{\pm}$ on the ends $[0, +\infty) \times Y_+$ and $(-\infty, 0] \times Y_-$ respectively.  Similar to  three dimensional case, there is a  decomposition
\begin{equation}\label{e3}
S_+= E \oplus EK_X^{-1},
\end{equation}
where $E$ is a Hermitian line bundle and $K_X^{-1}$ is the canonical bundle of $(X,J)$. Furthermore,  the splitting (\ref{e3}) agrees with the splitting (\ref{e2}) \ on the ends $[0, \infty) \times Y_+$ and $(-\infty, 0] \times Y_- $.
Again as three dimensional case, when $E$ is the trivial bundle,  there is a canonical connection $A_{K^{-1}}$ such that $D_{A_{K^{-1}}}(1,0)=0$.

Fix  a $2$-form $\wp_{3,\pm}$ on $Y_{\pm}$  such that $[\wp_{3, \pm}]= 2\pi c_1(\mathfrak{s}_{\pm}) $ and let $\wp_4$ be a closed $2$-form on $\overline{X}$ which agrees with   $\wp_{3,+}$ and $\wp_{3,-}$  on the ends. For a connection $A$ on $E$ and a section $\psi$ of $S_+$,  the perturbed version of  $4$-dimensional Seiberg Witten  equations is
\begin{equation} \label{e4}
  \begin{cases}
D_{A} \psi =0 \\
F_{A}^+= \frac{r}{2}(q_4 (\psi) - i \Omega_X) - \frac{1}{2}F_{A_{K^{-1}}}^+ - \frac{i}{2} \wp_4^+,
  \end{cases}
\end{equation}
where $q_4(\psi)$ is a certain quadratic form and $\wp_4^+$ is the self-dual part of $\wp_4$.  After a gauge transformation, we may assume that $\nabla_A= \partial_s + \nabla_{A(s)}$ (temporal gauge) on the ends, where $A(s)$ is connection of $E \vert_{\{s\} \times Y_{\pm}}$. Given $\mathfrak{c}_{\pm} \in {CM}^{*}(Y_{\pm},  -\pi \varpi_{\pm,r}, \mathfrak{s}_{\pm},  \Lambda^{\pm}_{S})$ and their representatives  $(A_{\pm}, \psi_{\pm})$, let $\mathfrak{M}^r_{\overline{X}, ind=0}(\mathfrak{c}_+, \mathfrak{c}_-)$ be  the  moduli space of solutions to (\ref{e4})(possibly with an additional choice of abstract perturbation) with boundary conditions
\begin{equation*}
\lim\limits_{s \to \pm\infty}(A(s), \psi(s))=(A_{\pm}, \psi_{\pm}),
\end{equation*}
 modulo  gauge equivalence. The solution $(A, \psi)$ is also called instanton.

Let $\mathcal{B}_X(\mathfrak{c}_+, \mathfrak{c}_-)= \{(\mathds{A}, \Psi)\} / \backsim  $, where %$(\mathds{A}, \Psi)$ is a pair consisting of a $Spin^c$ connection and a section  of $S_+$,
$(\mathds{A}, \Psi)$ is a configuration such that $(\mathds{A}, \Psi)$ is asymptotic to a  representative of $\mathfrak{c}_{\pm}$ on each end in the sense of last paragraph, modulo the gauge transformation. Given an element in $\mathcal{B}_X(\mathfrak{c}_+, \mathfrak{c}_-)$, its relative homotopy class refers to the element in $\pi_0 \mathcal{B}_X(\mathfrak{c}_+, \mathfrak{c}_-)$.
The cobordism maps  in chain level is given by
\begin{equation*}
{CM}(X,  \Omega_X, J, r,  \Lambda_X)  = \sum\limits_{\mathfrak{c}_{+}, \mathfrak{c}_{-}} \sum_{\mathcal{Z} \in \pi_0 \mathcal{B}_X(\mathfrak{c}_+, \mathfrak{c}_-)}\#\mathfrak{M}^r_{\overline{X}, ind=0}(\mathfrak{c}_+, \mathfrak{c}_-, \mathcal{Z}) \Lambda_X(\mathcal{Z}).
\end{equation*}
This map induces a homomorphism  in homology level,    $${HM}(X,  \Omega_X, J, r, \Lambda_X ): {HM}^{*}(Y_{+},  -\pi \varpi_{+,r}, \mathfrak{s}_{+}, J_+, \Lambda_S^+) \to {HM}^{*}(Y_{-},  -\pi \varpi_{-,r}, \mathfrak{s}_{-}, J_-, \Lambda_S^-).$$
The cobordism maps  ${HM}(X,  \Omega_X, J, r, \Lambda_X )$ satisfies the natural composition rule.(Cf. Chapter 26 of \cite{KM})
\begin{remark}
In general, to define either  ${HM}^{*}(Y,  \mathfrak{s}_{\Gamma}, -\pi \varpi_{r},  \Lambda_S)$ or ${HM}(X,  \Omega_X,  r, \Lambda_X )$, one needs  to choose a suitable abstract perturbation. But for the same reasons as in \cite{Te1} and \cite{HT}, they play   a minor role in the arguments. So we suppress them form notation.

\end{remark}

\subsection{Canonical isomorphism } \label{section17}
Although not explicitly exhibited in \cite{LT}, the isomorphism in Theorem \ref{Thm0} is canonical.  The purpose of this section is to make this point clear.
\subsubsection{ Invariance of HM and cobordism maps}
The  cohomology $HM^{*}(Y, \mathfrak{s}_{\Gamma}, -\pi \varpi_r, J, \Lambda_{S})$  defined in  Section  \ref{section3}, in fact,  is   an invariant in the following sense:
\begin{theorem}[Theorem 34.4.1 \cite{KM}]
 $HM^{*}(Y, \mathfrak{s}_{\Gamma}, -\pi \varpi_r, J, \Lambda_{S})$ is independent on $J$ and $r$. It depends only on $[\omega]$, $Y$, $\Gamma$ and $\Lambda_S$ .
\end{theorem}

To say more about the theorem,  given two pairs $(\omega_0, J_0)$ and $(\omega_1, J_1)$ such that $[\omega_0]=[\omega_1]$, then we can take a  homotopy  $\rho=\{(\omega_s, J_s)\}_{s\in[0,1]}$  from   $(\omega_0, J_0)$ to  $(\omega_1, J_1)$, where $[\omega_s]=[\omega_0]$ and  $J_s \in \mathcal{J}_{comp}(Y, \pi, \omega_s)$ for each $s \in [0,1]$. Let $\{r_s\}_{s\in[0,1]}$ be a homotopy between $r_0$ and $r_1$, and fix  a generic abstract perturbation $\mathfrak{p}=\{\mathfrak{p}_s\}_{s \in [0,1]}$,  then the data $\rho$,  $\{r_s\}_{s\in[0,1]}$ and $\mathfrak{p}$ induce an isomorphism
\begin{eqnarray*}
HM(\rho, \{r_s\})_{\mathfrak{p} }: HM^{*}(Y, \mathfrak{s}_{\Gamma}, -\pi \varpi^1_{r_1}, J_1, \Lambda_S)_{\mathfrak{p}_1} \to  HM^{*}(Y, \mathfrak{s}_{\Gamma}, -\pi \varpi^1_{r_0}, J_0, \Lambda_S)_{\mathfrak{p}_0}.
\end{eqnarray*}
In fact, $HM(\rho, \{r_s\})_{\mathfrak{p} }$ is induced by the chain map
\begin{eqnarray*}
CM(\rho,\{ r_s\})_{\mathfrak{p}}: CM^{*}(Y, \mathfrak{s}_{\Gamma}, -\pi \varpi^1_{r_1}, J_1, \Lambda_S)_{\mathfrak{p}_1} \to  CM^{*}(Y, \mathfrak{s}_{\Gamma}, -\pi \varpi^0_{r_0}, J_0, \Lambda_S)_{\mathfrak{p}_0},
\end{eqnarray*}
where $CM(\rho, \{r_s\})_{\mathfrak{p} }$ is defined by counting  index zero solutions to the following equations:
\begin{equation}
  \begin{cases}
  \frac{\partial}{\partial s}\Psi(s) + D_{\mathds{A}(s)} \Psi(s) =\mathfrak{S}_s \\
\frac{\partial}{\partial s}\mathds{A}(s) + *_3F_{\mathds{A}}= (q_3(\Psi )-i*_3r_s\omega_s) - \frac{1}{2}i*_3 \wp_3 + \mathfrak{T}_s,
  \end{cases}
\end{equation}
where the pair $(\mathfrak{S}_s, \mathfrak{T}_s)$ comes from the abstract perturbation $\mathfrak{p}$.
%\begin{equation}
%  \begin{cases}
%  \frac{\partial}{\partial s}\Psi(s) + D_{\mathds{A}(s)} \Psi(s) =\mathfrak{S}_s \\
%\frac{\partial}{\partial s}\mathds{A}(s) + *_3F_{\mathds{A}}= (q_3(\Psi )-i*_3r_s\omega_s) - \frac{1}{2}i*_3 \wp_3 + \mathfrak{T}_s,
%  \end{cases}
%\end{equation}
%where the pair $(\mathfrak{S}_s, \mathfrak{T}_s)$ comes from the abstract perturbation $\mathfrak{p}_s$.

\begin{lemma} \label{C18}
$HM(\rho, r_s)_{\mathfrak{p}}$ mentioned above satisfies the following properties:
\begin{enumerate}
\item
$HM(\rho, r_s)_{\mathfrak{p} }$ only depends  on the homotopy class of  $\rho$ and $\{r_s\}_{s \in [0,1]}$.
\item
Let  $(\rho, \{{r_{s}}\}_{s \in [0,1]})$ and $(\rho', \{r'_{s}\}_{s \in [0,1]})$ be two  homotopies  such that $(\rho, \{{r_{s}}\})_{s=0} = (\rho', \{r'_{s}\})_{s=1}$, then $HM(\rho' \circ \rho , {r'_{s} \circ r_{s}}) = HM(\rho', {r'_{s}}) \circ HM(\rho, {r_{s}}) $.
\end{enumerate}
\end{lemma}
Note that any two   such data $\{\rho, \{r_s\}_{s \in[0,1]}\}$ and  $\{\rho', \{r'_s\}_{s \in[0,1]}\}$ lie inside the same homotopy class.
Therefore, $HM(\rho, r_s)$   only depends  on $\{(\omega_i, J_i)\}_{i=0,1}$ and $\{r_i\}_{i=0,1}$.

 Let $(X,   \Omega_X)$ be a symplectic  cobordism from $(Y_+, \pi_+, \omega_+)$ to $(Y_-, \pi_-, \omega_-)$, then the cobordism map $HM(X, \Omega_X,\Lambda_X)$ is invariance in the following sense:
\begin{lemma} \label{C35}
Let $\rho_{\pm}=\{(\omega_{\pm}, J_{\pm s})\}_{s \in [0,1]}$ and  $\{r_s\}_{s \in [0,1]}$ be homotopies,  and  $\Omega_{X_0}$ and $\Omega_{X_1}$  be two symplectic forms  such that $\Omega_{X_1} = \Omega_{X_2}$ in a small  collar neighborhood of $Y_{\pm}$,      then  we have  the following commutative diagram.
\begin{displaymath}
\xymatrix{
HM^*(Y_+,  \mathfrak{s}_{\Gamma_+},  -\pi \varpi^+_{r_1},  J_{+1}, \Lambda^+_S) \ar[d]^{HM(X, \Omega_{X_1}, J_1, r_1, \Lambda_X)} \ar[r]^{HM(\rho_+, r_s)} & HM^{*}(Y_+, \mathfrak{s}_{\Gamma}, -\pi \varpi^+_{r_0}, J_{+0}, \Lambda^+_S) \ar[d]^{HM(X, \Omega_{X_0}, J_0, r_0, \Lambda_X)}\\
HM^*(Y_-,   \mathfrak{s}_{\Gamma_-},  -\pi \varpi^-_{r_1},  J_{-1}, \Lambda^-_S)   \ar[r]^{ HM(\rho_-, r_s)} &HM^{*}(Y_-, \mathfrak{s}_{\Gamma_-},  -\pi \varpi^-_{r_0}, J_{-0}, \Lambda^-_S)}
\end{displaymath}
\end{lemma}
\begin{proof}[Proof of Lemma \ref{C18} and \ref{C35}]
 The lemmas can be proved by using the general outline in chapter VII of \cite{KM}. It is worth noting that  one needs to establish the finiteness results for the parameterized moduli spaces when $\{r_s\}_{s \in [0, 1]}$ is not a constant path. This can be done by using the same   argument in Section 31 of \cite{KM}.
\end{proof}

In conclusion,  Lemma \ref{C18} implies  that  $HM^{*}(Y, \mathfrak{s}_{\Gamma}, -\pi \varpi_{r}, J,  \Lambda_S)_{\mathfrak{p}}= HM^{*}(Y, \mathfrak{s}_{\Gamma}, \omega, \Lambda_S)$  is independent of $r$ and almost complex structure $J$ and abstract perturbation.  Lemma \ref{C35} gives us a well-defined cobordism map $HM(X, \Omega_X, \Lambda_X)$ between $HM^{*}(Y_{\pm}, \mathfrak{s}_{\Gamma_{\pm}}, \omega_{\pm}, \Lambda^{\pm}_S)$.

\subsubsection{Review of $Q$-$\delta$ flat approximation.}
Let $\omega$ be a $Q$-admissible $2$-form over $\pi: Y \to S^1$ and $J \in \mathcal{J}_{comp}(Y, \pi, \omega)$. Let   $\omega_{\Sigma}= \omega \vert_{\pi^{-1}(0)}$ and  $\phi_{\omega}$ be time one flow of $R$ starting at $\pi^{-1}(0)$. It satisfies $\phi_{\omega}^*\omega_{\Sigma} =\omega_{\Sigma} $.  As explained in Section \ref{section11},  $Y$ can be identified with a mapping torus $\Sigma_{\phi_{\omega}}$ via diffeomorphism  (\ref{e47}).
The definition of $Q$-$\delta$ flat  approximation of $(\phi_{\omega}, J)$ is given by the following lemma:
\begin{lemma} [Lemma 2.1, Lemma 2.2, 2.4 of  \cite{LT}] \label{C41}
Given  $Q \ge 1$,  $0 <\delta \ll 1$ and $J \in \mathcal{J}_{comp}(Y, \pi, \omega)$,  we can   modify $(\phi_{\omega}, J)$ to $(\phi_{\omega'}, J')$ such that
\begin{enumerate}
\item
 $(\phi_{\omega'}, J') = (\phi_{\omega} ,J)$ outside  a $\delta$-neighborhood of the periodic orbits with degree less than $Q$.
\item
The  set of $\phi_{\omega'}$  periodic orbits with degree less than $Q$ is  identical to  set of $\phi_{\omega}$ periodic orbits with degree less than $Q$. Moreover, if $\gamma$ is elliptic or hyperbolic as defined using $\phi_{\omega}$, then  then it is respectively elliptic or hyperbolic as defined
using $\phi_{\omega'}$ with the same rotation number.
\item
Given a periodic orbit $\gamma$ of $\phi_{\omega'}$ with degree less than $Q$,  then  there is an embedding
$$\varphi: S_{\tau}^1 \times  D_z \to \Sigma_{\phi_{\omega'}} $$
such that
\begin{enumerate}
\item
The Reeb vector field is given  by
\begin{equation*}
q\varphi_*^{-1}(R) =\partial_{\tau}-2i(\nu_{\gamma} z + \mu_{\gamma} \bar{z} +\mathfrak{r}' ) \partial_{z} + 2i(\nu_{\gamma} \bar{z} + \bar{\mu}_{\gamma} z + \bar{\mathfrak{r}}') \partial_{\bar{z}}.
\end{equation*}
where
\begin{enumerate}
\item
If $\gamma$ is elliptic, then  $\nu_{\gamma}=\frac{1}{2}\theta$ and $\mu_{\gamma}=0$  and  $\theta$ is an irrational number.
\item
If $\gamma$ is positive hyperbolic,  then $\nu_{\gamma}=0$ and $\mu_{\gamma}=-\frac{i}{4\pi} \log|\lambda|$, where $\lambda $ is the eigenvalue of the linear Poincar\'e return map with $|\lambda|>1$.
\item
 If $\gamma$ is negative hyperbolic, then  $\nu_{\gamma}=\frac{1}{4}$, $\mu_{\gamma}=-\frac{i}{4\pi} \log|\lambda|e^{i  \tau}$, where $\lambda$ is the same as above.
 \item
 $\mathfrak{r'}=0$ in a small neighborhood of $z=0$.
\end{enumerate}
\item
$\varphi^*J'(\partial_{z})=i\partial_z $ and $J'$ satisfies $|J-J'| \le c_0 \delta$ and $|\nabla(J-J')| \le c_0$.
\item
$\omega' \vert_{\{\tau\} \times D} =\frac{i}{2} dz \wedge d \bar{z}$ and $[\omega']=[\omega]$.
\end{enumerate}

\end{enumerate}
\end{lemma}

In order to make the proof of  Lemma \ref{C34} that follows  clear, we  outline   the  construction of  $Q$-$\delta$ flat approximation $(\omega', J')$    as follows:
\begin{enumerate}
\item
Let $\gamma$ be a  periodic orbit with degree $q \le Q$ and  $z=(z_1, z_2)$ be   symplectic coordinates centered on the corresponding periodic points of $\phi_{\omega}$. Construct a tubular neighborhood $\varphi_{\gamma}: S^1 \times D^* \to \Sigma_{\phi_{\omega}}$ such that
\begin{enumerate}
\mathitem
\begin{align} \label{e77}
\varphi_{\gamma *}^{-1}(\frac{1}{q} \partial_t) = \partial_{\tau} -2 i(\nu z + \mu \bar{z} +\mathfrak{r}) \partial_{z} + 2 i(\nu \bar{z} + \bar{\mu}{z} + \bar{\mathfrak{r}})\partial_{\bar{z}},
\end{align}
where $\nu \in C^{\infty}(S^1; \mathbb{R})$, $\mu \in C^{\infty}(S^1; \mathbb{C})$, and $|\mathfrak{r}| \le c_0 |z|^2$, $|\nabla \mathfrak{r}| \le c_0 |z|$.
\item
$J \varphi_{\gamma *} \partial_{z}= i\varphi_{\gamma *} \partial_{z}$ along $\gamma$.
\end{enumerate}
\item
Fix a homotopy $\{(\nu^{\lambda}, \mu^{\lambda})\}_{\lambda \in [0,1]}$ satisfying the following properties:
\begin{enumerate}
\item
$(\nu^{0}, \mu^{0}) = (\nu_{\gamma}, \mu_{\gamma})$ and $(\nu^{1}, \mu^{1})=(\nu, \mu)$, where $(\nu_{\gamma}, \mu_{\gamma})$ is defined in \cite{LT}.
\item
Let $K^{\lambda} \in SL(2: \mathbb{R})$ such that $K^{\lambda}z = \eta(2\pi)$, where $\eta$ is  a solution to the following equations:
\begin{equation*}
  \begin{cases}
 \frac{i}{2} \frac{d \eta}{d \tau} + \nu^{\lambda} \eta +  \mu^{\lambda} \bar{\eta}=0  \\
\eta(0)=z.
  \end{cases}
\end{equation*}
We require that $K^{\lambda}$ is conjugated to $K^0$ in $SL(2: \mathbb{R})$ for any $\lambda \in [0,1]$.
\end{enumerate}

\item
Fix $\rho_0 \in (0, \delta^{16})$,   and let  $\lambda^*(z) = \chi(\frac{ \log |z|}{\log \rho_0})$,  we define
$$\mathfrak{r}^{(1)} = (\nu^{\lambda^*} - \nu_{\gamma})z + (\mu^{\lambda^*} - \mu_{\gamma})\bar{z}  + \lambda^* \mathfrak{r} + \mathfrak{e}, $$
where $\mathfrak{e}$ is a generic small  perturbation which is supported in $|z| \in (\rho_0^{\frac{7}{16}}, \rho_0^{\frac{5}{16}})$.

Let $D \subset D^*$ with radius $c_0^{-1} \delta$. Define $\{u^{(1)}_{\tau}: D \to D^*\}_{\tau \in [0, 2\pi]}$ to be $u^{(1)}_{\tau} (z) = \eta(\tau)$, where $\eta$ is a solution to the following equations
\begin{equation*}
  \begin{cases}
 \frac{i}{2} \frac{d \eta}{d \tau} + \nu_{\gamma} \eta +  \mu_{\gamma} \bar{\eta} + \mathfrak{r}^{(1)} (\eta)=0  \\
\eta(0)=z.
  \end{cases}
\end{equation*}

\item
In general, $u^{(1)}_{\tau}$ is not area-preserving. But one can use Morser's trick to construct a diffeomorphism $\rho_{\tau}: D^* \to D^*$ such that $u_{\tau}^*(\frac{i}{2} dz \wedge d \bar{z} )= \frac{i}{2} dz \wedge d \bar{z}$, where $u_{\tau} = \rho_{\tau} \circ u^{(1)}_{\tau}$. Define
\begin{equation*}
\phi_{\omega}'=
  \begin{cases}
 (\varphi_{\gamma}^{-1})^*u_{\frac{ 2\pi k}{ q}}   \mbox{ \ on \  $\varphi_{\gamma}(\frac{ 2\pi k}{ q} \times D)$ for any $ k \in \{0, 1, \dots, q-1\}$ }\\
 % \ and  \ $\gamma$ \ with degree $\le Q$}\\
\phi_{\omega} \ \ \ \ \ otherwise,
  \end{cases}
\end{equation*}
then  $\phi_{\omega}'$ induces a $Q$-admissible $2$-form $\omega'$ on $Y$, it is what we require.

\item
The almost complex structure $J'$ is defined to be $J$ outside the image of $\varphi_{\gamma}$ for periodic orbits $\gamma$ with degree less than $Q$. Let $R'$ be the Reeb vector of $\omega'$, we require that  $J' \partial_s =R'$, the only ambiguity concerns
the action of $\ker \pi_*$.  In the image of $\varphi_{\gamma}$,  it is straightforward to construct an almost complex structure $J' \in \mathcal{J}_{comp}(Y, \pi, \omega')$ such that it satisfies $\varphi_{\gamma*}J'(\partial_z) = i \varphi_{\gamma*}(\partial_z)$ in a small neighborhood of $\gamma$,  $|J-J'| \le c_0 \delta$ and $|\nabla(J-J')| \le c_0$.
\end{enumerate}
According to Proposition 2.1 of \cite{LT}, the tautological identification of the respective generators  induces   a canonical isomorphism $\Psi: HP(Y, \omega, \Gamma, J, \Lambda_P) \to HP(Y, \omega', \Gamma, J', \Lambda_P) $.

The following lemma is the counterpart of Proposition B.1 of \cite{Te1}.
\begin{lemma} \label{C34}
Let $(\omega_0, J_0)$ and $(\omega_1, J_1)$ be $Q$-$\delta_i$ flat approximations of $(\omega, J)$. Then there exists a family of pairs $\{(\omega_s, J_s)\}_{s \in [0,1]}$ such that $(\omega_s, J_s)_{s=i}=(\omega_i, J_i)$ for $i=0,1$ and $(\omega_s, J_s)$ is $Q$-$\delta$ flat  approximation of $(\omega, J)$ for each $s \in [0, 1]$, where $\delta= \min \{\delta_0, \delta_1\}$.
\end{lemma}
\begin{proof}
Given  tubular neighborhoods $\varphi_{\gamma}: S^1 \times D^* \to \Sigma_{\phi_{\omega}}$ described in the first point of the outline above.  Note that the constructions of $(\omega_0, J_0)$ and $(\omega_1, J_1)$ only depend on the choice of the radius $\rho$ and homotopy $\{(\nu^{\lambda}, \mu^{\lambda})\}_{\lambda \in [0,1]}$.

Let $\delta= \min \{\delta_0, \delta_1\}$ and $\rho_*  \ll \rho_0, \rho_1\in (0, \delta^{16})$. We can take a decreasing family $\rho_s$ from $\rho_0$ to $\rho_*$. For each $s$, we can follow above outline to construct a $Q$-$\delta$ flat approximation of $(\omega, J)$, and this give a homotopy of $Q$-$\delta$ flat approximation  from $(\omega_0, J_0)$  to the $\rho_*$ version of    $(\omega_0, J_0)$. Similar conclusion can be made for $(\omega_1, J_1)$.

Based on the above understanding,   we assume that $(\omega_0, J_0)$ and $(\omega_1, J_1)$  are defined by using $\rho_*  \in (0, \delta^{16})$ but different homotopies $\{(\nu_0^{\lambda}, \mu_0^{\lambda})\}_{\lambda \in [0,1]}$ and $\{(\nu_1^{\lambda}, \mu_1^{\lambda})\}_{\lambda \in [0,1]}$ respectively. In fact, we can insert a  family of homotopies   $\{(\nu_s^{\lambda}, \mu_s^{\lambda})\}_{s, \lambda \in [0,1]}$ between them so that it satisfies (a) and (b) in the second point of our outline  for each $s \in [0,1]$. (Cf. Proof of Proposition B.1 of \cite{Te1})
Therefore, for each $s$, we can follow the outline to  construct a $Q$-$\delta$ flat approximation $(\omega_s. J_s)$ for $(\omega, J)$.
\end{proof}

%\begin{remark}
% The construction in above lemma doesn't need to assume  the  conditions \Romannum{1}, or \Romannum{2}, or \Romannum{3} and  $(X, \pi_X)$ is relative minimal.
%\end{remark}

%\[\omega_{X}'=
%\left\{
%\begin{array}
%    {r@{\quad:\quad}l}
%    \omega_X   & \mbox{on $X \setminus ((-2\varepsilon, 0 ]\times Y_+ \bigcup  (0, 2\varepsilon]\times Y_- )$} \\
%    \omega_+ + d(\phi_+(s_+) \mathfrak{a}_+) & \mbox{ on $(-2\varepsilon, 0 ]\times Y_+ $}\\
%    \omega_- + d(\phi_-(s_-)\mathfrak{a}_-) & \mbox { on $(0, 2\varepsilon]\times Y_- $},
%\end{array}
%\right.
%\]

\subsubsection{Isomorphism between HP and HM}
With the previous understanding, we show that the isomorphism in Theorem \ref{Thm0} is canonical in this subsection.
\begin{definition}
Let $F, G: CP(Y_+, \omega_+, \Gamma_+,  \Lambda_P^+) \to CP(Y_-, \omega_-, \Gamma_-,  \Lambda_P^-)$ be two homomorphisms, where $F= \sum\limits_{\alpha_+,\alpha_-}\sum\limits_{Z \in H_2(X, \alpha_+, \alpha_-)} f_Z \Lambda_Z$ and $G= \sum\limits_{\alpha_+, \alpha_-}\sum\limits_{Z \in H_2(X, \alpha_+, \alpha_-)} g_Z \Lambda_Z$.  We say that $F$ equals to $G$ up to order $L$ if $f_Z=g_Z$ for any $Z$ with $\int_Z \omega_X <L$, denoted by $F=G +o(L)$. The analogy can be defined for the maps between chain complexes of Seiberg Witten cohomology.
\end{definition}

Suppose that $(\omega, J )$ is $Q$-$\delta$ flat, by Theorem 6.1  of \cite{LT}, we have  a bijective map
\begin{eqnarray} \label{e62}
T_r: CP_*(Y, \omega, \Gamma, J, \Lambda_P) \to  CM^{-*}(Y, \mathfrak{s}_{\Gamma}, -\pi \varpi_{r}, J, \Lambda_S).
\end{eqnarray}
In addition, given $L>0$, there exists $r_L>0$ such that for any $r> r_L$, we have
\begin{equation} \label{e28}
\#\mathcal{M}_{Y, I=1}^{J, L}(\alpha, \beta) / \mathbb{R}=\#\mathfrak{M}_{Y, ind=1}^{J, L}(T_r(\alpha), T_r(\beta)) / \mathbb{R}.
\end{equation}
A priori, the isomorphism $\mathcal{T}_{r*}$ is not induced by $T_r$ straightforward in general case, because the constant $r_L$ may go to infinity as $L \to \infty$.
%But in fact the following definition of $\mathcal{T}_r$ is chain homotopy  to $T_r$  because of Lemma \ref{C16}.

To define $\mathcal{T}_r$, recall that the cohomology $HM^{-*}(Y, \mathfrak{s}_{\Gamma}, -\pi \varpi_{r}, J, \Lambda_S)$ is independent on $r$. We can take unbounded generic sequences $\{L_n\}_{n=1}^{\infty}$ and $\{r_n\}_{n=1}^{\infty}$ such that $r_n> r_{L_n}$ and a sequences of homotopies $\{r_{n, s} \vert s \in [0,1]\}_{n=1}^{\infty}$ such that $r_{n, s=0}=r$ and $r_{n, s=1}=r_n$, also a sequence of generic abstract perturbations.
Define
\begin{eqnarray}\label{e61}
\mathcal{T}_r = \lim\limits_{n \to \infty}   CM(r_{n, s}) \circ T_{r_n} :  CP_*(Y, \omega, \Gamma, J, \Lambda_P) \to  CM^{-*}(Y, \mathfrak{s}_{\Gamma}, -\pi \varpi_{r}, J, \Lambda_S).
\end{eqnarray}
%The following Lemma \ref{C16} tells us that we can take $CM(r_{n, s})= T_{r} \circ T_{r_n}^{-1}+ o(L)$, then above limit is well-defined.  By (\ref{e28}),  $\mathcal{T}_r $ is a chain map.
%The shifting $\psi_n$ guarantees that above limit is well defined.
For each $n$, $CM(r_{n, s})$ can be factored as in Section 4 of \cite{Te1}. Using the limit argument in Lemma 4.6 of \cite{Te1}, one can show that the energy of the trajectory which contributes to  each small piece is positive unless that it is close to the constant  trajectory. Moreover, the energy of the  broken trajectory which consists of "constant  trajectory" is still close to zero. Use these properties one can show that   energy of the trajectory which contributes  to   $CM(r_{n, s})$ has an $n$-independent lower bound, provided that $r$ is sufficiently large. Hence, the above limit is well defined.  Also, note that $\mathcal{T}_r $  may dependent on the choice of the sequences and homotopies, but $\mathcal{T}_{r*} $ does not.

The following lemma is  a counterpart of Lemma 3.4 of \cite{HT}. To simplify the notation, for any two $Q$-$\delta$ flat approximations, we suppress the canonical isomorphism between them in   both  chain complex level and homology level.
\begin{lemma}
Suppose that $\rho=\{(\omega_s, J_s )\}_{s \in [0,1]}$ such that $(\omega_s, J_s)$ is  $Q$-$\delta$ flat  and  $J_s \in \mathcal{J}_{comp}(Y, \pi, \omega_s)^{reg}$  for each $s\in[0,1]$. Given $L>0$, then there exists $c_L>1$ such that if $\min\limits_{s \in[0,1]} r_s \ge c_L$, $CM(\rho, r_s)$  is chain homotopy to $T_{r_0} \circ T_{r_1}^{-1}+ o(L)$.
\label{C16}
\end{lemma}
\begin{proof} The argument basically is the same as Lemma 3.4 of \cite{HT}. Arguing by contradiction,   suppose that for each positive integer $j$, there is a path $\{ r_{j,s}\vert s \in [0,1] \}$ such that the conclusion fails and $\lim\limits_{j \to \infty} \min\limits_{s \in [0, 1]} r_{j,s} = +\infty$.

For each $j$ and positive integer $k$, we  choose a family of abstract perturbations $\{\mathfrak{p}_{j,k,s} \vert s\in [0,1]\}$ suitable for defining $HM(\rho, r_{j,s})$ and satisfying
\begin{enumerate}
\item
$|\mathfrak{p}_{j,k,s}|_{\mathcal{P}} < k^{-1}$,
\item
The $\mathfrak{p}_{j,k,s}$ instantons between generators of $CM^*(Y, \mathfrak{s}_{\Gamma}, -\pi\varpi_{r_{j,s}}, J_s ,\Lambda_S )$ have nonnegative index.
\end{enumerate}

Now fix $j$ and $k$, let $N$ be a  positive large  integer. Taking a partition of $[0,1]$,
\begin{eqnarray*}
0=s_0< s_1 < \dots s_N=1 \mbox{ and } |s_i-s_{i-1}| \le \frac{2}{N}.
\end{eqnarray*}

Let $I_i:   CM^{*}(Y, \mathfrak{s}_{\Gamma}, -\pi\varpi_{r_{i}}, J_i ,\Lambda_S) \to CM^{*}(Y, \mathfrak{s}_{\Gamma}, -\pi\varpi_{r_{i-1}}, J_{i-1} ,\Lambda_S)$ be the chain map which is defined by $\rho \vert_{[s_{i-1}, s_i]}$, $r_{j, s} \vert_{[s_{i-1}, s_i]}$ and $\mathfrak{p}_{j,k, s} \vert_{[s_{i-1}, s_i]}$. Let $I=I_1 \circ  \dots \circ I_N$,  then Lemma \ref{C18} implies that $I$  is chain homotopy to $CM(\rho, r_s)$.

If $I_i$ is canonical bijection of the generators up to order $L$, i.e., $I_i= T_{r_{s_{i-1}}} \circ T_{r_{s_{i}}}^{-1} + o(L)$,  then $I= T_{r_{0}} \circ T_{r_1}^{-1} + o(L)$, contradict with our assumption. Therefore, for each $N$, we can find $i_N$ such that $I_{i_N}$ is not canonical bijection of the generators up to order $L$.

Let $N \to \infty$,  using the compactness argument as the  proof of Theorem 34.4.1 \cite{KM} and the second property of our $\mathfrak{p}_{j,k,s}$, there exists  $s_{j,k} \in [0,1]$ and an index zero, non-$\mathbb{R}$ invariant  $\mathfrak{p}_{j,k,s_{j,k}}$-instanton $\mathfrak{d}_{j,k}$ between generators of
$ CM^{*}(Y, \mathfrak{s}_{\Gamma}, -\pi\varpi_{r_{s_{j,k}}}, J_{s_{j,k}} ,\Lambda_S)$. Moreover, the   total drop in  the $s_{j,k}$-version of Chern-Simon-Dirac functional along  $\mathfrak{d}_{j,k}$ is uniformly bounded by $r_j(c_0+L)$, where $r_j= \max\limits_{s \in [0,1]} r_{j, s}$.

Passing a subsequence, we can assume that $\lim\limits_{j \to \infty} \lim\limits_{k \to \infty} s_{j,k}=s_*$.  The argument in Section 5 of \cite{LT} and Section 4 of \cite{Te4}  can  be repeated almost verbatim
to conclude the following:  the instantons  $\mathfrak{d}_{j,k}$ converges to a non-$\mathbb{R}$ invariant  $J_{s_*}$- holomorphic currents (possibly broken) $\mathcal{C}$ in $\mathbb{R} \times Y$ between generators of $CP_*(Y, \omega_{s_*}, \Gamma, J_{s_*}, \Lambda_P)$. In addition,  by Theorem 5.1 of  \cite{21}, $I(\mathcal{C})=0$.  However, our $J_{s_*}$ here is generic, this is impossible.
\end{proof}

\begin{corollary}
Suppose that $(\omega_0, J_0)$ and $(\omega_1, J_1)$ are $Q$-$\delta$ flat approximations of $(\omega, J)$, then the following diagram commutes. \label{C17}
\begin{displaymath}
\xymatrix{
HP_*(Y,  \omega_0 ,\Gamma, J_0, \Lambda_P) \ar[d]^{\Psi_{0, 1}} \ar[r]^{\mathcal{T}_{r_0*}} & HM^{-*}(Y, \mathfrak{s}_{\Gamma}, -\pi \varpi^0_{r_0}, J_0, \Lambda_S) \ar[d]^{HM(\rho, r_s)}\\
HP_*(Y,  \omega_1 ,\Gamma, J_1, \Lambda_P)   \ar[r]^{ \mathcal{T}_{r_1 *}} &HM^{-*}(Y, \mathfrak{s}_{\Gamma}, -\pi \varpi^1_{r_1}, J_1, \Lambda_S)}
\end{displaymath}
The map $\Psi_{0,1}$ is induced by canonical bijection of the generators.
\end{corollary}
\begin{proof}  By Lemma \ref{C34},  we can find a homotpy $\rho=\{(\omega_s, J_s)\}_{s \in [0,1]}$ from  $(\omega_0, J_0)$ to $(\omega_1, J_1)$ so that $(\omega_s, J_s)$ is $Q$-$\delta$ flat pairs for each $s \in[0,1]$.

Taking  increasing unbounded sequences  $\{L_n\}_{n=1}^{\infty}$ and  $\{r_n\}_{n=1}^{\infty}$ such that $r_n > \max\{r_{L_n}, c_{L_n}\}$. Let $\rho_i=\{(\omega_i, J_i)\}_{s\in [0,1]}$ and  homotopy $\{ r^i_{n, s}\}_{s \in [0,1]}$ such that $r^i_{n, s=0} =r_i$ and $r^i_{n, s=1} =r_n $, $i=0,1$.      By Lemmas \ref{C16} and \ref{C18}, we know that $HM(\rho, r_s)$ is induced by the following map
\begin{eqnarray*}
 &&CM(\rho_0, r^0_{n, s}) \circ CM(\rho, r_n) \circ CM(\rho_1, r^1_{n, s})^{-1}\\
&=&CM(\rho_0, r^0_{n, s}) \circ (T_{r_n} \circ T_{r_n}^{-1} + o(L_n)) \circ CM(\rho_1, r^1_{n, s})^{-1}.
\end{eqnarray*}
 The corollary follows from (\ref{e61}) and taking $n \to \infty$.
\end{proof}

%Given $(\omega, J)$, let $(\omega', J')$ be a $Q$-$\delta$ flat approximation of $(\omega, J)$, then we have a canonical isomorphism $\Psi: HP_*(Y, \omega, \Gamma, J, \Lambda_P) \to HP_*(Y, \omega', \Gamma, J', \Lambda_P) $ induced by tautological bijection between the generators. Let $$\mathcal{T}_{r*}':  HP_*(Y, \omega', \Gamma, J', \Lambda_P) \to HM^{-*}(Y, \mathfrak{s}_{\Gamma}, -\pi \varpi_{r}', J', \Lambda_S)  $$ be the isomorphism in Theorem \ref{Thm0}.    Take a homotopy $\rho=\{(\omega_s, J_s)\}_{s \in [0,1]}$ from $(\omega, J)$ to $(\omega', J')$, then we have an isomorphism
Let $\Psi: HP_*(Y, \omega, \Gamma, J, \Lambda_P) \to HP_*(Y, \omega', \Gamma, J', \Lambda_P) $ be the canonical isomorphism induced by  $Q$-$\delta$ flat approximation. Let $$\mathcal{T}_{r*}':  HP_*(Y, \omega', \Gamma, J', \Lambda_P) \to HM^{-*}(Y, \mathfrak{s}_{\Gamma}, -\pi \varpi_{r}', J', \Lambda_S)  $$ be the isomorphism provided by  Theorem \ref{Thm0}.    Take a homotopy $\rho=\{(\omega_s, J_s)\}_{s \in [0,1]}$ from $(\omega, J)$ to $(\omega', J')$, then we have an isomorphism
 \begin{eqnarray}\label{e52}
\mathcal{T}_{*}: HP_*(Y, \omega, \Gamma, J, \Lambda_P) \to  HM^{-*}(Y, \mathfrak{s}_{\Gamma}, [\omega], \Lambda_S),
\end{eqnarray}
where $\mathcal{T}_*$ is defined by $\mathcal{T}_*=i_r \circ HM(\rho, r) \circ \mathcal{T}_{r*}' \circ \Psi $ and $i_r:  HM^{-*}(Y, \mathfrak{s}_{\Gamma}, -\pi \varpi_{r}, J,\Lambda_S) \to  HM^{-*}(Y, \mathfrak{s}_{\Gamma}, [\omega], \Lambda_S)$ is the direct limit isomorphism. By Corollary  \ref{C17} and Lemma \ref{C18},  $\mathcal{T}_* $ is independent of the choice of the  $Q$-$\delta$ flat approximation and $r$.

\begin{corollary}[Cf. Corollary 6.7 of \cite{LT}]
$HP_*(Y, \omega, \Gamma, \Lambda_P)$  is independent on $J$.
\end{corollary}
\begin{proof}
For any two almost complex structures $J_0, J_1 \in \mathcal{J}_{comp}(Y, \pi, \omega)^{reg}$, the isomorphism between $HP_*(Y, \omega, \Gamma, J_0, \Lambda_P)$ and $ HP_*(Y, \omega, \Gamma, J_1, \Lambda_P)$ is given by $\mathcal{T}_{1*}^{-1} \circ \mathcal{T}_{0*}$.
\end{proof}

\subsection{Convergence theorem} \label{section7}
%In this subsection, we  prove a convergence theorem for a sequence of Seiberg Witten solutions to  (\ref{e4}).
%Compare with the result in \cite{Lee}, the computation is much easier in our case since both of $\Omega$, $\wp_4$ and the metric $g$ are $\mathbb{R}$-invariant on the ends.
Proposition \ref{C32} that follows asserts that  a sequence of Seiberg Witten solutions to  (\ref{e4})  give rise to broken holomorphic curves.    It is a counterpart of Proposition 7.1 of \cite{HT}. In this subsection, we assume that $(\omega_{\pm}, J_{\pm})$ is $Q$-$\delta$ flat, unless otherwise stated. Recall that $T_r^{\pm}$ is the bijection in (\ref{e62}) and $s: \overline{X} \to \mathbb{R}$ is the coordinate function.

 We prove a  convergence result for a slightly general setting than we need here.  Given a symplectic form $\Omega_X$ over $\overline{X}$, assume that  there exists a  positive function $K:  ( -\infty, \epsilon] \cup [-\epsilon, \infty)  \to \mathbb{R}$ such that   $\Omega_X = \omega_{\pm}  + K^2 ds \wedge \pi^*_{\pm} dt $ over  $ ( -\infty, \epsilon]  \times Y_- \cup [-\epsilon, \infty) \times Y_+$.   Additionally, we assume that $K$ satisfies the following properties:
\begin{enumerate}
\item
$|K-1|_{C^2} \le c_0e^{-\frac{s}{c_0}}$.
\item
There exists a positive constant $C$ such that  $ C^{-1} \le  K \le   C$.
\end{enumerate}

Take a  $J \in \mathcal{J}_{comp}(X, \Omega_X )$,  %(Now the $\Omega_X$  in second condition of  Definition \ref{def8} is $\omega_X + \pi_X(K^2\omega_B)$.),
note that $J$ is still compatible with $\Omega_X$ over $\overline{X}$. Define a metric $g=\Omega_X(\cdot, J \cdot)$. Obviously, $g$ is not product metric on the ends.  But $\Omega_X$  still is a self-dual harmonic $2$-form with respective to $g$ and $|\Omega_X|_g =\sqrt{2}$. Along the ends of $\overline{X}$, the metric $g$ satisfies the following properties:
\begin{enumerate}
\item
$g= K^2ds^2 + g_3$,  where $g_3$ is  an  $s$-dependence metric on $Y_{\pm}$.
\item
$ |R_{\pm}|_{g_3} =K$ and $g_3(v, w)= \omega_{\pm}(v, J w)$ for any $v, w \in \ker \pi_{\pm*} $.
\item
$*_3 \omega_{\pm} =K \pi_{\pm}^*dt$.
\item
For any $a \in \Omega^1(Y) $ and $b \in \Omega^2(Y)$, we have
\begin{equation*}
*_4(ds \wedge a ) =K^{-1}*_3 a, \ and \ *_4 b =K ds \wedge *_3 b.
\end{equation*}
\end{enumerate}
Write $F_A =ds \wedge E_A + F_{A(s)}$ over the ends of $\overline{X}$, then (\ref{e4}) becomes
\begin{equation}   \label{e71}
  \begin{cases}
\frac{1}{K}\nabla_{A, s} \psi + D^{g_3}_{A(s)} \psi =0 \\
\frac{1}{K} E_A + *_3 F_{A(s)} + \frac{1}{2}*_3 F_{A_{K^{-1}}} -r(q_3(\psi) -iKdt) + \frac{i}{2}*_3\wp_3=0.
  \end{cases}
\end{equation}
Note that $E_A =\partial_s A $ if $A$ is in temporal gauge.

Note that the splitting (\ref{e3}) still holds. We decompose $\psi=(\alpha, \beta)$ with respective to this splitting. We have the following relation.
\begin{equation} \label{e73}
\begin{split}
& \mathfrak{cl}_X(\Omega_X )\alpha=-2i \alpha  \ \ \mbox{and}  \ \  \mathfrak{cl}_X(\Omega_X )\beta=2i \beta  \\
& \mathfrak{cl}_Y(K dt)\alpha=i \alpha  \ \ \mbox{and}  \ \    \mathfrak{cl}_Y(K dt)\beta= -i \beta.
\end{split}
\end{equation}

%In order to state the main result, we need the following definition.
The following definition  is  an analogy of the $\Omega_X$-energy of holomorphic curve.
\begin{definition}
Given a configuration $\mathfrak{d} = (A, \psi) \in \mathcal{B}_X( \mathfrak{c}_+, \mathfrak{c}_-)$,  the $\Omega_X$  energy of  $\mathfrak{v} $ is defined as follows: $$\mathcal{F}_{\Omega_X} ( \mathfrak{d} ) =  \frac{i}{2\pi} \int_{\mathbb{R}_- \times Y_-} F_A \wedge \omega_- +   \frac{i}{2\pi} \int_{X} F_A \wedge \Omega_X +   \frac{i}{2\pi} \int_{\mathbb{R}_+ \times Y_+} F_A \wedge \omega_+.$$
\end{definition}
Note that $\mathcal{F}_{\Omega_X} ( \mathfrak{d} )$ only depends  on $\mathfrak{c}_{\pm}$ and  the relative homotopy class of $\mathfrak{d}$.
%\begin{remark}
%Note that it is  an analogy of the $\Omega_X$-energy of holomorphic curve.  Also,  $\mathcal{F}_{\Omega_X} ( \mathfrak{d} )$ only depend on $\mathfrak{c}_{\pm}$ and  the relative homotopy class of $\mathfrak{d}$.
%\end{remark}

The main result of this section is as follows:
\begin{prop}\label{C32}
Let $(A_{+}, \psi_{+})$ and $(A_-, \psi_-)$ be   solutions  to $r$-version of (\ref{e1}) and  $[(A_{\pm}, \psi_{\pm})] = T_r^{\pm}(\alpha_{\pm})$.  Let  $\mathfrak{d}=(A, \psi)$ be  a  solution to $r$-version of  (\ref{e4})  which is asymptotic to $(A_{\pm}, \psi_{\pm})$ at the ends. Assume that $\mathcal{F}_{\Omega_X} ( \mathfrak{d} )  \le L$.  Given any $\delta>0$, there exists $\kappa_{\delta} \ge 1$ and $c_0 \ge 1$ such that  the following are true  for $r \ge \kappa_{\delta} $:
\begin{enumerate}
\item
Each point in $\overline{X}$ where $|\alpha| \le 1- \delta$ has distance less than $c_0 r^{-\frac{1}{2}}$ from $\alpha^{-1}(0)$.
\item
There exist
\begin{enumerate}
\item
a positive integer $N \le c_0$ and an open cover of $\mathbb{R}$ as  $\bigcup_{1 \le k \le N }I_k$, each of length at least $2\delta^{-1}$, with $[-1, 1] \subset I_{k_0}$ and
\item
a broken holomorphic current $\{\mathcal{C}_k\}_{1\le k \le N}$ from $\alpha_+ $ to $\alpha_-$
such that  for each $k$,  we have
\begin{enumerate}
\item
 $\begin{aligned}
\sup_{z \in C \in \mathcal{C}_k} dist(z, \alpha^{-1}(0)) + \sup_{z \in \alpha^{-1}(0)} dist \left(z,  \bigcup_{C  \in \mathcal{C}_k } C \right) < \delta.
 \end{aligned} $
\item
Let $I \subset I_k$ be an interval of length $1$ and let $\nu$
denote the restriction to $s^{-1}(I)$ with $|\nu|_{\infty}=1$ and $|\nabla \nu| \le \delta^{-1}$. Then
\begin{equation*}
|\frac{i}{2\pi}\int_{s^{-1}(I) }F_{\hat{A}} \wedge \nu - \int_{\mathcal{C}} \nu | \le \delta,
\end{equation*}
where $\hat{A} = A - \frac{1}{2}(\bar{\alpha} \nabla_A \alpha - \alpha \nabla_A \bar{\alpha})$.
\end{enumerate}
\end{enumerate}
\end{enumerate}
\end{prop}

%Let us sketch the key points of the proof of the convergence theorem in ECH setting and  illustrate how to adapt them  in our setting.
The essential difference between \cite{HT} and our case is  how to get the uniformly upper bound on the function $\underline{\mathcal{M}}$, where $\underline{\mathcal{M}}$ is defined as follows.   Under the assumption in Proposition \ref{C32} and let $\mathfrak{d}=(A, \psi)$ be  a  solution to the $r$-version of  (\ref{e4}),  we define a function $\underline{\mathcal{M}}: \mathbb{R} \to \mathbb{R}$ by
\begin{equation*}
\underline{\mathcal{M}}(s)= \int_{[s, s+1]} r(1-|\alpha|^2).
\end{equation*}

Firstly, let us  summarize the analytic results which we need in the proof of Proposition \ref{C32}.
\begin{enumerate}
\item
An $r$-independent upper bound on the function $\underline{\mathcal{M}}$. The precise statement  and the proof will be given in Lemma \ref{C10} latter.
\item    \label{C46}
Provided that the last point is true, we can  obtain  the counterparts of Lemmas  5.3, 5.6, 5.7, 5.9, 5.10, 5.11 and 5.12 in  \cite{LT}. Because these lemmas  carry over almost verbatim to our setting, we will not repeat them here. We only remind that   the counterpart of Lemmas 5.12, $|\beta| \le c_0 r^{-1}$ and  $|\nabla_A \beta| \le c_0 r^{-\frac{1}{2}}$ are true at the points $x \in \overline{X}$ that has distance $c_0^{-1}$ equal or less   to $\mathbb{R}_{\pm} \times \gamma_{\pm} $ and $|s(x)| \ge 1$. But this is enough for our purpose. The proof of these Lemmas can be copied with the relevant argument in Section 3 of \cite{Te4}.  The proof there relies on the elliptic estimates for certain differential inequalities. One can deduce these differential inequalities from the Seiberg Witten equations and   $Weizenb\ddot{o}ck$ formula.  Since  elliptic estimates  are local, the proof  in  \cite{Te4} can be applied  to our cases without essential change.
We refer to the lemmas mentioned above as priori estimates for the Seiberg Witten equations (\ref{e4}).

\end{enumerate}
\begin{remark}
One can compare the  parallel results in ECH setting (Sections 3  of \cite{Te4}) with PFH setting (our case here or \cite{LT}),  the only difference is that the assumptions ${A}_{\mathfrak{v}} < r^2 $ or $f_{\mathfrak{v}} > -r^2 $ and $E(A) \le L$  in ECH setting  %corresponding Lemmas in \cite{Te4}
are replaced by $\mathcal{F}_{\Omega_X} ( \mathfrak{d} )  \le L$. These kinds of assumptions are equivalent to give an upper bound on the topological energy of Seiberg Witten equations. In Remark \ref{r5}, we will explain  a little more about the difference.   %This will be proved  in Lemma \ref{C26}.
\end{remark}
With the above analytic results in hand, we   follow   Section 4 in \cite{Te4} to sketch how to use them to prove the convergence theorem.  For the convenience of readers, we  restate some key lemmas (Lemmas  \ref{C44}, \ref{C45}) in \cite{Te4} and \cite{HT}.

The following proposition is a local version of Proposition \ref{C32}.  Roughly speaking,  it asserts that given a sequence of solutions $\{(A_{r_n}, \psi_{r_n})\}$,   there exists a subsequence of  $\{(A_{r_n}, \psi_{r_n})\}$ converges  to a holomorphic current $\mathcal{C}$ in certain sense over any compact subset of $\overline{X}$.
\begin{prop}(Cf.Prop 4.1 \cite{Te4} and Prop 7.8 \cite{HT})  \label{C44}
Under the same assumptions in Proposition \ref{C32} and fixed a $\delta>0$. Let $\mathbb{I}$ be a connected subset of $\mathbb{R}$ of length at least $2\delta^{-1}+16$.  Let $I \subset \mathbb{I}$ be a connected subset of points with distance at least $7$ from $\partial \mathbb{I}$ and $|I| \ge 2 \delta^{-1}$. Then there exists $\kappa_{\delta} \ge 1$ and $c_0 \ge 1$ such that  the following are true  for $r \ge \kappa_{\delta} $:
\begin{enumerate}
\item
Each point in $s^{-1}(I)$ where $|\alpha| \le 1- \delta$ has distance less than $c_0 r^{-\frac{1}{2}}$ from $\alpha^{-1}(0)$.
\item
There exists holomorphic current $\mathcal{C}$ defined in a  neighborhood of the
closure of $s^{-1}(I)$,  such that
\begin{enumerate}
\item $\begin{aligned}
    \sup_{z \in C \in \mathcal{C} \cap s^{-1}(I)} dist(z, \alpha^{-1}(0)) + \sup_{z \in \alpha^{-1}(0)  \cap s^{-1}(I)} dist \left(z,  \bigcup_{C  \in \mathcal{C} } C \right) < \delta.
\end{aligned}$
%\item
%\begin{equation*}
%\sup_{z \in C \in \mathcal{C} \cap s^{-1}(I)} dist(z, \alpha^{-1}(0)) + \sup_{z \in \alpha^{-1}(0)  \cap s^{-1}(I)} dist \left(z,  \bigcup_{C  \in \mathcal{C} } C \right) < \delta.
%\end{equation*}
\item
Let  $\nu$ be a $2$-form with support in $s^{-1}(I)$  such that $|\nu| \le 1$ and $|\nabla \nu| \le \delta^{-1}$, then
\begin{equation*}
|\frac{i}{2\pi}\int_{s^{-1}(I) }F_{\hat{A}} \wedge \nu - \int_{\mathcal{C}} \nu | \le \delta,
\end{equation*}
where $\hat{A} = A - \frac{1}{2}(\bar{\alpha} \nabla_A \alpha - \alpha \nabla_A \bar{\alpha})$.
\end{enumerate}
\end{enumerate}
\end{prop}
\begin{proof}(Sketch)
With the  above analytic results in hand,  Proposition 4.1  and its proof in \cite{Te4}  can be carried  over almost verbatim to our setting.
%Roughly speaking,  Proposition 4.1 in \cite{Te4} states a local convergence results, i.e. a  subsequence of $(A_r, \psi_r)$ converges  to a holomorphic current $\mathcal{C}$ in certain sense over any compact subset of $\overline{X}$.
To see how these work, let $\mathcal{F}_r$ be the current associate   to the curvature of $A_r$, i.e., $\mathcal{F}_r(\nu) =\frac{i}{2\pi} \int_{\overline{X}} F_{A_r} \wedge \nu$ for any 2-form $\nu$ with compact support in a compact set $K \subset \overline{X}$. By the analytic result above, there is an $r$-independent bound  on the $L^1$ norm of curvature, i.e., $|F_{A_r}|_{L^1(K)} \le C_K$. As a result, the current $\mathcal{F}_r$ converges weakly to a current $\mathcal{F}_{\infty}$. By  the technique called   positive cohomology assignment  which is introduced in \cite{T1}, the support  of $\mathcal{F}_{\infty}$ is a holomorphic current $\mathcal{C}$. Moreover, it satisfies the properties (a) and (b).
%Moreover, the uniform bound $\frac{i}{2\pi} \int_{\overline{X}} F_{A_r} \wedge \omega_X \le L$ gives a upper bound on $\int_{\mathcal{C}}\omega_X$.

If $s^{-1}(I)$ is non-compact, then one can take a sequence of compact sets to cover $s^{-1}(I)$.  The  argument above  supplies a holomorphic current in each compact set. The argument on page 2874-2875 of \cite{Te4} shows that the holomorphic currents over these compact sets can be patched together.
\end{proof}
%\begin{remark}
%In the proof of following proposition and the proof of Theorem \ref{A10}, we also need the counterpart of Lemma 5.3, 5.6, 5.7, 5.9, 5.10, 5.11 and 5.12 in  \cite{LT}. Because these lemmas and their proofs carry over almost verbatim to our setting, we will not repeat them here. We only remind that in the counterpart of Lemmas 5.12, $|\beta| \le c_0 r^{-1}$ and  $|\nabla_A \beta| \le c_0 r^{-\frac{1}{2}}$ are true at the points $x \in \overline{X}$ that has distance or less  $c_0^{-1}$ to $\mathbb{R}_{\pm} \times \gamma_{\pm} $ and $|s(x)| \ge 1$. But this is enough for our purpose.
%\end{remark}
Basically, Lemmas 4.6, 4.7 and 4.8 in \cite{Te4} concern the following phenomenon:  If the energy of holomorphic curve is sufficiently small,  then it closes to the trivial cylinder in certain sense. The statement and the proof there can be adapted to our setting with only notation change. We omit them here.

The following lemma is an analogy of Lemma 4.9 in \cite{Te4} and  Lemma 7.10 \cite{HT}  in our setting.
\begin{lemma}(Cf. Lemma 4.9 in \cite{Te4} and  Lemma 7.10 \cite{HT}) \label{C45}
Under the same assumptions in Proposition \ref{C32}. Let $\mathbb{I} \subset \mathbb{R}-\{0\}$ be a connected subset with length at least $16$. Given $\varepsilon>0$,  let $\mathcal{I} =\{k \in \mathbb{Z} \vert \int_{s^{-1}[k, k+1]} F_{\hat{A}} \wedge \omega  \ge \varepsilon\}$,   where $\omega = \omega_+ $ or $\omega_-$.
Let $I' \subset \mathbb{I} -(\cup_{k \in \mathcal{I}} [k, k+1])  $ be a connected component. Then $\frac{i}{2\pi}\int_{s^{-1}(I')} F_{\hat{A}} \wedge \omega  \ge -\varepsilon^2$.
\end{lemma}
\begin{proof}
% We can almost copy the proof in Lemma 4.9  of \cite{Te4}.
 Let $\hat{I}'$ be the portion of $I'$ with distance at least $16 + R_{\varepsilon'}$ from $\partial I'$, where  $R_{\varepsilon'}$  is  defined in  Lemma 4.9 of \cite{Te4}.  Let $\hat{I}''$ be the rest, i.e., $\hat{I}''=I' - \hat{I}'$.  It suffices to estimate $\frac{i}{2\pi}\int_{s^{-1}(\hat{I}')} F_{\hat{A}} \wedge \omega $ and $\frac{i}{2\pi}\int_{s^{-1}(\hat{I}'')} F_{\hat{A}} \wedge \omega $.

To obtain the lower bound for $\frac{i}{2\pi}\int_{s^{-1}(\hat{I}'')} F_{\hat{A}} \wedge \omega $, note that the argument in \cite{Te4} only use the priori estimates for Seiberg Witten equations. Since the  relevant analytic results are also true in our setting,  we can  copy the  relevant  argument to get lower bounded for $\frac{i}{2\pi}\int_{s^{-1}(\hat{I}'')} F_{\hat{A}} \wedge \omega $,.
% However,  note that when Taubes prove the lower bound for  $\int_{s^{-1}(\hat{I}')} F_{\hat{A}} \wedge  \omega$, where $\hat{I}'$ will be introduce in next paragraph,  he  use the exactness of  $d a$  in ECH setting while our $\omega$ ($\omega=\omega_+$ or $\omega_-$) is not exact. But we can make the following adjustment:

 To estimate $\int_{s^{-1}(\hat{I}')} F_{\hat{A}} \wedge  \omega $,  the proof in \cite{Te4} use the exactness of  $d a$  in ECH setting while our $\omega $  is not exact. But we can make the following adjustment:
We  decompose $\omega = \omega_0 + \sum_{\gamma} d y_{\gamma}$, where  $\gamma$ runs over  simple periodic orbits with degree less than $Q$,   $y_{\gamma}$ is a $1$-form supports in a $10\varepsilon'$ neighborhood  of $\gamma$ and $\omega_0$  supports  away from $\gamma$. The term
$\int_{s^{-1}(\hat{I}')} F_{\hat{A}} \wedge  dy_{\gamma} = \int_{s^{-1}(\partial \hat{I}')} F_{\hat{A}} \wedge y_{\gamma}$ can be estimated as (4-14) in \cite{Te4}.  As a result, we obtain the lower bound for  $\int_{s^{-1}(\hat{I}')} F_{\hat{A}} \wedge  dy_{\gamma}$.   Let $\mathcal{C}$ be the holomorphic current provided by Proposition \ref{C44}, (\cite{Te4} uses notation $\vartheta$), therefore, $$|\int_{s^{-1}(\hat{I}')} F_{\hat{A}} \wedge  \omega_0 - \int_{\mathcal{C} \cap  s^{-1}(\hat{I}') } \omega_0| < \delta.$$ Note that the argument in Lemma 4.9 of  \cite{Te4} show that $\mathcal{C} \vert_s$ is contained in a $2\varepsilon'$ neighborhood of $\gamma$ for any $s \in \hat{I}'$, thus $\int_{\mathcal{C} \cap  s^{-1}(\hat{I}') } \omega_0 =0$.  Hence,  $\int_{s^{-1}(\hat{I}')} F_{\hat{A}} \wedge  \omega_0 > -\delta$.
\end{proof}
The consequence of Lemma \ref{C45} is that there are only finitely many elements in $\mathcal{I}$. The solutions $\{(A_r, \psi_r)\}_r$ over $\mathbb{I} -(\cup_{k \in \mathcal{I}} [k, k+1])$ have small energy. Apply Proposition \ref{C44} to $\{(A_r, \psi_r)\}_r$, $\{(A_r, \psi_r)\}_r$ supplies  a nontrivial holomorphic current in each connected component of $s^{-1}(\cup_{k \in \mathcal{I}} [k, k+1])) $, while in each connected component of $s^{-1}(\mathbb{I} -(\cup_{k \in \mathcal{I}} [k, k+1]) )$, the holomorphic current  is close to the trivial cylinders and they are corresponding to ends of holomorphic curves. Roughly speaking,  the above construction supplies us with  the broken holomorphic current promised in Proposition \ref{C32}.  For more  details, we refer the reader to Section 4 of \cite{Te4}.

 \subsection{Energy Bound} \label{section18}
In this subsection, we prove the promised bound for $\underline{\mathcal{M}}$.
Firstly, let us introduce some functionals as follows:
\begin{enumerate}
\mathitem
\begin{align*}
\mathfrak{S}(A, \psi)=\big( \frac{F^+_{\mathds{A}}}{2} - \mathfrak{cl}_X^{-1}(\Psi \Psi^*)_0 + i\frac{\eta}{4}, D_{\mathds{A}} \Psi \big),
\end{align*}
where $\mathds{A}=2A+A_{K^{-1}}$, $\Psi=\sqrt{2r}\psi$ and $\eta= 2r\Omega_X + 2\wp_4^+ $.

\item
As in \cite{KM}, we define the topological energy and analytic energy as follows:
%\begin{enumerate}
\begin{enumerate}
\mathitem
\begin{align*}
\mathscr{E}_{anal}(X_R)(A,\psi)&= \frac{1}{4} \int_{X_R} |F_{\mathds{A}}|^2 + \int_{X_R} |\nabla_{\mathds{A}} \Psi|^2  + \int_{X_R} 2|\frac{i}{4}\eta - \mathfrak{cl}_{X_R}^{-1}(\Psi \Psi^*)_0 |^2 \\
&\quad+\frac{1}{4}\int_{X_R} R_g |\Psi|^2 - i \int_{X_R} F_{\mathds{A}} \wedge    \frac{1}{2}*\wp_4,
\end{align*}
 where $X_R=\{x \in \overline{X} \vert | s(x)| \le R \}$.

\mathitem
\begin{align*}
\mathscr{E}_{top}(X_R)(A,\psi)&= \frac{1}{4} \int_{X_R} F_{\mathds{A}}\wedge F_{\mathds{A}}   - \int_{\partial X_R} <\Psi, D_{\mathds{A}}\Psi> +  i \int_{X_R} F_{\mathds{A}} \wedge (r\Omega_X + \frac{1}{2}\wp_4)\\
& + \int_{\partial X_R} \frac{H}{2} |\Psi|^2,
\end{align*}
 where   $H$ is mean curvature.
\end{enumerate}

%\item
%\begin{eqnarray*}  \displaystyle
%\mathscr{E}_{anal}(X_R)(A,\psi)&=& \frac{1}{4} \int_{X_R} |F_{\mathds{A}}|^2 + \int_{X_R} |\nabla_{\mathds{A}} \Psi|^2  + \int_{X_R} 2|\frac{i}{4}\eta - \mathfrak{cl}_{X_R}^{-1}(\Psi \Psi^*)_0 |^2 \\
%&+&\frac{1}{4}\int_{X_R} R_g |\Psi|^2 - i \int_{X_R} F_{\mathds{A}} \wedge (r \pi_X^*\omega_B  + \frac{1}{2}*\wp_4).
%\end{eqnarray*}
%
%\item
%\begin{equation*}
%\mathscr{E}_{top}(X_R)(A,\psi)= \frac{1}{4} \int_{X_R} F_{\mathds{A}}\wedge F_{\mathds{A}}   - \int_{\partial X_R} <\Psi, D_{\mathds{A}}\Psi> +  i \int_{X_R} F_{\mathds{A}} \wedge (r\omega_X + \frac{1}{2}\wp_4),
%\end{equation*}
% where $X_R=\{x \in \overline{X} \vert | s(x)| \le R \}$.
%\end{enumerate}
\item
We follow \cite{LT} to define a functional  $Q_F$  as follows:
 \begin{eqnarray*}
%\mathfrak{a}(A, \psi)&=& -\frac{1}{2}\int_{Y_+} (A-A_0)\wedge d(A-A_0)- \int_{Y_+}(A-A_0) \wedge (F_{A_0}+ \frac{1}{2}F_{A_{K^{-1}}})\\
%&& - \frac{i}{2}\int_{Y_+}(A-A_0)\wedge (2r\omega_+ + \wp_3) + r\int_{Y_+} \psi^*D_A \psi\\
&&Q_F(A)=i \int_{Y_{\pm}}(A-A_0) \wedge \omega_{\pm},
 \end{eqnarray*}
where $A_0$ is a fixed reference connection of $E \to Y_{\pm}$.  We abuse notation $A_0$ to denote a reference connection of $E \to \overline{X}$ such that $A_0$ is $\mathbb{R}$ invariant on the ends.
\item
Let $\mathfrak{a}_{\pm}$ be the Chern-Simon-Dirac functional for $Y_{\pm}$. Recall that we have defined them  in (\ref{e49}).
\end{enumerate}

\begin{lemma} \label{C50}
Assume that $\wp_4$ is a closed $2$-form such that $\wp_4= \wp_3$ when $|s|\ge 0$. Suppose that  $\mathfrak{d}=(A, \psi)$ is a  solution to $r$-version of  (\ref{e4})  which is asymptotic to  solutions  $(A_{+}, \psi_{+})$ and $(A_-, \psi_-)$ to (\ref{e1}) at the ends.  Then there exists a constant $c_0>0$ such that for $r \ge c_0$, we have
\begin{enumerate}
\item
$|\alpha| \le 1+ \frac{c_0}{r}$,
\item
$|\beta|^2 \le c_0 r^{-1}(1-|\alpha^2|) + \frac{c^2_0}{r^2}$.
\end{enumerate}
\end{lemma}
\begin{proof}
The proof is the same as Lemma 7.3 of \cite{HT}.
\end{proof}

The next lemma shows that the Chern-Simon-Dirac functional  is  almost decreasing. Since the metric $g_3$ varies along the ends, the  Chern-Simon-Dirac functional  $\mathfrak{a} (\mathfrak{d} \vert_{s })$ should be understood as the one  defined by metric $g_3 \vert_{\{s\} \times Y}$.
\begin{lemma}  \label{C51}
Assume that $\wp_4$ is a closed $2$-form such that $\wp_4= \wp_3$ when $|s|\ge 0$. Suppose that  $\mathfrak{d}=(A, \psi)$ is a Seiberg Witten solution to $r$-version of  (\ref{e15})  which is asymptotic to  solutions  $(A_{+}, \psi_{+})$ and $(A_-, \psi_-)$ to (\ref{e1}) at the ends.  For any $0<s_- \le s_+$ or $s_- \le s_+ <0$, we have
\begin{equation*}
\begin{split}
  &  \ \ \ \ \mathfrak{a} (\mathfrak{d} \vert_{s_-}) - \mathfrak{a} (\mathfrak{d}\vert_{s_+}) \\
  & = \frac{1}{2}  \int_{s_-}^{s_+}\Bigg( \int_Y(K |\mathfrak{B}(A,\psi)|^2 + K^{-1}|E_A |^2  + 2rK |D_A^{g_3} \psi|^2 + 2rK^{-1}| \nabla_{A,s} \psi|^2 )vol_{g_3} \Bigg)ds\\
&+   \int_{s_-}^{s_+}  \Bigg( \int_Y \left( r  <\psi, (\partial_sD^{g_3}) \psi>   +r h(\psi, D^{g_3}_A \psi) \right) vol_{g_3} \Bigg) ds, \\
 %+ (K^{-1}-1) \int_Y \partial_s A\wedge ir \omega. \\
%&+\int_{s \in [s_-, s_+]} ( K^{-1}-1) \left(  r^2( 1 - K|\alpha|^2  +  K|\beta|^2) vol_{g_3} + r^2O(|\alpha||\beta|)\right)
\end{split}
\end{equation*}
where $\mathfrak{B}(A,\psi)=*_3F_A-r(q_3(\psi) - i Kdt) +\frac{1}{2}*_3F_{A_{K^{-1}}} + \frac{i}{2}*_3\wp_3$, $h=\partial_s g$, and $\partial_s (D^{g_3})$ is derivative  of the Dirac operator defined by metric $g_3$ along $s$ direction.
Furthermore,
%$ \mathfrak{a}(\mathfrak{d} \vert_{s_-}) - \mathfrak{a}( \mathfrak{d}\vert_{s_+})  \ge -c_0 r.$
\begin{equation} \label{e72}
\begin{split}
    \mathfrak{a}(\mathfrak{d} \vert_{s_-}) - \mathfrak{a}(\mathfrak{d}\vert_{s_+}) & \ge  \frac{1}{4C}  \int_{s \in [s_-, s_+]}(  |\mathfrak{B}(A,\psi)|^2 + |E_A|^2) + 2r(  |D_A^{g_3} \psi|^2 + | \nabla_{A, s} \psi|^2 )-c_0r.
 %+ (K^{-1}-1) \int_Y \partial_s A\wedge ir \omega. \\
%&+\int_{s \in [s_-, s_+]} ( K^{-1}-1) \left(  r^2( 1 - K|\alpha|^2  +  K|\beta|^2) vol_{g_3} + r^2O(|\alpha||\beta|)\right)
\end{split}
\end{equation}
\end{lemma}
\begin{proof}
Without loss of generality, assume that $(A ,\psi)$ is in  temporal gauge. By direct  computation, we have
\begin{equation*}
\begin{split}
& \partial_s \mathfrak{a}(\mathfrak{d} \vert_{s})=
   -\int_Y \partial_s A \wedge (F_{A(s)} + \frac{1}{2}F_{A_{K^{-1}}} + ir\omega + \frac{i}{2}\wp_3) + r\int_Y \left(  < \partial_s \psi, D^{g_3}_A \psi>  +   <  \psi, D^{g_3}_A   ( \partial_s\psi)> \right)\\
&+  r\int_Y  <\psi, \mathfrak{cl}_Y(\partial_sA) \psi>  +   r\int_Y <\psi, (\partial_sD^{g_3}) \psi>  + r\int_Y h(\psi, D^{g_3}_A \psi)\\
  & =   -\int_Y \partial_s A \wedge (F_{A(s)} + \frac{1}{2}F_{A_{K^{-1}}} -r *_3q_3(\psi)+ ir\omega + \frac{i}{2}\wp_3) + r\int_Y 2Re < \partial_s \psi, D^{g_3}_A \psi> \\
 &+   r\int_Y <\psi, (\partial_sD^{g_3}) \psi>   +r\int_Y h(\psi, D^{g_3}_A \psi).
\end{split}
\end{equation*}
Then the conclusion follows from the flow line equations (\ref{e71}).

Using Cauchy-Schwarz  inequality, the term   $rh(\psi, D^{g_3}_A \psi)$ can be absorbed by the term  $2r K |D_{A}^{g_3} \psi|^2$. By the assumption on $K$, we have $  |h|+|\partial_s (D^{g_3})| \le c_0e^{-\frac{s}{c_0}}  $. Combine this with Lemma \ref{C50}, we get the inequality  (\ref{e72}).

\end{proof}

\begin{lemma}
Assume that $\wp_4$ is a closed $2$-form such that $\wp_4= \wp_3$ when $|s|\ge 0$. Suppose that  $\mathfrak{d}=(A, \psi)$ is a Seiberg Witten solution to $r$-version of  (\ref{e4})  which is asymptotic to  solutions  $(A_{+}, \psi_{+})$ and $(A_-, \psi_-)$  to (\ref{e1}) at the ends.  Assume that  $ \mathcal{F}_{\Omega_X}(\mathfrak{d}) \le L$, then there exists  a constant $c_0>0$ such that for $r \ge c_0$, we have \label{C26}
\begin{enumerate}
\item
$|\mathfrak{a}( \mathfrak{d} \vert_{\partial X})| \le c_0(1+ 2\pi L) r$, where $\mathfrak{a}( \mathfrak{d} \vert_{\partial X})= \mathfrak{a}_+( \mathfrak{d} \vert_{\{0\} \times Y_+}) -  \mathfrak{a}_-( \mathfrak{d} \vert_{\{0\} \times Y_-})$.
\item
$\mathscr{E}_{top}(X_R)(A,\psi) \le  c_0 (1+2\pi L )r + c_0 R r.$

\item
$  \frac{1}{8} \int_X |F_{\mathds{A}}|^2 + \int_X |\nabla_{\mathds{A}} \Psi|^2  + 2\int_X |\frac{i}{4}\eta - \mathfrak{cl}_X^{-1}(\Psi \Psi^*)_0 |^2  \le c_0(1+ 2\pi L) r$.
\item
$\underline{\mathcal{M}}(s)=\int_{[s, s+1]} r(1-|\alpha|^2) \le c_0 (1+2\pi L ) r^{\frac{1}{2}}$.
\end{enumerate}
\end{lemma}
\begin{proof} For the  sake of simplicity, we  assume that $Y_-=\emptyset$.
The proof is divided into the following four steps.
\begin{enumerate} [label=\textbf{Step \arabic*}]
\item
%Given $s_+ > s_- \ge 0$,   we have (Cf. \cite{HT} Lemma 7.6 or \cite{Te4}  Lemma 3.4 )
%\begin{equation} \label{e25}
%\begin{split}
%& \ \ \ \ \  \mathfrak{a}( \mathfrak{d} \vert_{s=s_-}) - \mathfrak{a}( \mathfrak{d} \vert_{s=s_+})\\
%&= \frac{1}{2} \int_{s \in [s_- ,s_+]} (|E_A|^2 + |\mathfrak{B}(A,\psi)|^2 + 2r( |\nabla_{A, s} \psi|^2 + |D_{A(s)} \psi|^2)),
%\end{split}
%\end{equation}
%where $E_A=F_A(\partial_s, \cdot)$ and $\mathfrak{B}(A,\psi)=*_3F_A-r(q_3(\psi) - i dt) +\frac{1}{2}*_3F_{A_{K^{-1}}} + \frac{i}{2}*_3\wp_3$.
%Therefore, $\mathfrak{a}( \mathfrak{d} \vert_{s=s_-}) \ge \mathfrak{a}( \mathfrak{d}\vert_{s=\infty})$.
 Note that  $\mathfrak{a}(A_+, \psi_+) + rQ_F(A_+)$ is gauge-invariant.  By Lemma 5.1 in \cite{LT},
\begin{equation} \label{e53}
 |\mathfrak{a}( A_+ ,\psi_+) + rQ_F(A_+, \psi_+)| \le c_0 r.
\end{equation}
  On the other hand, $Q_F(A_+, \psi_+) = 2\pi( \mathcal{F}_{\Omega_X}(\mathfrak{d}) -\mathcal{F}_{\Omega_X}(A_0))
 $. Therefore,
\begin{equation} \label{e63}
Q_F(A_+, \psi_+) \le 2\pi L+ c_0
\end{equation}
provided that $\mathcal{F}_{\Omega} (\mathfrak{d}) \le L$.    Combine Lemma \ref{C51}, (\ref{e53}) and (\ref{e63}), we have
\begin{eqnarray} \label{e26}
\mathfrak{a}( \mathfrak{d} \vert_{\partial X})\ge \mathfrak{a}(A_+, \psi_+) -c_0 r \ge - (c_0+ 2\pi L) r.
 \end{eqnarray}
By direct  computation, we have
\begin{equation} \label{e48}
\mathscr{E}_{top}(X)(A,\psi)= \frac{1}{4} \int_X F_{\mathds{A}_0}\wedge F_{\mathds{A}_0}   + \frac{i}{4} \int_X F_{\mathds{A}_0} \wedge (r\Omega_X + \frac{1}{2}\wp_4) - 2\mathfrak{a}( \mathfrak{d} \vert_{\partial X}) + \int_{\partial X} rH|\psi|^2.
\end{equation}
By our assumption, $|H| \le c_0$. Therefore, $\mathscr{E}_{top}(X)(A,\psi) \le c_0(1+ 2\pi L)r$.

\item
Replace $X$ by $X_R$  and $\mathfrak{a}( \mathfrak{d} \vert_{\partial X})$ by  $\mathfrak{a}( \mathfrak{d} \vert_{\partial X_R})$ in (\ref{e26}) and  (\ref{e48}), then we complete the proof of  the second assertion in the lemma. Note that  the  extra term $c_0 Rr $ comes from the energy $\frac{i}{4} \int_{X_R} F_{\mathds{A}_0} \wedge r\Omega_X $.

\item
Follows  from Lemma \ref{C50} and  Cauchy-Schwarz inequality, we have
\begin{eqnarray*}
\mathscr{E}_{anal}(X)(A,\psi) &\ge&  \frac{1}{4} \int_X |F_{\mathds{A}}|^2 + \int_X|\nabla_{\mathds{A}} \Psi|^2  +  \int_X 2|\frac{i}{4}\eta - \mathfrak{cl}_X^{-1}(\Psi \Psi^*)_0 |^2 \\
 &-& i\int_X F_{\mathds{A}} \wedge \frac{1}{2}*\wp_4- c_0r\\
&\ge & \frac{1}{4} \int_X |F_{\mathds{A}}|^2 + \int_X |\nabla_{\mathds{A}} \Psi|^2  + \int_X 2|\frac{i}{4}\eta - \mathfrak{cl}_X^{-1}(\Psi \Psi^*)_0 |^2  - \frac{1}{8}\int_X |F_{\mathds{A}}|^2 \\
&-& c_0\int_X |\wp_4|^2- c_0r\\
&\ge & \frac{1}{8} \int_X |F_{\mathds{A}}|^2 + \int_X|\nabla_{\mathds{A}} \Psi|^2  + \int_X 2|\frac{i}{4}\eta - \mathfrak{cl}_X^{-1}(\Psi \Psi^*)_0 |^2  - c_0r.
\end{eqnarray*}

By direct computation, we know that
$$|\mathfrak{S}(A, \psi)|^2_{L^2(X_R)}=\mathscr{E}_{anal}(X_R)(A,\psi)- \mathscr{E}_{top}(X_R)(A,\psi)$$ for any $R \ge 0$ (Cf. Section 4.5 of \cite{KM}).
In particular,  $\mathscr{E}_{top}(X)(A,\psi)= \mathscr{E}_{anal}(X)(A,\psi)$ when $(A, \psi)$ is a solution to (\ref{e4}).  Thus we have
\begin{eqnarray*}
&&\frac{1}{4} \int_X F_{\mathds{A}_0}\wedge F_{\mathds{A}_0}   + i  \int_X F_{\mathds{A}_0} \wedge (r\Omega_X +\frac{1}{2}\wp_4) - 2\mathfrak{a}( \mathfrak{d} \vert_{\partial X}) \\
&\ge & \frac{1}{8} \int_X |F_{\mathds{A}}|^2 +  \int_X|\nabla_{\mathds{A}} \Psi|^2  + \int_X 2|\frac{i}{4}\eta - (\Psi \Psi^*)_0 |^2  -c_0 r\\
&\Longrightarrow &    \frac{1}{8} \int_X |F_{\mathds{A}}|^2 + \int_X |\nabla_{\mathds{A}} \Psi|^2  + \int_X 2 |\frac{i}{4}\eta - (\Psi \Psi^*)_0 |^2  \le  - 2\mathfrak{a}( \mathfrak{d} \vert_{\partial X}) + c_0r.
\end{eqnarray*}
In particular, $\mathfrak{a}( \mathfrak{d} \vert_{\partial X}) \le c_0 r$, and hence $|\mathfrak{a}( \mathfrak{d} \vert_{\partial X})| \le c_0(1+ 2\pi L) r$. Moreover,
\begin{eqnarray} \label{e50}
\frac{1}{8} \int_X |F_{\mathds{A}}|^2 + \int_X |\nabla_{\mathds{A}} \Psi|^2  + \int_X 2 |\frac{i}{4}\eta - \mathfrak{cl}_X^{-1}(\Psi \Psi^*)_0 |^2  \le c_0(1+ 2\pi L) r.
\end{eqnarray}
These deduce the first and third bullets of the lemma.
\item
To prove the last bullet of the lemma.  By relation (\ref{e73}) and  Lemma \ref{C50}, we have
\begin{eqnarray} \label{e51}
\int_{X} 2|\frac{i\eta}{4} - \mathfrak{cl}_X^{-1}(\Psi \Psi^*)_0|^2 = \int_{X} r^2(1-|\alpha|^2)^2 + \mathfrak{e},
\end{eqnarray}
where $|\mathfrak{e}| \le \frac{1}{1000}\int_{X} r^2(1-|\alpha|^2)^2 +  c_0 r$.
Combine this with (\ref{e50}), we have
\begin{equation}\label{e64}
\int_X r(1-|\alpha|^2)^2 \le c_0(1+ 2\pi L) .
\end{equation}
By $H\ddot{o}lder$ inequality,
\begin{eqnarray}
\int_X r(1-|\alpha|^2) \le c_0 r^{\frac{1}{2}} \left(\int_{X} r(1-|\alpha|^2)^2 \right)^{\frac{1}{2}} \le c_0(1+ 2\pi L)^{\frac{1}{2}} r^{\frac{1}{2}}.
\end{eqnarray}

To estimate $\int_{[s, s+1] \times Y_+} r(1-|\alpha|^2)$ for $s \ge 0$,  the argument is the same as the proof of Lemma 5.4 of \cite{LT}. By the  first bullet of the lemma and inequalities  (\ref{e26}) and Lemma \ref{C51},
\begin{equation} \label{e65}
 \int_{\mathbb{R}_+ \times Y_+} (|E_A|^2 + |\mathfrak{B}(A,\psi)|^2 + 2r (|\nabla_{A, s} \psi|^2 + |D_{A(s)} \psi|^2) )\le c_0(1+ 2\pi L) r,
\end{equation}
in particular, $\int_{[s, s+1] \times Y_+} |E_A|^2 \le c_0(1+ 2\pi L) r $.

 By the Seiberg Witten equations (\ref{e4}), (\ref{e73}) and  Lemma \ref{C50},
\begin{align}\label{e15}
%\begin{split}
&i\int_{[s, s+1] \times Y_+} K^2ds \wedge \pi_+^*dt \wedge (K^{-1}*_3  E_A)  + i\int_{[s, s+1] \times Y_+} K^2ds \wedge \pi_+^*dt \wedge (F_{A(s)} + \frac{1}{2}F_{A_{K^{-1}}}) \nonumber \\
=& \int_{[s, s +1] \times Y_+} r(1-|\alpha|^2) + \mathfrak{e},
%\end{split}
\end{align}
where $|\mathfrak{e}| \le \frac{r}{100} \int_{[s,s+1] \times Y_+} (1-|\alpha|^2) + c_0$.

The first term  of (\ref{e15})
$$|i\int_{[s, s+1] \times Y_+} K  ds \wedge \pi_+^*dt \wedge *_3 E_A| \le c_0 \left(  \int_{[s, s+1] \times Y_+} |E_A|^2\right)^{ \frac{1}{2}} \le c_0(1+2\pi L) r^{\frac{1}{2}},$$
and the second term of (\ref{e15}) is $c_1(\mathfrak{s}_{\Gamma_+} ) \cdot [\pi_+^*(Kdt)]$, these give the promised  bound on  $\int_{[s, s+1] \times Y_+} r(1-|\alpha|^2)$  in the lemma.

\end{enumerate}
\end{proof}

\begin{lemma}\label{C10}
Assume that $\wp_4$ is a closed $2$-form such that $\wp_4= \wp_3$ when $|s|\ge 0$. Suppose that  $\mathfrak{d}=(A, \psi)$ is a Seiberg Witten solution to $r$-version of  (\ref{e4})  which is asymptotic to  solutions   $(A_{+}, \psi_{+})$ and $(A_-, \psi_-)$  to (\ref{e1}) at the ends to (\ref{e1}) at the ends. Assume that  $\mathcal{F}_{\Omega} (\mathfrak{d})  \le L$, then there exists a constant $c_0>0$ such that for $r \ge c_0$,
\begin{equation*}
\underline{\mathcal{M}}(s)= \int_{[s, s+1]} r(1-|\alpha|^2) \le c_0(1+2\pi L).
\end{equation*}
\end{lemma}
\begin{proof}
Fix $R \ge 2$, note that the conclusions in Lemma \ref{C26} are still true  by replacing $X$ by $X_R$.
Let $w=1-|\alpha|^2$, then $w$ satisfies the following equation (Cf. the proof of Lemma 2.2 of \cite{T1})
\begin{eqnarray} \label{e12}
\frac{1}{2}d^*dw + r|\alpha|^2 w -|\nabla_A \alpha|^2 + \mathfrak{e}_w=0,
\end{eqnarray}
where $|\mathfrak{e}_w| \le c_0(|\alpha|^2 + |\beta|^2 + |\nabla_A \beta|^2)$.

Let $\chi$ be a cut-off function such that $\chi=1$ on $X_{\frac{1}{2}R}$ and $\chi=0$ on $\overline{X} \backslash X_{\frac{3}{4}R}$. Multiply both sides of the equality (\ref{e12}) by $\chi$, then we have
\begin{equation} \label{e13}
\frac{1}{2}d^*d(\chi w)- \frac{1}{2}w(d^*d\chi) - <d\chi, dw>   + \chi r|\alpha|^2 w - \chi|\nabla_A \alpha|^2 + \chi\mathfrak{e}_w=0.
\end{equation}
Integrate both sides of (\ref{e13}), then
\begin{eqnarray*}
\int_{X_R} \chi r|\alpha|^2 w  &\le&  \int_{X_R} \chi|\nabla_A \alpha|^2   +  \int_{X_R} 2|d\chi||\nabla_A \alpha| +  \int_{X_R} \chi |\mathfrak{e}_w| +  c_0 \\
&\le& c_0 \int_{X_R} (|\nabla_A \alpha|^2 +  |\nabla_A \beta|^2 )  + c_0.
\end{eqnarray*}
By the third point of Lemma \ref{C26},  $\int_{X_R} (|\nabla_A \alpha|^2 +  |\nabla_A \beta|^2 ) \le c_0(1+2\pi L)$. Hence,  $\int_{X_{\frac{1}{2}R}} r|\alpha|^2 w \le c_0(1+2\pi L)$.

On the other hand,  inequality (\ref{e64}) tells us that that $$\int_{X_{\frac{1}{2}R}} r w^2=\int_{X_{\frac{1}{2}R}} r(1-|\alpha|^2)^2 \le c_0 (1+2\pi L).$$
%\begin{eqnarray*}
%\int_{X_R} 2|\frac{i\eta}{4} - \mathfrak{cl}_X^{-1}(\Psi \Psi^*)_0|^2 = \int_{X_R} r^2(1-|\alpha|^2)^2 + \mathfrak{e} \le c_0(1+2\pi L)r,
%\end{eqnarray*}
%where $|\mathfrak{e}| \le \frac{1}{1000}\int_{X} r^2(1-|\alpha|^2)^2 +  c_0 r$. Thus
Therefore, $ \int_{X_{\frac{1}{2}R}} r(1-|\alpha|^2) = \int_{X_{\frac{1}{2}R}} rw^2 + \int_{X_{\frac{1}{2}R}} r|\alpha|^2 w  \le c_0(1+2\pi L)$. In particular, $\underline{\mathcal{M}}(0) \le c_0 (1+2\pi L)$.

To bound $\underline{\mathcal{M}}(s)=\int_{[s,s+1]}r(1-|\alpha|^2)$ for $s>0$, we may assume that $\underline{\mathcal{M}}(s)$ attains its maximum at some $s_* \in (1, \infty)$. Let $\chi_*$ be a cut off function such that $\chi_*=1$ on $s_*-1 \le s \le s_* +1 $ and $\chi_*=0$ whenever $|s-s_*| \ge 2$.     %Lemma \ref{C26} and inequalities  (\ref{e26}) and (\ref{e25}),
%\begin{eqnarray*}
%&& \int_{s \in [0 , +\infty)} (|E_A|^2 + |\mathfrak{B}(A,\psi)|^2 + 2r |\nabla_{A(s), s} \psi|^2 + 2r |D_{A(s)} \psi|^2) \le  c_0(1 + 2\pi L) r.
%\end{eqnarray*}

Recall the inequality (\ref{e65}), apply $Weizenb\ddot{o}ck$ formula for $D_{A(s)}$ and use the fact that  $\int_{\{s\} \times Y_+} \frac{i}{2\pi}F_{A(s)} \wedge \pi_+^*(Kdt) = c_1(E) \cdot [\pi_+^*(Kdt) ] $, we have
\begin{eqnarray} \label{e54}
c_0(1+2\pi L)r &\ge& \int_{ [0 , +\infty) \times Y_+} (|E_A|^2 + |\mathfrak{B}(A,\psi)|^2 + 2r |\nabla_{A(s), s} \psi|^2 + 2r |D_{A(s)} \psi|^2) \nonumber \\
&\ge & \int_{ [0 , +\infty) \times Y_+}  \chi_*(|E_A|^2 + |\mathfrak{B}(A,\psi)|^2 + 2r |\nabla_{A(s), s} \psi|^2 + 2r |D_{A(s)} \psi|^2) \nonumber \\
&\ge &  \int_{ [0 , +\infty) \times Y_+}  \chi_* (|E_A|^2 + |B_A|^2 + 2r|\nabla_A \psi|^2 + r^2w^2 ) -c_0r.
\end{eqnarray}
Therefore, $\int_{ [0 , +\infty) \times Y_+}  \chi_* (r w^2 + 2|\nabla_A \alpha|^2 + 2 |\nabla_A \beta|^2 )\le c_0(1+2\pi L) $.  Replace the cut off function $\chi$ by $\chi_*$ in (\ref{e13}), and repeat the same argument, we obtain
\begin{eqnarray*}
\int_{s \in [s_*-1, s_*+1]} r|\alpha|^2w \le c_0(1+2\pi L).
\end{eqnarray*}
Hence, $\int_{s \in [s_*-1, s_*+1]} r(1-|\alpha|^2) \le \int_{s \in [s_*-1, s_*+1]} (r|\alpha|^2w +  rw^2) \le c_0(1+2\pi L) $.
\end{proof}

\begin{comment}
\begin{proof}
Firstly, Proposition 4.1  and its proof in \cite{Te4}  carry  over almost verbatim to our setting, except for the following change. The assumption $A_{\mathfrak{d}} < r^2$ or $i_{\mathfrak{d}}> r^2$ are replaced by $\frac{i}{2\pi}\int_{\overline{X}} F_A \wedge \omega_X \le L$. In addition,  the symplectic form $da + a\wedge ds$ is to be replaced by $\Omega$ here.

Secondly, Lemma 4.6, Corollary 4.7, Lemma 4.8, Lemma 4.9  and their proof also carry  over almost verbatim to our setting. Hence, our proposition follows from Lemma \ref{C10} and the part $1-4$ in the proof of Proposition 4.5 in \cite{Te4}.

\end{proof}
\end{comment}

\begin{remark}
Assume that $K\equiv 1$. Using the same limit argument in the proof of Proposition 5.2 of \cite{HT}, the following statement is true. Let  $\{(A^r_{\pm}, \psi^r_{\pm})\}_{r \ge r_0}$ be a family of solutions   to (\ref{e1}) with sufficiently small abstract perturbation $\mathfrak{p}_{\pm}$. Let  $\{\mathfrak{d}_r\}_{r\ge r_0}=\{(A_r, \psi_r)\}_{r\ge r_0}$ be  a  family of solutions to  (\ref{e4}) which is perturbed by  the small abstract perturbation $\mathfrak{p}$.  Assume that $\mathfrak{d}_r$ is asymptotic to $(A^r_{\pm}, \psi^r_{\pm})$ at the ends, where $\mathfrak{p}=\mathfrak{p}_{\pm}$ on the ends.  Suppose that $ \mathcal{F}_{\Omega_X}(\mathfrak{d}_r) \le L$ and $\{(A^r_{\pm}, \psi^r_{\pm})\}_{r \ge r_0}$  converges to orbit set $\alpha_{\pm}$ in the current sense, then there is  a broken  $J$ holomorphic current between $\alpha_+$ and $\alpha_-$.
\end{remark}

\begin{remark} \label{r5}
One could compare the result here with   Proposition 5.1 of \cite{Te4} and  Proposition 7.1 of \cite{HT}, the proof in our setting is easier  than the  ECH setting. The reasons are as follows.
\begin{itemize}
\item
Let $\mathfrak{d}=(A, \psi)$ be an instanton which is asymptotic to $(A_{\pm}, \psi_{\pm})$. In both PFH and ECH cases, we need to control the quantity $A_{\mathfrak{v}}=\mathfrak{a}_-(A_-, \psi_-) -\mathfrak{a}_+(A_+, \psi_+)  $. In PFH case, the Chern-Simon-Dirac functional is of the form
\begin{equation*}
\mathfrak{a}=\mathfrak{a}_0 -\frac{r}{2}Q_F,
\end{equation*}
where $\mathfrak{a}_0$ is gauge invariant. Choose a suitable gauge, Lemma 5.1 of \cite{LT} shows that $|\mathfrak{a}_0 | \le c_0r $. %Since there is only finitely many generators, the constant $c_0$ is independent of the generators.
Therefore, the quantity $A_{\mathfrak{v}}$ is completely controlled by $Q_F(A_+) -Q_F(A_-)$, while  $Q_F(A_+) -Q_F(A_-)$ is controlled by $\mathcal{F}_{\Omega_X}(\mathfrak{d}) $.

However, in ECH case, the relevant part  is  more complicated.  At generators, the Chern-Simon-Dirac functional is of the form
\begin{equation*}
\mathfrak{a}=\frac{1}{2}(\mathfrak{cs}  -rE).
\end{equation*}
The quantity $A_{\mathfrak{v}}$ is not determined by the energy $i\int_{\overline{X} }F_A \wedge d\lambda =E(A_+) -E(A_-)$.  The reason is as follows.  If the  spin-c structure is torsion, one may follow  the  same argument in PFH setting  to show that     $A_{\mathfrak{v}}$ is  controlled by  $E(A_+) $
 and  $E(A_-) $. However, the energy does not give any bound on $E(A_+) $
 and  $E(A_-) $. In fact, $E(A_-)$ and  $i\int_{\overline{X} }F_A \wedge d\lambda $ are controlled by $E(A_+)$, however,  this is nontrivial and is not obvious from the  perspective of Seiberg Witten theory.
The same problems happen when the spin-c structure is non-torsion. Also,  note that the   Chern-Simon functional  $\mathfrak{cs} $ is not gauge-invariant in this case, we need to take  into  account the grading of the generators. Thus one also needs the spectral flow estimate.
 %If spin-c structure is torsion, in this case $\mathfrak{cs} $ is  gauge invariant, however, there are infinitely many generators. In both cases,  we need to take  into  account the grading of the generators.  %That is why in the proof of ECH case, Taubes need to estimate the spectral flow.
%The perturbation in  (\ref{e1}) has nicer properties. The chain complex has only finitely generated.   In ECH setting, at first glance,  the Chern-Simon-Dirac function $\mathfrak{a}$ and energy $E(A)=\int_Y F_A \wedge \lambda $ are two different concepts. One need to do a lot of work to relate them by using the spectral  flow estimates.(Cf. Lemma 7.7 of \cite{HT} and  Proposition 4.11 of \cite{Te1}). While in our case, (\ref{e53}) provides us a simply way to relate these concepts.
\item
Another difference  is that the energy $i\int_{Y} F_A \wedge dt $ is a topological invariant, while it is not in the ECH setting. This difference simplifies the analysis in many places. In fact, this already causes differences  in the last bullet.  The corresponding  quantities of $E(A_{\pm}) $ are constant in the PFH setting. This is why we can get the uniform bound  $|\mathfrak{a}_0 | \le c_0r $.    Another example is that it is very difficult to get the  estimate (\ref{e54}) in ECH setting, because, one needs to  make effort to get the bound on $i\int_{\{s\} \times Y} F_{A(s)} \wedge \lambda$. Lemmas 5.2 and 5.3 of \cite{Te4} and argument in pages 79-81   are trying  to realize the similar bound for   $\int_{\{s\} \times Y} F_{A(s)} \wedge \lambda$.
\end{itemize}
\end{remark}

\subsection{Proof of Theorem \ref{Thm3}}
In order to reduce the proof of Theorem  \ref{Thm3} to the case that $(\Omega_X \vert_{Y_{\pm}}, J \vert_{\mathbb{R}_{\pm} \times Y_{\pm}})$ is $Q$-$\delta$ flat,  we prove the following lemma firstly. It plays the same role as Lemma 6.6 of \cite{HT}.

%In the following lemma, we perturb the   symplectic fibered cobordism such that it is $Q$-flat approximation on the ends.
\begin{lemma}
Let $(X, \pi_X, \omega_X )$ be a fiberwise symplectic cobordism from $(Y_+, \pi_+,\omega_+)$ to $(Y_-, \pi_-,\omega_-)$. Let $J \in \mathcal{J}_{tame}(X, \pi_X, \omega_X)^{reg}$ or $J \in \mathcal{J}_{comp}(X,  \Omega_X)^{reg}$ such that  $J=J_{\pm}$ on $\mathbb{R}_{\pm} \times Y_{\pm}$.  Given sufficiently small $\delta >0$, then there exist an admissible $2$-form $\omega_X'$ and an almost complex structure $J' \in \mathcal{J}_{tame}(X,\pi_X, \omega'_X)^{reg} $ or $J' \in \mathcal{J}_{comp}(X,  \Omega'_X)^{reg} $ accordingly so that $(\omega_X', J' )\vert_{\mathbb{R}_{\pm}\times Y_{\pm}} = (\omega_{\pm}', J_{\pm}')$ is a $Q$-$\delta$ flat approximation of $(\omega_{\pm}, J_{\pm})$, where $\Omega'_X=\omega_X' + \omega_B$. The almost complex structure $J'$ satisfies estimates $|J' - J| \le c_0 \delta$ and $|\nabla(J'-J)| \le c_0$. Moreover, $(X, \omega_X')$ is monotone if  $(X, \omega_X)$ is monotone.  \label{C27}
\end{lemma}
\begin{proof}   For simplicity, we assume that $Y_-=\emptyset$. Let $R_+$ be the  Reeb vector field of $\omega_+$. By Lemma \ref{C41}, we can find a $Q$-$\delta$ flat approximation of $(\omega_+, J_+)$, the result is denoted by $(\omega_+', J_+')$. % $(\omega_+', J_+')$ satisfies the  properties in Lemma \ref{C41}.

According to the construction of $Q$-$\delta$ flat approximation, we know that  $\omega_+'=\omega_+ + d(h_+ dt)$ and $|R_+-R_+'| \le c_0\delta$, where $h_+$ is a function which supports  in  $\delta$-tubular neighborhoods of periodic orbits with degree less than $Q$,  $R_{+}$ and $R_+'$ are Reeb vector fields  of $\omega_+$ and $\omega_+'$ respectively. It is easy to check that $|d^vh_+| \le c_0 \delta$,  where $d^vh_+$ is the component of $dh_+$ which doesn't contain $dt$. Consequently,  $|h_+| \le c_0 \delta^2$,

Let $U_+=[0, -\epsilon] \times Y_+$ be a collar neighborhood of $Y_+$ such that $\Omega_X \vert_{U_+}= \omega_+ + ds \wedge \pi_+^*dt$ and $J \vert_{U_+} =J_+$.  We define a new admissible $2$-form $\omega_X'$ on $X$ by
 \[\omega_X'=
\left\{
\begin{array}
    {r@{\quad:\quad}l}
     \omega_+ + d(\phi(s)h_+ )\wedge \pi_+^*dt & \mbox{on $U_+$} \\
    \omega_{X}  & \mbox{ on $X \setminus U_+$}, \\
\end{array}
\right.
\]
where $\phi(s)$ is a cut off function such that $\phi=1$ near $s=0$ and $\phi=0$ near  $s= -\epsilon$, and $|\phi'| \le c_0 \epsilon^{-1}$.   Note that $\omega_X' $ has the  same cohomology class as $\omega_X$, therefore, $\omega_X' $  is  monotone whenever  $\omega_X $  is  monotone. In addition, $\omega_X' + \pi_X^*\omega_B$ is still symplectic whenever $\delta \ll \epsilon$ is small enough.

Let $\omega_{+s} = \omega_+ + d(\phi(s)h_+ )\wedge dt \vert_{\{s\} \times Y_+} $ and $R_s$ be  Reeb vector field of $(\omega_{+s}, \pi_+^*dt)$.  Let $\{J_s\}_{s \in [0,1]}$ be a family of almost complex structures on $\ker \pi_+$ starting from $J_+ \vert_{\ker \pi_+}$ and ending  at $J_+' \vert_{\ker \pi_+}$,  and $J_s \vert_{\ker \pi_{+}}$ is compatible with $\omega_s \vert_{\ker \pi_{+}}$  for each $s$.    The almost complex structure $J'$ is defined as follows:
\begin{enumerate}
\item
 $J'(\partial_s) = R_s$,  $J'$ maps $\ker{ \pi_+}$ to itself  and $J' \vert_{\ker \pi_+} = J_s \vert_{\ker \pi_+}$ in $U_+$.
 \item
 $J'$ agrees with $J$ in $X \setminus U_+$.
\end{enumerate}
By construction, we have $|J'-J| \le c_0 \delta$ and $|\nabla(J'-J)| \le c_0$.
It is straight forward to check that $J' \in \mathcal{J}_{tame}(X, \pi_X, \omega'_X)$  and  $J' \in \mathcal{J}_{comp}(X,   \Omega'_X)$ whenever $J \in \mathcal{J}_{tame}(X, \pi_X, \omega_X)$ and $J \in \mathcal{J}_{comp}(X,  \Omega_X)$ respectively. %Hence, $J' \in  \mathcal{V}_{comp}(X, \pi_X, \omega'_X, j_B).$

Finally, we make a small perturbation  such that the resulting almost complex structure $J' \in \mathcal{J}_{tame}(X, \pi_X, \omega'_X)^{reg}$ or $\mathcal{J}_{comp}(X, \Omega_X)^{reg}$. In addition, we make  that such a perturbation   supports in $U_+$ and the image of $\varphi_{\gamma}$ but disjoints with the periodic orbits.
\end{proof}

\begin{remark} \label{r7}
The conclusions in Lemma \ref{C27} are also true when  $(X, \Omega_X)$ is a symplectic cobordism and $J \in \mathcal{J}_{comp}(X, \Omega_X)$, because, the construction only takes place in collar neighborhoods of $Y_{+}$ and $Y_-$ where $\Omega_X= \omega_{\pm} + ds\wedge \pi_{\pm}^* dt$.
\end{remark}

\begin{proof}[Proof of Theorem \ref{Thm3}]
The homomorphism ${HP}_{sw}(X, \Omega_X, \Lambda_X)$ is defined by $${HP}_{sw}(X, \Omega_X, \Lambda_X)=\mathcal{T}_-^{-1} \circ HM(X, \Omega_X, \Lambda_X) \circ \mathcal{T}_+.$$
Note that we have a natural decomposition
\begin{eqnarray*}
HP_{sw}(X, \Omega_X, \Lambda_X)&=&\mathcal{T}_-^{-1} \circ HM(X, \Omega_X, \Lambda_X) \circ \mathcal{T}_+\\
&=& \sum\limits_{\mathfrak{s}_X \in Spin^c(X), \mathfrak{s}_X \vert_{Y_{\pm}}= \mathfrak{s}_{\Gamma_{\pm}}}\mathcal{T}_-^{-1} \circ HM(X, \Omega_X, \mathfrak{s}_X, \Lambda_X) \circ \mathcal{T}_+\\
&=&\sum\limits_{\mathfrak{s}_X \in Spin^c(X), \mathfrak{s}_X \vert_{Y_{\pm}}= \mathfrak{s}_{\Gamma_{\pm}}} {HP}_{sw}(X, \Omega_X, \mathfrak{s}_X, \Lambda_X).
\end{eqnarray*}
It is well known that $Spin^c(X)$ is  an affine space over $H_2(X, \partial X, \mathbb{Z})$.  Therefore,
\begin{eqnarray*}
HP_{sw}(X, \Omega_X, \Lambda_X) = \sum\limits_{\Gamma_X \in H_2(X, \partial X, \mathbb{Z}), \partial_{Y_{\pm}} \Gamma_X= \Gamma_{\pm}} HP_{sw}(X, \Omega_X, \Gamma_X, \Lambda_X).
\end{eqnarray*}
The first three assertions of  Theorem  \ref{Thm1} are  direct consequences of the corresponding properties of  cobordism maps on Seiberg Witten Floer cohomology.

To prove  the fourth point in the theorem, first of all, we assume that $(\omega_{\pm}, J_{\pm})$ is $Q$-$\delta$ flat.  Fix an  almost complex structure $J \in \mathcal{J}_{comp}(X,  \Omega_X)$ such that $J \vert_{\mathbb{R}_{\pm} \times Y_{\pm}}=J_{\pm}$  is generic, then the  chain map ${CP}_{sw}(X, \Omega_X, J, \Lambda_X)$ is defined by
\begin{equation} \label{e29}
{CP}_{sw}(X, \Omega_X, J, \Lambda_X) = \lim\limits_{r \to \infty}(T_r^{-})^{-1} \circ CM(X, \Omega_X, J, r, \Lambda_X) \circ T_r^+.
 \end{equation}
Suppose that ${CP}_{sw}(X, \Omega_X, J, \Lambda_X) $ is non-zero,  then the existence of holomorphic current $\mathcal{C}$  follows immediately from Proposition \ref{C32}.  The analogy  of Theorem 5.1  of  \cite{21} tells us that $\mathcal{C}$ has zero ECH index.

For the case that $(\omega_{\pm}, J_{\pm})$ is not $Q$-$\delta$ flat, we can use the argument in Section 6 of \cite{HT} to show that the conclusion still holds.  The main points are sketched as follows: By Lemma \ref{C27} and Remark \ref{r7}, we can construct a sequence of  pair $(\Omega_{X n}, J_n)$ over $\overline{X}$ such that  each $(\Omega_{Xn} \vert_{ Y_{\pm}} , J_n\vert_{\mathbb{R}_{\pm} \times Y_{\pm}})$   is a $Q$-$\frac{1}{n}$ flat approximation of $(\omega_{\pm}, J_{\pm})$ and $J_n$ $C^0$-converges to $J$.   The pair $(\Omega_{X n}, J_n)$ plays the same role as $(\lambda_n', J_n')$ in Section 6 of \cite{HT}.
Define the  chain map by
\begin{equation} \label{e79}
{CP}_{sw}(X, \Omega_X, J, \Lambda_X)=\lim\limits_{n \to \infty}(\Psi_n^-)^{-1} \circ {CP}_{sw}(X, \Omega_{X_n}, J_n, \Lambda_X) \circ \Psi_n^+,
\end{equation}
where    $CP_{sw}(X, \Omega_{X_n}, J_n, \Lambda_X)$ is  $(\Omega_{X n}, J_n)$ version of (\ref{e29}) and $\Psi_n^{\pm}$ is the canonical isomorphism induced by $Q$-$\frac{1}{n}$ flat approximation.
%by the composition of   $(\omega_{X n}, J_n)$ version of (\ref{e29}) and the canonical identification $CP(Y_{\pm}, \omega_{\pm}, \Gamma_{\pm}, J_{\pm}, \Lambda_{P}^{\pm}) \to CP(Y_{\pm}, \omega'_{\pm}, \Gamma_{\pm}, J'_{\pm}, \Lambda_{P}^{\pm})  $.
The analogy of commutative diagram in  Lemma 6.5 of \cite{HT}  can be proved similarly, thus  $CP_{sw}(X, \Omega_X, J, \Lambda_X)$ induces the cobordism map $HP_{sw}(X, \Omega_X, J, \Lambda_X)$.   If the chain map  (\ref{e79}) is non-trivial, then  Proposition \ref{C32} supplies  us   a $J_n$ holomorphic current $\mathcal{C}_n$ for sufficiently large $n$.  By Taubes'Gromov compactness (Lemma 6.8 of \cite{HT}), $\mathcal{C}_n$ converges to a $J$ holomorphic current $\mathcal{C}$. This complete the proof of the holomorphic curve axiom.

Note that Lemma 6.8 of \cite{HT} is weaker than the usual one (eg. Lemma 9.9 of \cite{H1}), because, $J_n$ only $C^0$-converges to $J$. The new ingredient of the proof is Lemma 6.13 of  \cite{HT}, it also can be applied to our case.  In sum, the analogy of Lemma 6.8 of \cite{HT} is also true in our setting.
%For more detail, we refer reader to section 6 of \cite{HT}.

%Combine  Lemma \ref{C27} and the argument in section 6 of \cite{HT}, the conclusion still hold.

%The fourth point of the theorem follows from the definition and Theorem 1 of \cite{T4}.
%
%
%
%To see the last point of the theorem.  Fix  $J  \in \mathcal{J}_{comp}(X, \pi_X, \omega_X) \bigcap   \mathcal{J}_{tame}(X, \pi_X, \omega_X)$, note that the closed holomorphic curves in $\overline{X}$ are contained in fiber. Moreover, the ECH index of them are negative. The only closed curve in $\overline{X}$  with $I=0$ is empty set.  In addition, the argument in the proof of Theorem \ref{A10} can be applied to this case, then there is a bijection between $\mathcal{M}^{J}_{X, I=0}(\emptyset, \emptyset)$ and $\mathfrak{M}^{J, r}_{X, I=0}(T_r(\emptyset), T_r(\emptyset))$.   As a consequence, $\#\mathfrak{M}_{X, \rm{ind}=0}^{J, r}(T_r^+(\emptyset), T_r^-(\emptyset))=1$.

\end{proof}

\begin{remark} \label{r8}
%We will prove a  ``Seiberg-Witten to Gromov" type degeneration which is analogy to the one in \cite{HT}. (See Proposition \ref{C32} for precise statement.) In particular,
% Fix a relative homotopy class $\mathcal{Z} \in  \pi_0\mathcal{B}_X(\mathfrak{c}_+, \mathfrak{c}_-)$.
The convergence theorem  \ref{C32} implies that  a sequence of  instantons   $ \{\mathfrak{d}_r\}_r$ in $ \mathfrak{M}^r_{\overline{X}}(\mathfrak{c}_+, \mathfrak{c}_-)$     gives   a relative homology class $Z$, where $Z$ is the relative class of holomorphic curves given by the instantons. %This gives a correspondence between $\mathcal{Z}$ and $Z$.
Note that the relative homology class $Z$ only depends   on the relative homotopy class of $ \{\mathfrak{d}_r\}_r$.  Moreover, this correspondence is 1-1. If the instantons contribute to the cobordism maps, then the ECH index of $Z$ is zero. (Cf. Theorem 5.1 of \cite{21})

Assume that  $(\Gamma_X,\Omega_X)$ is monotone, then  there are only finitely many relative homology classes supporting  holomorphic curves with zero ECH index.(Cf. Corollary \ref{C53})   The observations in the  last paragraph imply that there are only   finitely many  relative homotopy classes such that  $ \mathfrak{M}^r_{\overline{X}, ind=0}(\mathfrak{c}_+, \mathfrak{c}_-, \mathcal{Z})$ is nonempty.  In particular,   the cobordism map  $HP_{sw}(X, \Omega_X, \Gamma_X)$  can be defined with $\mathbb{Z}_2$  or $\mathbb{Z}$ coefficient.
\end{remark}

\section{ Proof of  Proposition \ref{C42} } \label{section1}
%In this subsection, we prove the Proposition\ref{C42}.
%The construction in this section doesn't rely on the assumption $(\spadesuit)$.
 %we establish a proposition which is analogy of the Theorem 2.5.2 of \cite{VPK} in PFH setting.

Let  $(Y, \pi )$ be a surface fibration over circle  together with an admissible $2$-form $\omega$,  and $J \in \mathcal{J}_{comp}(Y, \pi, \omega)$.  By  Lemma \ref{C41}, we can assume that $(\omega, J)$ is $Q$-$\delta$ flat.  Let $\gamma $ be an elliptic orbit with degree $q \le Q$, then there is a small tubular neighborhood of $\gamma $ such that it can be identified with the following standard form
$$\left(S_{\tau}^1 \times D_z, qR= \partial_{\tau} -i\theta z \partial_z + i\theta \bar{z} \partial_{
\bar{z}}, \omega=\frac{i}{2} dz \wedge d \bar{z} + \left(\frac{\theta}{2} \bar{z}dz + \frac{\theta}{2} z d\bar{z} \right) \wedge d\tau \right) ,$$
together with the  standard almost complex structure $J(\partial_x) =\partial_y$. Assume that $0<\theta<1$, otherwise, we change of coordinate $(t=\tau, w=e^{-ik \tau} z)$ for suitable $k \in \mathbb{Z}$.

Let us assume that $0< \theta < \frac{1}{2}$, the construction  that follows also works for   $\frac{1}{2}<\theta<1$. The only difference  is that one needs  to choose a suitable function $C(r)$. (See Remark \ref{r3})  Take a disk $D_{\delta'} \subset D$ with radius $0< \delta' \ll 1$ and a smooth function $f: D \to \mathbb{R}$  such that $f$ is constant outside  $D_{\delta'}$ and $f=f(r)$.
Define a new admissible $2$-form by  $$\omega_f = \omega + df \wedge d\tau. $$
By directly computation,
%$\nabla f = \frac{x}{r} f' \partial_x + \frac{y}{r} f' \partial_y.$
the Reeb vector field of $\omega_f$  over $S^1 \times D_{\delta'}$ is
$$qR_f= \partial_{\tau} - ( \theta + \frac{f'}{r}) x\partial_y + (\theta + \frac{f'}{r})  y\partial_x. $$
Let $C(r)$ be a smooth function satisfying the following properties:
\begin{enumerate}
\item
 $C(r)=\frac{1}{2} -\theta $ on $D_{\frac{\delta''}{2}}$, where $0< \delta'' \le \frac{1}{2} \delta'$.
\item
$C(r)\ge0$ and $C(r)=0$ outside  $D_{\delta''}$. Also,  $C'(r) \le 0$.
\end{enumerate}
 Take $f'/ r =C(r)$, then we know that $f(r)=\int_0^r sC(s) ds$ and $f'(0)=0$. Therefore, $\gamma $ is still a periodic orbit of $R_f$ with degree $q$.

The  Reeb flow is
\begin{equation*}
\left\{
\begin{aligned}
&\frac{dx}{d\tau}=(\theta + C(r)) y \\
& \frac{dy}{d\tau}=-(\theta+ C(r)) x.
\end{aligned}
\right.
\end{equation*}
%\begin{equation*}
%\begin{split}
%& \frac{dx}{d\tau}=(\theta + C(r)) y\\
%& \frac{dy}{d\tau}=-(\theta+ C(r)) x.
%\end{split}
%\end{equation*}
When $r \le \frac{\delta''}{2}$, above equations  become
\begin{equation*}
\left\{
\begin{aligned}
& \frac{dx}{d\tau}=\frac{1}{2} y\\
&  \frac{dy}{d\tau}=- \frac{1}{2} x.
\end{aligned}
\right.
\end{equation*}
%\begin{equation*}
%\begin{split}
%& \frac{dx}{d\tau}=\frac{1}{2} y\\
%& \frac{dy}{d\tau}=- \frac{1}{2} x.
%\end{split}
%\end{equation*}
The linear  Poincar\'e  return map is given by sending $(x, y)$ to $(-x, -y)$ when $r \le \frac{\delta''}{2}$.
\begin{remark} \label{r3}
If $\frac{1}{2} <\theta< 1 $, we can take   $C(r)$ to be a smooth function such that
\begin{enumerate}
\item
 $C(r)=\frac{1}{2} -\theta $ on $D_{\frac{\delta''}{2}}$, where $0< \delta'' \le \frac{1}{2} \delta'$.
\item
$C(r)\le 0$ and $C(r)=0$ outside  $D_{\delta''}$. Also,  $C'(r) \ge  0$.
\end{enumerate}
\end{remark}

\begin{lemma}
The  set of $\phi_{\omega_f}$  periodic orbits with degree less than $q$ is  identical to  the set of $\phi_{\omega}$ periodic orbits with degree less than $q$.
\end{lemma}
\begin{proof}
Let $\gamma_{q'}$ be a periodic orbit with degree $q' \le q$. Suppose that there exists $\tau_0$ such that $\gamma_{q'}(\tau_0) \in D_{\delta''}$, otherwise, there is nothing to prove.  We  assume that $\tau_0=0$.

First of all, we claim that $\gamma_{q'}(\tau) \in D_{2\delta''}$ for any $\tau$.  Let $\tau_{\max} = \sup\{\tau \in [0,2\pi] \vert \gamma_{q'}(\tau) \in D_{\delta''}\}$, if $\tau_{\max}=2\pi$, then we have done.  If $\tau_{\max}<2\pi$, then for $\tau> \tau_{\max}$, $\gamma_{q'}(\tau)$ satisfies the ODE $\partial_{\tau} x =\theta y$ and $\partial_{\tau}y=-\theta x$.  It cannot escape outside $D_{2\delta''}$, because, it is a rotation. In particular, the degree of $\gamma_{q'}$ is at least $q$, and hence $q'=q$.

Now the Reeb vector is $qR_f=\partial_{\tau} - (\theta + C(r)) \partial_{\phi}$ in $D_{2\delta''}$, where $(r, \phi)$ is the polar coordinate. The Reeb flow equation becomes
\begin{equation*}
\left\{
\begin{aligned}
& \frac{dr}{d\tau}=0\\
& \frac{d\phi}{d\tau}= - (\theta + C(r)).
\end{aligned}
\right.
\end{equation*}
%\begin{equation*}
%\begin{split}
%& \frac{dr}{d\tau}=0\\
%& \frac{d\phi}{d\tau}= - (\theta + C(r)).
%\end{split}
%\end{equation*}
Therefore, $r=r_0$ and $\phi(\tau)=\phi_0 - (\theta + C(r_0)) \tau $. Since  $0<\theta \le \theta + C(r_0) \le \frac{1}{2}$, $\phi(2\pi) \ne \phi(0) \mod 2\pi$ unless that $r_0=0$. In other words, $\gamma_q=\{r=0\}=\gamma$.
\end{proof}

Define a new coordinate $(t' , x' , y' )$ by
\begin{equation*}
\begin{split}
& t'=\tau \ \ x'=\cos (\frac{1}{2} \tau)x -\sin (\frac{1}{2} \tau)y \ \ y'= \sin (\frac{1}{2} \tau) x + \cos (\frac{1}{2} \tau) y.
\end{split}
\end{equation*}
\begin{comment}
Then
\begin{equation*}
\begin{split}
& \partial_x =\cos(\frac{1}{2} \tau) \partial_{x'} + \sin(\frac{1}{2} \tau) \partial_{y'}\\
& \partial_y= -\sin(\frac{1}{2} \tau) \partial_{x'} + \cos(\frac{1}{2} \tau) \partial_{y'}\\
%& \partial_{\tau}= \partial_{t'} + (-\pi sin(\pi \tau)x - \pi \cos(\pi \tau)y ) \partial_{x'} +(\pi \cos(\pi \tau)x -\pi\sin(\pi \tau)y) \partial_{y'}.
&\partial_{\tau}= \partial_{t'} + \frac{1}{2}(-sin(\frac{1}{2} \tau)x -  \cos(\frac{1}{2} \tau)y ) \partial_{x'} + \frac{1}{2}(\cos(\frac{1}{2} \tau)x -\sin(\frac{1}{2} \tau)y) \partial_{y'}\\
%&=\partial_{t'} + \frac{1}{2}(x \partial_y -y\partial_x).
\end{split}
\end{equation*}
and satisfying $J(\partial_{x'})=\partial_{y'}$.
\end{comment}
Under the  coordinate above, $(S^1_{\tau} \times D_{\frac{\delta''}{4}}, \omega_f, qR_f)$ can be identified with $$\left( [0,2\pi] \times D_{\frac{\delta''}{4}} /((2\pi, x, y) \sim (0, -x,-y)), dx'\wedge dy', \partial_{t'} \right),$$
and the almost complex structure $J $ satisfies $J(\partial_{x'})=\partial_{y'}$.

Take a function $h$ with compact support on  $D_{\frac{\delta''}{4}}$ such that  $h=\frac{1}{2}( (x')^2-(y')^2)$ on  $D_{\frac{\delta''}{8}}$. Define
\begin{equation*}
\omega_{f}'= dx' \wedge dy' + \varepsilon dh\wedge dt'.
\end{equation*}
Then $\gamma=\{x'=y'=0\}$  still is a periodic orbit. Moreover, for  sufficiently small $\varepsilon>0$, there is no  periodic orbit  with period $q$ other than $\gamma$.
The linear Reeb flow is
\begin{equation*}
\left\{
\begin{aligned}
& \frac{dx'}{dt}=-\varepsilon y'\\
& \frac{dy'}{dt}=-\varepsilon x'.
\end{aligned}
\right.
\end{equation*}
%\begin{equation*}
%\begin{split}
%& \frac{dx'}{dt}=-\varepsilon y'\\
%& \frac{dy'}{dt}=-\varepsilon x'.
%\end{split}
%\end{equation*}
Then the linear Poincar\'e return map has two real eigenvalues $-e^{\varepsilon}$ and $-e^{-\varepsilon}$. Therefore, $\gamma $ is negative hyperbolic.

%\begin{prop}
%Let $(Y, \pi)$ be a surface fiber bundle over circle and $\omega$ be an admissible $2$-form. Given $L \ge Q \ge 0$, we can always find another admissible $2$-from $\omega'$ such that
%\begin{enumerate}
%\item
%Every periodic orbit with degree less than $Q$ is either $L$-negative (positive) elliptic  or hyperbolic.
%\item
%$[\omega']=[\omega].$
%\end{enumerate}
%\end{prop}

\begin{proof}[ Proof of Proposition \ref{C42}]
First of  all, we can always perturb $(\omega_{+}, J_+)$ such that it is $Q$-$\delta$ flat. If there exists a periodic orbit $\gamma $ of $\omega_{+}$ which is not $Q$-negative elliptic or hyperbolic, then we perform the above construction to make it become negative hyperbolic. This process may create another elliptic orbit, but they have larger degree. Repeat this process until  all  periodic orbits with degree less than $Q$ satisfy the conclusion. Obviously, the cohomology class of $\omega_+$ does not change under the above construction. Also, from the construction,  we know that
\begin{center}
$ \begin{aligned}
|\omega_+' -\omega_+| \le c_0 \delta.
\end{aligned}$
\end{center}
%\begin{equation*}
%|\omega_+' -\omega_+| \le c_0 \delta.
%\end{equation*}
We can modify  $\omega_-$ in the same way.

By using cut-off functions, we can always modify $\omega_X$ such that $\omega_X=\omega_{\pm}$ in a collar neighborhood of $Y_{\pm}$. Since $\omega_{\pm}'=\omega_{\pm}+ d\mu_{\pm}$ for some $\mu_{\pm}$, we can define $\omega_X'$ by the same formula (\ref{e46}).
\end{proof}

%%%%%%%%%%%%%%%%%%%%%%%%%%%%%%%%%%%%%%%%%%%%%%%%%%%%%%

\section{Cobordism maps on HP via holomorphic curve  }
%In this section, we  present a  result about defining   the cobordism  map $HP(X, \omega_X, J, \Lambda_X)$ on PFH via holomorphic curve method in certain cases.

%Except the case that $X$ is a Lefschetz fibration over a disk,  we can always modify $\omega_{\pm}$ and $\omega_X$ so that it satisfies \Romannum{1} or \Romannum{2} or \Romannum{3} by making use of Proposition \ref{C42}.
In principle, the cobordism maps  in chain level should be defined by
\begin{equation} \label{e21}
CP(X, \Omega_X, J, \Lambda_X)=\sum\limits_{\alpha_{\pm} \in  \mathcal{ P}( Y_{\pm}, \omega_{\pm}, \Gamma_{\pm})} \sum\limits_{Z \in H_2(X, \alpha_+, \alpha_-)} \#_2 \mathcal{M}_{X, I=0}^J( \alpha_+, \alpha_-, Z) \Lambda_X(Z).
\end{equation}
However,  Section 5.5 of  \cite{H3} points out that  (\ref{e21}) doesn't make sense in general, because, the moduli space $\mathcal{M}_{X, I=0}^J( \alpha_+, \alpha_-, Z)$ is not compact due to the appearance of holomorphic curves with negative ECH index.

In this section, we show that actually  definition   (\ref{e21})  works  for fiberwise symplectic cobordism  $(X, \pi_X, \omega_X)$  satisfying  the assumptions $(\spadesuit)$. Firstly, we  show that the ECH index is nonnegative and investigate the holomorphic currents with small indexes  $(0\le I\le 1)$.  Then we construct the cobordism maps  separately for the following three cases: $(X, \pi_X)$ contains no separating fiber, monotone case and $(X, \pi_X)$ is not relatively  minimal. Finally, we show that the cobordism maps satisfy certain natural properties.

\begin{remark} \label{r9}
It is worth  mentioning  that recently C.Gerig defines the cobordism maps on  ECH  for certain symplectic cobordisms \cite{CG}.  It is worth comparing our results with C.Gerig's on ECH. %  that   He considers the symplectic cobordism which comes from nearly symplectic manifold.
%In his case, he shows that the ECH index is nonnegative but he need to deal with certain  multiply covered holomorphic curves.
In both cases, we borrow the idea in \cite{H4} to use the $Q$-positive (negative) orbits to guarantee the nonnegative of the ECH index.   Besides, our proof here is  also based on  the following two simple observations:
\begin{enumerate} [label=\textbf{O.\arabic*},ref=O.\arabic*]
\item
If $C$ is a holomorphic curve in $\overline{X}$ with at least one end, then $g(C) \ge g(B)$. (See Lemma \ref{C1} below.)
\item \label{o2}
If $Y_+ \ne \emptyset$ and $Y_- \ne \emptyset$,  and $C$ is a holomorphic curve in $\overline{X}$ with at least one end, then   $C$ has at least one positive end and one negative end.
\end{enumerate}
Both of these  observations come from the fibration structure, which  does not appear in  the ECH setting.
In the ECH case, even the ECH index is nonnegative,  C.Gerig  needs to deal with certain  multiply covered holomorphic planes.
Compare to his case, our situation is even better,   the  multiply covered holomorphic curves cannot contribute to the cobordism maps   because of the observations above. For example, there is no holomorphic plane  in case  \ref{assumption3} due to \ref{o2}.
\end{remark}

\textbf{In the remaining of this paper, we focus on the fiberwise symplectic cobordism case.       We always assume that the  fiberwise symplectic   cobordism $(X, \pi_X, \omega_X)$  satisfies the assumptions $(\spadesuit)$,
unless otherwise stated. }

\subsection{Index computation  }
In this subsection, we show that the ECH index of holomorphic current is nonnegative under   assumptions $(\spadesuit)$. Moreover, the moduli space $\mathcal{M}^J_{X, I=i}(\alpha_+, \alpha_-)$ only consists of embedded holomorphic curves in the sense of Definition \ref{def3}, where $i=0$ or $1$.

%Firstly, let us make sure that the ECH index is nonnegative.
\begin{lemma} \label{C43}
Let $\alpha_{+}$ and $\alpha_-$ be  orbit  sets  of $(Y_{+} , \pi_{+}, \omega_{+})$  and $(Y_{-} , \pi_{-}, \omega_{-})$ respectively,  and $J \in \mathcal{J}_{tame}(X, \pi_X, \omega_X)$. Suppose that $C$ is an irreducible somewhere injective $J$ holomorphic curve  from $\alpha_+$ to $\alpha_-$  in $\overline{X}$ and $C$ is not closed,  then $g(C) \ge g(B)$. \label{C1}
\end{lemma}
\begin{proof}
Let $d=[\alpha_{\pm}] \cdot [\Sigma]$, then $C \cap [\Sigma]=d$. The intersection positivity  implies that $C$ intersects a generic fiber with $d$ distinct points.  Since $\pi_X: \overline{X} : \to \overline{B}$ is complex linear,  $\pi_X: C \to \overline{B}$ is a degree $d$ branched cover.

Suppose that $C$ has $k_+$ positive ends and $k_-$ negative ends, $1 \le k_{\pm} \le d$, then by  Riemann-Hurwitz formula,
\begin{eqnarray} \label{e55}
&&2-2g(C) - k_+ - k_-= d (-2g(B)) -b \nonumber \\
&\Rightarrow& g(C)=1+ dg(B) +\frac{1}{2}b - \frac{1}{2}(k_+ + k_-)\\
&\Rightarrow& g(C) \ge 1+dg(B)-d \ge g(B)+ (d-1)(g(B)-1) \ge g(B). \nonumber
\end{eqnarray}
where $b \ge 0$ is the sum over all the branch points of the order of multiplicity minus one.
\end{proof}

Recall that $C \star C$ is the generalized self-intersection number in Definition \ref{def1}. The following lemma tells us  that it is nonnegative.
\begin{lemma} \label{C40}
%Let $\alpha_{\pm}$ be  orbit set of $(Y_{\pm} , \pi_{\pm}, \omega_{\pm})$(not necessarily admissible)
Let $\alpha_{+}$ and $\alpha_-$ be  orbit  sets  of $(Y_{+} , \pi_{+}, \omega_{+})$  and $(Y_{-} , \pi_{-}, \omega_{-})$ respectively (not necessarily admissible),   and  $C \in \mathcal{M}_X^J(\alpha_+, \alpha_-)$ be an irreducible simple holomorphic curve. Assume that $C$ is not closed, then
$C\star C \ge i \ge 0$ for generic $J \in \mathcal{J}_{tame}(X, \pi_X, \omega_X)$, where $i=1$ in case  \ref{assumption1}, $i=\frac{1}{2}$ in case \ref{assumption2} and $i=0$ in case \ref{assumption3}.
\end{lemma}
\begin{proof}
We  divide the proof into three  cases, and we prove the result case by case.
\begin{itemize}
\item
In case \ref{assumption1}, the conclusion follows from Lemma \ref{C1} and Definition \ref{def1}.
\item
In case \ref{assumption2}, $C\star C \ge 0$ by Lemma \ref{C1}. It suffices  to rule out the  possibility that $C\star C=0$.  If $C\star C=0$, it follows from   Definition \ref{def1} that  $C$ satisfies
\begin{equation} \label{e56}
g(C)=1 \ \ \mbox{and} \ \ {\rm{ind }}C = e_Q(C)=\delta(C)=h(C)=0.
\end{equation}
By (\ref{e56}) and the assumption \ref{assumption2},  all  ends of $C$ are asymptotic to elliptic orbits with degree larger than $1$.  Substitute $g(C)=g(B)=1$ for (\ref{e55}), we deduce that $k_{\pm}=d$, where $k_{+}$ and $k_-$ are the numbers of positive and negative ends of $C$ respectively,  and $d$ is the degree of $\pi_X: C \to \overline{B}$. Consequently,  all the ends of $C$ are asymptotic to periodic orbits with degree $1$, we get a contradiction.
%However, follows from the proof of Lemma \ref{C1},  $k_{\pm}=d$ whenever $g(C)=1$, where $k_{\pm}$ are the number of positive and negative ends of $C$ respectively and $d$ is the degree of $\pi_X: C \to \overline{B}$.  This implies that all the ends of $C$ are asymptotic to periodic orbits with degree $1$, we get contradiction.
\item
In case \ref{assumption3}, by observation \ref{o2},   every  holomorphic curve in $\overline{X}$ which is not closed has at least one positive end and at least one negative end.  Therefore, $2e_Q(C_a)+ h(C_a) \ge 2$  under our assumption. Then  the conclusion follows  immediately  from  Definition \ref{def1}.

% obverse  that  every  holomorphic curve in $\overline{X}$ which is not closed has at least one positive end and at least one negative end.  As a consequence, $2e_Q(C_a)+ h(C_a) \ge 2$  under our assumption.
\end{itemize}
\end{proof}

%Let $\mathcal{M}_X^J(\alpha)=\{\mbox{$J$-holmorphic currents asymptotic to $\alpha$}\}$.

\begin{corollary}
%Let $\alpha_{\pm}$ be  orbit set of $(Y_{\pm} , \pi_{\pm}, \omega_{\pm})$(not necessarily admissible) and  $\mathcal{C} \in \widetilde{\mathcal{M}}_X^J(\alpha_+, \alpha_-)$.
 Let $\alpha_{+}$ and $\alpha_-$ be  orbit  sets  of $(Y_{+} , \pi_{+}, \omega_{+})$  and $(Y_{-} , \pi_{-}, \omega_{-})$ respectively (not necessarily admissible),  and  $\mathcal{C} \in \widetilde{\mathcal{M}}_X^J(\alpha_+, \alpha_-)$.  Assume that $\mathcal{C}$ doesn't contain closed components. Then for generic $J \in \mathcal{J}_{tame}(X, \pi_X, \omega_X)$, $I(\mathcal{C}) \ge 0$.  \label{C3}
\end{corollary}

\begin{proof}
\begin{comment}
By Theorem 5.1 of \cite{H2},
\begin{equation} \label{e18}
I(\mathcal{C}) \ge \sum\limits_a d_aI(C_a) + \sum\limits_a \frac{d_a(d_a-1)}{2} (2g(C_a)-2 + {\rm ind}(C_a) + h(C_a) +4\delta(C_a)) + 2\sum\limits_{a \neq b }d_ad_b C_a \cdot C_b,
\end{equation}
where $C_a \cdot C_b  \ge 0$ is the algebraic count of intersections of $C_a$ and $C_b$.
\end{comment}
%By Theorem 4.15 of \cite{H2},
\begin{comment}
By Theorem \ref{ECH1}, $I(C_a) \ge {\rm ind}(C_a) \ge 0$ for generic $J \in \mathcal{J}_{tame}(X, \pi_X, \omega_X)$.
Keep in mind   that every holomorphic curve in $\overline{X}$ which is not closed has at least one positive end and at least one negative end, so we have $2e_D(C_a)+ h(C_a) \ge 2$ in case   \Romannum{3}.
By Lemma \ref{C1} and our assumption,  we have
\begin{equation*}
 C_a \cdot C_a=\frac{1}{2}(2g(C_a)-2 + {\rm ind}(C_a) + 2e_D(C_a)+ h(C_a) +4\delta(C_a))\ge 0.
\end{equation*}
In fact, $C_a \cdot C_a \ge 1$ in case \Romannum{1} and  $C_a \cdot C_a \ge \frac{1}{2}$ in case \Romannum{2}.  According to inequality (\ref{e18}), we have $I(\mathcal{C}) \ge 0$.
\end{comment}
By Theorem \ref{ECH1}, $I(C_a) \ge {\rm ind}(C_a) \ge 0$ for generic $J \in \mathcal{J}_{tame}(X, \pi_X, \omega_X)$. According to Lemma \ref{C40} and  inequality (\ref{e18}), we have $I(\mathcal{C}) \ge 0$.
\end{proof}

The following two lemmas  analyze the holomorphic curves with ECH index zero and one.
\begin{lemma}
%Let $\alpha_{\pm}$ be  orbit sets of $(Y_{\pm} , \pi_{\pm}, \omega_{\pm})$
 Let $\alpha_{+}$ and $\alpha_-$ be  orbit  sets  of $(Y_{+} , \pi_{+}, \omega_{+})$  and $(Y_{-} , \pi_{-}, \omega_{-})$ respectively, and $\mathcal{C}=\{(C_a, d_a)\}\in \widetilde{\mathcal{M}}_X^J(\alpha_+, \alpha_-)$ without closed components. Assume that at least one of $\alpha_+$ and $\alpha_-$ is admissible.  For generic $J \in \mathcal{J}_{tame}(X, \pi_X, \omega_X)$, if $I(\mathcal{C}) = 0$, then  \label{C2}
\begin{enumerate}
\item
For each $a$, $I(C_a)={\rm ind}(C_a)=0$.
%and $d_a=1$ for any $a$.
\item
%If $a \neq b$, then $C_a \star C_b =0$.
$\mathcal{C}$ is embedded in the sense of Definition \ref{def3}.
\item
If $u$ is a $J$ holomorphic curve whose associated  current is $\mathcal{C}$, then $u$ is admissible.
\end{enumerate}
\end{lemma}
\begin{proof} For cases  \ref{assumption1} and \ref{assumption2},    the conclusions follows immediately from   Lemma \ref{C40},   Theorem \ref{ECH1} and \ref{ECH2}.
%inequality (\ref{e18}) and
% Theorem 4.15 of \cite{H2}.

For case \ref{assumption3}, if $d_a>1$ for some $a$, then we must have $C_a \star C_a=0$.  % Observe that   $C$ has at least one positive end and one negative end, thus $h(C_a) +2 e_Q(C_a) \ge 2$.
Recall that  $h(C_a) +2 e_Q(C_a) \ge 2$ because of observation \ref{o2}.
The equality  $C_a \star C_a=0$ forces
\begin{equation*}
g(C_a)= {\rm ind} C_a = \delta(C_a)=0 \ \ \ and \ \ \  h(C_a) +2 e_Q(C_a) = 2.
\end{equation*}
If $e_Q(C_a) =1$  and $ h(C_a) =0$, then $C_a$ has only one end, which contradicts with \ref{o2}.  As a result, the only possibility is that
%Since $C_a$ is not closed, we have
\begin{equation} \label{e40}
g(C_a)={\rm ind }C_a =\delta(C_a)=e_Q(C_a)=0, \ \ \mbox{and} \ \ h(C_a)=2.
\end{equation}
%Then we obtain contradiction, because, $C_a$ has at least one positive and negative end and $\alpha_{\pm}$ are admissible.
Thus $C_a$ has respectively one positive end and one negative end at hyperbolic orbit. However, $\alpha_{+}$ or $\alpha_-$ is admissible, this forces $d_a =1$.
\end{proof}

\begin{lemma}
%Let $\alpha_{\pm}$ be  admissible orbit set of  $(Y_{\pm} , \pi_{\pm}, \omega_{\pm})$.
Let $\alpha_{+}$ and $\alpha_-$ be admissible  orbit  sets  of $(Y_{+} , \pi_{+}, \omega_{+})$  and $(Y_{-} , \pi_{-}, \omega_{-})$ respectively.  Let $\mathcal{C}=\{(C_a, d_a)\}\in \widetilde{\mathcal{M}}_X^J(\alpha_+, \alpha_-)$ without closed components. For generic $J \in \mathcal{J}_{tame}(X, \pi_X, \omega_X)$,  if $I(\mathcal{C}) = 1$, then \label{C11}
\begin{enumerate}
\item
%$d_a=1$ for any $a$.
$\mathcal{C}$ is embedded in the sense of Definition \ref{def3}.

\item
There exists a unique $a_0$ such that $I(C_{a_0})= {\rm ind}(C_{a_0})=1 $. Without loss of generality, we always assume that $a_0=1$. For $a \ge 1$, $I(C_a)={\rm ind}(C_a)=0$.

%\item
%If $a \neq b$, then $C_a \star C_b =0$.

\item
If $u$ is a $J$ holomorphic curve whose associated current is $\mathcal{C}$, then $u$ is admissible.
\end{enumerate}
\end{lemma}

\begin{proof}
We  divide the proof into three  cases as before, and we prove them case by case.
\begin{itemize}
\item
Fist of all, we consider  case \ref{assumption1}.  The first and third  conclusions  and uniqueness of $a_0$ follow from  the formula (\ref{e18}).  The  third conclusion of the lemma follows from
Theorem  \ref{ECH1}. To show the second conclusion,  it   suffices  to rule out the case that all $C_a$ have $I(C_a)=0$ but $I(\mathcal{C})=1 $.

Arguing by contradiction. Suppose that $I(\mathcal{C})=1 $ but $I(C_a)=0$ for any $a$. By (5.5) of \cite{H2}, we have
\begin{eqnarray*}
I(\sum\limits_a C_a)= \sum\limits_a I(C_a) + 2 \sum\limits_{a \ne b }(C_a\star C_b-l_{\tau}(C_a, C_b) ) +\mu_{\tau}(\sum\limits_a C_a) - \sum\limits_a \mu_{\tau}(C_a),
\end{eqnarray*}
%\begin{eqnarray*}
%I(C_a +C_b)- I(C_a) -I(C_b)=2C_a \star C_b +\mu_{\tau}(C_a+C_b) -\mu_{\tau}(C_a)-\mu_{\tau}(C_b)- 2l_{\tau}(C_a,C_b).
%\end{eqnarray*}
Suppose that each $C_a$  has   ends at $\gamma$ with total multiplicity $m_a$ and $m=\sum_a m_a$, where $\gamma$ is an embedded periodic orbit. If $\gamma$ is elliptic and its  rotation number is $\theta$, then
\begin{eqnarray*}
&&\sum_{q=1}^m\mu_{\tau}(\gamma^q)  - \sum_a \sum_{q=1}^{m_a}\mu_{\tau}(\gamma^q)
%&=&\left( \sum\limits_{q=1}^{m+n} -  \sum\limits_{q=1}^{m}- \sum\limits_{q=1}^{n}\right)\mu_{\tau}(\gamma^q)\\
%&=& \left( \sum\limits_{q=1}^{m} -  \sum_a\sum\limits_{q=1}^{m_a}\right)(2  \lfloor q\theta \rfloor  +1)\\
= \left( \sum\limits_{q=1}^{m} -  \sum_a\sum\limits_{q=1}^{m_a}\right)2  \lfloor q\theta \rfloor =even.
\end{eqnarray*}
%Suppose that $C_a$ and $C_b$ have common ends at $\gamma$, where $\gamma$ is an embedded periodic orbit. If $\gamma$ is elliptic orbit with rotation number $\theta$, then
%\begin{eqnarray*}
%&&\mu_{\tau}(C_a+C_b) -\mu_{\tau}(C_a)-\mu_{\tau}(C_b)\\
%%&=&\left( \sum\limits_{q=1}^{m+n} -  \sum\limits_{q=1}^{m}- \sum\limits_{q=1}^{n}\right)\mu_{\tau}(\gamma^q)\\
%&=& \left( \sum\limits_{q=1}^{m+n} -  \sum\limits_{q=1}^{m}- \sum\limits_{q=1}^{n}\right)(2  \lfloor q\theta \rfloor  +1)\\
%&=& \left( \sum\limits_{q=1}^{m+n} -  \sum\limits_{q=1}^{m}- \sum\limits_{q=1}^{n}\right)2  \lfloor q\theta \rfloor =even.
%\end{eqnarray*}
Since $\alpha_+$ and $\alpha_-$ are admissible, the  same conclusion automatically  holds  when $\gamma$ is hyperbolic. Hence,
%$I(C_a +C_b) -I(C_a)- I(C_b) = 0 \mod 2 .$  In general, using the same argument and formula
%\begin{eqnarray*}
%I(\sum\limits_a C_a)= \sum\limits_a I(C_a) + 2 \sum\limits_{a \ne b }(C_a\star C_b-l_{\tau}(C_a, C_b) ) +\mu_{\tau}(\sum\limits_a C_a) - \sum\limits_a \mu_{\tau}(C_a),
%\end{eqnarray*}
 we  have $I(\sum\limits_a C_a)= \sum\limits_a I(C_a) \mod 2$ which  contradicts with $I(\mathcal{C})=1$.

Therefore, there exists $a_0=1$ such that $I(C_1)=1$. Since ${\rm ind}(C_1) \le I(C_1)$ and $I(C_1)={\rm ind}(C_1)\mod 2 $,  we must have ${\rm ind}(C_1)=1$.

\item
Now we consider case \ref{assumption2}. If $d_a >1$ for some $a$, then  $C_a \star C_a =\frac{1}{2}$.
Recall that $g(C_a) \ge 1$ under the assumptions.  Definition \ref{def1} and  equality $C_a \star C_a =\frac{1}{2}$ deduce  that $C_a $ satisfies any one of the following:
\begin{equation} \label{e58}
\begin{split}
&g(C_a)-1=h(C_a) =\delta(C_a)=e_Q(C_a)=0, \ \ \mbox{and} \ \ {\rm ind }C_a=1, \\
\end{split}
\end{equation}
\begin{equation}\label{e41}
\begin{split}
&g(C_a)-1={\rm ind }C_a =\delta(C_a)=e_Q(C_a)=0, \ \ \mbox{and} \ \ h(C_a)=1, \\
&\mbox{and  $C_a$ has no ends at elliptic orbits}.
\end{split}
\end{equation}
The reasons for second line in (\ref{e41}) are  as follows. If $g(C_a)=1$, the Riemann-Hurwitz formula (\ref{e55}) tells us that all the ends of $C_a$ are asymptotic to  periodic orbits with degree $1$.
On the other hand, the assumption  \ref{assumption2} and $e_Q(C_a)=0$ imply  that  $C_a$ has no ends at elliptic orbits with  degree $1$.

Firstly,  the case (\ref{e58}) cannot happen. If $h(C_a)=0$,  then all the ends of $C_a$ are asymptotic to elliptic orbits. As a result,  ${\rm{ind}} C_a =0 \mod 2$.  We obtain a contradiction.

%Also note that (\ref{e41}) can  happen only when $Y_+=\emptyset $ or $Y_-=\emptyset $.
In the case  (\ref{e41}), all the ends of $C_a $ are asymptotic to hyperbolic orbits.  The  assumption that $\alpha_{+}$ and $\alpha_-$ are admissible forces $d_a=1$.

  Repeat the same argument in case \ref{assumption1}, we obtain the  conclusion.

\item
For case \ref{assumption3}, if $d_a>1$ for some $a$, then $C_a \star C_a =\frac{1}{2}$ or $C_a \star C_a =0$. If $C_a \star C_a =0$, follow from the proof of Lemma \ref{C2}, then $C_a$ satisfies (\ref{e40}).

If $C_a \star C_a =\frac{1}{2}$, by $2e_Q(C_a) + h(C_a) \ge 2$, one can check that $C_a$ satisfies any one of the following:
%\begin{enumerate}
%\item
\begin{equation}\label{e42}
g(C_a)=\delta(C_a)=e_Q(C_a)=0, \ \ \mbox{and} \ \ {\rm ind }C_a=1, \ \  h(C_a)=2,
\end{equation}
\begin{equation}\label{e43}
\mbox{or} \ \ g(C_a)=\delta(C_a)= {\rm ind }C_a=e_Q(C_a)=0, \ \ \mbox{and} \ \  h(C_a)=3,
\end{equation}
\begin{equation}\label{e44}
\mbox{or} \ \ g(C_a)=\delta(C_a)= {\rm ind }C_a=0, \ \ \mbox{and} \ \  e_Q(C_a)=1, \ \  h(C_a)=1.
\end{equation}
%\end{enumerate}
In either of these cases, $C_a$ has at least one end at hyperbolic orbits.  Then $d_a>1$ contradicts with the assumption that $\alpha_+$ and $\alpha_-$ are admissible. As a result, $d_a=1$ for any $a$.  Repeat the same argument in case \ref{assumption1}, we get the same conclusion.
\end{itemize}
 \end{proof}

Assume that $u$ is a  holomorphic curve with $0\le I(u) \le 1$. If we remove the admissible assumption  on  $\alpha_{\pm}$,   then $u$ may not be embedded. The following lemma asserts that the Fredholm index of $u$ is still nonnegative for generic $J$.
\begin{lemma} \label{C39}
%Let $\alpha_{\pm}$ be  orbit sets (not necessary  to be admissible)of  $(Y_{\pm} , \pi_{\pm}, \omega_{\pm})$.
Let $\alpha_{+}$ and $\alpha_-$ be  orbit  sets  of $(Y_{+} , \pi_{+}, \omega_{+})$  and $(Y_{-} , \pi_{-}, \omega_{-})$ respectively (not necessary  to be admissible). For generic $J \in \mathcal{J}_{tame}(X, \pi_X, \omega_X)$, any $u \in \mathcal{M}^J_{X, I =i} (\alpha_+, \alpha_-)$ without closed irreducible components  has  nonnegative Fredholm index, i.e.,
\begin{equation*}
{\rm ind} u \ge 0,
\end{equation*}
where $i=0$ or $i=1$.
\end{lemma}
\begin{proof}
From the proof of Lemmas \ref{C2} and \ref{C11}, we know that $u$ has no nodes but $u$ could be  multiple covers of simple curves,  as our $\alpha_{\pm}$  may not be admissible anymore.

Without loss of generality,  we may assume that $u$ is an irreducible non-simple curve. Then $u$ is branched cover of an embedded curve $C$. According to the discussion in Lemmas \ref{C2} and \ref{C11},  $C$ satisfies (\ref{e40}) or (\ref{e41}) or (\ref{e42}) or (\ref{e43}) or (\ref{e44}). The covering degree is denoted by $d$. In the cases  (\ref{e40}) or (\ref{e41}) or (\ref{e42}) or (\ref{e43}), all the ends of $u$ are asymptotic to hyperbolic orbits, by Riemann-Hurwitz formula, we have
$${\rm ind}u= d {\rm ind}C +b \ge 0, $$
where $b \ge 0$ is the sum over all the branch points of the order of multiplicity minus one.

In the case (\ref{e44}), $C$ is a holomorphic cylinder from $\gamma_+^m$ to $\gamma_-^n$. One of them is elliptic orbit while the other one is hyperbolic.  Without loss of generality, assume that $\gamma_+$ is an embedded elliptic orbit and   $\gamma_-$ is an embedded hyperbolic orbit. Note that $\gamma_-$ is negative hyperbolic, otherwise, ${\rm ind}C=1 \mod 2$.  We can take a trivialization $\tau $ such that $\mu_{\tau}(\gamma_+^p)=-1$ for any $p\le Q$ and $\mu_{\tau}(\gamma_-)=1$.  Therefore, ${\rm ind}C=0$ implies that $2c_{\tau}(C)=1+n$.

 %By our assumption, we can take a trivialization $\tau $ such that $\mu_{\tau}(\gamma_+^p)=-1$ for any $p\le Q$ and $\mu_{\tau}(\gamma_-)=1$. Therefore, ${\rm ind}C=0$ implies that $2c_{\tau}(C)=1+n$ and $\gamma_-$ is negative hyperbolic.

By Riemann-Hurwitz formula, we have
\begin{equation*}
\begin{split}
{\rm ind}u&= b + 2d c_{\tau}(C) + \sum_i \mu_{\tau}(\gamma_+^{p_i}) -\sum_j \mu_{\tau}(\gamma_-^{q_j})\\
&= b+ d(1+n) - \sum_i - \sum_j q_j  \ge 0.
\end{split}
\end{equation*}
The last step in above inequality is a consequence of $\sum_i  \le d$ and   $ \sum_j q_j =dn$.

One can perform the same argument when  $\gamma_+$ is hyperbolic and $\gamma_-$ is elliptic, then we obtain the same conclusion.
\end{proof}

\subsection{$(X, \pi_X)$ contains no separating singular fiber.}
 In this subsection, we define  the  cobordism maps $HP(X, \Omega_X, J, \Lambda_P)$ by counting holomorphic curves  in the case that  $X$ contains no separating singular fiber.
\subsubsection{Compactness and Transversality}
Recall that the fibers  of $(X,\pi_X)$  are holomorphic curves. The following lemma rules out the fibers provided that the degree $\Gamma_{\pm}\cdot [\Sigma]$ is large enough.
\begin{lemma} \label{C25}
Suppose that $(X, \pi_X)$ contains no  separating singular fiber. Let $\mathcal{C}=\{(C_a, d_a)\}\in \widetilde{\mathcal{M}}_X^J(\alpha_+, \alpha_-)$. Assume that  $d=\alpha_{\pm} \cdot [\Sigma]> g(\Sigma)-1$. If $\mathcal{C}$ contains closed components, then $I(\mathcal{C}) \ge 2$ for generic $J \in \mathcal{J}_{tame}(X, \pi_X, \omega_X)$.
\end{lemma}
\begin{proof}
By our assumptions,  the homology class of regular fibers and  singular fibers are $[\Sigma]$, thus we can rewire  $\mathcal{C}= \mathcal{C}'+ m[\Sigma]$ for some $m \ge 1$ and $\mathcal{C}'$ has no closed component.   By definition of ECH index,
\begin{eqnarray*}
I(\mathcal{C})-I(\mathcal{C}')= m<c_1(T\overline{X}), [\Sigma]> +  2m Q_{\tau}(\mathcal{C} , [\Sigma]) + m^2Q_{\tau}([\Sigma] , [\Sigma]).
\end{eqnarray*}
According to  the definition of $Q_{\tau}$ and adjunction  formula, we have
 \begin{eqnarray*}
&&Q_{\tau}(\mathcal{C}, [\Sigma]) = \mathcal{C} \cap[\Sigma] =d,  Q_{\tau}( [\Sigma], [\Sigma])= [\Sigma] \cdot [\Sigma]=0, \\
&& <c_1(T\overline{X}), [\Sigma]>=\chi(\Sigma) + [\Sigma] \cdot [\Sigma]=2-2g(\Sigma).
\end{eqnarray*}
Therefore, $I(\mathcal{C})=I(\mathcal{C}')+ 2m(d-g(\Sigma)+1) \ge 2$.

\end{proof}

\begin{lemma}
Let  $\alpha_{+}$ and $\alpha_-$ be  admissible  orbit sets  of $(Y_{+} , \pi_{+}, \omega_{+})$  and $(Y_{-} , \pi_{-}, \omega_{-})$ respectively.   Assume that  $d=\alpha_{\pm} \cdot [\Sigma]> g(\Sigma)-1$  and  $(X, \pi_X)$ contains no  separating singular fiber. Then for generic $J \in \mathcal{J}_{tame}(X, \pi_X, \omega_X)$ and any $L>0$, the moduli space $\mathcal{M}_{X,I=0}^{J,L}(\alpha_+, \alpha_-)$ is compact. \label{C4}
\end{lemma}
\begin{proof}
%For the sake of simplicity, we assume that $Y_-=\emptyset $.
The proof here is similar to Lemma 7.19 of \cite{HT1}.
Let $\{u_n\}_{n=1}^{\infty} \subset\mathcal{M}_{X,I=0}^{J,L}(\alpha_+, \alpha_-) $ be a sequence of holomorphic curves and their underlying currents are $\{\mathcal{C}_n\}_{n=1}^{\infty}$. By Taubes' Gromov compactness (Cf. Lemma 6.8 of [1]),  $\{\mathcal{C}_n\}_{n=1}^{\infty}$ converges to a broken holomorphic current
$\mathcal{C}_{\infty}= \{ \mathcal{C}_0, \dots , \mathcal{C}_N \}$
in the sense of Section 9 of \cite{H1}. Convergent as a  current implies that $[\mathcal{C}_n]= [\mathcal{C}_{\infty}]$ for large $n$. Then $J_0(\mathcal{C}_n)=J_0(\mathcal{C}_{\infty})$, where $J_0$ is a number which only  depends on relative homology class. For the precise definition of $J_0$, please refer to $\cite{H2}$.  According to Corollary 6.11 and  Proposition 6.14 of  \cite{H2},  $J_0(\mathcal{C}_n)=J_0(\mathcal{C}_{\infty})$ implies  that the genus of  domains of $\{u_n\}_{n=1}^{\infty}$ have an $n$ independent upper bound.%(Cf. section 6 of  [15])

With the above understanding, we can apply the  Gromov compactness in  \cite{FYHKE} to $\{u_n\}_{n=1}^{\infty}$.  After passing a subsequence,   $\{u_n\}_{n=1}^{\infty}$ converges to a broken holomorphic curve $u_{\infty}=\{u^{-N_-}, \dots, u^0, \dots, u^{N_+}\}$, where $u^0 \in \mathcal{M}_X^J(\alpha^0, \beta^0)$ and  $u^i \in \mathcal{M}_{Y_+}^{J_+}(\alpha^i, \beta^i)$ for $i>0$ and  $u^i \in \mathcal{M}_{Y_-}^{J_-}(\alpha^i, \beta^i)$ for $i<0$ , $\alpha^i= \beta^{i+1}$, $\alpha^{N_+}=\alpha_+$ and $\beta^{-N_-}=\alpha_-$.    Moreover,
\begin{equation} \label{e74}
I(u_{\infty})=I(u^{-N_-}) \dots + I(u^0) + \dots +I(u^{N_+})=0.
%&&{\rm ind}(u^0) + \dots + {\rm ind}(u^N)=0.
\end{equation}
By Lemma \ref{C25}, $u_{\infty}$ doesn't contain  closed components.

By Lemma \ref{C3}  and the fact that $I(u^i) \ge 0$ in $\mathbb{R} \times Y_{\pm}$, (\ref{e74}) forces $I(u^i)=0$ for all $i$. For $i \ne 0$, $I(u^i)=0 $ implies that $u^i$ are branched cover of cylinders with nonnegative Fredholm index. %By Lemma \ref{C2}, $u^0$ is disjoint union of embedded curves with  ${\rm ind}(u^0) =0$.
By Lemma \ref{C39}, ${\rm ind}(u^0)  \ge 0$.  The additivity of the Fredholm index implies ${\rm ind}(u^i)= 0$ for  $i$. In particular, $ u^i $ is   a   connector  for $i \ne 0$.

Since $u_n$ is  admissible for all $n$, $u^{N_+}$ is  positively admissible. Recall that there is no nontrivial connector from $p_{+}(\gamma, m)$ to other partitions of $\gamma^m$. As a result, $u^{N_+}$ is  a trivial connector, which is ruled out in  the holomorphic building.  Therefore, $u_{\infty}$ has no positive level. Similarly, $u_{\infty }$ has no negative level. In conclusion,  $u_{\infty}=u^{0} \in \mathcal{M}_{X,I=0}^{J,L}(\alpha_+, \alpha_-)$. %In conclusion,  $\mathcal{M}_{X,I=0}^{J,L}(\alpha_+, \alpha_-)$ is compact with respect to the SFT topology.
\end{proof}

\begin{corollary}
 For $ J \in \mathcal{J}_{1tame}(X, \pi_X,  {\omega_X}, J_{\pm}) $ and  any $L>0$, the moduli space  $\mathcal{M}_{X,I=0}^{J,L}(\alpha_+, \alpha_-)$ is a  finite set.  In particular, the equation (\ref{e21}) is meaningful. \label{C22}
\end{corollary}
\begin{proof}
The corollary follows from  Lemmas \ref{C4}  and \ref{C15}.
\end{proof}
%In conclusion, (\ref{e21}) is meaningful because of  Corollary \ref{C22}.

\subsubsection{Chain map}
To show that (\ref{e21}) is a chain map, we carry out the following compactness result.    It is similar to Lemma 7.23 of \cite{HT2}.
\begin{lemma}
%Let $\alpha_{\pm}$ be admissible  orbit set of $(Y_{\pm}, \pi_{\pm}, \omega_{\pm})$ with degree $ g(\Sigma)-1 < d \le Q$.
Let  $\alpha_{+}$ and $\alpha_-$ be  admissible  orbit sets  of $(Y_{+} , \pi_{+}, \omega_{+})$  and $(Y_{-} , \pi_{-}, \omega_{-})$ respectively with degree $ d> g(\Sigma)-1$.   Suppose that $(X, \pi_X)$ contains no  separating singular fiber.  For any $L>0$ and generic $J \in \mathcal{J}_{tame}(X, \pi_X, \omega_X)$, let $\{u_n\}_{n=1}^{\infty} \subset \mathcal{M}_{X,I=1}^{J,L}(\alpha_+, \alpha_-)$, then $\{u_n\}_{n=1}^{\infty} $ converges to a broken holomorphic curve $u_{\infty}$, either $u_{\infty} \in \mathcal{M}_{X,I=1}^{J,L}(\alpha_+, \alpha_-)$ or $u_{\infty}= \{u^0 , \dots, u^i, \dots, u^{N_+}\}$ or $u_{\infty}= \{u^{-N_-} , \dots u^i, \dots u^0\}$, where $u^0 \in \mathcal{M}_{X,I=0}^{J,L}(\beta_+, \beta_-)$ and it is embedded, $u^{N_+} \in \mathcal{M}_{Y_{+},I=1}^{J_{+},L}(\alpha_{+}, \beta_{+}) $, $u^{-{N_-}} \in \mathcal{M}_{Y_{-},I=1}^{J_{-},L}(\beta_{-}, \alpha_{-}) $ for some $\beta_{\pm}$ and $\{u^i\}_i$ are connectors for $i\ne 0, \pm N_{\pm}$. \label{C5}
\end{lemma}
\begin{proof}% For the sake of simplicity, we assume that $Y_-=\emptyset $.
Arguing as Lemma \ref{C4}, $\{u_n\}_{n=1}^{\infty} \subset \mathcal{M}_{X,I=1}^{J,L}(\alpha_+, \alpha_-)$ converges to a  broken holomorphic curve $u_{\infty}= \{u^{-N_-}, \dots, u^0 , \dots u^{N_+}\} $ in the sense of \cite{FYHKE} and
\begin{eqnarray*}
&&I(u^{-N_-}) \dots + I(u^0) + \dots +I(u^{N_+})=1.
\end{eqnarray*}
%We assume that $u^{N}$ is the top level all time.
%Since $I(u^i) \ge 0$, there exists a  unique level such that ECH index equals to one. Such level should be either the top level or bottom level.
Here we only consider case \ref{assumption3}, the other two cases are similar.

\begin{comment}
If $I(u^1)=0$, then $u^1$ is union of branched cover of trivial cylinder. As consequences of Lemma \ref{C2} and \ref{C11}, $I(u^0) \le 1$ and  ${\rm ind}u^0 \ge 0$ and $u^0$ has no nodes. By the fact that ${\rm ind}(u^1) \ge 0$ and  ${\rm ind}(u^1)=0 \mod 2$ whenever $\alpha_+$ is admissible, %When $\gamma$ is hyperbolic, $u^1$ is trivial cylinder since $\alpha$ is admissible.
 we must have ${\rm ind}(u^1)=0$ and $u^1$ is a connector.    Since $u_n$ is admissible, $u_{\infty}$ is also admissible. So $u^1$ must be a trivial connector.
Hence, $u_{\infty}$ has only one level, i.e. $u_{\infty}=u_0 \in \mathcal{M}_{X,I=1}^{J,L}(\alpha_+)$.

If $I(u^1)=1$, then $I(u^0)=0$ and $I(\tau^i)=0$. By Lemma \ref{C2},  $u^0$ is embedded with ${\rm ind}(u^0)=0$.   By Lemma 1.7 of \cite{HT1},  $\tau^i$ is branched cover of trivial cylinder and ${\rm ind}(\tau^i) \ge 0$. By Lemma 9.5 of \cite{H1} and $u^1$ is  admissible with ${\rm ind}(u^1)=1$. Additivity of Fredholm implies that   ${\rm ind}(\tau^i)=0$ for all $i$.
%Since ${\rm ind}(u^1)=I(u^1) \mod 2$ and $I(u^1) \ge {\rm ind}(u^1)$,   ${\rm ind}(u^1)=1$.
In conclusion, $u^0 \in  \mathcal{M}_{X,I=0}^{J,L}(\beta)$ and $\tau^i$ are connectors.
\end{comment}
First of all, keep in mind that  $0\le I(u^i) \le 1$  and  $u^i$ has nonnegative Fredholm index, for any $N_- \le i\le N_+$.  These follow  from Corollary \ref{C3}, Lemma \ref{C39} and Proposition 7.15 of \cite{HT1}. %we still have   ${\rm ind}u^0 \ge 0$.
Note that  $u^0$ could be a multiply covered holomorphic curve.

If $I(u^{N_+})=I(u^{-N_-})=0$, then $u^{\pm N_{\pm}}$ is a union of branched cover of trivial cylinders.
By the facts that ${\rm ind}(u^{\pm N_{\pm}}) \ge 0$ and  ${\rm ind}(u^{\pm N_{\pm}})=0 \mod 2$ whenever $\alpha_{+}$ and $\alpha_-$ are admissible, %When $\gamma$ is hyperbolic, $u^1$ is trivial cylinder since $\alpha$ is admissible.
 we   have ${\rm ind}(u^{\pm N_{\pm}})=0$ and $u^{\pm N_{\pm}}$ is connector.    Since each $u_n$ is admissible, $u_{\infty}$ is also admissible. As a result,  $u^{\pm N_{+}}$  and $u^{\pm N_{-}}$ must be   trivial connectors, which are ruled out in the holomorphic building.
Hence, $u_{\infty}=u^0 \in \mathcal{M}_{X,I=1}^{J,L}(\alpha_+, \alpha_-)$.

If $I(u^{N_+})=1$, then  $I(u^i)=0$ for all $i< N_+$. %By Lemma \ref{C39},  ${\rm ind}(u^0)\ge 0$.
 By Lemma 1.7 of \cite{HT1},  $u^i$ is branched cover of trivial cylinders with ${\rm ind}(u^i) \ge 0$ for $i \ne 0, N_+$. By Lemma 9.5 of \cite{H1}, $u^{N_+}$ contains a nontrivial component with $I={\rm ind}=1$. Additivity of Fredholm index implies that   ${\rm ind}(u^i)=0$ for all $i < N_+$. Argue as before, $u^{-N_-}$ is a trivial connector, which is ruled out in the holomorphic building. Hence, $u_{\infty}$ has no negative level. If $u^0$ contains  a multiply covered component, then  (\ref{e40}) tells us that it  is a branched  cover of a holomorphic cylinder   with   nonempty  positive  end  and  nonempty  negative end at hyperbolic orbits. However,  this contradicts  with our assumption that $\alpha_-$ is admissible.  Therefore, $u^0$ is embedded.
%Since ${\rm ind}(u^1)=I(u^1) \mod 2$ and $I(u^1) \ge {\rm ind}(u^1)$,   ${\rm ind}(u^1)=1$.
In conclusion, $u^0 \in  \mathcal{M}_{X,I=0}^{J}(\beta_+, \alpha_-, Z')$, $u^{N_+} \in \mathcal{M}^{J_+}_{Y_+, I=1}(\alpha_+, \beta_+)$ and $\{u^i\}_i$ are connectors for $ 0<i< N_+$.

If $I(u^{-N_-})=1$, then we can get the similar conclusion.
\end{proof}

\begin{lemma}
For  generic  $J \in \mathcal{J}_{tame}(X, \pi_X, \omega_X)$, $CP(X, \Omega_X, J, \Lambda_X)$ is  a chain map, namely, \label{C14}
$$ CP(X, \Omega_X, J, \Lambda_X) \circ \partial_+=\partial_- \circ CP(X, \Omega_X, J, \Lambda_X),$$
where $\partial_{\pm}$ is the differential of $CP_*(Y_{\pm}, \omega_{\pm}, \Gamma_{\pm}, J_{\pm}, \Lambda_P^{\pm})$.
\end{lemma}
\begin{proof}
We assume that $Y_-=\emptyset$ for simplicity. %%The statement is equivalent to
%\begin{eqnarray*}
%0&=&CP(X, \omega_X, J,  \Lambda_X)\circ \partial_+ \\
%&=&\sum\limits_W  \left(\sum\limits_{Z_1+Z_2=W} \sigma(\alpha, \beta ,Z_1 )\phi(\beta, Z_2) \right)  \Lambda_X(W),
%\end{eqnarray*}
%where $\sigma(\alpha, \beta ,Z_1 ) = \#_2 \mathcal{M}^{J_+}_{Y_+, I=1}(\alpha, \beta, Z_1)/ \mathbb{R}$ and $\phi(\beta, Z_2) = \#_2 \mathcal{M}_{X, I=0}^J(\beta, Z_2)$.
To show  that $CP(X, \Omega_X, J, \Lambda_X)$ is a chain map,   it suffices to  prove that
\begin{eqnarray} \label{e57}
\sum\limits_{Z_1+Z_2=W} \#_2 \mathcal{M}^{J_+}_{Y_+, I=1}(\alpha, \beta, Z_1) \#_2 \mathcal{M}_{X, I=0}^J(\beta, Z_2)=0,
\end{eqnarray}
where $ W \in H_2(X, \alpha)$, $Z_2 \in H_2(X, \beta)$ and $Z_1 \in H_2(Y_+, \alpha, \beta)$.

The usual strategy is to express (\ref{e57}) as  a  count of $ \partial \mathcal{M}_{X, I=1}^J(\alpha, W)$. However, the compactness result in Lemma \ref{C5} tells us that the broken holomorphic curves from  $ \partial \mathcal{M}_{X, I=1}^J(\alpha, W)$ not only involves elements from  $\mathcal{M}^{J_+}_{Y_+, I=1}(\alpha, \beta, Z_1)$  and $ \mathcal{M}_{X, I=0}^J(\beta, Z_2)$, but also involves connectors.

In fact, the similar phenomenon already happen in the proof of $\partial^2=0$. (See Lemma 7.23 of \cite{HT2}.) In the proof  of $\partial^2=0$, the key point is to glue a pair $(u_-, u_+)$ consisting of  $I=1$ holomorphic curves $u_-$ and $u_+$  in  $\mathbb{R} \times Y $, by inserting connectors in between. In our case,  the pair of  holomorphic curves above is  substituted by $(u_-, u_+)$, where $u_-$ is a curve  in $\overline{X}$ with $I(u_-)=0$ and $u_+$ is a curve in $\mathbb{R} \times Y_+$ with $I(u_+)=1$.

The strategy of gluing two holomorphic curves with $I(u_{\pm})=1$ in \cite{HT1} and \cite{HT2} still can be applied to our cases without essential change, because, their gluing analysis   mainly depends   on what is happening near the neck region. In addition, our  curves  $u_+$ and $u_-$ are admissible.  By Theorem 1.13 of \cite{HT1},   when the  orbit set $\beta$ is admissible,  the gluing coefficient   is $1$, otherwise, it is zero.  In other words, there is exactly one way (in mod 2 sense) to glue $(u_-, u_+)$  together when $\beta$ is admissible.
The argument in Section 7.3 of \cite{HT1} leads to  (\ref{e57}).

%holomorphic curves $u_-$ with $I(u_-)=0$ in $\overline{X}$ and $u_+$ with $I(u_+)=1$ in $\mathbb{R} \times Y_+$, by inserting connectors in between. The strategy of gluing two holomorphic curves with $I(u_{\pm})=1$ in \cite{HT1} and \cite{HT2} still can be applied to our cases without essential change, because, their gluing analysis   mainly depends   on what is happening near the neck region. In addition, both of our orbit sets $\alpha$, $\beta$ and curves $u_{\pm}$ are admissible.  By Theorem 1.13 of \cite{HT1}, the gluing coefficient   is $1$.

 A  summary of Hutchings and Taubes' gluing argument can be found in Section $6.5$ of \cite{VPK}, besides,  it indicates the minor changes in the cobordism case.

\end{proof}
In conclusion,  $CP(X, \omega_X, J,  \Lambda_X)$ induces a homomorphism
$$HP(X, \omega_X, J, \Lambda_X): HP_*(Y_+,  \omega_+,\Gamma_+, J_+, \Lambda^+_P) \to HP_*(Y_-,  \omega_-,\Gamma_-, J_-, \Lambda^-_P)  $$
in homology level.

\begin{comment}
\begin{remark} \label{r1}
If $X$ contains separating singular fibers,  then each  separating singular  fiber is union of pairs of embedded surfaces $(\Sigma_i, \Sigma_i')$.  Fix a  $\Gamma_X \in H_2(X, \partial X, \mathbb{Z})$ such that
$\partial_{Y_{\pm}} \Gamma_X =\Gamma_{\pm}$ and $\Gamma_X \cdot [\Sigma_i] \ge g(\Sigma_i)-1$ and  $\Gamma_X \cdot [\Sigma'_i] \ge g(\Sigma')-1$ for each $i$. Then $HP(X, \omega_X, J, \Gamma_X, \Lambda_X)$ is still well defined because of  the following reasons:  If $\mathcal{C}$ is a holomorphic current whose relative class is $\Gamma_X$ and $\mathcal{C}=\mathcal{C}' + \sum_i (m_i [\Sigma_i] + n_i [\Sigma_i']) $, where $\mathcal{C}'$ doesn't contain any $\Sigma_i$ or $\Sigma_i'$ component.   Without loss of generality, we may assume that $m_i \ge n_i$. By definition and adjunction formula, we have
\begin{equation*}
\begin{split}
I(\mathcal{C}) -I(\mathcal{C}')&=2 \sum_i  \left(\Gamma_X \cdot(m_i [\Sigma_i] + n_i [\Sigma_i'] ) + (m_i-n_i)^2 + m_i(1-2g(\Sigma_i)) + n_i(1-2g(\Sigma_i')) \right)\\
&\ge 2 \sum_i \left( n_i (d+1 -g(\Sigma)) + (m_i-n_i)(\Gamma_X \cdot [\Sigma_i] + 1- g(\Sigma_i))  \right) \ge 2.
\end{split}
\end{equation*}
Above inequality we have used $g(\Sigma)=g(\Sigma_i) + g(\Sigma_i')$. The conclusion in Lemma \ref{C25}  still holds and the fiber bubble still can be ruled.  The cobordism map is defined by previous arguments.
\end{remark}
\end{comment}
\subsection{Monotone case} \label{section21}
As  explained in Section \ref{section12}, in order to compare the  maps $HP_{sw}(X, \Omega_X, J, \Lambda_X)$ and $HP(X, \Omega_X, J, \Lambda_X)$, we need to find almost complex structures  in $\mathcal{J}_{comp}(X,  \Omega_X)$  such that these two maps can be defined  simultaneously. Under the  monotonicity assumptions, we show that there is  a  subset $\mathcal{V}_{comp}(X, \pi_X, \omega_X)^{reg}$ which meets our requirement.

%In order to define the cobordism maps  $HP(X,  \omega_X, J, \Lambda_X)$  when $X$ contains separating singular fiber and to  show that the cobordism maps $HP(X,  \omega_X, J, \Lambda_X)$  are equivalent to $\overline{HP}(X,  \omega_X,  \Lambda_X)$, we need to define them by using almost complex structures in $\mathcal{J}_{comp}(X, \pi_X, \omega_X)$ instead of $\mathcal{J}_{tame}(X, \pi_X, \omega_X)$.

Recall that when we define the cobordism maps in cases  \ref{assumption1} and   \ref{assumption2}, we need to use the genus bound $g(C) \ge g(B)$ to guarantee that  $C \star C \ge \frac{1}{2}$. However, this may not be true if we use $J \in \mathcal{J}_{comp}(X,  \Omega_X)$, because, the projection $\pi_X$   generally is not complex linear anymore. However, at least intuitively,  the bound $g(C) \ge g(B)$  should be  still  true when $J$ is close to   $\mathcal{J}_{tame}(X, \pi_X, \omega_X)$. The  following lemma summarizes the goal of this subsection.
\begin{lemma} \label{C48}
 Assume that $(X, \omega_X)$ is monotone,
 %or $(\omega_X, \Gamma_X)$ is monotone for $\Gamma_X \in H_2(X, \partial X, \mathbb{Z})$,
 then  there is a nonempty open subset $\mathcal{V}_{comp}(X, \pi_X, \omega_X, J_{\pm})$ of $\mathcal{J}_{comp}(X,  \Omega_X, J_{\pm})$ such that it satisfies following properties: Fix a generic $J\in \mathcal{V}_{comp}(X, \pi_X, \omega_X, J_{\pm})$, let $C$ be a simple irreducible  $J$-holomorphic curve  and it is not a fiber on the end of $\overline{X}$. Assume additionally that $C$  is not an exceptional sphere or torus with zero self intersection number, then  $C$ satisfies
\begin{equation} \label{e31}
2C \star C =2g(C)-2 + {\rm{ind}}(C) + h(C) + 2e_Q(C) + 4\delta(C) \ge  i,
\end{equation}
where $i=2$ in case  \ref{assumption1} and $i=1$ in case   \ref{assumption2}.
\end{lemma}
We  postpone the proof for the lemma to Section \ref{section13}.

This nice property guarantees that  we can define $HP(X,  \omega_X, J, \Lambda_X)$  for generic $J \in \mathcal{V}_{comp}(X, \pi_X, \omega_X, J_{\pm}) $ as before.  Moreover, using such an almost complex structure $J$, the assumption that $(X, \pi_X)$ only contains nonseparating fibers can be dropped. In the case that $(\Gamma_X, \omega_X)$ is monotone, the similar conclusion can be obtained. We use  $\mathcal{V}_{comp}(X, \pi_X, \omega_X) $ to denote   $\cup_{J_{\pm}}\mathcal{V}_{comp}(X, \pi_X, \omega_X, J_{\pm}) $, where $J_{\pm}$ runs over all  symplectization admissible almost complex structures.
 %We assume that $(X, \pi_X)$ is relative minimal in this  subsection. The case for non relative minimal case will be discussed in next subsection.

\begin{remark}
Note  that we don't use the  property $g(C) \ge g(B)$ in case  \ref{assumption3}. So in this section, we only consider cases  \ref{assumption1} and  \ref{assumption2}, unless otherwise stated.   In case \ref{assumption3},    we can define the cobordism maps  for generic $J \in \mathcal{J}_{comp}(X,  \Omega_X, J_{\pm}) $ as before, without the  monotonicity assumptions.  The only difference  is that the holomorphic torus and sphere may appear. Note that if they appear, then they lie inside the interior of $X$, because, their intersection number with the fibers on the ends is zero.

Firstly, let us consider the case that there is no holomorphic sphere. Then the ECH index is still nonnegative. But the curves  in  $\mathcal{M}_{X, I=i}^J(\alpha_+, \alpha_-)$ may contain multiple covers of torus for $i=0, 1$.  According to \cite{T2}, the transversality of multiple tori still can be obtained for  generic $J$. Also note that the multiple tori cannot shift to the symplectization  ends or other levels, or break into curves with asymptotic ends under Gromov compactness.  Under these observations, the  cobordism maps can be defined by the previous argument.

 %In fact, the embedded torus with $C \cdot C=0$ cannot appear in this case. Since $0=c_1(C)=\tau \int_C \omega_X=\tau \int_C \Omega_X=0$, $C $ must be constant.
For the cases that the holomorphic sphere with $C \star C=-1$  appears, the discussion is the same as  when $X$ is not relatively  minimal. We  discuss this case in  Section \ref{section4}.
%For the case that  $X$ is not relative minimal, please see the discussion in section 3.3.
\end{remark}

Without loss of generality,  in this section and Section \ref{section4} subsequently, we  assume that $(X, \pi_X)$ has only one singular fiber. If the singular fiber is  separating, then it is a union of two embedded surfaces $\Sigma_1$ and $\Sigma_2$. %In addition, we assume that $(X, \pi_X)$ is relative minimal.
Sections \ref{section20} and \ref{section13} consider the case that $(X, \pi_X)$ is relatively minimal.   The case that $(X, \pi_X)$ is not relatively   minimal  will be considered in  Section \ref{section4}.

%%%%%%%%%%%%%%%%%%%%%%%%%%%%%%%%%%%%%%%%%%%%%%%%%%%%%%%%%%%%%%%%%%%%%%%%%%%%%%%%%%%%

\begin{definition} \label{def5}
Given $L>0$, we define $\mathcal{V}^L_{comp}(X, \pi_X, \omega_X, J_{\pm}) $ to be a subset of $ \mathcal{J}_{comp}(X,  \Omega_X, J_{\pm})$ such that  $J \in \mathcal{V}^L_{comp}(X, \pi_X, \omega_X, J_{\pm}) $  if the following is true:
\begin{enumerate} [label=\textbf{S.\arabic*}]
\item  \label{s1}
For any   irreducible $J$ holomorphic curve  $C$ in $\overline{X}$ which has at least one end and $\int_C \omega_X  \le L$, then  $g(C) \ge g(B)$.

\item  \label{s2}
Let $C \in \mathcal{M}^J(\alpha_+, \alpha_-)$ with $d=[\alpha_{\pm}] \cdot [\Sigma] \ge 1$ and $\int_C \omega_X  \le L$. If $g(C)=g(B)=1$, then on each end of $\overline{X}$,  $C$ has exactly $d$-ends.
\item  \label{s3}
For any   irreducible  closed holomorphic curve $C$  in $\overline{X}$ with $\int_{C} \omega_X \le L$, then $C$ is homologous to $m[\Sigma] + m_1[\Sigma_1] + m_2[\Sigma_2]$ for some nonnegative intergers $m$, $m_1$ and $m_2$.
\end{enumerate}

\end{definition}

\begin{lemma}

\label{C19}
 $\mathcal{V}_{comp}^L(X, \pi_X, \omega_X, J_{\pm}) $ is a nonempty open subset of $ \mathcal{J}_{comp}(X, \Omega_X, J_{\pm})$.
\end{lemma}
\begin{proof} Firstly, it is worth noting that
$$ \mathcal{J}_{comp}(X,   \Omega_X, J_{\pm}) \cap \mathcal{J}_{tame}(X, \pi_X, \omega_X, J_{\pm}) \ne \emptyset \subset \mathcal{V}_{comp}^L(X, \pi_X, \omega_X, J_{\pm}) $$
 for any $L> 0$. In particular, $ \mathcal{V}_{comp}^L(X, \pi_X, \omega_X, J_{\pm}) $ is nonempty for any $L>0$.

For openness, we argue by contradiction. Suppose that   $\mathcal{V}_{comp}^L(X, \pi_X, \omega_X, J_{\pm})$ is not open,
then there exists a sequence of  irreducible  holomorphic curves $\{(C_n, J_n)\}_{n=1}^{\infty}$ with $\int_{C_n} \omega_X  \le  L$  that  do not satisfy the conclusions and $J_n$ converges to  $J_{\infty} \in \mathcal{V}_{comp}^L(X, \pi_X, \omega_X, J_{\pm})$ in $C^{\infty}$ topology. By Gromov compactness \cite{FYHKE},  $\{(C_n, J_n)\}_{n=1}^{\infty}$ converges to $(u_{\infty}, J_{\infty})$, where $u_{\infty}$ is possibly broken with nodes.

Firstly,  we deal with the case that each $C_n$ has at least one end, so does  $u_{\infty}$.  %We may assume that $u_{\infty}$ is not broken and irreducible, otherwise, we consider the an irreducible component in cobordism level of $u_{\infty}$ which has at least one end.  %According to Theorem 6.19 of \cite{Wen2}, we factorize $u_{\infty}= v_{\infty} \circ \varphi$, where $v_{\infty}$ is simple holomorphic curve.
Let $u^0$ be  an irreducible component in cobordism level of $u_{\infty}$ which has at least one end.  %According to Theorem 6.19 of \cite{Wen2}, we factorize $u_{\infty}= v_{\infty} \circ \varphi$, where $v_{\infty}$ is simple holomorphic curve
If   $C_{n}$ violates \ref{s1}, then the genus of  the domain of $u^0$ is strictly less than $g(B)$.
%has genus less than $g(B)$, so does  $u^0$.  % so does $v_{\infty}$ byth Riemann-Hurwitz formula.
     However, it is impossible since $J_{\infty} \in \mathcal{V}_{comp}^L(X, \pi_X, \omega_X, J_{\pm})$.

If   $C_{n}$ violates \ref{s2}, i.e., $g(C_n)=g(B)=1$, but $C_n$ has strictly less than $d$-ends. If $d=1$, then there is nothing to prove, because, we assume that $C_n$ has at least one end. Let us assume that $d\ge 2$.  Since $J_{\infty} \in \mathcal{V}_{comp}^L(X, \pi_X, \omega_X, J_{\pm})$,  the cobordism level of $u_{\infty}$ has exactly one irreducible component $u^0$ with at least one end, otherwise, combine with the discussion in the last paragraph,  the arithmetic genus of the domain of  $u_{\infty}$ is strictly greater than $1$, contradict with $g(C_n)=1$. Moreover, the domain of  $u^0$ has genus one and $u^0$ has  $d$-ends. As a result, $u_{\infty}$ is broken. Obviously, the arithmetic genus of the domain of $u_{\infty}$ is strictly greater than $1$, we get a contradiction.

Assume that each $C_{n}$ is closed and it violates \ref{s3}. Then $u_{\infty}$ cannot be broken, otherwise, the top (bottom) level of   $u_{\infty}$  has no positive (negative) ends.  Then its homology class must be $m[\Sigma]+ m_1[\Sigma_1] + m_2 [\Sigma_2]$  for some nonnegative integers $m, m_1$ and $m_2$ because of  $J_{\infty} \in \mathcal{V}_{comp}^L(X, \pi_X, \omega_X, J_{\pm})$. As a result, this is also true for large $n$, then we get a  contradiction.

\end{proof}

\begin{corollary} \label{C47}
For generic $J \in \mathcal{V}_{comp}^L(X, \pi_X, \omega_X, J_{\pm})$, let $C$ be a simple irreducible $J$-holomorphic curve which has at least one and  $\int_C \omega_X  \le L$.  Then
$2C \star C  \ge  i,$
where $i=2$ in case  \ref{assumption1} and $i=1$ in case   \ref{assumption2} .
\end{corollary}
\begin{proof}
Follows from Lemma \ref{C19} and repeat the same argument in Lemma \ref{C40}.
\end{proof}
\begin{remark}\label{r6}
Note that  any closed simple holomorphic curve $C$ in $\overline{X}$  such that $C \bigcap X \ne \emptyset$ must lie inside the interior of $X$.  For generic $J \in  \mathcal{V}^L_{comp}(X, \pi_X, \omega_X, J_{\pm})$, such a curve $C$ is Fredholm regular.
\end{remark}

\subsubsection{$(\Gamma_X, \omega_X)$ is monotone} \label{section20}
\begin{corollary} \label{C53}
Suppose that $(X, \pi_X)$ is relatively minimal  and  $(\Gamma_X, \omega_X)$ is monotone,  then there exists a nonempty open subset  $\mathcal{V}_{comp}(X, \pi_X, \omega_X, J_{\pm}) $ of  $\mathcal{J}_{comp}(X, \pi_X, \omega_X, J_{\pm})$ with the following significance: For generic $J \in \mathcal{V}_{comp}(X, \pi_X, \omega_X, J_{\pm}) $, the cobordism map  $HP(X, \Omega_X, \Gamma_X, J, \Lambda_X)$ is well defined.
\end{corollary}
\begin{proof}
Let $\mathcal{C}_1$ and $\mathcal{C}_2$ be holomorphic currents with the same  relative homology class $\Gamma_X$. Then by definition of ECH index, we have
\begin{equation*}
I(\mathcal{C}_1) - I(\mathcal{C}_2)=<c_1(T\overline{X}) + 2PD(\Gamma_X), \mathcal{C}_1- \mathcal{C}_2 >=\tau \left( \int_{\mathcal{C}_1} \omega_X - \int_{\mathcal{C}_2} \omega_X \right).
\end{equation*}
Therefore, the $\omega_X$-energy of the elements in  $\bigsqcup\limits_{ Z \in H_{2}(X, \alpha_+, \alpha_-), [Z]=\Gamma_X}\mathcal{M}_{X, I=i}^{J}(\alpha_+, \alpha_-, Z)$ are constant  which  only depends on $\alpha_{\pm}$. Their maximum is denoted by $L_*$. Take $\mathcal{V}_{comp}(X, \pi_X, \omega_X, J_{\pm})=\mathcal{V}^{L_*}_{comp}(X, \pi_X, \omega_X, J_{\pm}) $.

Assume that $g(\Sigma) \ge 2$. Since $(X, \pi_X)$ is relatively  minimal,   $g(\Sigma_i) \ge 1$. By adjunction formula, it is easy to check that
\begin{equation} \label{e70}
I(m [\Sigma] + m_1 [\Sigma_1] + m_2 [\Sigma_2])<0.
\end{equation}
  According to Lemma \ref{C19} and Lemma 2.4 in \cite{H3}, for generic $J$ in $\mathcal{V}_{comp}(X, \pi_X, \omega_X, J_{\pm})$, the holomorphic current consists of  closed holomorphic curves must be of the form $(C, d)$, where $C$ is a fiber in the ends of $\overline{X}$. If $g(\Sigma) \le 1$, then the singular fiber is nonseparate, otherwise, $(X, \pi_X)$ is not relatively minimal.

In both cases, by Corollary \ref{C47},  (\ref{e31}) is true for all irreducible simple holomorphic curves with energy less than $L$, except   for the close curves with  the same  homology  as  the fibers.  Replace Lemma \ref{C40} by Corollary \ref{C47},  then repeat the previous arguments in Lemmas \ref{C3}, \ref{C2}, \ref{C11}, \ref{C39}, \ref{C25}, \ref{C4}, \ref{C5} and \ref{C14},  $HP(X, \Omega_X, \Gamma_X, J,  \Lambda_X)$ is well defined.

\end{proof}
%%%%%%%%%%%%%%%%%%%%%%%%%%%%%%%%%%%%%%%%%%%%%%%%%%%%%%%%%%%%%%%%%%%%%%%%%%%%%%%%%%%%%%%%%%%%%%%%%%%%%%%%%%%%%%%%%

\subsubsection{$(X, \omega_X)$ is monotone } \label{section13}
Observe that there are only finitely many pairs of orbit  sets $(\alpha_+, \alpha_-)$ with $[\alpha_{\pm}] \cdot [\Sigma] \le Q$. We label these pairs by $\{(\alpha_+^k, \alpha_-^k)\}_{k=1}^{N_Q}$.   For each pair  $(\alpha_+^k, \alpha_-^k)$, fix a $J_k$ holomorphic curve $u_k: C_k \to \overline{X}$ which violates  the  conditions  \ref{s1} or   \ref{s2}  in Definition \ref{def5}, where $J_k \in \mathcal{J}_{comp}(X,  \Omega_X, J_{\pm})$. Denote $L_0 = \max_k\{\int_{C_k}u_k^*\omega_X\}$ and $I_0= \max_k \{ |{\rm ind} C_k|\}$. Note that in either case, $g(C_k) \le g(B)$.

The consequence of the following lemma is that if $C$ is a simple holomorphic curve which violates  the requirements  \ref{s1} or  \ref{s2} in Definition \ref{def5}, then the Fredholm index of $C$ is larger than $3$.
\begin{lemma}
Suppose that $(X, \omega_X)$ is monotone, there exists  $L_* >0$  with the following   significance: Given $L \ge L_*$,    for $J \in  \mathcal{V}_{comp}^L(X, \pi_X, \omega_X, J_{\pm}) $,  \label{C24}
%\begin{enumerate}
%\item[]
%For  any irreducible simple $J$ holomorphic curve  $C$ in $\overline{X}$ between  orbit sets (at least one is nonempty)  of $Y_+$ and $Y_-$ with $0 \le \rm{ind}(C) \le 2$, then  $g(C) \ge g(B)$.
given orbit  sets $\alpha_{+}$ and $\alpha_-$ (at least one is nonempty) and  an  irreducible  $J$ holomorphic curve  $C \in \mathcal{M}^J(\alpha_+, \alpha_-)$ (not necessarily simple) with $0 \le {\rm{ind} }C \le 3$,    then the following are true:
\begin{enumerate}
\item
$g(C) \ge g(B)$,
\item
If $g(C)= g(B)=1$, then $C$ has exactly $d=\alpha_{\pm} \cdot [\Sigma]$-ends.
\end{enumerate}
%\end{enumerate}
%Then there exists $L_0>0$ such that $ \mathcal{V}^L_{comp}(X, \pi_X, \omega_X, J_{\pm})  \subset \mathcal{U}_{comp}(X, \pi_X, \omega_X, J_{\pm}) $ for any $L \ge L_0$.
\end{lemma}
\begin{proof}
Let $J \in \mathcal{V}_{comp}^{L}(X, \pi_X, \omega_X,   J_{\pm}) $ and  $C$ be  an irreducible $J$ holomorphic curve from $\alpha_+$ to $\alpha_-$ with  $0 \le \rm{ind}(C) \le 3$ and $  \int_{C} \omega_X > L$. Assume additionally that $g(C) \le g(B)$.  %but  $C $ violates  the conclusions above.   Note that $g(C) \le g(B)$ in both cases.

There exists  $k$ such that $ \alpha_{\pm}= \alpha_{\pm}^k $. By  definition of the Fredholm index,
\begin{eqnarray}\label{e66}
{\rm{ind}}(C)- {\rm{ind}}(C_k) = \chi(C_k)- \chi(C) + 2<c_1(T\overline{X}), C-C_k> + \mu_{\tau}^{\rm{ind}} (C)-\mu_{\tau}^{\rm{ind}} (C_k).
\end{eqnarray}
Since there are only finitely many  orbit sets $\alpha_{\pm}$  with degree less than $Q$ and $g(C), g(C_k) \le g(B)$,   we can find a constant $c_Q>0$ which only  dependents  on $Q$ and $g(B)$  such that
\begin{equation} \label{e67}
|\mu_{\tau}^{\rm{ind}} (C)-\mu_{\tau}^{\rm{ind}} (C_k)| +|\chi(C_k)- \chi(C)| \le c_Q.
\end{equation}
Combine (\ref{e66}) and $(\ref{e67})$, we have $ |<c_1(T\overline{X}), C-C_k>| \le c_0$, where $c_0 =I_0 +c_Q +3$.

On the other hands, by the  monotonicity assumption,
\begin{equation} \label{e11}
|\int_C \omega_X - \int_{C_k} \omega_X|= |<\omega_X, C-C_k>|= |\tau <c_1(T\overline{X}), C-C_k>| \le |\tau| c_0,
\end{equation}
In other words, the $\omega_X$-energy  of holomorphic curve which violates the desired properties is uniformly bounded.
    %We give more detail of this point as follows.

Take $L \ge L_* \ge L_0 + |\tau| c_0$.  Suppose that  there exists  $J \in \mathcal{V}_{comp}^{L}(X, \pi_X, \omega_X,   J_{\pm}) $ and    an irreducible $J$ holomorphic curve $C$ violated the conclusions of the lemma.  Then $C$ must satisfy  $0 \le \rm{ind}(C) \le 3$ ,  $  \int_{C} \omega_X > L$ and  $g(C) \le g(B)$.
However,
$$ \int_{C} \omega_X - \int_{C_k} \omega_X > L_* -  \int_{C_k} \omega_X \ge  |\tau| c_0$$
contradicts with (\ref{e11}).

\end{proof}

\begin{corollary} \label{C49}
Assume  that $(X, \omega_X)$ is monotone. For    $J \in  \mathcal{V}_{comp}^{L_*}(X, \pi_X, \omega_X, J_{\pm})^{reg} $, let  $C$ be an  irreducible simple $J$ holomorphic curve with at least one end, then
$2C \star C  \ge  i,$
where $i=2$ in case  \ref{assumption1} and $i=1$ in case   \ref{assumption2}.
\end{corollary}
\begin{proof}
If $\int_C \omega_X \le L_*$,  then the conclusion follows from Corollary \ref{C47}.

If $\int_C \omega_X >L_*$, by  Lemma \ref{C24}, either ${\rm ind } C <0 $ or ${\rm ind } C \ge 4$.  Since $J$ is generic, the former case cannot happen. The conclusion follows from the definition of $C \star C$.
\end{proof}

Now we  examine (\ref{e31}) for closed  simple holomorphic curve $C$.  By Remark  \ref{r6}, we   assume that  $C$ lies inside  the interior of $X$, otherwise, $C$ is  a  fiber over  the ends  of $\overline{X}$.
% $C \bigcap X \ne \emptyset$, otherwise, $C$ is  a  fiber on the end of $\overline{X}$.
\begin{lemma} \label{C36}
Assume  that $(X, \omega_X)$ is monotone and $(X, \pi_X)$ is relatively  minimal.  For generic $J \in \mathcal{V}^{L_*}_{comp}(X, \pi_X, \omega_X, J_{\pm})$, let $C$ be a  simple closed $J$ holomorphic curve in $\overline{X}$, then   either (\ref{e31}) is still true for  $C$ or $[C]=m[\Sigma]$ for $m \ge 0$.
\end{lemma}
\begin{proof}

 %By Remark  \ref{r6}, $C$  is Fredholm regular for generic $J$.
%Let $C$ be an irreducible simple holomorphic curve in $X$. Then $C$  is Fredholm regular for generic $J$.% Since  $[C] \ne m[\Sigma]$, we must have  $C \bigcap X \ne \emptyset$ and $C$ is Fredholm regular.

By Definition  \ref{def5}, any such a closed curve  satisfies  either $[C]=m[\Sigma]+ m_1[\Sigma_1]+m_2[\Sigma_2]$ or $\int_C \omega_X  > L_*$.   By (\ref{e70}) and Remark  \ref{r6},  the former case cannot happen if $g(\Sigma) \ge 2$, and $[C] =m[\Sigma]$ when $g(\Sigma) \le 1$.   Hence, it suffices to consider  $\int_C \omega_X  > L_*$.

If $g(C) \ge 2$, then (\ref{e31}) is automatically true for $C$.   Suppose that $g(C) \le 1$,  by the monotonicity assumption,
\begin{equation*}
{\rm{ind}}(C) = 2<c_1(T \overline{X}), C> + 2g(C)-2 =  \frac{2}{\tau}\int_C \omega_X  + 2g(C)-2.
\end{equation*}
Since $L_* > 2 |\tau|$,  either ${\rm{ind}}(C)<0$ or ${\rm{ind}}(C) \ge 4$.  The former case cannot appear for generic $J$ and  (\ref{e31}) is  true for $C$ in the latter case.
 \end{proof}

\begin{proof}[Proof of Lemma \ref{C48}]
Define $\mathcal{V}_{comp}(X, \pi_X, \omega_X, J_{\pm})=\mathcal{V}^{L_*}_{comp}(X, \pi_X, \omega_X, J_{\pm})$.
The lemma is a direct consequence of Corollary  \ref{C49} and Lemma  \ref{C36}.
\end{proof}

\begin{lemma}
Suppose that $(X, \omega_X)$ is monotone and $(X, \pi_X)$ is relatively minimal. Then there exists a nonempty open subset $\mathcal{V}_{comp}(X, \pi_X, \omega_X, J_{\pm})$  of $\mathcal{J}_{comp}(X,   \Omega_X, J_{\pm})$ such that for $J \in \mathcal{V}_{comp}(X, \pi_X, \omega_X, J_{\pm})^{reg}$, $HP(X, \Omega_X, J , \Lambda_X)$ is well defined.
\end{lemma}
\begin{proof}
%For  $J \in \mathcal{V}_{comp}(X, \pi_X, \omega_X, J_{\pm})^{reg}$, by the definition of $\mathcal{V}^{L_*}_{comp}(X, \pi_X, \omega_X, J_{\pm})$, if $C$ is a closed holomorphic curve such that   $[C] \ne m[\Sigma]$, then   $\int_C \omega_X > L_0$.
%%$$I(m[\Sigma_i]) =m(2-2g(\Sigma_i)) + m(m+1) \Sigma_i \cdot \Sigma_i<0.$$
%%As a result, for generic $J$, any simple closed holomorphic curve with $[C] \ne m[\Sigma]$ in $\overline{X}$ must have $\int_C \omega_X > L_0$.
%According to Lemma \ref{C24} and  \ref{C36},  (\ref{e31}) is true for any simple holomorphic curve except the curve with homology class $m[\Sigma]$.
Replace Lemma \ref{C40} by Lemma \ref{C48}, then we    repeat the same argument in Lemmas \ref{C3}, \ref{C2}, \ref{C11}, \ref{C25}, \ref{C4}, \ref{C5}  and \ref{C14} to define the cobordism maps $HP(X, \Omega_X, J, \Lambda_X)$.
\end{proof}

\begin{comment}
\begin{remark} \label{r4}
The set  of almost complex structures $ \mathcal{J}_{comp}(X, \pi_X, \omega_X)^{reg}$ in Lemma \ref{C27} can be replaced by $ \mathcal{V}_{comp}(X, \pi_X, \omega_X)^{reg}$ for the following reasons.   Using Taubes' Gromov compactness (Cf. Lemma 6.8 of \cite{HT}) and  limit argument  in  Lemma \ref{C19}, $J' \in  \mathcal{V}_{comp}(X, \pi_X, \omega'_X)$ provided that $\delta>0$ small enough and $J \in  \mathcal{V}_{comp}(X, \pi_X, \omega_X)$.
\end{remark}
\end{comment}

\begin{comment}
When $X$ is relatively minimal, then $g(\Sigma_i) \ge 1$.  By adjunction formula, it is easy check that
$$I(m[\Sigma]+m_1[\Sigma_1]+ m_2[\Sigma_2])<0.$$
%$$I(m[\Sigma_i]) =m(2-2g(\Sigma_i)) + m(m+1) \Sigma_i \cdot \Sigma_i<0.$$
Therefore, there is no simple closed holomorphic curve with homology class $m[\Sigma]+m_1[\Sigma_1]+ m_2[\Sigma_2]$ for generic $J \in  \mathcal{V}^{L_0}_{comp}(X, \pi_X, \omega_X, J_{\pm})$. According to Lemma \ref{C24} and  \ref{C36},  (\ref{e31}) is true for any simple holomorphic curve except the curve with homology class $m[\Sigma]$. Then one can  repeat the argument in Lemma \ref{C3}, \ref{C2}, \ref{C11}, \ref{C25}, \ref{C4}, \ref{C5} and \ref{C14} to define the cobordism map $HP(X, \omega_X, J, \Lambda_X)$ when $d > g(\Sigma)-1$.
\end{comment}
\subsection{$(X, \pi_X)$ is not relatively minimal } \label{section4}
 This section concerns the case that $(X, \pi_X)$ is not  relatively minimal.  Without loss of generality,   assume that $\Sigma_2$ is  an exceptional sphere. Its homology class is  denoted by    $E$. The discussion here follows \cite{T2}.  To define the cobordism maps, we need further constrains on the holomorphic currents.
\begin{definition}
 A holomorphic current $\mathcal{C}=\{(C_a, d_a)\}$ is called admissible if $d_a=1$ whenever $C_a$ is a holomorphic sphere with $C_a \cdot C_a=-1$.
 Given orbit  sets $\alpha_{+}$ and $\alpha_-$, and a relative homology class $Z$, let $\mathcal{M}_{X}^{J, ad}(\alpha_+, \alpha_-, Z)$ denote the moduli space of holomorphic curves whose underlying holomorphic currents  are admissible.
\end{definition}

Under certain assumptions, we claim that  the  cobordism maps are still well defined.  The results are summarized in the following lemma.
%and then induce a well-defined homomorphism in homology level.
%  The only difference is that  we need to replace the moduli space $\mathcal{M}_{X, I=0}^J(\alpha_+, \alpha_-, Z)$ by $\mathcal{M}_{X, I=0}^{J, ad}(\alpha_+, \alpha_-, Z)$.
\begin{lemma} \label{C52}
Let $\Gamma_X \in H_2(X, \partial X, \mathbb{Z})$,  we consider any one of the following cases:
\begin{enumerate} [label=\textbf{C.\arabic*}]
\item \label{case1}
$(X, \omega_X)$ or $(\Gamma_X, \omega_X)$ is  monotone.  Take  $J \in \mathcal{V}_{comp}(X, \pi_X, \omega_X)^{reg}$.
\item \label{case2}
The relative class satisfies $\Gamma_X \cdot E \le 0$.   Take $J \in \mathcal{J}_{tame}(X, \pi_X, \omega_X)^{reg}$.
\end{enumerate} Replace the moduli space $\mathcal{M}_{X, I=0}^J(\alpha_+, \alpha_-, Z)$ by $\mathcal{M}_{X, I=0}^{J, ad}(\alpha_+, \alpha_-, Z)$  in equation (\ref{e21}), then  it is still meaningful. Moreover, the map defined by  equation (\ref{e21}) is a chain map,  and it induces a well-defined homomorphism
$HP(X, \Omega_X,  \Gamma_X, J,\Lambda_X)$.
\end{lemma}

%The same argument in Lemmas \ref{C2},
% \ref{C11} and \ref{C25} deduces that the element in $\mathcal{M}_{X, I=i}^{J, ad}(\alpha_+, \alpha_-, Z)$ is embedded, where $0 \le i \le 1$.  The following lemma tells us that the multiply covered exceptional sphere cannot  appear when we take Grmov compactness.

It is worth noting that $I(m[\Sigma]+ m_1[\Sigma_1] + m_e E) \le 0 $ and equality holds if and only if $m=m_1=0$ and $m_e=1$ when $g(\Sigma) \ge 2$.
%Let $C$ be a  simple closed holomorphic curve in $\overline{X}$ with homology class $m[\Sigma]+ m_1[\Sigma_1] + m_e E$.
In the monotone case \ref{case1}, by the  observation  above and Lemma \ref{C48},  any holomorphic current can be written as $\mathcal{C}+ m_eE$, where $\mathcal{C}$ doesn't contain any $E$ component and $I(\mathcal{C}) \ge 0$. When $g(\Sigma)=0$, let $C$ be a exceptional sphere with $[C]= m[\Sigma]+ m_1[\Sigma_1] + m_2 [\Sigma]$. Without loss of generality, assume that $m_2 \ge m_1$, hence we can write $[C] =m [\Sigma] + m_2 [\Sigma_2]$. By adjunction formula, it is easy to check that $C$ is an exceptional sphere if  and only if $m=0$ and $m_2 =1$.  Therefore, we still have  the decomposition $\mathcal{C}+ m_eE$. For the case $g(\Sigma)=1$, the homology class of exceptional sphere is not necessarily  to be $E$,  but the discussion that follows still works, we omit this case here.

In the case \ref{case2}, for any holomorphic current $\mathcal{C}$ with relative class $\Gamma_X$,  we can write $\mathcal{C} =  \mathcal{C}_0 + m[\Sigma] + m_1[\Sigma_1] + m_e E$, where  $\mathcal{C}_0 $ doesn't contain any closed component.   Then $\mathcal{C} \cdot E =\Gamma_X \cdot E \le 0$ implies that $m_e \ge m_1$.  Since $E + [\Sigma_1]=[\Sigma]$, we write $\mathcal{C} = \mathcal{C}' + (m_e-m_1) E$, where $\mathcal{C}' =\mathcal{C}_0 + (m +m_1)[\Sigma]$.  By  Corollary \ref{C3} and Lemma  \ref{C25},  we have $I(\mathcal{C}') \ge 0$.

  %Without loss
%of generality, we may assume that  $\mathcal{C}=\mathcal{C}' + m_e E $.

%Note  that $I(C) \ge 0$ if and only if $[C]=E$, i.e $C$ is a holomorphic sphere with self intersection $-1$.
%By this observation and Lemmas \ref{C24} and \ref{C36}, in the cases that $(X, \omega_X)$ or $(\Gamma_X, \omega_X)$ is monotone,

\begin{lemma} \label{C37}
%Suppose that $(X, \omega_X)$ or $(\Gamma_X, \omega_X)$ is monotone.
For generic $J$, let $\{\mathcal{C}_n\}_{n=1}^{\infty}$ be a sequence of admissible holomorphic currents in  $\widetilde{\mathcal{M}}^{J, L}_{X, I=i}(\alpha_+, \alpha_-)$,   where $0\le i \le 1$. Then $\{\mathcal{C}_n\}_{n=1}^{\infty}$ converges to a broken holomorphic currents $\mathcal{C}_{\infty}$ in the current sense. Decompose $\mathcal{C}_{\infty}=\mathcal{C}+ m_e E$,   where $\mathcal{C}$ doesn't contain any $e$ component. Assume that $I(\mathcal{C}) \ge 0$,  then  $m_e=1$. %$\mathcal{C}_{\infty}$ is still admissible.
\end{lemma}
\begin{proof}
By Taubes' Gromov compactness,  $\{\mathcal{C}_n\}_{n=1}^{\infty}$  converges to $\mathcal{C}_{\infty}$ (possibly broken) in the sense of Lemma 9.9 of \cite{H1}.
%Suppose that $\mathcal{C}_{\infty}=\mathcal{C}+ m E$, where $\mathcal{C}$ doesn't contain any $e$ component. Then
By definition,
\begin{equation*}
I(\mathcal{C}_{\infty})=I(\mathcal{C}+ m_e E)=I(\mathcal{C})+ I(m_eE)+ 2 m_e \mathcal{C} \cdot E = I(\mathcal{C})+ m_e-m_e^2 + 2 m_e \mathcal{C} \cdot E.
\end{equation*}
On the other hand, $\mathcal{C}_{\infty} \cdot E =\mathcal{C} \cdot E -m_e =\mathcal{C}_{n} \cdot E\ge -1$.
As a result, $I(\mathcal{C}_{\infty}) \ge m_e(m_e-1)$. By our assumption, $m_e=1$. %In other words, $\mathcal{C}_{\infty}$ is admissible.
\end{proof}

\begin{comment}
For our goal, we need to ensure that the  compactness of  $\mathcal{M}_{X, I=0}^{J, ad}(\alpha_+, \alpha_-, Z)$. Let $\{\mathcal{C}_n \}_{n=1}^{\infty}\subset \widetilde{\mathcal{M}}^{J, L}_X(\alpha_+, \alpha_-)$ be a sequence of admissible holomorphic currents with ECH index $0$ or $1$. By intersection positivity and our assumption that $\mathcal{C}_n$ is admissible, we have $\mathcal{C}_n \cdot e \ge -1$. By Taubes' Gromov compactness,  $\mathcal{C}_n$ converges to $\mathcal{C}_{\infty}$(possibly broken) in the sense of section 9 of \cite{H1}, then $\mathcal{C}_{\infty} $ is still admissible for the following reason.  Suppose that $\mathcal{C}_{\infty}=\mathcal{C}+ m E$, where $\mathcal{C}$ doesn't contain any $e$ component. Then
\begin{eqnarray*}
I(\mathcal{C}_{\infty})=I(\mathcal{C}+ m e)=I(\mathcal{C})+ I(mE)+ 2 m \mathcal{C} \cdot E = I(\mathcal{C})+ m-m^2 + 2 m \mathcal{C} \cdot E.
\end{eqnarray*}
On the other hand, $\mathcal{C}_{\infty} \cdot e =\mathcal{C} \cdot e -m \ge -1$. Thus, $I(\mathcal{C}_{\infty}) \ge m(m-1)$. Hence, $I(\mathcal{C}_{\infty}) \le 1$ implies that $m=1$, i.e. $\mathcal{C}_{\infty}$ is admissible.
\end{comment}

\begin{proof}[Proof of Lemma \ref{C52}]
Let us consider case \ref{case2}. Write $\mathcal{C} =\mathcal{C}' + m_e E $ as before.  Observe that  if    $\mathcal{C}$ is admissible with $0 \le I(\mathcal{C})\le 1$, then by Lemma \ref{C25},   $\mathcal{C}'$ doesn't contain any closed component and $m_e=0$ or $1$. % We can use Lemma \ref{C25} to rule the closed holomorphic curve except the exceptional sphere.
 Note that the exceptional sphere is Fredholm regular due to  the automatic transversality  theorem.(Cf. Page 51 of \cite{H3}).   Take a generic   $J \in \mathcal{J}_{tame}(X, \pi_X, \omega_X) $,  replace the   moduli space $\mathcal{M}_{X, I}^J(\alpha_+, \alpha_-, Z)$  by $\mathcal{M}_{X, I}^{J, ad}(\alpha_+, \alpha_-, Z)$.
  By Lemma \ref{C37}, the compactness results in Lemmas \ref{C4} and \ref{C5} still hold, because, there is no holomorphic curve with negative ECH index created after taking Gromov compactness.   Moreover, all the holomorphic curves which contribute to the cobordism maps are embedded. By using the same argument as before, we define  cobordism maps $HP(X,\Omega_X, \Gamma_X, J, \Lambda_X)$, even though $X$ is not relatively  minimal. This leads to  the last bullet in Remark \ref{r2}.

Similarly, by the same reasons as above, the cobordism maps can be defined by counting  embedded holomorphic curves  in monotone case \ref{case1}.
\end{proof}

\subsection{Properties of $HP(X, \Omega_X, J, \Lambda_X)$}
\subsubsection{Composition rule} \label{section2}
Let $(X_+, \pi_{X_+}, \omega_{X_+})$ be a fierwise symplectic  cobordism from $(Y_+, \pi_+, \omega_+)$ to $(Y_0, \pi_0, \omega_0)$ and $(X_-, \pi_{X_-}, \omega_{X_-})$ be a fiberwise  symplectic   cobordism from $(Y_0, \pi_0, \omega_0)$ to  $(Y_-, \pi_-, \omega_-)$. Given $J_{X_{\pm}} \in \mathcal{J}_{tame}(X_{\pm}, \pi_{X_{\pm}}, \omega_{X_{\pm}})$, define $(X_R, \pi_{X_R}, \omega_R)$ and $J_R$ as Section \ref{section22}.
%Let $(X, \pi_X, \omega_X)$ be the composition of $(X_{+}, \pi_{X_+}, \omega_{X_+})$ and $(X_{-}, \pi_{X_-}, \omega_{X_-})$,  and $J_X = J_{X_+} \circ J_{X_-}$, then we have the following composition rule.

\begin{lemma}
Assume that  $(X_{+}, \pi_{X_+}, \omega_{X_+})$ and $(X_-, \pi_{X_-}, \omega_{X_-})$ satisfy the  assumptions $(\spadesuit)$  and they   contain no  separating singular fiber.  Introduce   $(X_R, \pi_{X_R}, \omega_R)$ and $J_R$ as Section \ref{section22}.  Suppose that  $J_{X_{\pm}} $ is generic, and there exists $R_0 \ge0$ such that  $J_R$ is generic  for any $R \ge R_0$.  Then \label{C20}
\begin{equation*}
\begin{split}
&\sum_{\Gamma_{X_{R_0}} \vert_{X_{\pm} }= \Gamma_{X_{\pm}}, \Gamma_{X_{R_0}} \in H_2(X_{R_0}, \partial X_{R_0}, \mathbb{Z}) }HP(X_{R_0},  \Omega_{X_{R_0}}, \Gamma_{X_{R_0}},  J_{R_0}, \Lambda_{X_{R_0}})\\
&= HP(X_-,\Omega_{X_-}, \Gamma_{X_-},  J_{X_-},\Lambda_{X_-}) \circ HP(X_+,\Omega_{X_+},   \Gamma_{X_+}, J_{X_+}, \Lambda_{X_+}),
\end{split}
\end{equation*}
where $\Lambda_{X_R}=\Lambda_{X+} \circ \Lambda_{X-}$.
\end{lemma}
\begin{proof}  To simplify the notation, we assume that $R_0=0$.  Let $\alpha_-$ and $\alpha_+$ be admissible orbit  sets. Fix a relative homology class $Z \in H_2(X_0, \alpha_+, \alpha_-)$ with $[Z] = \Gamma_{X_0}$. %Define $\mathcal{M}= \bigsqcup_{R \ge 0} \bigsqcup_{ A_{\vert_{X_{\pm}}}=A_{\pm}} \mathcal{M}^{J_R}_{X_{R}, I=0}(\alpha, A) $.
 Let  $\{R_n\}_{n=1}^{\infty}$ be an increasing unbounded sequence and  $u_n \in \mathcal{M}^{J_{R_n}}_{X_{R_n},I=0} (\alpha_+ ,  \alpha_-, Z)$. Arguing as Lemma \ref{C4},  $\{u_n\}_{n=1}^{\infty}$ converges to a broken holomorphic curve $u_{\infty}= \{v^-_m, \dots u^- , \tau_1 \dots \tau_k, u^+, \dots v^+_l\}$ in the sense of SFT \cite{FYHKE} with $I(u_{\infty})=0$,   where
%$u^+ \in \mathcal{M}^{J_{X_+}}_{X_+} (\beta_+, \alpha_0, Z^+)$, $u^- \in \mathcal{M}^{J_{X_-}}_{X_- } (\beta_0, \alpha_-, Z^-)$,
$u^+$, $u^-$, $\tau_i$, $v_i^+$ and  $v_j^-$ are respectively holomorphic curves in $\overline{X_+}$,  $\overline{X_-}$, $\mathbb{R} \times Y_0$, $\mathbb{R} \times Y_+$ and $\mathbb{R} \times Y_-$.% $Z=Z^++Z^-$ and $[Z^{\pm}]=\Gamma_{X_{\pm}}$.
%$\{\tau_i\} \in \mathcal{M}^{J_0}_{Y_0}(\alpha_0, \beta_0) $,  $\{v^+_i\} \in \mathcal{M}^{J_+}_{Y_+}(\alpha_+, \beta_+) $,  $\{v^+_i\} \in \mathcal{M}^{J_-}_{Y_-}(\beta_-, \alpha_-) $,

Corollary \ref{C3} forces $I(u^+)=I(u^-)=I(\tau_i)=I(v^+_j)=I(v^-_j)=0$.  Furthermore, by the argument similar to Lemmas \ref{C4} and \ref{C5},  $\{\tau_i\}_j$, $\{v^+_j\}_j$ and $\{v^-_j\}_j$ are connectors.  Since $u_{\infty}$ is admissible, $\{v^+_j\}_j$  and  $\{v^-_j\}_j$ are trivial connectors and they  are ruled out by the holomorphic building. Thus negative ends of $u^-$ are asymptotic to $\alpha_-$ and   positive  ends of $u^+$ are asymptotic to $\alpha_+$.   By Lemma \ref{C2},   both of $u^{+}$ and $u^{-}$ are embedded.

 %Keep track the proof of Lemma \ref{C2}, one can find that the admissible conditions on $\alpha_{\pm}$ is not necessarily in  cases ($\spadesuit$\Romannum{1}) and  ($\spadesuit$\Romannum{2}).  %In case ($\spadesuit$ \Romannum{3}), if we only assume that one of $\alpha_+$ and $\alpha_-$ is admissible, the conclusion of  Lemma \ref{C2} still hold.

% Note that here  $\deg(\alpha)= \deg(\beta)$, in particular, $\beta$ is nondegenerate.

Let $\mathcal{M}= \bigsqcup_{R\ge0} \mathcal{M}^{J_R}_{X_R,I=0} (\alpha_+, \alpha_-, Z) $, then $\mathcal{M}$ is a $1$-dimensional manifold, because, $J_R$ is generic for $R \ge 0$.  The  boundary of $\mathcal{M}$ has two components $\partial_1 \mathcal{M}$ and $\partial_2 \mathcal{M}$, where  $\partial_1 \mathcal{M}= \mathcal{M}^{J_X}_{X, I=0}(\alpha_+, \alpha_- Z)$, and $\partial_2 \mathcal{M}$  is contributed by the broken holomorphic curves described in the previous paragraph.
%According the discussion in previous paragraph, $\partial_2 \mathcal{M}$ is
%$$ \bigsqcup_{\beta} \bigsqcup_{ Z=Z^++Z^-}\mathcal{M}^{J_{X^+}}_{X^+, I=0} (\alpha_+, \beta, Z^+) \times \{\mbox{connectors}\} \times \mathcal{M}^{J_{X^-}}_{X^-, I=0} (\beta, Z^-).$$
By the gluing arguments in \cite{HT1} and \cite{HT2}, we have
\begin{equation*}
\begin{split}
& \ \ \ \ \ \ \#_2 \mathcal{M}^{J_X}_{X_0, I=0}(\alpha_+, \alpha_-, Z) \\
&= \sum\limits_{\alpha_0 \in \mathcal{P}(Y_0, \omega_0, \Gamma_0)} \sum\limits_{ Z=Z^++Z^-} \#_2 \mathcal{M}^{J_{X_+}}_{X_+, I=0} (\alpha_+, \alpha_0, Z^+) \#_2 \mathcal{M}^{J_{X_-}}_{X_-, I=0} (\alpha_0, \alpha_-,  Z^-).
\end{split}
\end{equation*}
This formula implies the conclusion of the lemma.
\end{proof}

\begin{comment}
\begin{remark}
Given symplectic fibred cobordisms $(X_{\pm }, \pi_{X_{\pm}}, \omega_{X_{\pm}})$  and almost complex structures $J_{X_{\pm}}$ and $J_X$ as above. Assume that $(X_+, \pi_{X_+}, \omega_{X_{+}})$ and $(X_-, \pi_{X_-}, \omega_{X_-})$ are monotone and  assume that both of
$J_{X_+} \in \mathcal{V}_{comp}(X_+, \pi_{X_+}, \omega_{X_+}, J_+, J_0)$, $J_{X_-} \in \mathcal{V}_{comp}(X_-, \pi_{X_-}, \omega_{X_-}, J_0, J_-)$ and $J_{X} \in \mathcal{V}_{comp}(X, \pi_{X}, \omega_{X}, J_{\pm})$ are generic. Then we obtain the same  conclusions in Lemma \ref{C20} by using the same argument.
\end{remark}
\end{comment}

\subsubsection{Blowup formula}
In this subsection, we prove a simple version of blowup formula.  The proof is  the same as Theorem 4.2 in \cite{LL}.

 Take a generic almost complex structure $J \in \mathcal{J}_{tame}(X, \pi_X, \omega_X)$  which is integral in a small neighborhood $U$ of $x \in X$, where $x$ lies in a regular fiber.     Then we perform a complex blowing up at $x$, the result is denoted by $(\overline{X'}, J')$, also, $J'=J$ outside $U$.
Then $\pi_{X'} =\pi_X \circ pr: \overline{X'} \to  \overline{B} $  still is a Lefschetz fibration and $\pi_{X'}$ is complex linear with respect to $J'$ and $j_{B}$, where $pr: \overline{X'} \to \overline{X}$ is the blowing down map.
After blowing up, there is one additional critical point and its critical value is $\pi_X(x)$. One of the irreducible components of $\pi_{X'}^{-1}(\pi_X(x))$ is an exceptional sphere with self interesting $-1$. Let $E $ denote  its homology class. There is a standard way to find a symplectic form $\Omega_{X'}$ on $X'$ such that $\Omega_{X'}$ tames $J'$ and $\Omega_{X'}= pr^* \Omega_X$ outside a neighborhood  of $pr^{-1}(x)$.  (Cf. \cite{MS} ) Then we can define an admissible $2$-form by $\omega_{X'}=\Omega_{X'}-\omega_B$ and clearly we have $J' \in \mathcal{J}_{tame}(X', \pi_{X'}, \omega_{X'})$. According to Lemma \ref{C52},  $HP(X', \Omega_{X'}, J'', \Gamma_X, \Lambda_{X'})$ is well defined for any  generic $J''$, provide that $\Gamma_X \cdot E \le 0$.

%the exceptional class arising from blow up.

\begin{lemma} \label{C38}
Suppose that $(\Gamma_X, \omega_X)$ is monotone.   Introduce $(\overline{X'}, \omega_{X'},  J')$ and $E$ as above. Regard  $\Gamma_X$ as an element in $ H_2(X',  \partial X', \mathbb{Z})$ and assume that  $\Gamma_X \cdot E=0$, then $HP(X', \Omega_{X'}, J', \Gamma_X, \Lambda_{X'})$ is well defined, where $\Lambda_{X'}$ is a $X'$-morphism such that $\Lambda_X(Z)=\Lambda_{X'}(Z)$ for any $Z$ with relative class $\Gamma_X$.        Furthermore,  $$HP(X, \Omega_X, J, \Gamma_X, \Lambda_X)=HP(X', \Omega_{X'}, J', \Gamma_X, \Lambda_{X'}). $$

%then $\#\mathcal{M}_{X, I=0}^J(\alpha_+, \alpha_-, Z) = \#\mathcal{M}_{X', I=0}^{J'}(\alpha_+, \alpha_-, Z)   $.
\end{lemma}
\begin{proof}
By  the  monotonicity assumption, the energy of curves in $ \mathcal{M}_{X, I=0}^{J}(\alpha_+, \alpha_-, Z)$ is fixed, denoted by $L$,  where $Z$ is a  relative homology class with $[Z]=\Gamma_X$.  By Corollary \ref{C22}, $\mathcal{M}_{X, I=0}^{J, L}(\alpha_+, \alpha_-)$ is a finite set. So we can  take a sufficiently small neighborhood $U$ of $x$ such that any  holomorphic curve in $\mathcal{M}_{X, I=0}^{J, L}(\alpha_+, \alpha_-)$  does not interest $U$.

Let $e=pr^{-1}(x)$ be the exceptional sphere representing  the class $E$. $e$ is a $J'$ holomorphic curve in $\overline{X'}$. For every $u \in \mathcal{M}_{X', I=0}^{J'}(\alpha_+, \alpha_-, Z)$ and $[Z]=\Gamma_X$, $u$ does not intersect $e$ due to the assumption $\Gamma_X\cdot E=0$. As a result, $I(pr \circ u)=0$.  Since $pr$ is biholomorphic outside $pr^{-1}(x)$, $pr \circ u$ is a holomorphic curve in $\overline{X} $.  That is $pr \circ u \in \mathcal{M}_{X, I=0}^J(\alpha_+, \alpha_-, Z)$.  Moreover,  $u$ is  Fredholm regular.   As a consequence, there is a bijection between $\mathcal{M}_{X', I=0}^{J'}(\alpha_+, \alpha_-, Z)$ and $\mathcal{M}_{X, I=0}^{J}(\alpha_+, \alpha_-, Z)$ and this lead to the conclusion.
\end{proof}

\subsubsection{Commute with $U$-map} \label{section5}
The argument here is similar to Section 2.5 of  \cite{HT3}.  We sketch the main points as follows.

Fix a point $x \in X$, we   define the  cobordism  maps with mark point $x$ as follows: Let $\mathcal{M}^J_{X, I=2}(\alpha_+, \alpha_-, x, Z)$ be the moduli space of holomorphic curves whose image intersect $x$.  Given $ u\in \mathcal{M}^J_{X, I=2}(\alpha_+, \alpha_-, x, Z)$ and let $\mathcal{C}$ be  the associated holomorphic current. Note that the components of   $\mathcal{C}$ which interest $x$   have ${\rm ind}\ge 2$.

%Let $\mathcal{M}^{J, simple}_{X, 1}(\alpha_+, \alpha_-, Z)$  be the moduli space of simple holomorphic curves with  one mark point.  The dimension near $u \in \mathcal{M}^{J, simple}_{X, 1}(\alpha_+, \alpha_-, Z)$ is ${\rm ind}u+2$. Define evaluate map $ev:  \mathcal{M}^{J, simple}_{X, 1}(\alpha_+, \alpha_-, Z) \to \overline{X}$, then the dimension of $ev^{-1}(x)$ is ${\rm ind}-2$. So  if $ev^{-1}(x)$ is nonempty, we must have ${\rm ind}\ge 2$.

Based on the above understanding,  we have $I(u) \ge 2$ for $u \in \mathcal{M}^J_{X}(\alpha_+, \alpha_-, x, Z)$. Repeat the argument in Lemma \ref{C4} to show that   $\mathcal{M}^J_{X, I=2}(\alpha_+, \alpha_-, x, Z)$  is a compact zero  dimensional manifold. We  define the cobordism  maps with one mark point by
\begin{equation*}
CP(X, \Omega_X, J, \Lambda_X)_x= \sum_{\alpha_{\pm}} \sum_{Z\in H_2(X, \alpha_+, \alpha_-)} \#_2\mathcal{M}^J_{X, I=0}(\alpha_+, \alpha_-, x, Z)\Lambda_X(Z).
\end{equation*}
It is a chain map due to the same argument in Lemma \ref{C14}. So $HP(X, \omega_X, J,   \Lambda_X)_x$  is well defined.

Let $y_{\pm} \in Y_{\pm}=\partial_{\pm} X$, and $\gamma_{\pm}$ be a path in $X$ such that $\gamma_{\pm}(0)=x$ and $\gamma_{\pm}(1)= y_{\pm}$,  and $\gamma_{\pm}(t)$ lies in the interior of $X$ for $0\le t<1$.
Let  $\mathcal{M}^J_{X, I=i}(\alpha_+, \alpha_-, \gamma_+, Z)$ be  the moduli space holomorphic curves whose image intersect $\gamma_{\pm}$. When $i=1$, using the same argument as in Lemma \ref{C4} to show that it is a $0$ dimensional compact manifold. As a result,
\begin{equation*}
K_{+}=\sum_{\alpha_{\pm} \in \mathcal{P}(Y_{\pm}, \omega_{\pm}, \Gamma_{\pm})}\sum_{Z \in H_{2}(X, \alpha_+, \alpha_-)}  \#\mathcal{M}^J_{X, I=1}(\alpha_+, \alpha_-, \gamma_+, Z)\Lambda_Z
\end{equation*}
is well defined.

When $i=2$, $ \mathcal{M}^J_{X, I=2}(\alpha_+, \alpha_-, \gamma_{\pm}, Z) $ is a manifold of dimension $1$. The broken holomorphic curve  $u_{\infty}=\{u^{-N_-}, \dots, u^0, \dots, u^{N_+} \}$  which comes from the limits of curves in  $ \mathcal{M}^J_{X, I=2}(\alpha_+, \alpha_-, \gamma_{\pm}, Z) $ consists of  the following possibilities:
\begin{enumerate}
\item
$u^{N_+}$ is a holomorphic curve with $I=1$ and $u^0 \in  \mathcal{M}^J_{X, I=1}(\beta, \alpha_-, \gamma_+, Z')$, and the other levels are connectors. Moreover, there is no negative level.
\item
 $u^{N_+}$ is a holomorphic curve with $I=2$ and it interest $y_+$.   $u^0 \in  \mathcal{M}^J_{X, I=0}(\beta, \alpha_-,  Z')$, and the other levels are connectors. Moreover, there is no negative level.
\item
$u^{-N_-}$ is a holomorphic curve with $I=1$ and $u^0 \in  \mathcal{M}^J_{X, I=1}(\beta, \alpha_-, \gamma_+, Z')$, and the other levels are connectors. Moreover, there is no  positive level.

\item
$u_{\infty}=u^0 \in  \mathcal{M}^J_{X, I=2}(\alpha_+, \alpha_-, x,  Z)$.
\end{enumerate}
The gluing analysis in \cite{HT1} and \cite{HT2}  shows  that
\begin{equation*}
CP(X, \Omega_X, J, \Lambda_X)_x- U_{+,y} \circ CP(X, \Omega_X, J, \Lambda_X)=\partial_+\circ K_+ - K_+ \circ \partial_-.
\end{equation*}
Therefore, $HP(X, \Omega_X, J,  \Lambda_X)_x=U_+\circ HP(X, \Omega_X, J,  \Lambda_X)$. Similarly, we show that  $HP(X, \Omega_X, J,   \Lambda_X)_x= HP(X, \Omega_X, J,  \Lambda_X) \circ U_-$. As a consequence, $$ U_+\circ HP(X, \Omega_X, J,  \Lambda_X)=HP(X, \Omega_X, J, \Lambda_X) \circ U_-.$$

%%%%%%%%%%%%%%%%%%%%%%%%%%%%%%%%%%%%%%%%%%%%%%%%%%%%%%%%%%%%%%%%%%%%%%%%%%%%%%%%%%%%%%%%%%%%%%%%%%%%%%%
\subsection{$\mathbb{Z}$-coefficient} \label{section19}
So far, the cobordism maps  $HP(X, \Omega_X, J, \Lambda_X)$  are defined over the local coefficient of $\mathbb{Z}_2$-module.   In fact, they also can be defined over the local coefficient of $\mathbb{Z}$-module  by endowing a   coherent orientation on the moduli space $ \mathcal{M}_{X, I=0}^J(\alpha_+, \alpha_-, Z)$.

Given an admissible orbit set $\alpha_{\pm} \in \mathcal{P}(Y_{\pm}, \pi_{\pm}, \Gamma_{\pm} )$,  we can define the determinant line bundle $\det D\to \mathcal{M}^J(\alpha_+, \alpha_-) $ with fiber $\det D_u= \Lambda^{max} \ker D_u \otimes \Lambda^{max}  {\rm coker} D_u$, where $D_u$ is the linearization  of $\bar{\partial}_J u$.
 %Note that for each $u \in \mathcal{M}_{X, I=0}^J(\alpha_+, \alpha_-, Z) $ consisting of pairwise disjoint embedded holomorphic curves $\{C_a\}$ and $\mathcal{D}_{C_a}$ has trivial cokernel, thus $\det \mathcal{D}_u= \oplus_a \ker \mathcal{D}_{C_a} $.
%Now let $\mathcal{M}_*^J(\alpha, \beta)$ to be $\mathcal{M}_{Y, I=1}^J(\alpha, \beta)$ or $\mathcal{M}_{X, I=0}^J(\alpha, \beta)$,  then one can assign a $2$-element set $\Lambda(\alpha, \beta)$ to $\mathcal{M}_*^J(\alpha, \beta)$ to be the orientation of  $\det \mathcal{D} $.
Introduce a $2$-elements set $\Lambda(\alpha_+, \alpha_-)$ to be the  orientation sheaf of $\det D$. The gluing principle in Section 3.b of \cite{Te3} tells us that the collection of $\{\Lambda(\alpha_+, \alpha_-)\}_{\alpha_{\pm}}$ is a  coherent system of orientations in the following sense: Given orbit sets  $\alpha_{+}$ and $\alpha_-$, there is  a canonically associated $\mathbb{Z}_2$ module, which is denoted by $\Lambda(\alpha_{\pm})$. Moreover, there is a canonical isomorphism from $\Lambda(\alpha_+, \alpha_-)$ to $\Lambda(\alpha_+)\Lambda(\alpha_-)$.

With the above understanding, a set of  orientations $\{\mathfrak{o}_{PFH}(\alpha_+ , \beta_+) \in \Lambda(\alpha_+, \alpha_-) \}$ is called coherent if there exists $\{\mathfrak{o}_{PFH}(\alpha_{\pm}) \in \Lambda(\alpha_{\pm}) \}_{\alpha_{\pm} }$ such that $\mathfrak{o}_{PFH}(\alpha_+ , \alpha_-)=\mathfrak{o}_{PFH}(\alpha_+)\mathfrak{o}_{PFH}(\alpha_-)$.

Fix a choice of coherent orientation, the moduli space $\mathcal{M}_{X, I=0}^J(\alpha_+, \alpha_-, Z)$  is a finite set with signs which are determined by the orientations.   Granted this  understood, we can define  the cobordism maps $HP(X, \Omega_X, J, \Lambda_X)$ over the local coefficient of   $\mathbb{Z}$-module  through replacing the mod two count $"\#_2"$ by  integer count $"\#"$.

\subsection{Proof of  Theorem \ref{Thm1}}
\begin{proof} %[Proof of  Theorem \ref{Thm1}]
Suppose that $(X, \pi_X, \omega_X)$ is a fiberwise symplectic  cobordism  from $(Y_+, \pi_+, \omega_+) $ to $(Y_-, \pi_-, \omega_-) $ satisfying assumptions $(\spadesuit)$.  Also, assume that  $ \Gamma_{\pm} \cdot [\Sigma]> g(\Sigma)$. Up to now, we have shown that the cobordism maps $HP(X, \Omega_X, J, \Lambda_X)$  are  well defined    for generic $J \in \mathcal{J}_{tame}(X, \pi_X, \omega_X)$  whenever $(X, \pi_X)$   contains  no separating  singular fiber,  and it is well defined for generic  $J \in \mathcal{V}_{comp}(X, \pi_X, \omega_X)$  whenever $(X, \omega_X)$ is monotone.  The  promised properties of $HP(X, \Omega_X, J, \Lambda_X)$ follow  from Lemmas \ref{C20}, \ref{C38} and the discussion in Section \ref{section5}.

Up to now, we have completed the proof of  Theorem \ref{Thm1}.
\end{proof}

%Now we consider the case that $d=\Gamma_+ \cdot [\Sigma]=\Gamma_- \cdot [\Sigma]=0$.
%It is natural to define $$CP(X ,\omega_X, J, \Lambda_X) \Lambda_{\emptyset} = \sum\limits_{A \in H_2(X, \mathbb{Z})} \#_2\mathcal{M}_{X, I=0}^{J}(\emptyset, \emptyset, A) \Lambda_X(A) \circ \Lambda_{\emptyset}.$$
%
%When $Y_+ \ne \emptyset$ or $Y_- \ne \emptyset$, we assume  that $X$ is relative minimal.  For $J \in \mathcal{J}_{tame}(X, \pi_X, \omega_X)$ (not necessary generic), the closed holomorphic curves in $\overline{X}$ have homology class $m[\Sigma_i]$ or $m [\Sigma]$ for some $m \ge 1$. In either case, $I<0$. The empty curve is the only one element of $\mathcal{M}_{X, I=0}^{J}(\emptyset, \emptyset)$.   Therefore, $HP(X, \omega_X, J, \Lambda_X)$ is well defined and  $HP(X, \omega_X, J, \Lambda_X)$ is a canonical isomorphism.
%
%
%When $Y_+=Y_-=\emptyset$, $\#\mathcal{M}_{X, I=0}^{J}(\emptyset, \emptyset, A) $ is well-defined and independent on $J$. (Cf. \cite{T2})  Assume that  $b_2^+(X)>1$, by  the simple type property of closed symplectic 4-manifold, we have
%\begin{equation*}
%\begin{split}
%HP(X, \omega_X, \Lambda_X)\Lambda_{\emptyset}&=\sum\limits_{A \in H_2(X, \mathbb{Z})} \#\mathcal{M}_{X, I=0}^{J}(\emptyset, \emptyset, A) \Lambda_X(A) \circ \Lambda_{\emptyset}\\
%&=\sum\limits_{A \in H_2(X, \mathbb{Z})}Gr(X, \Omega, A) \Lambda_X(A) \circ \Lambda_{\emptyset}.
%\end{split}
%\end{equation*}

\section{Embedded holomorphic curves and Seiberg-Witten equations}
In this section, we  aim to prove the following theorem.
\begin{theorem}
 Suppose that there exists $R_0 \ge 0$  such that  $\wp_4$ agrees with $\wp_{3 \pm}$ when $|s| \ge R_0$ and $J \in \mathcal{V}_{comp}(X, \pi_X, \omega_X)$ is generic. Also, assume that  $(\omega_X, J) \vert_{\mathbb{R}_{\pm} \times Y_{\pm}}$ is $Q$-$\delta$ flat. Given the same assumptions in Theorem \ref{Thm2} and any $L$, there exists $r_L>0$ such that for any $r \ge r_L$, there is a bijective map
\begin{eqnarray*}  \label{C31}
\Psi^r : \mathcal{M}_{X, I=0}^{J,L}(\alpha_+, \alpha_-) \to \mathfrak{M}_{X, {\rm ind}=0}^{J,L}(T^+_r (\alpha_+), T_r^-(\alpha_-)),
\end{eqnarray*}
where $ \mathfrak{M}_{X, {\rm ind}=0}^{J,L}(T^+_r (\alpha_+), T_r^-(\alpha_-))$ is the  moduli space of index zero solutions to (\ref{e4}) with  $\frac{i}{2\pi}\int_{\overline{X}} F_A \wedge \omega_X  \le L$. Moreover, the image of $\Psi^r $ is nondegenerate.  \label{A10}
\end{theorem}

The proof of Theorem \ref{A10} relies heavily on \cite{Te2}, \cite{Te3}, \cite{Te4}  and   \cite{LT}  with only minor changes. In these papers,  they prove   a similar result that there   is a $1$-$1$ correspondence between instantons and  $I=1$ holomorphic curves  when  $\overline{X}$ is  symplectization of a contact three manifold or a mapping torus. Note that Theorem \ref{A10} differs from the analogs in    \cite{Te2}, \cite{Te3}, \cite{Te4}  and   \cite{LT} in the following aspects: The manifold $\mathbb{R} \times Y$,   symplectization admissible almost complex structure and symplectic form are respectively replaced by $\overline{X}$, cobordism admissible almost complex structure and $\Omega_X$ here.  Besides,
the  $I=1$ holomorphic curves are replaced by $I=0$ holomorphic curves in our case, because, $\overline{X}$ does not admit global $\mathbb{R}$ action.  Note that our holomorphic curves  are embedded with $I=0$,  they are    counterparts  of the nontrivial  embedded $I=1$ holomorphic curves  in   symplectization case.
%Compare to symplectization cases in  \cite{Te2}, \cite{Te3}, \cite{Te4} or \cite{LT}, our symplectic form and almost complex structure agree with  symplectization case on the ends of $\overline{X}$, i.e. $\Omega_X=\omega_{\pm} + ds \wedge \pi_{\pm}^*dt$  and $J \in \mathcal{J}_{comp}(Y_{\pm,}, \pi_{\pm}, \omega_{\pm})$.

Except for  the convergence results in Sections 4 and 5 of \cite{Te4},  the analysis in   \cite{Te2}, \cite{Te3}, \cite{Te4} or \cite{LT} mainly takes place either near a  holomorphic curve  $\mathcal{C}$ or its ends.  Since our symplectic form  and  cobordism admissible almost complex structure agree with  symplectization case on the ends of $\overline{X}$, i.e., $\Omega_X \vert_{\mathbb{R}_{\pm} \times Y_{\pm}}=\omega_{\pm} + ds \wedge \pi_{\pm}^*dt$  and $J \vert_{\mathbb{R}_{\pm} \times Y_{\pm}} \in \mathcal{J}_{comp}(Y_{\pm,}, \pi_{\pm}, \omega_{\pm})$, the relevant analysis on the ends of $\mathcal{C}$ can be applied to our situation without modification. For the   relevant analysis  on the tubular  neighborhood  of $\mathcal{C}$, the form of  $\Omega_X$ plays no role there,   so their arguments also apply equally well in the current situation.

Let us briefly outline the proof, it can be divided into the following three parts:
\begin{enumerate}
\item
Firstly, we construct the map $\Psi^r$.  The construction  is sketched as follows, it  is corresponding to Sections 4-7 of \cite{Te2}. Given an embedded holomorphic current $\mathcal{C}$,  we  star by following Section 5 of \cite{Te2} to construct a line bundle $E \to \overline{X}$ and a pair $(A^*, \psi^*)$ which  is close to solving   (\ref{e4}). Then using the perturbation analysis to deform $(A^*, \psi^*)$  to be a true solution $(A_{\mathcal{C}}, \psi_{\mathcal{C}} )$ and define $\Psi^r(\mathcal{C})=[(A_{\mathcal{C}}, \psi_{\mathcal{C}} )]$.  We will review the ideal of the construction later.    We point out that the analysis of this part is local, and whether $\Omega_X$  is exact or not plays no role to the analysis,  the arguments in \cite{Te2} also can be applied to our setting without essential change. When $\mathcal{C} =\emptyset$, the manifold $\mathbb{R} \times Y$ has natural $\mathbb{R}$ action, so \cite{Te2}  can take $ (A_{\emptyset}, \psi_{\emptyset} )$ simply by the $\mathbb{R}$ invariant solution corresponding to empty orbits.  Since our $(X, \pi_X)$ is not  $\mathbb{R}$ invariant, we will    construct  $ (A_{\emptyset}, \psi_{\emptyset} )$  in Section \ref{section6}  in detail.

\item
Secondly, to verify that  the map $\Psi^r$ is well defined, we need to check that  $ (A_{\mathcal{C}}, \psi_{\mathcal{C}} )$ is  regular with zero Fredholm index.  This corresponding to Section 3.a of \cite{Te3}. The main points of this part have been summarized in Section 5.1 of \cite{21}.   Thus we skip this part here.

\item
From the construction, it is not difficult to check that $\Psi^r$ is injective. To show that $\Psi^r$ is  surjective,  the proof also  is  divided into three parts.  \cite{LT} calls them  \textbf{estimation}, \textbf{convergence}, and  \textbf{perturbation} respectively. They are corresponding to Section 3, Sections 4-5 and Sections 6-7  of \cite{Te4} respectively.   We will clarify their meaning latter.  We point out that there is no essential difference in  \textbf{estimation}  and  \textbf{perturbation} parts.  For the  \textbf{convergence}  part,   we already established it in Section \ref{section7}.
\end{enumerate}
For simplicity, we assume that $Y_- =\emptyset$  in this section.

\subsection{Construct a trivial Seiberg Witten solution.} \label{section6}
We begin to construct  $\Psi^r$ in Theorem \ref{A10} by considering a special case that $\mathcal{C} =\emptyset$.  We  follow the argument in  \cite{Te2} to construct a solution $(A_{\emptyset}, \psi_{\emptyset})$ to Seiberg Witten equations (\ref{e4}) which  corresponds to   the emptyset of $\overline{X}$.
Also, in Section \ref{section8}, we show that this solution is small.

\textbf{Deformation operator.} Given a configuration $(A, \psi)$,
%connection $1-$form $A$ and a section $\psi$ of $S_+$,
the deformation operator $\mathfrak{D}: \Gamma(iT^*\overline{X} \oplus S_+) \to \Gamma( i(\Omega^{2+}(\overline{X}) \oplus i\mathbb{R})\oplus S_-)  $ at $(A, \psi)$ is  defined by
\begin{equation} \label{e75}
  \begin{cases}
d^+b- 2^{-\frac{3}{2}} r^{\frac{1}{2}}(q(\psi, \eta) + q(\eta, \psi)) \\
D_A \eta + (2r)^{\frac{1}{2}} \mathfrak{cl}(b) \psi \\
*d*b - 2^{-\frac{1}{2}} r^{\frac{1}{2}} (\eta^{\dagger} \psi - \psi^{\dagger} \eta),
  \end{cases}
\end{equation}
where $(b, \eta ) \in  \Gamma(iT^*\overline{X} \oplus S_+)$. The first two lines come from the linearization of Seiberg Witten equations and the last line is gauge fixing condition.

Take $E$ to be the trivial line bundle, denoted by $I_{\mathbb{C}}$,  then $S_+=I_{\mathbb{C}} \oplus K_X^{-1}$ and $S_-=\wedge^{0,1}T^{*}\overline{X} $.
Let $(A_I, \psi_I)$ be the trivial approximation solution, i.e., $A_I$ is the trivial connection on $I_{\mathbb{C}}$ and $\psi_I=(1_{\mathbb{C}},0)$.  Using $\mathfrak{D}_I$ to denote the deformation operator  at $(A_I,\psi_I)$. Note that $\mathfrak{D}_I$  is an elliptic operator from $\Gamma (i\Omega^1 \oplus S_+) $ to $\Gamma ((i\Omega^0 \oplus \Omega^{2+} )\oplus S_-) $.

Let $\mathfrak{h} \in \Gamma (i\Omega^1 \oplus S_+) $, we define the $H$-norm as follows,
\begin{eqnarray*}
|\mathfrak{h}|^2_H = \int_{\overline{X}} |\nabla \mathfrak{h}|^2 + \frac{r}{4}|\mathfrak{h}|^2,
\end{eqnarray*}
where the covariant derivative acts on section of $iT^*\overline{X}$ as the Levi-Civita  covariant derivative and on section of $S^+$ as  covariant derivative that is defined by  Levi-Civita  covariant and $2A_I + A_{K^{-1}}$. One can define the $H$-norm for section of $\Gamma ((i\Omega^0 \oplus \Omega^{2+} )\oplus S_-) $ similarly. Keep in mind that  $|\mathfrak{h}|_{L^4} \le c_0|\mathfrak{h}|_H $ because of  the Sobolev inequality.

\textbf{ $Weizenb\ddot{o}ck$ formula:}
\begin{equation} \label{e20}
\begin{split}
&\mathfrak{D}_I \mathfrak{D}_I ^*= \nabla^* \nabla +2r +\mathcal{R}_0 + \sqrt{r}  \mathcal{R}_1 \\
&\mathfrak{D}_I ^* \mathfrak{D}_I = \nabla^* \nabla +2r +\tilde{\mathcal{R}_0} + \sqrt{r} \tilde{ \mathcal{R}_1}.
\end{split}
\end{equation}
Here $ \mathcal{R}_i$ and  $\tilde{ \mathcal{R}_i}$ $(i= 0, 1)$ are bounded endomorphism of $i\mathbb{R} \oplus \Omega^{2+} \oplus S_-$ and $iT\overline{X}^* \oplus S_+$ respectively.

\begin{lemma}
Let $\mathfrak{h} \in \Gamma (i\Omega^1 \oplus S_+) $ and $\mathfrak{f} \in \Gamma ((i\Omega^0 \oplus \Omega^{2+} )\oplus S_-) $ with compact supports, then there exists $c_0>0 $ such that for any $r \ge c_0$, $\mathfrak{D}_I $ and $\mathfrak{D}_I ^*$ satisfy the following estimates:
\begin{center}
$\frac{1}{c_0}|\mathfrak{h}|^2_H  \le |\mathfrak{D}_I \mathfrak{h}|^2_{L^2} \le c_0 |\mathfrak{h}|_H^2,$
\end{center}

\begin{center}
$\frac{1}{c_0}|\mathfrak{f}|^2_H  \le |\mathfrak{D}_I^*\mathfrak{f}|^2_{L^2} \le c_0 |\mathfrak{f}|_H^2,$
\end{center}   \label{A1}
\end{lemma}

\begin{proof} The lemma follows from  the $Weizenb\ddot{o}ck$ formula (\ref{e20}) and integration by part.
\end{proof}

Let $\mathbb{H}$  and $\mathbb{L}$  be the completion of smooth sections in  $\Gamma (i\Omega^1 \oplus S_+) $ and $\Gamma ((i\Omega^0 \oplus \Omega^{2+} )\oplus S_-) $ with compact supports   with respect to $H$ norm and $L^2$ norm respectively.

\begin{lemma}
There exists $\kappa>0 $  and $c_0>0$ such that for  $r \ge c_0$, $\mathfrak{D}_I : \mathbb{H} \to \mathbb{L}$ is invertiable. Moreover, $|\mathfrak{D}_I ^{-1}| \le \kappa$. \label{A2}
\end{lemma}
\begin{proof} By Lemma \ref{A1}, $\mathfrak{D}_I$ is injective.

To show that $\mathfrak{D}_I$ is surjective, given  $g \in  \mathbb{L}$,  we define a functional $\mathcal{F}_g: \mathbb{L} \cap L_1^2((i\Omega^0 \oplus \Omega^{2+} )\oplus S_-) \to \mathbb{R} $ by
\begin{equation*}
 \mathcal{F}_g(\mathfrak{f}) = \int_{\overline{X}} |\mathfrak{D}_I ^*\mathfrak{f}|^2-2<g,\mathfrak{f}> ,
\end{equation*}
where $ L_1^2((i\Omega^0 \oplus \Omega^{2+} )\oplus S_-)$ is the  completion of smooth sections with compact support with respect to $L_1^2$-norm.
By $H\ddot{o}lder$ inequality and Lemma \ref{A1}, we have
\begin{eqnarray*}
\mathcal{F}_g(\mathfrak{f})=\int_{\overline{X}} |\mathfrak{D}_I ^*\mathfrak{f}|^2-2<g,\mathfrak{f}>
\ge \frac{1}{c_0} |\mathfrak{f}|_H^2 - \frac{c_0}{r} |g|^2_{L^2}
\ge - \frac{c_0}{r} |g|^2_{L^2} > -\infty.
\end{eqnarray*}
Consequently, $m=\inf\limits_{\mathfrak{f} \in \mathbb{L} \cap L_1^2} \mathcal{F}_g$ is well-defined. We claim that there exists a minimizer of $\mathcal{F}_g$. The reasons  are given as follows:

Firstly, there exists a sequence $ \{\mathfrak{f}_i\} \subset \mathbb{L} \cap L_1^2 ((i\Omega^0 \oplus \Omega^{2+} )\oplus S_-)$ such that
%\begin{eqnarray*}
$\lim\limits_{i \to \infty}\mathcal{F}_g(\mathfrak{f}_i)=m.$
%\end{eqnarray*}
The  limit and  Lemma \ref{A1} imply that $\{\mathfrak{f}_i\}$ have a uniform  bound on $\mathbb{H}$  norm. After passing a subsequence, $\{\mathfrak{f}_i\}$ converges weakly to $\mathfrak{f}$ in the sense of $\mathbb{H}$ norm. By Lemma \ref{A1}, $\mathfrak{D}_I^*$ is a bounded linear operator, thus  $\mathfrak{D}_I^* \mathfrak{f}_i \to \mathfrak{D}_I^*\mathfrak{f}$   converges weakly in the sense of $L^2$ norm.  Therefore, $\varliminf \limits_{i \to \infty} |\mathfrak{D}_I^*\mathfrak{f}_i|_{L^2} \ge |\mathfrak{D}_I^*\mathfrak{f}|_{L^2} $. As a result,
\begin{equation*}
m=\varliminf \limits_{i \to \infty} \mathcal{F}_g(\mathfrak{f}_i)
\ge \varliminf \limits_{i \to \infty} |\mathfrak{D}_I^*\mathfrak{f}_i|_{L^2} -2\int_{\overline{X}} <\mathfrak{f},g>
\ge\mathcal{F}_g(\mathfrak{f}).
\end{equation*}
Hence,  $\mathfrak{f}$ is a minimizer of $\mathcal{F}_g$.
Furthermore, it is easy to check that $\mathcal{F}_g$ is convex. Therefore, the  minimum of $\mathcal{F}_g$ is unique.

Let $\mathfrak{f} \in \mathbb{L} \cap L_1^2 ((i\Omega^0 \oplus \Omega^{2+} )\oplus S_-)$  be  the minimizer above. The standard variational argument  show that $\mathfrak{D}_I ^*\mathfrak{f}$ is  a weak solution  to  $\mathfrak{D}_I \mathfrak{D}_I ^*\mathfrak{f}=g$. Elliptic regularity (Cf. \cite{GT}) implies that $\mathfrak{h}=\mathfrak{D}_I^* \mathfrak{f}$ is $L_{1, loc}^2$.   Then $\mathfrak{h}=\mathfrak{D}_I^* \mathfrak{f} \in \mathbb{H}$ follows from Lemma \ref{A1}. Hence, $\mathfrak{D}_I $ is surjective.

Finally,  the   bound  on $|\mathfrak{D}_I^{-1}|$ follows from Lemma \ref{A1}.
\end{proof}

\subsubsection{Contraction mapping argument} \label{section16}
In this section,   let $\kappa$ denote the bound of $|\mathfrak{D}_I ^{-1}|$ in Lemma \ref{A2}. Also, we use $\mu$ to denote $-\frac{1}{2}\wp_4^+ + \frac{i}{2} F^+_{A_{K^{-1}}}$.

Let $\mathfrak{b}_0=(b_0, \eta_0)$ denote the small solution to (3-34) in \cite{Te2} such that $(A_{\emptyset +}, \psi_{\emptyset +}) =(A_I + (2r)^{\frac{1}{2}}b_0, \psi_I+ \eta_0)$ satisfies equations (\ref{e1}) with an additional gauge fixed condition. We regard $\mathfrak{b}_0$ as   an  $\mathbb{R}$ invariant section on $\mathbb{R} \times Y_+$. By Lemma 3.10 in \cite{Te2}, $|\mathfrak{b}_0| \le c_0 r^{-1}$ and   $|\nabla \mathfrak{b}_0| \le c_0 r^{-1}$.

 Let $\chi$ be a cut-off function such that $\chi=1$ when $s \ge R_0$, $\chi=0 $ on $s \le R_0 -\kappa^{-4}$ and $|\nabla \chi| \le c_0 \kappa^4.$ Define $(A_{\emptyset}, \psi_{\emptyset}) =(A_I,\psi_I)+ ((2r)^{\frac{1}{2}}( b_h+ \chi b_0), (\chi \alpha_0+ \alpha_h, \chi \beta_0+\beta_h))$. We want to find $\mathfrak{h}_{\emptyset}=(b_h, \eta_h) \in \mathbb{H} $ such that $(A_{\emptyset}, \psi_{\emptyset}) $ satisfies the Seiberg Witten equations (\ref{e4}) together with the gauge fixing condition in the last line of (\ref{e75}).  It is equivalent to solve the following equation:
\begin{eqnarray} \label{e59}
\mathfrak{D}_I(\chi \mathfrak{b}_0+ \mathfrak{h}_{\emptyset}) +  r^{\frac{1}{2}} (\chi \mathfrak{b}_0+\mathfrak{h}_{\emptyset})*(\chi \mathfrak{b}_0+\mathfrak{h}_{\emptyset})= \mathfrak{v},  \label{E1}
\end{eqnarray}
where $\mathfrak{v}=(2r)^{-\frac{1}{2}} i\mu$. Equation (\ref{e59}) is equivalent to
 \begin{eqnarray*}
\mathfrak{v}  =\mathfrak{D}_I \mathfrak{h}_{\emptyset}+ r^{\frac{1}{2}}\mathfrak{h}_{\emptyset}*\mathfrak{h}_{\emptyset} +\chi(\mathfrak{D}_I\mathfrak{b}_0 + r^{\frac{1}{2}}\mathfrak{b}_0*\mathfrak{b}_0) + r^{\frac{1}{2}}\chi(\chi-1)\mathfrak{b}_0*\mathfrak{b}_0 + 2r^{\frac{1}{2}}\mathfrak{h}_{\emptyset} * \chi \mathfrak{b}_0 + \nabla{\chi} *\mathfrak{b}_0.
\end{eqnarray*}

%
%\begin{eqnarray*}
%\mathfrak{v}&=& \mathfrak{D}_I(\chi \mathfrak{b}_0+\mathfrak{h}) +  r^{\frac{1}{2}} (\chi \mathfrak{b}_0+\mathfrak{h})*(\chi \mathfrak{b}_0+\mathfrak{h})\\
%&=&\mathfrak{D}_I \mathfrak{h}+ r^{\frac{1}{2}}\mathfrak{h}*\mathfrak{h} +\chi(\mathfrak{D}_I\mathfrak{b}_0 + r^{\frac{1}{2}}\mathfrak{b}_0*\mathfrak{b}_0) + r^{\frac{1}{2}}\chi(\chi-1)\mathfrak{b}_0*\mathfrak{b}_0 + 2r^{\frac{1}{2}}\mathfrak{h} * \chi \mathfrak{b}_0 + \nabla{\chi} *\mathfrak{b}_0.
%\end{eqnarray*}

Note that $\mathfrak{D}_I\mathfrak{b}_0 + r^{\frac{1}{2}}\mathfrak{b}_0*\mathfrak{b}_0 = \mathfrak{v}_+$ and $\mathfrak{v}_+=\mathfrak{v}$ when $s \ge R_0$ by assumption. By Lemma \ref{A2}, we can rewrite   equation  (\ref{e59}) as follows:
\begin{eqnarray*}
\mathfrak{h}_{\emptyset}=\mathfrak{D}_I^{-1}((\mathfrak{v}-\chi\mathfrak{v}_+)-r^{\frac{1}{2}}\mathfrak{h}_{\emptyset}*\mathfrak{h}_{\emptyset} +  r^{\frac{1}{2}}\chi(1-\chi)\mathfrak{b}_0*\mathfrak{b}_0  -  2r^{\frac{1}{2}}\mathfrak{h}_{\emptyset} * \chi \mathfrak{b}_0 - \nabla{\chi} *\mathfrak{b}_0).
\end{eqnarray*}
Define a map $\mathcal{T} : \mathbb{H} \to \mathbb{H} $ by
\begin{eqnarray*}
\mathcal{T}(\mathfrak{h}_{\emptyset})=\mathfrak{D}_I^{-1}((\mathfrak{v}-\chi\mathfrak{v}_+)-r^{\frac{1}{2}}\mathfrak{h}_{\emptyset}*\mathfrak{h}_{\emptyset} +  r^{\frac{1}{2}}\chi(1-\chi)\mathfrak{b}_0*\mathfrak{b}_0  -  2r^{\frac{1}{2}}\mathfrak{h}_{\emptyset}* \chi \mathfrak{b}_0 - \nabla{\chi} *\mathfrak{b}_0).
\end{eqnarray*}

\begin{lemma}
Let $\kappa \ge 1$  be the constant in  Lemma \ref{A2}. Suppose that $|(\mathfrak{o}-\chi\mathfrak{o}_+)|_{L^2} \le \frac{1}{\kappa^4} r^{-\frac{1}{2}}$ and $r$ is sufficiently large. Let $\mathcal{B}$ be a ball in $\mathbb{H}$  centered at the origin with radius  $\frac{1}{\kappa^2} r^{-\frac{1}{2}}$. Then $\mathcal{T}: \mathcal{B} \to \mathcal{B}$ is a contraction mapping. \label{A3} In particular, equation (\ref{e59}) has a unique solution.
\end{lemma}

\begin{proof} Let $\mathfrak{h} \in \mathcal{B}$. Then we have
\begin{eqnarray*}
&&|\mathcal{T}(\mathfrak{h})|_H \\
&\le& \kappa \left( |(\mathfrak{v}-\chi\mathfrak{v}_+)|_{L^2}+ r^{\frac{1}{2}}|\mathfrak{h}*\mathfrak{h}|_{L^2}+r^{\frac{1}{2}}|\chi(1-\chi)\mathfrak{b}_0*\mathfrak{b}_0|_{L^2} + 2r^{\frac{1}{2}}|\mathfrak{h}*\chi\mathfrak{b}_0|_{L^2} + |\nabla{\chi} * \mathfrak{b}_0 |_{L^2} \right)\\
&\le& \frac{1}{\kappa^3} r^{-\frac{1}{2}}+ c_0\kappa r^{\frac{1}{2}} |\mathfrak{h}|^2_H+ c_0\kappa r^{-\frac{3}{2}} +  c_0\kappa r^{-\frac{1}{2}} |\mathfrak{h}|_H + c_0\kappa r^{-1}
\le \frac{1}{\kappa^2} r^{-\frac{1}{2}}.
\end{eqnarray*}
Therefore, $\mathcal{T}$ maps $\mathcal{B}$ to itself.

For $\mathfrak{h}_1, \mathfrak{h}_2 \in \mathcal{B}$,
\begin{eqnarray*}
&&|\mathcal{T}(\mathfrak{h}_1)- \mathcal{T}(\mathfrak{h}_2)|_H \\
&=&|\mathfrak{D}_I^{-1}(r^{\frac{1}{2}}\mathfrak{h}_1*\mathfrak{h}_1- r^{\frac{1}{2}}\mathfrak{h}_2*\mathfrak{h}_2 + 2r^{\frac{1}{2}}(\mathfrak{h}_1-\mathfrak{h}_2)*\chi\mathfrak{b}_0)|_H\\
&\le& \kappa  \left(r^{\frac{1}{2}}|\mathfrak{h}_1|_H + r^{\frac{1}{2}}|\mathfrak{h}_2|_H + |\mathfrak{b}_0|_{\infty} \right)|\mathfrak{h}_1-\mathfrak{h}_2|_H
\le\delta |h_1-h_2|_H,
\end{eqnarray*}
where $0<\delta<1$. Therefore,  $\mathcal{T}: \mathcal{B} \to \mathcal{B}$ is a contraction mapping.
\end{proof}

\subsubsection{Sup-norm estimate} \label{section8}
We show that the sup-norm of $\mathfrak{h}_{\emptyset}$ is small in the following sense:
\begin{lemma}
 Suppose that $|(\mathfrak{v}-\chi\mathfrak{v}_+)|_{L^2} \le \frac{1}{\kappa^4} r^{-\frac{1}{2}}$ and $r$ is sufficiently large enough. Let $\mathfrak{h}_{\emptyset}$ be  the solution in Lemma \ref{A3}  and $\kappa$ be the constant in Lemma \ref{A2},  then there exists $\kappa_0>1$ such that if $\kappa \ge \kappa_0$, we have   $|\mathfrak{h}_{\emptyset}|_{\infty} \le c_0 r^{-1}$. \label{A4}
\end{lemma}
\begin{proof} Recall that $\mathfrak{h}$ satisfies the equation:
\begin{equation}
\mathfrak{D}_I \mathfrak{h}_{\emptyset} + r^{\frac{1}{2}} \mathfrak{h}_{\emptyset}*\mathfrak{h}_{\emptyset} + r^{\frac{1}{2}}\mathfrak{h}_{\emptyset}*\mathfrak{A} = \mathfrak{B}, \label{A25}
\end{equation}
where $\mathfrak{A}= 2\chi \mathfrak{b}_0$ and  $\mathfrak{B}= r^{\frac{1}{2}}\chi(1-\chi)\mathfrak{b}_0* \mathfrak{b}_0 -\nabla{\chi} \mathfrak{b}_0 + (\mathfrak{v} - \chi\mathfrak{v}_+)$.  By Lemma 3.10 in \cite{Te2}, it is easy to check that $\mathfrak{A}$ and $\mathfrak{B}$ satisfy:
\begin{equation*}
\begin{split}
&|\mathfrak{A}| \le c_0 r^{-1}, \ \  |\nabla{\mathfrak{A}}| \le c_0, \\
&|\mathfrak{B}| \le c_0 r^{-\frac{1}{2}}, \ \ |\nabla{\mathfrak{B}}| \le c_0.
\end{split}
\end{equation*}
Keep in mind that the operators $\mathfrak{D}_I \mathfrak{h}$ and $\mathfrak{D}^*_I \mathfrak{h}$  are both of the  form $\nabla\mathfrak{h} + r^{\frac{1}{2}} O(1) *\mathfrak{h}$.

By  the $Weizenb\ddot{o}ck$ formula (\ref{e20}), we obtain
\begin{eqnarray*}
&&\nabla^*\nabla{\mathfrak{h}_{\emptyset}}+ 2r\mathfrak{h}_{\emptyset}  + \tilde{\mathcal{R}_0} \mathfrak{h}_{\emptyset} + \sqrt{r} \tilde{ \mathcal{R}_1}\mathfrak{h}_{\emptyset}\\
&=&\mathfrak{D}_I^*\mathfrak{B} - r^{\frac{1}{2}}\mathfrak{D}^*_I(\mathfrak{h}_{\emptyset}*\mathfrak{h}_{\emptyset}) -  r^{\frac{1}{2}}\mathfrak{D}^*_I(\mathfrak{h}_{\emptyset} *\mathfrak{A})\\
&=&I+ II+ III.
\end{eqnarray*}
 Take inner product of above equation with $\mathfrak{h}_{\emptyset}$, then  we have
\begin{eqnarray*}
|<III, \mathfrak{h}_{\emptyset}>|&=&|r^{\frac{1}{2}}<\mathfrak{D}^*_I(\mathfrak{h}_{\emptyset}  *\mathfrak{A}), \mathfrak{h}_{\emptyset} >| \le c_0
( r^{\frac{1}{2}}|\mathfrak{A}||\nabla{\mathfrak{h}_{\emptyset} }||\mathfrak{h}_{\emptyset} | +r|\mathfrak{A}| |\mathfrak{h}_{\emptyset} |^2 + r^{\frac{1}{2}}|\nabla{\mathfrak{A}}| |\mathfrak{h}_{\emptyset} |^2) \\
 &\le& c_0 r^{-\frac{1}{2}}|\mathfrak{h}_{\emptyset}||\nabla{\mathfrak{h}_{\emptyset} }| + c_0r^{\frac{1}{2}} |\mathfrak{h}_{\emptyset} |^2.
\end{eqnarray*}
\begin{equation*}
|<I, \mathfrak{h}_{\emptyset}>|=| < \mathfrak{D}^*_I\mathfrak{B},\mathfrak{h}_{\emptyset} >|
\le c_0|\nabla \mathfrak{B}||\mathfrak{h}_{\emptyset}|+ c_0r^{\frac{1}{2}} | \mathfrak{B}||\mathfrak{h}_{\emptyset}|
\le c_0|\mathfrak{h}_{\emptyset}|.
\end{equation*}
By definition,
\begin{equation}
r^{\frac{1}{2}}\mathfrak{h}_{\emptyset}*\mathfrak{h}_{\emptyset}=
  \begin{cases}
   -2^{-\frac{3}{2}}(r)^{\frac{1}{2}}q(\eta_h, \eta_h) \\
(2r)^{\frac{1}{2}}\mathfrak{cl}(b_h)\eta_h\\
0\\
  \end{cases}
\end{equation}
then
\begin{eqnarray*}
&&II=r^{\frac{1}{2}}\mathfrak{D}_I^*(\mathfrak{h}_{\emptyset}*\mathfrak{h}_{\emptyset})\\
&=&\left(-2d^*(r^{\frac{1}{2}}q(\eta_h,\eta_h)) +2(2)^{\frac{1}{2}}irIm(\psi_I^*(\mathfrak{cl}( b_h)\eta_h) ), D_{A_{I}}((2r)^{\frac{1}{2}}\mathfrak{cl}(  b_h) \eta_h)+ (2)^{\frac{1}{2}}r \mathfrak{cl}(q( \eta_h,\eta_h)) \psi_I\right)\\
&=&\left( A, B \right).
\end{eqnarray*}
Hence,
\begin{eqnarray*}
|A| &\le& c_0 r^{\frac{1}{2}}  | \nabla \eta_h| |\eta_h | + c_0 r | b_h||\eta_h|,\\
|B| &\le&  c_0 r^{\frac{1}{2}}|\nabla  b_h||\eta_h| + c_0 r^{\frac{1}{2}} |b_h | |\nabla \eta_h| + c_0r | \eta_h||\eta_h|.
\end{eqnarray*}
By our construction, $(A_{\emptyset}, \psi_{\emptyset}) =(A_I,\psi_I)+ ((2r)^{\frac{1}{2}}( b_h+ \chi b_0), (\chi \alpha_0+ \alpha_h, \chi \beta_0+\beta_h))$ satisfies the Seiberg Witten equations. By Lemma \ref{C50}, we have  $|\psi_{\emptyset}| \le 1+ \frac{c_0}{r}$.    In particular, $|\eta_h|\le c_0$.  Therefore,
\begin{eqnarray*}
|<II, \mathfrak{h}_{\emptyset}>| &\le &  c_0|A|| b_h|+ c_0|B| | \eta_h| \\
&\le&  c_0 r^{\frac{1}{2}} | \nabla  \eta_h||\eta_h| | b_h|  +  c_0 r | b_h|^2|\eta_h|  \\
&+& c_0 r^{\frac{1}{2}}|\nabla  b_h||\eta_h|| \eta_h| +  c_0 r^{\frac{1}{2}} | b_h | |\nabla \eta_h| | \eta_h| + c_0r | \eta_h|^3\\
&\le&  c_0 r^{\frac{1}{2}}|\nabla \mathfrak{h}_{\emptyset}||\mathfrak{h}_{\emptyset}| + c_0r | \mathfrak{h}_{\emptyset}|^2.
\end{eqnarray*}
Here we have used the fact that $|\eta_h|$ is bounded. On the other hand,
\begin{eqnarray*}
&&<\nabla^*\nabla{\mathfrak{h}_{\emptyset} }+2r  \mathfrak{h}_{\emptyset} + \tilde{\mathcal{R}_0} \mathfrak{h}_{\emptyset} +   \sqrt{r} \tilde{ \mathcal{R}_1} \mathfrak{h}_{\emptyset},  \mathfrak{h}_{\emptyset}>\\
&\ge & \frac{1}{2} d^*d | \mathfrak{h}_{\emptyset}|^2 +2 r | \mathfrak{h}_{\emptyset}|^2  + |\nabla  \mathfrak{h}_{\emptyset}|^2 - c_0(1+ r^{\frac{1}{2}}) |\mathfrak{h}_{\emptyset}|^2.
\end{eqnarray*}
Hence, we obtain a differential inequality:
\begin{eqnarray*}
d^*d|\mathfrak{h}_{\emptyset}|^2+ r| \mathfrak{h}_{\emptyset} |^2 \le c_0| \mathfrak{h}_{\emptyset}| +c_0r| \mathfrak{h}_{\emptyset}|^2.
\end{eqnarray*}

Let $g(x,y)$ be the Green function of $d^*d+ r$ over $\overline{X}$.  It  satisfies the following properties:
 \begin{itemize}
\item
$0< g(x, y) \le c_0 dist(x,y)^{-2} e^{-\sqrt{r} dist(x, y)/c_0}$ and $| \nabla g(x, y)| \le c_0 dist(x,y)^{-1} e^{-\sqrt{r} dist(x, y)/c_0}$.
\item
$\int_{\overline{X}} g(x, y) vol_y \le c_0r^{-1}$.
\end{itemize}

 The standard elliptic regularity implies that $|\mathfrak{h}_{\emptyset}| \to 0$ as $s \to +\infty$, (Cf. (4-11) of \cite{T})  thus we may assume that $|\mathfrak{h}_{\emptyset}|(x)=|\mathfrak{h}_{\emptyset}|_{\infty} $ for some $x \in \overline{X}$. Then
\begin{equation}
 \begin{split}
 |\mathfrak{h}_{\emptyset}|_{\infty}=& |\mathfrak{h}_{\emptyset} |(x) \le c_0 r \int_{\overline{X}} | \mathfrak{h}_{\emptyset}|g(x,y)dy  + c_0r^{-1} \\
&\le c_0r\left( \int_{\overline{X}} | \mathfrak{h}_{\emptyset}|^3\right)^{\frac{1}{3}}\left( \int_{\overline{X}}   \label{E2} g^{\frac{3}{2}}\right)^{\frac{2}{3}}  +c_0r^{-1} \\
&\le c_0   r^{\frac{2}{3}} \left( \int_{\overline{X}} | \mathfrak{h}_{\emptyset}|^3\right)^{\frac{1}{3}}+ c_0r^{-1} \\
&\le c_0  r^{\frac{2}{3}}  | \mathfrak{h}_{\emptyset}|_{L^4}^{\frac{2}{3}}  |\mathfrak{h}_{\emptyset}|_{L^2}^{\frac{1}{3}}+  c_0r^{-1} \\
&\le c_0 r^{\frac{1}{2}}|\mathfrak{h}_{\emptyset}|_H   +  c_0r^{-1} \le c_0 \kappa^{-2}.
 \end{split}
\end{equation}
Hence,  $|\mathfrak{h}_{\emptyset}| \le c_0 \kappa^{-2}$ for sufficiently large $r$.

On the other hand,
\begin{eqnarray*}
|<II, \mathfrak{h}_{\emptyset}>| &\le&  c_0 r^{\frac{1}{2}} | \nabla  \eta_h||\eta_h| | b_h|  +  c_0 r | b_h|^2|\eta_h|  \\
&+& c_0 r^{\frac{1}{2}}|\nabla  b_h||\eta_h|| \eta_h| +  c_0 r^{\frac{1}{2}} | b_h | |\nabla \eta_h| | \eta_h| + c_0r | \eta_h|^3\\
&\le&  c_0 r^{\frac{1}{2}}|\nabla \mathfrak{h}_{\emptyset}||\mathfrak{h}_{\emptyset}|^2 + c_0r | \mathfrak{h}_{\emptyset}|^3.
\end{eqnarray*}
Then we obtain
\begin{eqnarray}
d^*d|\mathfrak{h}_{\emptyset} |^2+ r| \mathfrak{h}_{\emptyset}|^2 \le c_0|\mathfrak{h}_{\emptyset}| + c_0 r| \mathfrak{h}_{\emptyset}|^4  + c_0r | \mathfrak{h}_{\emptyset}|^3. \label{A11}
\end{eqnarray}
Again using the Green function, we have
\begin{eqnarray*}
| \mathfrak{h}_{\emptyset}|_{\infty} &\le&   c_0 \left(  \int_{\overline{X}} g(x,y)r | \mathfrak{h}_{\emptyset}|^2 dy + \int_{\overline{X}} g(x,y)r |\mathfrak{h}_{\emptyset}|^3 dy + r^{-1} \right)\\
&\le& c_0(|\mathfrak{h}_{\emptyset}|_{\infty}+ |\mathfrak{h}_{\emptyset}|_{\infty}^2)| \mathfrak{h}_{\emptyset}|_{\infty}+ c_0 r^{-1}.
\end{eqnarray*}
Therefore, $|\mathfrak{h}_{\emptyset}| \le  c_0  r^{-1}$ if $\kappa$ is  sufficiently  large.

\end{proof}

\subsection{Iteration }
This subsection is not necessary if we choose $\wp_3$ and $\wp_4$ in the following ways: Recall that
$[\wp_3]=2\pi c_1(\mathfrak{s}_{\Gamma})$. Since $\mathcal{C} =\emptyset$, the line bundle $E$ is trivial.  We can take $\wp_3=iF_{A_{K^{-1}}} + d\mu_+$.  Using   a suitable cut-off function $\chi$ which supports  in the ends, we can construct  $\mu=\chi\mu_+$ and take $\wp_4=iF_{A_{K^{-1}}} + d\mu$. With these choices, we have $|(\mathfrak{v}-\chi\mathfrak{v}_+)|_{L^2} \le \kappa^{-4} r^{-\frac{1}{2}}$ provided that $|\mu_+| \ll 1$ sufficiently small. Then we can apply Lemma \ref{A3} directly to construct a  solution to (\ref{e59}).

If we choose $\wp_3$ and $\wp_4$ in an arbitrary way,  then the assumption $|(\mathfrak{v}-\chi\mathfrak{v}_+)|_{L^2} \le \kappa^{-4} r^{-\frac{1}{2}}$ may not hold.  But we can use the following argument to  solve (\ref{E1}). Firstly, we reintroduce the coordinate function $s$ of $\overline{X}$. Let $f: X \to \mathbb{R}$ be a Morse function such that $f=4$ on $\partial X= Y_+$. Then we define a new coordinate function $\tilde{s}$ on $\overline{X}$ such that $\tilde{s} = f$ on $X$ and $\tilde{s}=s+4$ on $\mathbb{R}_+ \times Y_+$. In general $\tilde{s}$ is not smooth, but we can use mollifier to make it smoothly.

Let $\chi_n = \chi(\tilde{s}+ n\kappa^{-4}-4)$ and $\mu_n= \chi_n \mu$. Note that $|\chi_{n+1}-\chi_n| \le c_0 \kappa^{-4}$.
First of all, we solve the following equations:
\begin{equation}
  \begin{cases}
   D_{A_0} \psi_0 =0 \\
F_{A_0}^{+}= \frac{r}{2}(q(\psi_0,\psi_0)- i\Omega_X)  +  i\mu_0 \\
*d*b_0 -2^{-\frac{1}{2}}r^{\frac{1}{2}}(\eta_0^* \psi_I-\psi_I^* \eta_0)=0,
  \end{cases}
\end{equation}
where $A_0=A_I + (2r)^{\frac{1}{2}}b_h^0 + \chi(2r)^{\frac{1}{2}} b_0$, $\psi_0= \psi_I + \eta_h^0 +\chi \eta_0$, $b_0= b_h^0 + \chi b_0$ and $\eta_0= \eta_h^0 +\chi \eta_0 $.  Since $|(\mathfrak{v}-\chi\mathfrak{v}_+)|_{L^2}= (2r)^{-\frac{1}{2}}|\chi(\mu- \mu)|_{L^2} \le c_0 \kappa^{-4}r^{-\frac{1}{2}}$, by Lemma \ref{A3}, we can find  a unique solution $\mathfrak{h}_0=(b^0_{h}, \eta_h^0) \in \mathbb{H}$. Moreover,  $|\mathfrak{h}_0|_H \le \kappa^{-2}r^{-\frac{1}{2}}$ and $|\mathfrak{h}_0|_{\infty} \le c_0r^{-1}$.

Next, we solve the equations
\begin{equation}
  \begin{cases}
   D_{A_1} \psi_1 =0 \\
F_{A_1}^{+}= \frac{r}{2}(q(\psi_1,\psi_1)- i\Omega_X)  +  i\mu_1\\
*d*b_1 -2^{-\frac{1}{2}}r^{\frac{1}{2}}(\eta_1^* \psi_I-\psi_I^* \eta_1)=0,
  \end{cases}
\end{equation}
where $A_1=A_I + (2r)^{\frac{1}{2}}b_h^0 +  (2r)^{\frac{1}{2}}\chi b_0 +(2r)^{\frac{1}{2}} b_h^1$,  $\psi_1= \psi_I + \eta_h^0 +\chi \eta_0 + \eta_h^1$, $b_1=  (2r)^{-\frac{1}{2}}(A_1-A_I)$, $\eta_1=\psi_1-\psi_I$ and $\mathfrak{h}_1=(b_h^1, \eta_h^1) \in \mathbb{H}$.  Then we need to solve the following equation
\begin{eqnarray*}
\mathfrak{D}_0 \mathfrak{h}_1 +  r^{\frac{1}{2}} \mathfrak{h}_1*\mathfrak{h}_1= \mathfrak{v}_1,
\end{eqnarray*}
where $\mathfrak{D}_0 \mathfrak{h}_1= \mathfrak{D}_I \mathfrak{h}_1 + r^{\frac{1}{2}}(\chi \mathfrak{b}_0 +\mathfrak{h}_0 ) *\mathfrak{h}_1 $ and $\mathfrak{v}_1= (2r)^{-\frac{1}{2}}( \chi_1- \chi_0) \mu $. Note that $|\mathfrak{v}_1|_{L^2} \le c_0 \kappa^{-4}r^{-\frac{1}{2}} $.

As before, define $\mathcal{T}(\mathfrak{h}_1) = -\mathfrak{D}_I^{-1} (r^{\frac{1}{2}}(\chi \mathfrak{b}_0 +\mathfrak{h}_0 ) *\mathfrak{h}_1 + r^{\frac{1}{2}} \mathfrak{h}_1*\mathfrak{h}_1- \mathfrak{v}_1  )$.   Let $\mathcal{B}$ be a ball in $\mathbb{H}$ with center at origin and radius less than $\frac{1}{\kappa^2} r^{-\frac{1}{2}}$ and $\mathfrak{h}_1 \in \mathcal{B}$. Then
\begin{eqnarray*}
|\mathcal{T}(\mathfrak{h}_1)|_H &\le& \kappa( r^{\frac{1}{2}} |(\chi \mathfrak{b}_0 +\mathfrak{h}_0 ) *\mathfrak{h}_1 |_{L^2} +  r^{\frac{1}{2}} |\mathfrak{h}_1*\mathfrak{h}_1|_{L^2} + |\mathfrak{v}_1|_{L^2}) \\
&\le&  \kappa(  |\mathfrak{h}_0|_{\infty} |\mathfrak{h}_1 |_{H} + |\mathfrak{b}_0 |_{\infty}|\mathfrak{h}_1|_H + c_0r^{\frac{1}{2}}|\mathfrak{h}_1 |^2_{H} + c_0 \kappa^{-4} r^{-\frac{1}{2}}) \\
&\le& c_0 \kappa^{-3} r^{-\frac{1}{2}}  +  c_0 \kappa^{-4} r^{-\frac{1}{2}}  \le   \kappa^{-2}r^{-\frac{1}{2}}.
\end{eqnarray*}
Also, it is not hard to check that $|\mathcal{T}(\mathfrak{h}_1)- \mathcal{T}(\mathfrak{h}'_1)|_H \le \delta | \mathfrak{h}_1 - \mathfrak{h}_1'|_H$ and $0< \delta<1$. By the contraction mapping theorem, we can find the solution $(A_1, \psi_1)$.   One can replace $\mathfrak{h}$  by  $\mathfrak{h}= \mathfrak{h}_0 + \mathfrak{h}_1$  and    repeat the argument  in Lemma \ref{A4} to show that $|\mathfrak{h}|_{\infty} \le c_0r^{-1}$. Note that in (\ref{E2}) we need to use the bound $|\mathfrak{h}|_H \le \kappa^{-2} r^{-\frac{1}{2}}$, but now $\mathfrak{h}= \mathfrak{h}_0+ \mathfrak{h}_1$. Fortunately,   we can still show that $|\mathfrak{h}|_{\infty} \le c_0\kappa^{-2}$ by the following inequalities:
\begin{equation} \label{e30}
 \begin{split}
 |\mathfrak{h}|_{\infty}=& | \mathfrak{h} |(x) \le c_0 r \int_{\overline{X}} | \mathfrak{h}|g(x,y)dy  + c_0r^{-1} \\
 &\le c_0 r \int_{\overline{X}} | \mathfrak{h}_1|g(x,y)dy  + |\mathfrak{h}_0|_{\infty} +  c_0r^{-1}\\
 & \le c_0r \int_{\overline{X}} | \mathfrak{h}_1|g(x,y)dy  + c_0r^{-1}\\
 & \le c_0r^{\frac{1}{2}} |\mathfrak{h}_1|_H + c_0 r^{-1} \le c_0\kappa^{-2}.
 \end{split}
\end{equation}
Therefore, with the above sightly modification,  we   still have  $|\mathfrak{h}|_{\infty} \le c_0r^{-1}$.

By induction, suppose that we have solved the $n$-th equation and the solution satisfies $|\mathfrak{h}_0 + \dots + \mathfrak{h}_n | \le c_0 r^{-1}$. To solve the $(n+1)$-th equation, let $A_{n+1}=A_n +(2r)^{\frac{1}{2}} b_h^{n+1}$, $\psi_{n+1}= \psi_n +\eta_{h}^{n+1}$ and  $\mathfrak{h}_{n+1} \in \mathbb{H}$. As before, it equivalent to find the fixed point of
\begin{eqnarray*}
\mathcal{T}(\mathfrak{h}_{n+1}) = -\mathfrak{D}_I^{-1} (r^{\frac{1}{2}}(\chi \mathfrak{b}_0 +\mathfrak{h}_0 + \dots + \mathfrak{h}_n ) *\mathfrak{h}_{n+1} + r^{\frac{1}{2}} \mathfrak{h}_{n+1} *\mathfrak{h}_{n+1}- \mathfrak{v}_{n+1}),
\end{eqnarray*}
where $\mathfrak{v}_{n+1}= i(2r)^{-\frac{1}{2}}(\mu_{n+1}-\mu_n)$.  Now the contribution of $(\chi \mathfrak{b}_0 +\mathfrak{h}_0 + \dots + \mathfrak{h}_n ) *\mathfrak{h}_{n+1}$ is
\begin{eqnarray*}
&&\kappa r^{\frac{1}{2}} |(\mathfrak{b}_0 +\mathfrak{h}_0 + \dots  + \mathfrak{h}_n ) *\mathfrak{h}_{n+1}|_{L^2} \\
&\le & \kappa (|\mathfrak{b}_0|_{\infty} + | \mathfrak{h}_0 \dots  +  \mathfrak{h}_n|_{\infty} ) |\mathfrak{h}_{n+1}|_H   \le c_0  \kappa^{-1}r^{-\frac{3}{2}}.
\end{eqnarray*}
 As before, we can show that $\mathcal{T}$ is contraction mapping. Moreover, repeat the argument in Lemma \ref{A4} with  the sightly modification as (\ref{e30}),  we still have $$|\mathfrak{h}_0+ \dots +\mathfrak{h}_{n+1}|_{\infty} \le c_0r^{-1}. $$

Hence,  there exists a positive integer   $N =O(\kappa^4)$ such that $\mathfrak{h}_{\emptyset}=(b_h, \eta_h)= \mathfrak{h}_0 + \dots \mathfrak{h}_N$ and  $A_{\emptyset}=A_I +  (2r)^{\frac{1}{2}} b_h + \chi (2r)^{\frac{1}{2}} b_0$, $\psi_{\emptyset}= \psi_I + \eta_h + \chi \eta_0$, $b= b_h+\chi b_0 $, $\eta=\eta_h+ \chi\eta_0$ satisfies the  following equations
\begin{equation}
  \begin{cases}
   D_{A} \psi =0 \\
F_{A}^{+}= \frac{r}{2}(q(\psi,\psi)- i\Omega)  +  i\mu\\
*d*b -2^{-\frac{1}{2}}r^{\frac{1}{2}}(\eta^* \psi_I-\psi_I^* \eta)=0.
  \end{cases}
\end{equation}
Moreover, according to our construction, $(A_{\emptyset}, \psi_{\emptyset})$ converges to $(A_{\emptyset + }, \psi_{\emptyset +})$ as $s\to \pm\infty$.

\subsection{General case that $\mathcal{C} \ne \emptyset$}
\begin{prop} [Cf. Prop 7.1 of \cite{Te2}]
Let $\mathcal{C}=\{(C_a,1)\}$ be an embedded holomorphic curve  in $\overline{X}$.  There exists $c_0>1$  and a finite dimensional normed vector space $V_0$. Let $\mathcal{B}$  be a radius $c_0^{-1}$ ball in $\mathcal{K}$, where  $\mathcal{K}$ is the Banach space defined in Section 5.b of \cite{Te2}. Then for any $r \ge c_0$, there exists  a linear map $q:\mathcal{K} \to V_0$ such that \label{A9}
\begin{enumerate}
\item
$|q(\xi)| \le c_0|\xi|_{L^2}$ and $q$ is surjective.
\item
For any $v \in V_0$, there exists a unique  $\xi_{v} \in q^{-1}(v)\cap\mathcal{B}$ such that $\mathfrak{b}(\xi_{v})= \mathfrak{q}(\xi_{v})+ \mathfrak{h}(\xi_{v})$ satisfies the following equations:
\begin{equation} \label{e32}
  \begin{cases}
 \mathfrak{D}\mathfrak{b}+r^{\frac{1}{2}} \mathfrak{b}*\mathfrak{b}=\mathfrak{v}\\
 \lim\limits_{s \to \infty} \mathfrak{b}=\mathfrak{b}_{\pm}.\\
  \end{cases}
\end{equation}
Moreover, $|\xi_{v}|_{\mathcal{K}} \le (r^{-\frac{1}{2}+16\sigma} + |v|)$.

\item
$dim(V_0)=I(\mathcal{C})$.

\item
Let $\mathfrak{b}(\xi)=(b, \eta)$ be the solution to (\ref{e32}), then $(A_r,\psi_r)=(A^{\xi}+(2r)^{\frac{1}{2}}b, \psi^{\xi} + \eta)$ is a solution to Seiberg Witten equations  (\ref{e7}).  Moreover, $(A_r, \psi_r)$ is non-degenerate with Fredholm index $I(\mathcal{C})$.
\item
$\lim\limits_{r\to \infty} \frac{i}{2\pi}\int_{\overline{X}} F_{A_r} \wedge \omega_X = \int_{\mathcal{C}} \omega_X.$

\end{enumerate}
\end{prop}

\begin{proof}
%%%%%%%%%%%%%%%%%%%%%%%%%%%%%%%%%%%%%%%%%%%%%%%%%%%%%%%%%%%%%%
The proof of the first and second bullets of the proposition closely tracks Taubes' argument in Sections 4-7 of \cite{Te2}.  We summarize the main points in the following five steps.

%Since they will give us a picture about the argument in \cite{Te2} and  also clarify why his argument holds in our case.
 %We assume that $(\overline{X}, \Omega)$ is the completion of symplectic fibered cobordism and $\mathcal{C}=\{(C_a, 1)\} \in \widetilde{\mathcal{M}}^{J, L}_{X, I=0}(\alpha_+, \alpha_-)$, where $\alpha_{+}$ and $\alpha_-$ are admissible orbit sets. The construction in \cite{Te2} is also available in our setting.
% In what follows, we use notation and terminology of  \cite{Te2}.

\textbf{Step 1:} The first step is to construct a line bundle $E \to \overline{X}$.  Let $\Xi_{\mathcal{C} }$ be a collection of the limit of the $s$-constant slice of $\mathcal{C}$ as $s \to \infty$.     One can define an open cover  $\{U_C\}_{C \in \mathcal{C}}$, $\{U_{\gamma}\}_{\gamma \in \Xi_{\mathcal{C} } }$, $U_0$  of $\overline{X}$ as in \cite{Te2}, where $U_C$ is a  tubular neighborhood of $C$, $U_{\gamma }$ is  a tubular neighborhood of $\mathbb{R}_{ \ge R_0} \times \gamma$, and $U_0$ is the rest.
Let $C \in \mathcal{C}$ and $\pi: N_C \to C $ be the normal bundle.  $E$ is defined by gluing $\pi^*N_C \vert_{U_C}$, and the trivial line bundles over $\{U_{\gamma}\}_{\gamma \in \Xi_{\mathcal{C} } }$ and $U_0$.  The Spin-c bundle is $S_+= E \oplus (E  \otimes K_X^{-1})$.

\textbf{Step 2:} The next step is to construct  an approximation solution $(A^*, \psi^*)$  that  is close to solving   (\ref{e4}). We need to introduce  vortices firstly.  The Vortex equations is  an   two dimensional analogy of Seiberg Witten equations. A pair  $(A, \alpha)$ is called $n$-vortices  if they satisfy the following equations
\begin{equation*}
  \begin{cases}
 \bar{\partial}_A \alpha =0\\
 *F_A= -i(1-|\alpha|^2)\\
 \frac{1}{2\pi}\int_{\mathbb{C}} (1-|\alpha|^2)=n,
  \end{cases}
\end{equation*}
where $A \in \Omega^1(\mathbb{C}, i\mathbb{R})$ and $\alpha$ is a complex function over $\mathbb{C}$.  The moduli space of $n$-vortices is denoted by $\mathfrak{C}_n$.  It  is diffeomorphic  to symmetric product   $Sym^{n} (\mathbb{C})$. For  more details about  the vortices,  please refer  to Section 2 of \cite{Te2}.

The fiber of $N_C$ is a complex plane, so one can consider the vortex equations over there. Assign each fiber of  $N_C$ a vortex ,  a collection of these vortices give us a pairs  $(A^{C, r}, \psi^{C, r})$ defined near $C$, where $A^{C, r}$ and $\psi^{C, r}$ are  connection and section of $E$ respectively. Since our $\mathcal{C}$ is embedded, there is a canonical choice of vortices over each fiber which is called symmetric vortices. This vortex  is corresponding to the origin under isomorphism     $\mathfrak{C}_n \simeq Sym^{n} (\mathbb{C})$

Let $(\gamma, m_{\gamma}) \in \Xi_{\mathcal{C}}$, i.e., $\mathcal{C}$ is asymptotic to $\gamma$ with total multiplicities $m_{\gamma}$.  Thus the ends of $\mathcal{C}$ lie inside $U_{\gamma}$. Again, we   assign vortices on each fiber of the normal bundle of $\mathbb{R}_{ \ge R_0} \times \gamma$. The choice of these vortices  is determined by the asymptotic behavior (\ref{e35}). A collection of them gives us a pair $(A^{\gamma, r}, \psi^{\gamma, r})$ is defined over $U_{\gamma}$.

Let $(A_0, \psi_0=(1_{\mathbb{C}}, 0))$ be the trivial solution over $U_0$, where $A_0$ is trivial connection. The approximation solution $(A^*, \psi^*)$ is defined by gluing $\{(A^{C, r}, \psi^{C, r})\}_{C \in \mathcal{C}}$, $\{(A^{\gamma, r}, \psi^{\gamma, r})_{\gamma \in \Xi_{\mathcal{C}}}\}$ and $(A_0, \psi_0)$ together. It is worth noting that  the constructions  of   $\{(A^{C, r}, \psi^{C, r})\}_{C \in \mathcal{C}}$ and $\{(A^{\gamma, r}, \psi^{\gamma, r})_{\gamma \in \Xi_{\mathcal{C}}}\}$ take place near $\mathcal{C}$ and the choice of $(A_0, \psi_0)$ is canonical,   that is why the constructions of \cite{Te2} can be applied to our case without essential change.

\textbf{Step 3:} Given   vortices in $\mathfrak{C}_n$, one can perturb the vortices by using  sections     of  $T_{1,0} \mathfrak{C}_n$ and exponential map. Since our $(A^*, \psi^*)$ comes from vortices over each fiber of the normal bundle, we can deform $(A^*, \psi^*)$ in a similar way. Follow Section 5.b of \cite{Te2}, we can  define a Banach space  $\mathcal{K}$ which is an  analogy of  $T_{1,0} \mathfrak{C}_n$.  Given $\xi \in \mathcal{K}$, we can deform $(A^*, \psi^*)$ by using   $\xi$,  the result is denoted by $(A^{\xi}, \psi^{\xi})$.

 Suppose that $\mathcal{C}$ is asymptotic to an orbit set $\alpha_+$, according to  Section $3$ of \cite{Te2} (also see \cite{LT}), we can construct a  configuration $(A_{+}, \psi_{+})$ from $\alpha_{+}$ satisfying equations (\ref{e1}) together with an additional gauge fixed condition.  Let $\mathfrak{b}=(b,\eta) \in \Gamma(i\Omega^1 \oplus S_+)$, we want to find  $\mathfrak{b}$  such that $(A, \psi)=(A^{\xi}+ (2r)^{\frac{1}{2}}b, \psi^{\xi}+\eta)$ satisfies the following equations
\begin{equation} \label{e7}
  \begin{cases}
   D_{A} \psi_ =0 \\
F_{A}^{+}= \frac{r}{2}(q(\psi,\psi)- i\Omega_X)  - \frac{1}{2} F^+_{A_{K^{-1}}} - \frac{i}{2} \wp_4^+\\
*d*b -2^{-\frac{1}{2}}r^{\frac{1}{2}}(\eta^* \psi^{\xi}-{\psi^{\xi}}^* \eta)=0,
  \end{cases}
\end{equation}
with asymptotic behavior $\lim\limits_{s \to \pm \infty}(A, \psi)=(A_{\pm}, \psi_{\pm})$. In particular, $[(A, \psi)] \in \mathfrak{M}^{J, r }_{X}(\mathfrak{c}_+, \mathfrak{c}_-)$.
 The equations  (\ref{e7}) can be  schematically
as
\begin{equation}\label{e8}
  \begin{cases}
 \mathfrak{D}\mathfrak{b}+r^{\frac{1}{2}} \mathfrak{b}*\mathfrak{b}=\mathfrak{v}\\
 \lim\limits_{s \to \pm \infty} \mathfrak{b}=\mathfrak{b}_{\pm},\\
  \end{cases}
\end{equation}
where $\mathfrak{D}$ is the deformation operator of Seiberg Witten equations at $(A^{\xi}, \psi^{\xi})$.

 Observed by Taubes, the bundle $iT^*X \oplus S_+$ over $U_C$ can be identified as $\mathbb{V}_{C0} \oplus \mathbb{V}_{C1}$, where $\mathbb{V}_{C0}=\pi^*N \oplus \pi^*N$ and $\mathbb{V}_{C1}=\pi^*T^{0,1}C \oplus (\pi^*N^2 \otimes \pi^*T^{0,1}C)$. Granted this identification, the deformation operator $\mathfrak{D}$  can be rewritten as
\begin{equation}  \label{e36} \mathfrak{D}=
\left(
  \begin{array}{ccc}
    \bar{\partial}_{\theta}^H &  \vartheta^*_{C_{\xi}, r} \\
    \vartheta_{C_{\xi}, r} &  \partial_{\theta}^H\
  \end{array}
\right)   + \mathfrak{r},
\end{equation}
where $\mathfrak{r}$ is a small error term, $  \vartheta_{C_{\xi}, r}$ is the deformation operator of the vortex equations and $\vartheta^*_{C_{\xi}, r}$ is its $L^2$ adjoint.   The  similar identification also holds on $U_0$ and $U_{\gamma_{\pm}}$. The operator  $\vartheta^*_{C_{\xi}, r}$ is not surjective, so does $\mathfrak{D}$. One cannot expect to use   contraction mapping principle directly.

 As \cite{Te2}, we can define a $L^2$ projection $\Pi_{\xi}: L^2((i\Omega^0 \oplus \Omega^{2+} )\oplus S_-) \to L^2((i\Omega^0 \oplus \Omega^{2+} )\oplus S_-)$.     Roughly speaking, the $\Pi_{\xi}$ maps elements in $  \mathbb{V}_{C0} \oplus \mathbb{V}_{C1}$  to the $coker  \vartheta^*_{C_{\xi}, r}$.  Using the properties of vortices, identification (\ref{e36}) and the same argument in Lemma 6.1 of \cite{Te2}, we know that $(1-\Pi_{\xi})\mathfrak{D}$
is surjective  in suitable Sobolev space.  Hence, we  want to use the contraction mapping principle to solve the following equation firstly,
\begin{eqnarray} \label{e14}
(1-\Pi_{\xi}) (\mathfrak{D} \mathfrak{b} + r^{\frac{1}{2}}\mathfrak{b}* \mathfrak{b})=(1-\Pi_{\xi}) \mathfrak{v}.
\end{eqnarray}

\textbf{Step 4: }
To solve (\ref{e14}), however, the error term $(1-\Pi_{\xi}) \mathfrak{v}$ is not small enough  to use  the contraction mapping principle.  Write  $\mathfrak{b}= \mathfrak{q}+\mathfrak{h}$,  equation (\ref{e14}) is equivalent to
\begin{eqnarray} \label{e9}
(1-\Pi_{\xi}) (\mathfrak{D} \mathfrak{q} + r^{\frac{1}{2}}\mathfrak{q}* \mathfrak{q} +2 r^{\frac{1}{2}}\mathfrak{q}* \mathfrak{h})=(1-\Pi_{\xi}) (\mathfrak{v} - \mathfrak{v}_{\mathfrak{h}}),
\end{eqnarray}
where $\mathfrak{v}_{\mathfrak{h}} = \mathfrak{D} \mathfrak{h} + r^{\frac{1}{2}}\mathfrak{h}* \mathfrak{h}$ and $\Pi_{\xi} \mathfrak{q}=0$.

Section 6.d of \cite{Te2} constructs a  $\mathfrak{h}=\mathfrak{h}(\xi)$ so that the error term  $(1-\Pi_{\xi}) (\mathfrak{v} - \mathfrak{v}_{\mathfrak{h}})$ is small enough.  The construction of  $\mathfrak{h}$ in our case is same as Lemma 6.3 of \cite{Te2} with the following adjustment:  Replace $\mathfrak{b}_0$  by  $\chi \mathfrak{b}_0+ \mathfrak{h}_{\emptyset}$ in proof of Lemma 6.3, where $\chi$ and $\mathfrak{h}_{\emptyset}$ is defined  in Section \ref{section16}. Note that Lemma \ref{A4}  guarantees that  the   estimates  (6-31)  in Lemma 6.3 of \cite{Te2} hold, and the remaining  part of proof takes place near the holomorphic curves, hence the construction in \cite{Te2} can be applied to cobordism case with only notation changes.

Given $\mathfrak{h}(\xi)$, for any sufficiently small $\xi$,  then we can find $\mathfrak{q}=\mathfrak{q}(\xi)$ satisfying equation (\ref{e9}) by using contraction mapping principle.

\textbf{Step 5: } To deal the remaining  part  of
 (\ref{e8}), we need to consider equation
\begin{eqnarray} \label{e10}
\Pi_{\xi}(\mathfrak{D} \mathfrak{q} + r^{\frac{1}{2}}\mathfrak{q}* \mathfrak{q} +2 r^{\frac{1}{2}}\mathfrak{q}* \mathfrak{h} -\mathfrak{v} + \mathfrak{v}_{\mathfrak{h}})=0.
\end{eqnarray}
The argument in Step 4 works for any small $\xi$, thus one hope to vary $\xi$ such that  (\ref{e10}) holds.
Let $\mathcal{T}(\xi)$ to be $r^{\frac{1}{2}}$ times the left hand side of (\ref{e10}).
We can write
\begin{eqnarray*}
\mathcal{T}(\xi)= \mathcal{T}_0 + \mathcal{T}_1(\xi) + \mathcal{T}_2(\xi),
\end{eqnarray*}
where
$\mathcal{T}_0$ is the $\xi=0$ version of $\mathcal{T}$, $\mathcal{T}_1$ is the linerization of $\mathcal{T}$, and $\mathcal{T}_2$ is a quadric error term. The argument in \cite{Te2} can show that  $\mathcal{T}_0$ is closed to zero. Under the  identification in Section 7.d of \cite{Te2},  $\mathcal{T}_1$ is  surjective provided that  $\{\mathcal{D}_C\}_{C \in \mathcal{C}}$ are surjective, where $\mathcal{D}_C$ is the linearization of the holomorphic curve.
Therefore, we can use the contraction mapping argument to find solutions  $\xi$ to equation (\ref{e10}).

%%%%%%%%%%%%%%%%%%%%%%%%%%%%%%%%%%%%%%%%%%%%%%%%%%%%%%%%%%%%%%%%%%

%Because the estimates in section $4$-$7$ of \cite{Te2} are locally and $\Omega$, $g$ and $\wp_4$ are $\mathbb{R}$-invariant on the ends,  the proof of the first and second bullets of the proposition follow the  corresponding arguments in section $4$-$7$ of  \cite{Te2} with only notational changes and the following adjustment: Replacing $\mathfrak{b}_0$ in Lemma 6.3 of \cite{Te2} by $\chi \mathfrak{b}_0+ \mathfrak{h}$ in our Lemma \ref{A3}, note that $|\mathfrak{h}| \le c_0 r^{-1}$ guarantees  that all estimates that are required in Lemma 6.3 \cite{Te2} still hold.  In addition, the vector bundle $iTM \oplus i \mathbb{R} \oplus S_+$ is to be replaced by $iT\overline{X} \oplus S_+$ and $i(\mathbb{R} \oplus \Omega^{2+} ) \oplus S_-$ in corresponding situations.
%
%The analysis in section 3.a of \cite{Te3} are also locally,  their arguments can be borrowed almost verbatim to prove that $(A_r, \psi_r)$ is non-degenerate with index $I(\mathcal{C})$ whenever $r$  is sufficiently large.

To  prove  the last point of the proposition,  rewrite  $ (A^{\xi}, \psi^{\xi})= (A^{*}, \psi^{*}) + \mathfrak{t}_{\xi}$, then
\begin{eqnarray*}
(A_r, \psi_r)= (A^{\xi}, \psi^{\xi})+((2r)^{\frac{1}{2}}b_r,\eta_r)
= (A^{*}, \psi^{*})+ ((2r)^{\frac{1}{2}}b_{\xi}, \eta_{\xi}),
\end{eqnarray*}
where $( (2r)^{\frac{1}{2}}b_{\xi}, \eta_{\xi}) = ( (2r)^{\frac{1}{2}}b_r,\eta_r) +   \mathfrak{t}_{\xi}$.
Note that the pair  $(b_{\xi}, \eta_{\xi}) $  is asymptotic to $(b_{\xi_{\alpha_+}},  \eta_{\xi_{\alpha_+}}) $, where   $(b_{\xi_{\alpha_+}},  \eta_{\xi_{\alpha_+}}) $ is used to construct $(A_+, \psi_+)$ from orbit
 set $\alpha_+$, corresponding to the solution to equation (3-6) in \cite{Te2}.  By Lemma 3.10 of \cite{Te2}, $d\omega_X =0$ and Stokes' formula, we have
\begin{eqnarray*}
&&|\int_{\overline{X}} (2r)^{\frac{1}{2}} db_{\xi} \wedge \omega_X |= |\int_{Y_+} (2r)^{\frac{1}{2}} b_{\xi_{\alpha_+}} \wedge \omega_+| \le c_0 r^{\frac{1}{2}}|b_{\xi_{\alpha_+}}|_{L^2} \le c_0 r^{-\frac{1}{2}}.
\end{eqnarray*}
Hence, it suffices to prove the statement  for  $A^*$. Then the last point of the proposition follows from the following Lemma \ref{A27}.

\end{proof}

\begin{lemma} \label{A27}
Let $(A^*, \psi^*)$ be the approximation solution which is constructed  in Proposition \ref{A9}, then \label{A26}
\begin{eqnarray*}
\lim_{r\to \infty} \frac{i}{2\pi}\int_{\overline{X}} F_{A^*} \wedge \omega_X = \int_{\mathcal{C}} \omega_X.
\end{eqnarray*}
\end{lemma}
\begin{proof}
The proof of this lemma relies heavily on the construction of $(A^*, \psi^*)$ in  Section 5 of \cite{Te2} and the properties of vortices.  In what follows, we use the notation and terminology of  \cite{Te2}.

 Recall that when we construct  $(A^* ,\psi^*)$, we need to  assign vortices to each fiber of normal bundle of $\mathcal{C}$. In fact, the collection of these vortices form a tuple of vortices sections $(\{\mathfrak{c}_C\}_{C \in \mathcal{C}}, \{ \mathfrak{c}_{\gamma_{+} } \}_{\gamma_+ \in \Xi_{\mathcal{C}}})$ of certain vector bundles, called vortices bundle. For the  details,  please refer to  Section 2.e of \cite{Te2}. In particular, $\mathfrak{c_{\gamma_+} }$ here is a map from $ [R, +\infty) \times S^1$ to $ \mathfrak{C}_{m_{\gamma_+}}$.

Given $C \in \mathcal{C}$,    by the construction and properties of vortices,
$F_{A^*}= d^VA^{C,r} = r(1-|\alpha^{C,r}|^2)  \frac{1}{2} \nabla_{\theta}\mathfrak{s} \wedge \nabla_{\theta}\bar{\mathfrak{s}}$ on the region  where $\chi_C=1$.  Also,   $ |F_{A^*}| \le c_0e^{-r^{\sigma}}$ on the part where $\chi_C \neq 1$. Thus
\begin{eqnarray} \label{e68}
&&\int_{{\overline{X}}\bigcap U_{C}}  \frac{i}{2\pi} F_{A^*} \wedge \omega_X \nonumber \\
&=&\int_{{\overline{X}}\bigcap U_{C}}  \frac{1}{2\pi} r(1-|\alpha^{C,r}|^2)   \frac{i}{2}\nabla_{\theta}\mathfrak{s} \wedge \nabla_{\theta}\bar{\mathfrak{s}} \wedge \omega_X  +  O(e^{-r^{\sigma}}),   \\
&=&\int_{\mathcal{C} \bigcap U_{C}} \omega_X + O(e^{-r^{\sigma}}).  \nonumber
\end{eqnarray}
The last step follows from the fact that   $\int_{N_z} \frac{1}{2\pi}(1-|\alpha^{C,r}|^2) \frac{i}{2}\nabla_{\theta}\mathfrak{s} \wedge \nabla_{\theta}\bar{\mathfrak{s}}=1$ along each fiber of $N_C$.

Given $\gamma_+ \in \Xi_{\mathcal{C}+}$, define $\mathfrak{c_{\gamma_+} }_0: [R, +\infty) \times S^1 \to \mathfrak{C}_{m_{\gamma_+}}$  be a smooth vortex section so that $\mathfrak{c_{\gamma_+}}_0 = \mathfrak{c_{\gamma_+}}$ on $s \le R_* -R$ and $\mathfrak{c_{\gamma_+}}_0$ is   vortices when $s \ge R_*$, where $R_*$ is the constant defined in (4-9) of \cite{Te2}.   Using  $\mathfrak{c}_0=\{(\mathfrak{c}_C)_{C \in \mathcal{C}}, (\mathfrak{c_{\gamma_+}}_0)_{\gamma \in \Xi_{\mathcal{C}+}}\}$, we can construct an approximation solution $(A_0^*, \psi_0^*)$ as before. Note  that $\int_{\overline{X}} F_{A^*} \wedge \omega_X = \int_{\overline{X}} F_{A_0^*} \wedge \omega_X$ because of Stokes' theorem.

On $U_{\gamma}$, we can divide it into three parts $X_1$, $X_2$ and $X_3$. For  the definition of $X_i$,  please refer to Section $6$ of \cite{Te2}.  By the observation in the  last paragraph, we only need to prove the statement for $F_{A_0^*}$. Since $F_{A_0^*} \wedge \omega_+=0$ on $s \ge R_*$,  there is no contribution from $X_1$.   Now we consider the case on $ X_2$. Suppose $\chi_{\gamma}=1$, then $A^*=\theta_0+ A_0^{\gamma, r}+ A^{\gamma,r}$ and $F_{A^*}= d^VA^{\gamma,r} + x_{\gamma,r} - \overline{x_{\gamma,r}} $, where   $x_{\gamma,r}$ is defined in (2-37)  of \cite{Te2}. Roughly speaking, it is  derivative of  $A^*$ in the horizontal direction.  As before, on the part where $\chi_{\gamma} \neq 1$, the contribution is  $O(e^{-r^{\sigma}})$ and
\begin{eqnarray} \label{e69}
&&\int_{{\overline{X}}\bigcap (U_{\gamma} \cap  X_2)}  \frac{i}{2\pi} F_{A^*} \wedge \omega_X  \nonumber\\
&=&\int_{{\overline{X}}\bigcap (U_{\gamma} \cap  X_2)}  \frac{1}{2\pi} r(1-|\alpha^{\gamma,r}|^2) \frac{i}{2} dz \wedge d\bar{z}  \wedge \omega_X  +  \int_{{\overline{X}}\bigcap (U_{\gamma} \cap  X_2)}  ( x_{\gamma} - \overline{x_{\gamma,r}} )\wedge \omega_X       + O(e^{-r^{\sigma}}),   \nonumber\\
&=&m_{\gamma}\int_{\mathbb{R} \times \gamma \bigcap U_{\gamma} \cap  X_2} \omega_X + O(e^{-r^{\sigma}}) + \int_{{\overline{X}}\bigcap (U_{\gamma} \cap  X_2)}  ( x_{\gamma} - \overline{x_{\gamma,r}} )\wedge \omega_X \\
&=& \int_{{\overline{X}}\bigcap (U_{\gamma} \cap  X_2)}  ( x_{\gamma,r} - \overline{x_{\gamma,r}} )\wedge \omega_X     + O(e^{-r^{\sigma}}).  \nonumber
\end{eqnarray}
 By (2-2) of \cite{Te2}, $|x_{\gamma, r}| \le c_0r^{\frac{1}{2}} \sum_{z' \in \mathfrak{Z}_{\gamma}(w, t)}e^{- r^{\frac{1}{2}} |z-z'|}$, where $\mathfrak{Z}_{\gamma}(w, t) $ is defined in Section 5.d of \cite{Te2}. Thus we have
$|\int_{{\overline{X}}\bigcap (U_{\gamma} \cap X_2)} x_{\gamma,r}\wedge \omega_X | \le c_0 r^{-\frac{1}{2} +8\sigma}. $

On the intersection of $U_{\gamma} \cap X_3$ and near an end $\mathpzc{E}$ of $\mathcal{C}$,  recall that
\begin{equation*}
A^*= \chi_{\mathpzc{E}}( (1-\chi_{\gamma})(\theta_0 - u_{\gamma}^{-1}du_{\gamma})  + \chi_{\gamma}(\theta_0 + A_0^{\gamma, r}+ A^{\gamma, r})    ) + (1-\chi_{\mathpzc{E}})\mathbb{A}^{\mathpzc{E}},
\end{equation*}
where $\mathbb{A}^{\mathpzc{E}}$, $\chi_{\mathpzc{E}}$ and $\chi_{\gamma}$ are defined in page 2648 of \cite{Te2}.
Then
\begin{eqnarray*}
F_{A^*}&=&d\chi_{\mathpzc{E}}(1-\chi_{\gamma}) \left( (\theta_0 - u_{\gamma}^{-1}du_{\gamma}) -  \mathbb{A}^{\gamma, r}\right) + \chi_{\mathpzc{E}} d\chi_{\gamma}  \left( \mathbb{A}^{\gamma, r} -  (\theta_0 - u_{\gamma}^{-1}du_{\gamma})\right) \\
 &+&  d\chi_{\mathpzc{E}} ( \mathbb{A}^{\gamma, r}-\mathbb{A}^{\mathpzc{E}})+ (1-\chi_{\mathpzc{E}}) d \mathbb{A}^{\mathpzc{E}} + \chi_{\mathpzc{E}} \chi_{\gamma} d\mathbb{A}^{\gamma, r}\\
 &=&(I) +(II) +(III) +(IV) +(V).
\end{eqnarray*}
By the exponential decay of vortices and $(4$-$9)$ of \cite{Te2}, the contribution of $(I)$ and $(II)$ are  $c_0 e^{-r^{\sigma}}$.

By $(6$-$45)$, $(6$-$46)$ and $(6$-$48)$ of \cite{Te2},
\begin{eqnarray*}
| \mathbb{A}^{\gamma, r}-\mathbb{A}^{\mathpzc{E}}| \le c_0 r^{\frac{1}{2}}|z|^2 e^{-\frac{\sqrt{r}|\eta|}{c_0}}+c_0e^{-\frac{\sqrt{r}|\eta|}{c_0}}.
\end{eqnarray*}
Hence, the contribution of $(III)$ is $c_0r^{-1+8\sigma}$.

Arguing as the cases in (\ref{e68}) and (\ref{e69}), also using $(6$-$45)$, $(6$-$46)$ and $(6$-$48)$ of \cite{Te2},  the contributions of  $(IV)+(V)$ are $\int_{C \cap X_3}{ \omega_X} + O(r^{-\frac{1}{2}+8\sigma})$.
\end{proof}

\subsection{Proof of Theorem \ref{A10}}

\begin{proof}

%[Proof of Theorem \ref{A10}]
For each $\mathcal{C}=\{C_a\} \in \mathcal{M}_{X, I=0}^{J,L}(\alpha_+ , \alpha_-) $,   Proposition \ref{A9} gives us a   unique solution $(A_r,\psi_r)$ to  Seiberg Witten equations (\ref{e4}). Moreover, $(A_r, \psi_r)$ is nondegenerate with zero Fredholm index  and $\lim\limits_{r \to \infty} \frac{i}{2\pi} \int_{\overline{X}} F_{A_r} \wedge \omega_X = \int_{\mathcal{C}} \omega_X$.

So we can define  the following meaningful map
\begin{eqnarray*}
\Psi^r : \mathcal{M}_{X, I=0}^{J,L}(\alpha_+ , \alpha_-) \to \mathfrak{M}_{X, {\rm ind}=0}^{J,L}(T^+_r(\alpha_+), T_r^-(\alpha_-)),
\end{eqnarray*}
by sending $\mathcal{C} \in \mathcal{M}_{X, I=0}^{J, L}(\alpha_+,  \alpha_-)$ to $[(A_r, \psi_r)] $,
where $\mathfrak{M}_{X, {\rm ind}=0}^{J,L}(T^+_r(\alpha_+), T_r^-(\alpha_-))$ is the moduli space of solutions to Seiberg Witten equations (\ref{e4}) with zero index.

Note that $\Psi^r$ is injective for the following reasons. Let $(A_r, \psi_r) = \Psi^r(\mathcal{C})$, then $F_{A_r}$ satisfies the following properties:
\begin{enumerate}
\item
The $L^2$-norm of $F_{A_r}$ over the $|s| \le R$ portion $U_0$ is bounded by $c_0$.
\item
For each $C_a \in \mathcal{C}$, for every $|s_0| \le R-1$, the $L^2$-norm of $F_{A_r}$ over the $[s_0-1, s_0+1]$ portion $U_C$ is greater than $c_0^{-1} r^{\frac{1}{2}}$.
\end{enumerate}
The  two properties above  can be checked directly from the construction. Let $\mathcal{C}, \mathcal{C}'  \in \mathcal{M}_{X, I=0}^{J, L}(\alpha_+,  \alpha_-)$ and $\mathcal{C} \ne \mathcal{C}'$. There exists $c_0' \ge c_0^{10}$ such that  for $r \ge c_0'$, then there exists a
component $C_a' $ of $\mathcal{C}'$ and $|s_0| \le R-1 $ such that the $[s_0-1, s_0+1]$ part of $U_{C_a'}$ lies inside the $\mathcal{C}$ version's $U_0$ part. Above two properties imply that $\Psi^r(\mathcal{C}) \ne \Psi^r(\mathcal{C}')$.

To show that  $\Psi^r$ is onto for sufficiently large $r$, as summarized in  \cite{LT}, we can divide the proof into three parts, which are called \textbf{estimation}, \textbf{convergence} and  \textbf{perturbation}.
\begin{enumerate}
\item
The estimation part: As remark  on  page 38, the  relevant estimates in \cite{LT} can be  carried
over almost verbatim to our setting. % More detail about this part please refer to \cite{Lee}.
\item
The convergence part:   Suppose that $\Psi^r$ is not onto for any $r$, then there exists a sequence $\{(A_n ,\psi_n)\}_{n=1}^{\infty} $ such that   each $(A_n ,\psi_n) \in \mathfrak{M}_{X, ind=0}^{J,L}(T^+_{r_n}(\alpha_+) , T^-_{r_n}(\alpha_-)   ) $ is  not image of $\Psi^{r_n}$ and $r_n \to \infty$. By Proposition \ref{C32} and Theorem 5.1 of \cite{21},  after passing a subsequence,  $\{(A_n ,\psi_n)\}_{n=1}^{\infty} $  converges to a broken holomorphic current $\mathcal{   C}$ with $I(\mathcal{C})=0$.
%By the same argument  in the proof of Theorem 5.1 of \cite{21}, $I(\mathcal{C})=0$. Note the the set up in Theorem 5.1 of \cite{21} is symplectic cobordism between contact manifold, but his proof can be adapted  with only notation changes to prove that $I(\mathcal{C})=0$.

According to  our assumptions, properties of $\mathcal{V}_{comp}(X, \pi_X, \omega_X)$ and  the counterparts of  Lemmas \ref{C3}, \ref{C2} and \ref{C25}, $\mathcal{C}$ is not broken and  it is embedded.

\item
The perturbation part: Let $\{(A_n, \psi_n)\}_{n=1}^{\infty}$ and $\mathcal{C}$ as above. We have the following refinement of the convergence (Cf. Lemma 6.2 of \cite{Te4}): Given any $\delta>0$,  then there exists a subsequence of $\{(A_n, \psi_n)\}_{n=1}^{\infty}$  such that
\begin{equation} \label{e33}
\sup_{z \in C \in \mathcal{C}} dist(z, \alpha_n^{-1}(0)) + \sup_{z \in \alpha_n^{-1}(0)} dist \left(z,  \bigcup_{C  \in \mathcal{C} } C \right) < \delta r_n^{-\frac{1}{2}}
\end{equation}
for sufficiently large $n$.
To prove (\ref{e33}), the argument is the same as Section 7 of \cite{Te4}. We define the special sections $\mathfrak{o}_k\vert_p$ as (7-9) of  \cite{Te4}, where $p$ lies in $ C \in \mathcal{C}$ or half $\mathbb{R}$-invariant cylinders.   Note that the proof of Lemmas 4.10,  7.1, 7.2, 7.3 and 7.4 in \cite{Te4} are local, they are still true in our case. Copy the argument in Lemmas 7.5 and 7.7 in \cite{Te4}, we can derive (\ref{e33}).

With (\ref{e33}), we follow Section 6.b  of \cite{Te4} to construct a gauge transformation $u_n$ such that $u_n\cdot (A_n, \psi_n)$ is close to the approximation solution $(A^*, \psi^*)$.  The construction relies on  Lemma 6.4 of \cite{Te2}. The proof only concerns the behavior of $(A_n, \psi_n)$ near $\mathcal{C}$, so it is still true in our case.  Write $$u_n \cdot (A_n, \psi_n) =(A^{\xi}, \psi^{\xi}) + ((2r_n)^{\frac{1}{2}}b^{\xi}, \eta^{\xi}),$$ where $(A^{\xi}, \psi^{\xi}) $ is defined in Step 3 of Proposition \ref{A9}. By contraction mapping principle  and argument in 6.c of \cite{Te4},  we find a new gauge transformation, still denoted by $u_n$, such that $\mathfrak{b}(\xi)=(b^{\xi}, \eta^{\xi})$ satisfies (\ref{e8}).  Let $\mathfrak{q}(\xi)=\mathfrak{b}(\xi)-\mathfrak{h}(\xi)$, where $\mathfrak{h}(\xi)$  is defined in step 4 of Proposition \ref{A9}.

If $\Pi_{\xi}\mathfrak{q}(\xi)=0$ and $\xi \in \mathcal{B}$, where $\mathcal{B} \subset \mathcal{K}$ is provided by Proposition \ref{A9}, then the uniqueness of equations (\ref{e9}), (\ref{e10}) implies that   $u_n^*(A_n, \psi_n)=\Psi^{r_n}(\mathcal{C})$, this lead to our goal.  The tasks of Sections 6.d and 6.e  in \cite{Te4} are to find such  $\xi \in \mathcal{B}$ such that $\Pi_{\xi}\mathfrak{q}(\xi)=0$. The analysis in  6.d and 6.e of \cite{Te4} are again local, they can be applied to our situation  without modification.
% copy the same argument in section 6 of \cite{Te4}  to find a gauge transformation $u_n$ such that $u_n^*(A_n, \psi_n)=\Psi^{r_n}(\mathcal{C})$. Then we obtain contradiction.
\end{enumerate}

\end{proof}

%%%%%%%%%%%%%%%%%%%%%%%%%%%%%%%%%%%%%%%%%%%%%%%%%%%%%%%%%%%%%%%%%%%%%

\section{Proof of Theorem \ref{Thm2}}
%In this section, we assume that  $(X, \pi_X)$ is minimal and $d=\Gamma_{\pm} \cdot [\Sigma ]> g(\Sigma)-1 $.\subsection{$Q$-$\delta$ flat approximation}

%Let $U$ be a tubular neighborhood of $Y$ so that $(U, \Omega_X)$ is symplecticmorphism to $([0, -\epsilon) \times Y, \omega+ds\wedge dt)$, where $\omega=\omega_X \vert_Y$. Moreover, $J_X=J$ on $U$.

Before we prove Theorem \ref{Thm2}, we want to reduce the proof to the case that   $(\omega_X, J ) \vert_{ \mathbb{R}_{\pm} \times Y_{\pm}}$ is $Q$-$\delta$ flat firstly.  Assume that $(X, \omega_X)$ or $( \Gamma_X, \omega_X)$ is monotone, then the existence  of nonempty open subset  $\mathcal{V}_{comp}(X, \pi_X, \omega_X)$ has been proved in Section \ref{section21}.  Take an almost complex structure $J \in \mathcal{V}_{comp}(X, \pi_X, \omega_X)^{reg}$,  by Lemma \ref{C27}, we can always perturb $(\omega_X, J)$, the result is denoted by  $(\omega'_X, J')$,  such   that  $(\omega'_X, J' ) \vert_{ \mathbb{R}_{\pm} \times Y_{\pm}}$ is $Q$-$\delta$ flat. In what follows, we show that the cobordism  maps are unchanged under the perturbation.
To simplify the notation, assume that $Y_- =\emptyset$. Given any $L>0$, our goal is to find a pair $(\omega_X', J')$  satisfying  the    conclusions  in  Lemma \ref{C27} and there exists a bijection  between   $\mathcal{M}^{J,L}_{X, I=0} (\alpha_+) $ and  $\mathcal{M}^{J', L}_{X, I=0} (\alpha_+)$. We follow the  strategy in Appendix of \cite{Te1} to construct such a  pair  $(\omega_X', J')$.
%The strategy to compare the moduli space $\mathcal{M}^{J,L}_{X, I=0} (\alpha_+)$ and  $\mathcal{M}^{J', L}_{X, I=0} (\alpha_+)$ is as follows: %we follow Appendix of \cite{Te2} to insert a sequence of $\{J_k\}_{k=0}^N$ between $J$ and $J'$, and show that
 For every periodic orbit  $\gamma$ with degree less than $Q$,  reintroduce  the coordinate $\varphi_{\gamma}: S_{\tau}^1 \times D_z \to Y_+$, $(\nu, \mu)$  in (\ref{e77}), and  $(\nu_{\gamma}, \mu_{\gamma})$ is defined in  Lemma \ref{C41}. We fix  a homotopy $\{(\nu_{\tau}, \mu_{\tau})\}_{\tau \in [0, 1]}$ starting   from  $(\nu, \mu)$ to $(\nu_{\gamma}, \mu_{\gamma})$ which is used to define the $Q$-$\delta$ approximation. Furthermore,    there is a uniform  lower bound   to  the absolute value of  the eigenvalue of  $(\nu_{\tau}, \mu_{\tau})$-version of  (\ref{e60}).

 Take a $N \gg 1$,  using the data $\{ (\nu_{\frac{k}{N}}, \mu_{\frac{k}{N}})\}_{k=0}^N$,  we  want to  construct a sequence of pair  $\{( \omega_{X,k}, J_k) \}_{k=0}^N$  inductively, where $\omega_{X, k}$ is an admissible $2$-form and  $J_k $ is a cobordism admissible almost complex structure with respective to $\omega_{X,k}$,
  such that
\begin{enumerate} [label=\textbf{F.\arabic*}]
\item \label{f1}
 $(\omega_{X, 0}, J_0)= (\omega_X, J)$  and $(\omega_{X, N}, J_N)=(\omega_X',J')$. Also,   $(\omega_X',J') \vert_{[0, \infty) \times Y_+}$ is $Q$-$\delta$ approximation of $(\omega_+, J_+)$.
\item \label{f2}
The difference $(\omega_{X, k+1} -\omega_{X, k}, J_{k+1} -J_k)$ is supported in  $[ -\frac{3}{2}\epsilon, \infty) \times U_{\gamma, \rho_{k+1}}$, where $U_{\gamma, \rho_{k+1}} $ is  image via the map  $\varphi_{\gamma}$  from  of the set in $S^1 \times D$  where  $ |z| \le \rho_{k+1} $.

%image of $ \{(\tau, z) \in S^1 \times D \vert  |z| \le \rho_{k+1} \}$.
\item \label{f3}
The half cylinder $[ -2\epsilon, \infty) \times \gamma$ is  $J_k$ holomorphic for any $k$ and any  periodic  orbit $\gamma$ with degree less than  $Q$.

\item  \label{f4}
We have estimates $|\omega_{X, k} -\omega_{X, k+1}| \le c_0 \rho_{k+1} N^{-1}$ , $|J_k -J_{k+1}| \le c_0 \rho_{k+1} N^{-1}$ and $| \nabla(J_k -J_{k+1})| \le c_0 N^{-1}$.

\item \label{f5}
There exists a bijective map between  $\mathcal{M}^{J_k,L}_{X, I=0} (\alpha_+)$ and $\mathcal{M}^{J_{k+1},L}_{X, I=0} (\alpha_+)$.
\item \label{f6}
For any $k$,  $\mathcal{M}^{J_{k},L}_{X, I<0} (\alpha_+) =\emptyset$ and  the curves  in  $\mathcal{M}^{J_{k},L}_{X, I=0} (\alpha_+)$ are embedded and  Fredholm regular.
%\item
%Let $|\lambda| $ be the smallest absolute eigenvalue of $(\mu_k, \nu_k)$ version of $L$, then $|\lambda|$ has a lower bounded which is independent of $k$, $\delta$ and $N$.
\end{enumerate}
%These sequence $\{J_k\}_{k=0}^N$ can be defined as follows.
The sequence   $0<\rho_N  <  \rho_{N-1} < \cdots< \rho_0 \ll \delta $ will be determined during the construction.   Appendix of \cite{Te1} (Also see \cite{LT}.) constructs such sequence of pairs  inductively for symplectization case. In other words,  given $(\omega_+, J_+)$,  we can  define a sequence $ \{(\omega_{+,k}, J_{+, k})\}_{k=0}^N $  with  properties \ref{f1},  \ref{f2}, \ref{f3}, \ref{f4}, \ref{f5}, and \ref{f6}.  For cobordism case,  we just need to interpolate this sequence  with $(\omega_X, J)$ over a collar neighborhood  of $\partial X$. Assume that $\{J_0, \cdots, J_k\}$ have been defined. We can define  $J_{k+1}$ as follows: Define $(\omega_{X, k+1}, J_{k+1}) =(\omega_X, J) $  over $X -[ \frac{3}{2}\epsilon, 0] \times Y_+ $ and $(\omega_{X, k+1}, J_{k+1}) =(\omega_{k+1, +}, J_{k+1, +})$ over $[0, \infty)\times Y_+$. Over the region $[-\frac{3}{2}\epsilon, 0] \times Y_+$, define $(\omega_{X, k+1}, J_{k+1})$ to be the interpolation between  $(\omega_{k+1, +}, J_{k+1, +})$ and $(\omega_{+}, J_{+})$, which is similar to the proof of Lemma \ref{C27}.   Obviously,  we can  arrange the interpolation such that $J_{k+1}$ satisfies the conditions  \ref{f1},  \ref{f2} and \ref{f3}.  Since $J_{+,k} -J_{+,k+1}$  satisfies the condition  \ref{f4}, we can further arrange that $J_{k+1} -J_k$ satisfies this condition.

%To prove \ref{f5} and \ref{f6}, we follows the argument in Appendix  of \cite{Te1}.
To prove \ref{f5} and \ref{f6}, firstly observe that we have the following facts.
Let $  \mathcal{C} \in \mathcal{M}^{J_k,L}_{X, I=0} (\alpha_+)$ and $C $ be an irreducible component of $\mathcal{C}$, then  %By  \ref{f1}, \ref{f2}, \ref{f3} and \ref{f4},  we have the following consequences.
\begin{enumerate}
\item
$J_{k+1}$ is  $\frac{1}{N}$-close to $J_{k}$ in $C^1$ sense. Also, the factor  $\frac{1}{N}$ is very small.
\item
If $C  \cap [ -\frac{3}{2}\epsilon, \infty)  \times Y_+$ equals to union of  $[ -\frac{3}{2}\epsilon, \infty) \times \gamma$, then $C$ is  a $J_{k+1}$     holomorphic by  \ref{f2} and \ref{f3}.  Otherwise,  $C$ is  $J_{k+1} $ holomorphic except for where it intersects the region  $[ -\frac{3}{2}\epsilon, \infty) \times U_{\gamma, \rho_{k+1}}$. The intersection of $C$ with  $[ -\frac{3}{2}\epsilon, \infty) \times U_{\gamma, \rho_{k+1}}$ can only occur in the following two ways.  (Cf. (A-14) of \cite{Te1})
\begin{enumerate}
\item
Intersection occurs in a disk of radius $R^{-1}$ in $C$ centered around each intersection points  $C \cap ( -\frac{3}{2}\epsilon, \infty) \times \gamma$. (By slightly increasing  $\epsilon$, we may assume that $C$ doesn't intersect  $\{-\frac{3}{2}\epsilon\} \times \gamma$.)  %Also, the number of intersections  of  $C $ over  $[ -\frac{3}{2}\epsilon, \infty) \times \gamma$ is independent of $C$ and $k$.
\item
The ends of $C $  which are  asymptotic to covered of $\gamma$.
\end{enumerate}

\item
The condition \ref{f4}  and  the last point  implies that $C$ is nearly $J_{k+1}$ holomorphic in $L^2$ sense.
\end{enumerate}

According to  the  above observations, we can follow the Appendix of \cite{Te1} to perturb $C$ to be $J_{k+1}$ holomorphic. The main points are sketched as follows.  Let $N_C$ be the normal bundle of $C$ and $\eta$ be a section of $N_C$. The deformation of $C$ along $\eta$ is $J_{k+1}$ holomorphic if and only  if $\eta$ satisfies the following equation:
\begin{equation} \label{e78}
{D}_{C }\eta + \mathfrak{p}_1 \eta + (\mathfrak{R}_1 + \mathfrak{p}_2 ) \nabla \eta +  \mathfrak{R}_0(\eta) + \mathfrak{p}_0=0,
\end{equation}
where $D_{C }$ is the deformation operator of $C$ with respect to $J_k$. Here $\{\mathfrak{p}_i\}_{i=0}^3$ and  $\{\mathfrak{R}_i\}_{i=0}^1$ are certain  operators.  They satisfy the same estimates (A-17) in \cite{Te1}.   Lemma A.2 of  \cite{Te1}   shows that   ${D}_{C }\eta + \mathfrak{p}_1 \eta +  + \mathfrak{p}_2  \nabla \eta$ is invertible as a map between suitable Banach spaces.  To see how this works.   Write $\eta= \eta_1 + \eta_2$, where $\eta_1$ has support far out on the ends of $C$ and $\eta_2$ is supported in the rest of $C$.   Since  we   have a relevant version of (A-14) \cite{Te1} and the same estimates on $\{\mathfrak{p}_i\}_{i=1}^2$, these can be used to make sure  that the contributions of  $\{\mathfrak{p}_i\}_{i=1}^2$ terms are small. Recall  that there is a uniform  lower bound   to the absolute value of  the eigenvalue of  $(\nu_{\tau}, \mu_{\tau})$-version of  (\ref{e60}). Over the ends of $C$,  this property   implies that   $||{D}_{C }\eta_1 + \mathfrak{p}_1 \eta_1   + \mathfrak{p}_2  \nabla \eta_1 ||_1 \ge c_0^{-1} ||\eta_1||_2$, where $c_0$ is independent of $k$ and $||\cdot||_{i}$, $ i=1,2$, are  suitable norms.   By induction, we have  lower bound  $||{D}_{C }\eta_2 ||_1 \ge c_k^{-1} ||\eta_2||_2$.  Therefore,    ${D}_{C }\eta + \mathfrak{p}_1 \eta   + \mathfrak{p}_2  \nabla \eta$ is invertible. Combine these with the relevant version  (A-17) in \cite{Te1} and   the  relevant estimates  on $\{\mathfrak{p}_i\}_{i=0}^3$, by contraction mapping principle, there exists a  unique solution to (\ref{e78}).  With the above understanding, we have an injective map
\begin{equation*}
 \mathcal{F}: \mathcal{M}^{J_k,L}_{X, I=0} (\alpha_+) \to \mathcal{M}^{J_{k+1}, L}_{X, I=0} (\alpha_+).
\end{equation*}
To see why $ \mathcal{F}$ is   onto  for  sufficiently small $\rho_{k+1}$,  by Taubes'Gromov compactness, as $\rho_{k+1} \to 0$, the holomorphic curves     $\mathcal{C}_{k+1} \in \mathcal{M}^{J_{k+1}, L}_{X, I=0} (\alpha_+)$   converge   to an unbroken curve in $ \mathcal{M}^{J_k,L}_{X, I=0} (\alpha_+)$. Since the curves in  $\mathcal{M}^{J_{k}, L}_{X, I=0} (\alpha_+)$ is regular, thus   $\mathcal{C}_{k+1}$ is in the image of $\mathcal{F}$ for sufficiently small $\rho_{k+1}$.  In sum, $J_{k+1}$ satisfies all the properties  \ref{f1}, \ref{f2}, \ref{f3}, \ref{f4}, \ref{f5}   and \ref{f6}. Follow Lemma A.4 in \cite{Te1} to compare the deformation operator $D_{\mathcal{F}(C)}$ and ${D}_{C }  + \mathfrak{p}_1   +   \mathfrak{p}_2  \nabla $, their difference is controlled by $\rho_{k+1}$.  Therefore,  $D_{\mathcal{F}(C)}$  is also invertible.

 For $(\omega_{X, N}, J_N)$, we can further perform a small perturbation such that  $J_N$ is generic,  Moreover,  the moduli space   $ \mathcal{M}^{J_{N}, L}_{X, I=0} (\alpha_+)$ is unchanged, because, the curve in  $ \mathcal{M}^{J_{N}, L}_{X, I=0} (\alpha_+)$ is Fredholm regular.

We summarize the   above discussion into the following lemma.
\begin{lemma} \label{C30}
Given $L>0$, there exists a $Q$-$\delta$ flat  pair $(\omega_X', J')$ satisfying the conclusions in  Lemma \ref{C27}. Moreover, there is a bijection between $ \mathcal{M}^{J,L}_{X, I=0} (\alpha_+, \alpha_-)$  and $ \mathcal{M}^{J', L}_{X, I=0} (\alpha_+, \alpha_-)$.
\end{lemma}

%%%%%%%%%%%%%%%%%%%%%%%%%%%%%%%%%%%%%%%%%%%%%%%%%%%%%%%%%%%%%%%%%%%%%%%%%%%%%%%%%%%%%%%%%%
\begin{proof} [Proof of Main Theorem \ref{Thm2}]
 Let $\omega_X$ be a $Q$-admissible $2$-form.  Let $J$ be a generic almost complex structure in $\mathcal{V}_{comp}(X, \pi_X, \omega_X)^{reg}$.  Taking an unbounded increasing sequence $\{L_n\}_{n=1}^{\infty}$ and take $\delta=\frac{1}{n}$ in Lemma \ref{C30}, we can find a sequence of pairs $\{(\omega_{Xn}, J_n)\}_{n=1}^{\infty}$ such that $(\omega_{Xn}, J_n) \vert_{\mathbb{R}_{\pm} \times Y_{\pm}}$ is $Q$-$\frac{1}{n}$ flat approximation of $(\omega_X , J )\vert_{\mathbb{R}_{\pm} \times Y_{\pm}}$. By Lemma \ref{C30}, we have
\begin{equation} \label{e22}
CP(X, \Omega_X, J, \Lambda_X)=CP(X, \Omega_{Xn}, J_n, \Lambda_X)+o(L_n),
\end{equation}
where $\lim\limits_{n \to \infty}o(L_n)=0$.

%Taking a sequence $\{r_n\}_{n=1}^{\infty}$ such that $r_n > r_{L_n}$, according to Theorem \ref{C31} in next section, we have
%\begin{equation} \label{e23}
%CP(X, \omega_{Xn}, J_n, \Lambda_X)= ({T}_{r_n}^-)^{-1} \circ CM(X, \omega_{Xn}, J_n, r_n, \Lambda_X) \circ {T}_{r_n}^++o(L_n).
%\end{equation}
For each $n$,  according to Theorem \ref{A10}, there exists a constant $r_{L_n}>0$ such that
\begin{equation} \label{e23}
CP(X, \Omega_{Xn}, J_n, \Lambda_X)= ({T}_{r}^-)^{-1} \circ CM(X, \Omega_{Xn}, J_n, r, \Lambda_X) \circ {T}_{r}^++o(L_n).
\end{equation}
for any $r > r_{L_n}$.  Let  $r \to \infty$, by (\ref{e29}), we have
\begin{equation} \label{e24}
\begin{split}
CP(X, \Omega_{Xn}, J_n, \Lambda_X)= CP_{sw}(X, \Omega_{Xn}, J_n, \Lambda_X) + o(L_n).
\end{split}
\end{equation}
%Fix a large $n_0$, let $\rho^n_{\pm}$ be  a homotopy from $(\omega_{\pm, n_0}, J_{\pm, n_0})$ to $(\omega_{\pm, n}, J_{\pm, n})$ provided by Lemma \ref{C34}.  By Lemmas  \ref{C18} and \ref{C35}, up to chain homotopy, we have
%\begin{equation} \label{e24}
%\begin{split}
% &  \ \ \ \ \ \ \ ({T}_{r_n}^-)^{-1} \circ CM(X, \omega_{Xn}, J_n, r_n, \Lambda_X) \circ {T}_{r_n}^+\\
%& =({T}_{r_n}^-)^{-1} \circ CM(\rho^n_-, r_n)^{-1} \circ CM(X, \omega_{Xn_0}, J_{n_0}, r_n, \Lambda_X) \circ CM(\rho^n_+, r_n) \circ {T}_{r_n}^+.
%\end{split}
%\end{equation}
Combine above equalities (\ref{e22}) and (\ref{e24}), (\ref{e79})  and  by taking $n \to \infty$,  then  we obtain
$$CP(X, \Omega_{X}, J, \Lambda_X)= {CP}_{sw}(X, \Omega_X, J, \Lambda_X).  $$
\end{proof}
%%%%%%%%%%%%%%%%%%%%%%%%%%%%%%%%%%%%%%%%%%%%%%%%%%%%%%%%%%%%%%%%%%%%%%%%%%%%%%%%%%%%%%%%%%%%%%%%%%%%%%%%%%%%%

%%%%%%%%%%%%%%%%%%%%%%%%%%%%%%%%%%%%%%%%%%%%%%%%%%%%%%%%%%%%%%%%%%%%%

(Guanheng Chen) Department of Mathematics, The Chinese University of Hong
Kong.

\verb| E-mail adress: ghchen@math.cuhk.eduh.hk|


\begin{thebibliography}{unsrt}
\addcontentsline{toc}{chapter}{Bibliography}


\bibitem{FYHKE}
F. Bourgeois, Y. Eliashberg, H. Hofer, K. Wysocki, and E. Zehnder,
{\it Compactness results in symplectic field theory},
Geom. Topol. 7 (2003), 799-888.


\bibitem{21}
D. Cristofaro-Gardiner,
{\it The absolute gradings in embedded contact homology
and Seiberg-Witten Floer cohomology}, Alg. Geom. Topol. 13 (2013), 2239-2260.


\bibitem{VPK}
V. Colin, P. Ghiggini, and K. Honda,
{\it The equivalence of Heegaard Floer homology
and embedded contact homology via open book decompositions I}, arXiv:1208.1074.







\bibitem{DS}
S. Donaldson and I. Smith,
{\it Lefschetz pencils and the canonical class for symplectic fourmanifolds},
Topology 42 (2003) 743-785.


\bibitem{CG}
C. Gerig,
{\it
Taming the pseudoholomorphic beasts in $\mathbb{R} \times S^1 \times S^2$},  arXiv:1711.02069.

\bibitem{CG2}
C.  Gerig,
{\it
 Seiberg-Witten and Gromov invariants for self-dual harmonic 2-forms},  arXiv:1809.03405.






\bibitem{GT}
D. Gilbarg and N. S. Trudinger,
{\it Elliptic Partial Differential Equations of Second Order},
Grundlehren, Vol. 224, Springer-Verlag, Berlin, 1983.



\bibitem{H1}
M. Hutchings,
{\it An index inequality for embedded pseudoholomorphic curves in symplectizations},
J. Eur. Math. Soc. 4 (2002) 313-361.

\bibitem{H2}
M. Hutchings,
{\it The embedded contact homology index revisited},
New perspectives
and challenges in symplectic field theory, 263-297, CRM Proc. Lecture
Notes 49, Amer. Math. Soc., 2009.




\bibitem{H3}
M. Hutchings,
{\it Lecture notes on embedded contact homology}, Contact and Symplectic Topology, Bolyai Society Mathematical Studies, vol. 26, Springer, 2014, pp. 389-484.


\bibitem{H4}
M. Hutchings,
{\it Beyond ECH capacities},
Geom, Topol. 20 (2016), 1085-1126.

\bibitem{HT}
M. Hutchings and C. H. Taubes,
{\it Proof of the Arnold chord conjecture in three dimensions II
},
Geom, Topol. 17 (2003), 2601-2688.




\bibitem{HT1}
M. Hutchings and C. H. Taubes,
{\it Gluing pseudoholomorphic curves along
branched covered cylinders I, } J. Symplectic Geom. 5 (2007), 43-137.


\bibitem{HT2}
M. Hutchings and C. H. Taubes,
{\it Gluing pseudoholomorphic curves along
branched covered cylinders II},  J. Symplectic Geom. 7 (2009), 29-133.



\bibitem{HT3}
M. Hutchings and C. H. Taubes,
{\it  The Weinstein conjecture for stable
Hamiltonian structures}, Geometry and Topology 13 (2009), 901-941.

\bibitem{Kim}
Inyoung Kim,
{\it Almost Kahler Anti-Self-Dual Metrics. }
Ph.D. thesis, State University of New York, Stony Brook
(2014).

\bibitem{KM}
P. Kronheimer, T. Mrowka,
{\it Monopoles and three-manifolds}, New Math. Monogr. 10,
Cambridge Univ. Press (2007) MR2388043.



%\bibitem{Lee}
%Y-J. Lee,
%{\it From Seiberg-Witten to Gromov: MCE andd Singular Symplectic Form},
%preprint 2015. (first draft 2006).

\bibitem{LT}
Y-J. Lee and C.H.Taubes,
{\it Periodic Floer homology and Seiberg-Witten-Floer cohomology},
J.Symplectic Geom. 10 (2012), no. 1, 81-164.

\bibitem{LL}
T-J. Li and A-K. Liu,
{\it The equivalence between SW and Gr in the case where $b^+ = 1$},
Int. Math. Res. Not. 7 (1999), 335-345.

\bibitem{JR}
J. Rooney,
{\it Cobordism maps in  embedded contact homology}, arXiv:1912.01048.

\bibitem{MS}
D. McDuff and D. Salamon,
{\it J-holomorphic curves and symplectic topology},
AMS, 2004.

\bibitem{REG}
 R. E. Gompf and A.I.Stipsicz,
{\it 4-Manifolds  and  Kirby Calculus, }
Graduate Studies in  Mathematics 20, A.M.S., Providence, RI, 1999.


\bibitem{T1}
C. H. Taubes,
{\it SW $\Rightarrow$ Gr:  From the Seiberg Witten equations to pseudo-holomorphic
curves},
J. Amer. Math. Soc. 9 (1996) 845-918.



\bibitem{T2}
C. H. Taubes,
{\it Counting pseudo-holomorphic curves in dimension 4},
J. Differential Geom. 44 (1996), 818-893.


\bibitem{T3}
C. H. Taubes,
{\it Gr $\Rightarrow$ SW: From pseudo-holomorphic curves to Seiberg-Witten solutions},
J. Differential Geom. 51 (1999) 203-334 MR1728301


\bibitem{T4}
C. H. Taubes,
{\it Gr $=$ SW: Counting curves and connections},
J. Differential Geom. 52 (1999) 453-609.


\bibitem{T}
C. H. Taubes,
{\it
  The Seiberg-Witten equations and the Weinstein conjecture. II. More
closed integral curves of the Reeb vector field}, Geom. Topol. 13 (2009) 1337-1417.
%MR2496048



\bibitem{Te1}
C. H. Taubes,
{\it Enbedded contact homology and Seiberg-Witten Floer homology I},
Geometry and Topology 14 (2010), 2497-2581.

\bibitem{Te2}
C. H. Taubes,
{\it Enbedded contact homology and Seiberg-Witten Floer homology II},
Geometry and Topology 14 (2010), 2583-2720.

\bibitem{Te3}
C. H. Taubes,
{\it Enbedded contact homology and Seiberg-Witten Floer homology III},
Geometry and Topology 14 (2010), 2721-2817.

\bibitem{Te4}
C. H. Taubes,
{\it Enbedded contact homology and Seiberg-Witten Floer homology IV},
Geometry and Topology 14 (2010), 2819-2960.

\bibitem{Te5}
C. H. Taubes,
{\it Enbedded contact homology and Seiberg-Witten Floer homology V},
Geometry and Topology 14 (2010), 2961-3000.






%
%\bibitem{bib9}
%M.Hutchings,
%{\it https://floerhomology.wordpress.com/2013/09/04/a-magic-trick-for-defining-obstruction-bundles/}


















\bibitem{U}
M. Usher,
{\it Vortices and a TQFT for Lefschetz fibrations on 4-manifolds}, Algebr. Geom. Topol. 6 (2006) 1677-1743.









%\bibitem{GW}
%Chris Gerig, Chris Wendl
%{\it Generic transversality for unbranched covers of closed pseudoholomorphic curves}.


%
%\bibitem{BH}
%E. Bao, K. Honda, D
%{\it Definition of Cylindrical Contact Homology in dimension three}.


%\bibitem{JN}
%
%{\it Automatic transversality in contact homology I: Regularity}



%\bibitem{Wen}
%C. Wendl,
%{\it , Lectures on holomorphic curves in symplectic
%and contact geometry}.


\bibitem{Wen2}
C. Wendl,
{\it  Lectures on Symplectic Field Theory}, arXiv:1612.01009.








\end{thebibliography}
\end{document}